\RequirePackage{fix-cm}
\documentclass[smallextended]{svjour3}       
\smartqed  
\usepackage{graphicx}
\usepackage{amsmath,amssymb,bbm,url}
\usepackage{mathtools}
\usepackage{pgf, tikz}
\usetikzlibrary{arrows, automata}

\usepackage[letterpaper, margin=1.2in]{geometry}

\spnewtheorem*{conj}{Conjecture}{\bf}{\it}

\newtheorem{lem}{Lemma}
\newtheorem{assumption}{Assumption}
\usepackage{xcolor}
\usepackage[hidelinks]{hyperref}       
\usepackage{url}            
\usepackage{booktabs}       
\usepackage{amsfonts}       
\usepackage{nicefrac}       
\usepackage{microtype}      
\usepackage{cancel}
\usepackage{subfiles}



\global\long\def\rl{{\mathbb{R}}}%

\global\long\def\tr{\mathop{{\bf tr}}}%

\global\long\def\prox{\mathbf{prox}}%

\global\long\def\argmin{\mathop{{\rm argmin}}}%

\global\long\def\rank{\mathop{{\bf rank}}}%

\global\long\def\FZerInf{\mathcal{F}_{0,\infty}}%

\global\long\def\FZerL{\mathcal{F}_{0,L}}%

\let\emph\textit

\begin{document}

\title{Computer-Assisted Design of 
\\Accelerated Composite Optimization Methods:
OptISTA
}

\titlerunning{
OptISTA}        

\author{Uijeong Jang         \and
        Shuvomoy Das Gupta \and Ernest K. Ryu 
}

\authorrunning{U. Jang, S. Das Gupta, E. K. Ryu} 

\institute{Uijeong Jang \at
              University of California, Los Angeles\\
              \email{uijeongjang@math.ucla.edu}           
           \and
Shuvomoy Das Gupta \at
             Rice University, Houston, Texas \\
              \email{sd158@rice.edu}           
           \and
Ernest K. Ryu \at
              University of California, Los Angeles\\
              \email{eryu@math.ucla.edu}           
}

\date{Received: date / Accepted: date}

\makeatletter

\renewcommand{\paragraph}{%
  \@startsection{paragraph}{4} {\z@}
  {1ex \@plus 1ex \@minus .2ex}
  {-1ex}%
  {\bfseries}%
}                                  
 
\makeatother

\maketitle

\begin{abstract}
The accelerated composite optimization method FISTA (Beck, Teboulle 2009) is suboptimal by a constant factor, and we present a new method OptISTA that improves FISTA by a constant factor of 2. The performance estimation problem (PEP) has recently been introduced as a new computer-assisted paradigm for designing optimal first-order methods. In this work, we present a double-function stepsize-optimization PEP methodology that poses the optimization over fixed-step first-order methods for composite optimization as a finite-dimensional nonconvex QCQP, which can be practically solved through spatial branch-and-bound algorithms, and use it to design the exact optimal method OptISTA for the composite optimization setup. We then establish the exact optimality of OptISTA {\color{black}{under the large-scale assumption}} with a lower-bound construction that extends the semi-interpolated zero-chain construction (Drori, Taylor 2022) to the double-function setup of composite optimization. By establishing exact optimality, our work concludes the search for the fastest first-order methods, with respect to the performance measure of worst-case function value suboptimality, for the proximal, projected-gradient, and proximal-gradient setups involving a smooth convex function and a closed proper convex function.
{\color{red}(After the publication of this work, we identified an error in the proof of Theorem 3 and several other miscellaneous typos. Although none of these affect the final results, we have corrected them and highlighted the changes in red.)}
\end{abstract}

\section{Introduction \label{sec:Introduction}}
Since the seminal work by Nesterov on accelerated gradient methods \cite{nesterov83MomentumPaper} and by Nemirovski on matching complexity lower bounds \cite{nemirovski1991optimality,nemirovsky1992information,nemirovski1995information}, accelerated first-order methods have been central to the theory and practice of large-scale convex optimization. Recently, the performance estimation problem (PEP) \cite{drori2014performance,taylor2017InterpolationFMuLPEP} has been introduced as a new computer-assisted paradigm for designing optimal first-order methods and was used to discover OGM, which surprisingly achieves a factor-$2$ speedup over Nesterov's method \cite{drori2014performance,kim2016OGM}, and several other new accelerated methods such as OGM-G \cite{kim2021OGMG}, APPM \cite{kim2021AcceleratedPPM}, and EAG \cite{yoon2021accelerated}.

However, the PEP methodology of searching over first-order methods to find globally optimal ones has so far been limited to the unconstrained optimization setup with a single function or monotone operator; it has not been extended to the composite optimization setup, where the objective is a sum of two functions, or in the constrained optimization setup, where methods use projections onto the constraint. In particular, whether it is possible to improve FISTA \cite{beck2009FISTA} or projected-gradient-type methods and achieve a speedup similar to OGM was an open problem.

In this work, we present a PEP methodology that poses the optimization over fixed-step first-order methods for composite optimization as a finite-dimensional nonconvex QCQP, which can be practically solved through spatial branch-and-bound algorithms \cite{Gurobi,locatelli2013global,horst2013global,das2024branch}. Using this methodology, we obtain \ref{eq:OptISTA_Alg}, a new composite optimization method that improves the prior state-of-the-art rates by a constant factor, including that of FISTA \cite{beck2009FISTA}.

At the same time, improving the prior lower bound constructions of Nesterov and Nemirovski \cite{nesterov2003Introductory,nemirovsky1992information,nemirovski1983problem} has been an active area of research that parallels the PEP line of work. New lower bounds have certified the exact optimality of methods such as a variant of Kelley's cutting plane method \cite{drori2016Kelley}, OGM \cite{drori2014performance,kim2016OGM}, and ITEM \cite{taylor2021optimal}. However, these prior lower bounds do not exactly match the upper bound of \ref{eq:OptISTA_Alg}. In this work, we provide a lower-bound construction that establishes the exact optimality of \ref{eq:OptISTA_Alg} for both the proximal-gradient and projected-gradient setups. By exact optimality, we mean the upper and lower bounds are exactly equal (rather than matching only up to a constant). The key insight is to extend the semi-interpolated zero-chain construction to the double-function setup of composite optimization.

\paragraph{Preliminaries and notation.}

Write $\rl^{d}$ for the underlying Euclidean space.  Write $\left\langle \cdot,\,\cdot\right\rangle $
and $\|\cdot\|$ to denote the standard inner product and norm on
$\rl^{d}$. For $a,b\in\mathbb{Z}$, denote \[
[a:b]=\{a,a+1,a+2,\ldots,b-1,b\}\subset \mathbb{Z}.
\]
 We follow
standard convex-analytical definitions \cite{boyd2004Convex,nesterov2003Introductory}.
A set $S\subseteq\rl^{d}$ is convex if 
\[
\theta x+(1-\theta)y\in S \, \mathrm{for \ any} \, x,y\in S \, \mathrm{and} \, \theta\in[0,1].
\] 
A function $f\colon\rl^{d}\to\rl\cup\{\infty\}$
is convex if 
\[f\left(\theta x+(1-\theta)y\right)\leq\theta f(x)+(1-\theta)f(y)
\, \mathrm{for \ all}\,  x,y\in\rl^{d} \, \mathrm{and} \,\theta\in(0,1).\]
The subdifferential
of $f\colon\rl^{d}\to\rl$ at $x$, denoted by $\partial f(x)$, is
\[\partial f(x)=\{g\in\rl^{d}\,|\,f(y)\geq f(x)+\langle g, y-x\rangle,\,\forall\,y\in\rl^{d}\}.\]
An optimization method is usually designed for a specific class of
functions. The class of $L$-smooth convex functions is denoted by
$\FZerL$, and any function $f$ in this class is defined by 
\[
f(y)\geq f(x)+\left\langle \nabla f(x),\,y-x\right\rangle +(1/2L)\|\nabla f(x)-\nabla f(y)\|^{2}
\]
for all $x,y\in\rl^{d}$ \cite[Theorem 2.1.5, (2.1.10)]{nesterov2003Introductory}.
The class of closed, convex, and proper functions that are potentially nonsmooth is denoted
by $\FZerInf$, and any function $h:\rl^{d}\to\rl\cup\{\infty\}$
in this class is defined by 
\[
h(y)\geq h(x)+\left\langle u,\,y-x\right\rangle 
\]
for all $x,y\in\rl^{d}$ and for all $u\in\partial h(x)$ \cite[Definition 3.1]{taylor2017InterpolationFMuLPEP}. For a closed, convex, and proper function $h$ and $\gamma>0$, we define the proximal operator $\mathbf{prox}_{\gamma h}\colon \rl^d\rightarrow\rl^d$ as
\[
\mathbf{prox}_{\gamma h}(x)=\argmin_{z\in\mathbb{R}^d}\Big\{ h(z)+\frac{1}{2\gamma}\|z-x\|^2\Big\}.
\]

\subsection{Prior work}
\label{ss:prior-work}

\paragraph{FISTA and its variants.} 
To solve composite convex optimization problems, one of the first methods is
the proximal gradient method that can be traced back to the works
by \cite{bruck1977weak,passty1979ergodic,LionsMercier1979_splitting} and
has a convergence rate of $\mathcal{O}(1/N)$
in function value suboptimality. When the nonsmooth function
in the composite setup is the $\ell^1$-norm, the proximal
gradient method is called iterative shrinkage thresholding algorithm
(ISTA) \cite{daubechies2004iterative,hale2008fixed,wright2009sparse,elad2006simple}.
The celebrated FISTA method due to Beck and Teboulle \cite{beck2009FISTA}
accelerates the convergence rate to $\mathcal{O}(1/N^{2})$
in function values. MFISTA is a variant of FISTA that maintains a
monotonically nonincreasing sequence of function values and preserves
the same convergence rate of FISTA \cite{beck2009MFISTA}. A comprehensive
review of FISTA and FISTA-based methods can be found in the monograph
\cite[Chapter 10]{beck2017first}. In \cite{kim2018another}, Kim and Fessler propose two FISTA-variants called GFPGM and FPGM-OCG, both having the same convergence rate  $\mathcal{O}(1/N^2)$ but a worse constant compared to FISTA.

Finally, in \cite[Section 4.2]{taylor2017CompositePEP},
a new method called FPGM2 is proposed, which has a worse convergence rate than
\ref{eq:OptISTA_Alg} approximately by a factor of 2. 

\paragraph{Proximal point method and its variants.}
The proximal point method, which traces its origins back to the 1970s, has been the subject of extensive research in the field of optimization \cite{martinet1970regularisation,martinet1972determination,Rockafellar1976_monotone,brezis1978produits}. G{\"u}ler's 1991 work \cite{guler1991convergence} studies the convergence rate for the proximal point method in terms of reducing the function values of convex functions. An accelerated proximal point method for maximally monotone operators through the lens of monotone operator theory was proposed independently in \cite{lieder2021convergence,kim2021AcceleratedPPM}.

In Section~\ref{ss:oppa-lb}, we describe a proximal method \ref{alg:OPPA}, parameterized by $\gamma_0,\gamma_1,\dots,\gamma_{N-1}$ and establish its exact optimality by providing an exact matching lower bound {\color{black}{under the large-scale assumption}}.  In \cite{guler1992new}, G{\"u}ler proposed the so-called G{\"u}ler's second method and provided a bound that improves upon G{\"u}ler's prior work \cite{guler1991convergence}. G{\"u}ler's second method turns out to be an instance of \ref{alg:OPPA} with $\gamma_0=\gamma_1=\dots=\gamma_{N-1}$. Monteiro and Svaiter proposed Accelerated Hybrid Proximal Extragradient (A-HPE) \cite{Monteiro2013hybrid} as an inexact accelerated proximal method that generalizes G{\"u}ler's second method. A-HPE with its parameter $\sigma=0$ turns out to include \ref{alg:OPPA} when the proximal operators are evaluated exactly. Barr\'e, Taylor, and Bach \cite{Barre2023principle} proposed Optimized Relatively Inexact Proximal Point Algorithm (ORI-PPA) also as an inexact accelerated proximal method generalizing G{\"u}ler's second method. ORI-PPA with its parameter $\sigma=0$ also turns out to be equivalent to \ref{alg:OPPA} when proximal operators are evaluated exactly. Although A-HPE and ORI-PPA reduce to the same method \ref{alg:OPPA} under the aforementioned conditions, their analyses are slightly different. In Section~\ref{ss:oppa-lb}, we reference the analysis of \cite{Barre2023principle} as it is tighter than the analysis of \cite{Monteiro2013hybrid} by a factor of $2$ in the case of exact proximal evaluations.

\paragraph{Lower bounds.}
Leveraging the information-based complexity framework \cite{nemirovski1991optimality,nemirovsky1992information}, iteration complexity lower bounds have been thoroughly explored for single-function first-order convex optimization methods \cite{nesterov2003Introductory,nemirovsky1992information,drori2017LowerBoundOGM,carmon2020stationary1,drori2020complexity,carmon2021stationary2,dragomir2021optimal,drori2022LowerBoundITEM}. In this study, we extend the semi-interpolated zero-chain construction introduced by \cite{drori2022LowerBoundITEM} to the double-function scenario of composite optimization.

\paragraph{Performance estimation problem: A foundation for optimal methods.}

Our primary methodology for constructing optimal methods in this paper is rooted in the well-established performance estimation problem (PEP) framework, a computer-assisted approach for analyzing and designing optimization methods, with a particular focus on first-order methods. Since the pioneering work of \cite{drori2014performance}, PEP has been extensively employed to discover and analyze numerous first-order and operator splitting methods \cite{taylor2017InterpolationFMuLPEP,taylor2017CompositePEP,taylor2017Thesis,Barre2023principle,barre2020complexity,de2017worst,de2020worst,dragomir2021optimal,ryu2020OperatorSplittingPEP,kim2016OGM,kim2017convergence,kim2018another,kim2021AcceleratedPPM,kim2021OGMG,das2024branch,das2024nonlinear}. The majority of these works employ the SDP-based methodology introduced by \cite{drori2014performance,taylor2017InterpolationFMuLPEP,taylor2017CompositePEP}, demonstrating that computing the worst-case performance measure for a known fixed-step first-order method is equivalent to solving a convex SDP. The work \cite{das2024branch} developed a nonconvex QCQP framework to determine optimal fixed-step first-order methods for a single-function setup.


\subsection{Contributions}
This paper presents two major contributions: one concrete and one conceptual. The first contribution is the exact optimal accelerated composite optimization method \ref{eq:OptISTA_Alg}, the combination of the upper and lower bound results. The second major contribution is the methodologies: the double-function stepsize-optimization PEP, used to discover \ref{eq:OptISTA_Alg}, and the double-function semi-interpolated zero-chain construction, used to establish the exact matching lower bound. In a sense, the upper and lower bounds of \ref{eq:OptISTA_Alg} are merely one application demonstrating the strength of the presented methodology.

Finally, as an auxiliary contribution, we consider the proximal minimization setup, where the proximal oracle but not the gradient oracle is used.  We show that the prior method \ref{alg:OPPA} \cite{Monteiro2013hybrid,Barre2023principle}, which we review in Section~\ref{ss:oppa-lb}, is exactly optimal by adapting the semi-interpolated zero-chain construction to proximal setup and providing an exact matching lower bound.

\section{New exact optimal method}\label{s:2}
We now present the main result: \ref{eq:OptISTA_Alg}. In this section, we first state the method and describe its convergence result. Later in Section~\ref{s:3}, we will present the methodology used to discover the method.
Later in Section~\ref{s:4}, we will provide complexity lower bounds establishing exact optimality of OptISTA and \ref{alg:OPPA}.

\subsection{OptISTA: Exact optimal composite optimization method}

Consider the composite minimization problem
\begin{equation}
\begin{array}{ll}
\underset{x\in\mathbb{R}^{d}}{\mbox{minimize}} & f(x)+h(x),\end{array}  \tag{\ensuremath{\mathcal{P}}}\label{eq:main_problem}
\end{equation}
where $f\colon \rl^{d}\to\rl$ is $L$-smooth convex and $h\colon\rl^{d}\to\rl\cup\{\infty\}$ is closed, convex, and proper (possibly nonsmooth). Assume  \eqref{eq:main_problem} has a minimizer $x_{\star}$ (not necessarily unique).

Let $N>0$ be the total iteration count and $L>0$. We propose
Optimal Iterative Shrinkage Thresholding Algorithm \begin{equation}
\begin{aligned} & y_{i+1}=\prox_{\frac{\gamma_{i}}{L}h}\Big(y_{i}-\frac{\gamma_{i}}{L}\nabla f(x_{i})\Big)\\
 & \ensuremath{z_{i+1}=x_{i}+\frac{1}{\gamma_{i}}\left(y_{i+1}-y_{i}\right)}\\
 & x_{i+1}=z_{i+1}+\frac{\theta_{i}-1}{\theta_{i+1}}\left(z_{i+1}-z_{i}\right)+\frac{\theta_{i}}{\theta_{i+1}}\left(z_{i+1}-x_{i}\right)
\end{aligned}
\tag{\textup{OptISTA}}\label{eq:OptISTA_Alg}
\end{equation}
for $i=0,\dots,N-1$, where $z_{0}=y_{0}=x_{0}\in \rl^d$ is a starting point and
\[
 \gamma_{i}=\frac{2\theta_{i}}{\theta_{N}^{2}}\left(\theta_{N}^{2}-2\theta_{i}^{2}+\theta_{i}\right)\,\,\,\textup{for }i=0,\dots,N-1,
 \quad
\begin{aligned} & \theta_{i}=\begin{cases}
1 & \textup{if }i=0,\\
\frac{1+\sqrt{1+4\theta_{i-1}^{2}}}{2} & \textup{if }1\leq i\leq N-1,\\
\frac{1+\sqrt{1+8\theta_{N-1}^{2}}}{2} & \textup{if }i=N.
\end{cases}
\end{aligned}
\]

Since the $\gamma$-coefficients, the proximal stepsizes, depend on $\theta_N$, the total iteration count $N$ must be chosen prior to the start of the method. To practically implement \ref{eq:OptISTA_Alg}, one should pre-compute $\theta_N$ in one for-loop and then perform the main iteration in a second for-loop. (The for-loops are not nested.)

\begin{theorem}
\label{thm:OptISTA-rate}
Let $f\colon \rl^{d}\to\rl$ be $L$-smooth convex and $h\colon\rl^{d}\to\rl\cup\{\infty\}$ be closed, convex, and proper. Assume $x_\star\in \argmin(f+h)$ exists.
Let $N>0$.
Then, \ref{eq:OptISTA_Alg} exhibits the rate
\[
f(y_{N})+h(y_{N})-f(x_\star)-h(x_\star)\leq\frac{L\|x_{0}-x_{\star}\|^{2}}{2(\theta_{N}^{2}-1)}
\le
\frac{L\|x_{0}-x_{\star}\|^{2}}{(N+1)^2},
\]
where $\theta_{N}$ is as defined for \ref{eq:OptISTA_Alg}.
\end{theorem}

A noteworthy prior work that we compare Theorem~\ref{thm:OptISTA-rate} to is OGM \cite{drori2014performance,kim2016OGM}, which achieves a factor-2 speedup over the classical Nesterov's accelerated gradient method. Let $N>0$ be the total iteration count and $L>0$. OGM is given by
    \begin{equation}
    \begin{aligned}
        &y_{i+1}=x_i-\frac{1}{L}\nabla f(x_i) \\
        &x_{i+1}=y_{i+1}+\frac{\theta_i-1}{\theta_{i+1}}\left(y_{i+1}-y_i\right)+\frac{\theta_i}{\theta_{i+1}}\left(y_{i+1}-x_i\right)
            \end{aligned}
        \tag{OGM}\label{eq:OGM}
    \end{equation}
    with starting point $x_0=y_0$, $f$ is $L$-smooth and convex, and with the sequence $\theta_i$ satisfying 
    \[\theta_{i}=\begin{cases}
1 & \textup{if }i=0,\\
\frac{1+\sqrt{1+4\theta_{i-1}^{2}}}{2} & \textup{if }1\leq i\leq N-1,\\
\frac{1+\sqrt{1+8\theta_{N-1}^{2}}}{2} & \textup{if }i=N.
\end{cases}
\]
Its convergence rate is
\[
f(x_N)-f(x_\star)\leq\frac{L\|x_0-x_\star\|^2}{2\theta_{N}^2}\leq\frac{L\|x_0-x_\star\|^2}{(N+1)^2}.
\]
The $\theta$-coefficients of \ref{eq:OptISTA_Alg} are identical to the $\theta$-coefficients of OGM \cite{drori2014performance,kim2016OGM}, which are in turn equal to the standard $\theta$-coefficients of Nesterov's FGM \cite{nesterov83MomentumPaper} for $i=0,1,\dots,N-1$, but the last coefficient $\theta_N$ differs. (Roughly, $\theta_i\sim i/2$ for $i<N$ and $\theta_N\sim N/\sqrt{2}$ \cite[Theorem 7]{park2023factor}.)

Interestingly, \ref{eq:OptISTA_Alg} reduces to OGM when $h=0$. To see this, note that $\prox_{\alpha h\equiv 0}(\cdot)$ is an identity operator for any scalar $\alpha>0$. Hence, the iterates of \ref{eq:OptISTA_Alg} reduces to 
\begin{align*} & y_{i+1}=y_{i}-\frac{\gamma_{i}}{L}\nabla f(x_{i})\\
 & \ensuremath{z_{i+1}=x_{i}+\frac{1}{\gamma_{i}}\left(y_{i+1}-y_{i}\right)}=x_i-\frac{1}{L}\nabla f(x_i)\\
 & x_{i+1}=z_{i+1}+\frac{\theta_{i}-1}{\theta_{i+1}}\left(z_{i+1}-z_{i}\right)+\frac{\theta_{i}}{\theta_{i+1}}\left(z_{i+1}-x_{i}\right),
\end{align*}
where $z$-iterates are now identical to the $y$-iterates of OGM.

The new rate of Theorem~\ref{thm:OptISTA-rate} improves the prior rates of FISTA $L\|x_{0}-x_{\star}\|^{2}/(2\theta_{N-1}^2)\sim 2L\|x_{0}-x_{\star}\|^{2}/N^2$ \cite{beck2009FISTA} 
and FPGM2  $2L\|x_{0}-x_{\star}\|^{2}/(N^2+7N)$ \cite{taylor2017Thesis} by a factor of $2$ in the leading $\mathcal{O}(1/N^2)$-term. We note that a similar factor-two improvement appears in OGM as well. However, even for the authors, it is difficult to pinpoint why this improvement specifically amounts to a factor of two.

Furthermore, the rate of Theorem~\ref{thm:OptISTA-rate} {\color{black}{under the large-scale assumption}} exactly matches the lower bound we later present in Theorem~\ref{thm:OptISTA-lb} of Section~\ref{s:4} and, therefore, provably cannot be improved without further assumptions. We say \ref{eq:OptISTA_Alg} is exactly optimal in this sense.

The rate of Nesterov's FGM \cite{nesterov83MomentumPaper} is identical to that of FISTA. This rate was improved by a factor of 2 with OGM, which has the rate $L\|x_{0}-x_{\star}\|^{2}/(2\theta_{N}^2)\sim L\|x_{0}-x_{\star}\|^{2}/N^2$ \cite{drori2014performance,kim2016OGM}. The rate of OGM is slightly better than that of \ref{eq:OptISTA_Alg} by a factor of $\theta_N^2/(\theta_{N}^{2}-1)$, but this difference does not manifest in the leading $\mathcal{O}(1/N^2)$-term.

Also, it is worth noting that there is a difference of $2$ in the denominators of the worst-case rate of OGM ($2\theta_N^2$) and OptISTA ($2\theta_N^2-2$) is reminiscent of some previous results. In prior work, the denominators of the convergence rates of gradient descent on a single function ($4N+2$) \cite{drori2014performance} and that of projected/proximal gradient on composite function ($4N$) \cite{teboulle2023elementary} also had a difference of $2$.

\subsection{OptISTA proof outline}

The proof of Theorem~\ref{thm:OptISTA-rate} utilizes the Lyapunov analysis, which is commonly used in analyzing the convergence of first-order methods \cite{taylor2018Lyapunov,park2023factor,lee2021geometric,kim2023}. The main structure of the proof is usually twofold. First, define an equivalent form of the given method. Usually, this equivalent form introduces an ``auxiliary sequence" that may not be present in the original form. Then, by constructing a nonincreasing Lyapunov sequence using the auxiliary sequence, we conclude the convergence result. For example, the $\mathcal{O}(1/N^2)$ rate of Nesterov's method \cite{nesterov83MomentumPaper} can be proved by defining
\begin{align*}
    \mathcal{U}_k =     2k^2    \big(f(x_{k-1}^+) - f_\star\big) + L \|z_{k} - x_\star\|^2\\[-0.25in]
\end{align*}
where $x_{k-1}^{+}=x_{k-1}-\frac{1}{L}\nabla f(x_{k-1})$ with $z_0=x_0$ and $z_{k+1} = z_{k} - \frac{k+1}{2L}\nabla{f(x_{k})}$,
and showing $\mathcal{U}_k\le\cdots\le \mathcal{U}_0$. 

However, the most challenging part of this style of proof is identifying the auxiliary sequence and the nonincreasing Lyapunov sequence. This process is far from straightforward, and in the case of composite functions, as in our paper, it is inevitably more complex and lengthy than for single function setups. In this section, we will introduce only the key ideas, leaving the detailed calculations to the appendix.

Define the following sequences, which will be soon shown as an alternative representation of \eqref{eq:OptISTA_Alg} in Lemma~\ref{c:equivalentform}:
\begin{equation}
\begin{aligned} & y_{i+1}=\prox_{\frac{\gamma_{i}}{L}h}\big(y_{i}-\frac{\gamma_{i}}{L}\nabla f(x_{i})\big)\\
 & \ensuremath{z_{i+1}=x_{i}+\frac{1}{\gamma_{i}}\left(y_{i+1}-y_{i}\right)}\\
 & w_{i+1}= w_i - \frac{2\theta_i}{L}\nabla f(x_i)- \frac{2\theta_i}{L}h'(y_{i+1}) \\
 & x_{i+1}=\left(1-\frac{1}{\theta_{i+1}}\right)z_{i+1}+\frac{1}{\theta_{i+1}}w_{i+1}
\end{aligned}\tag{\textup{OptISTA-A}}\label{eq:OptISTA_Alg_A}
\end{equation}
for $i={\color{red}{0}},\dots,N-1$, where $w_0=x_0$ and $h'(y_{i+1})=\frac{L}{\gamma_i}\left(y_i-\frac{\gamma_i}{L}\nabla f(x_i)-y_{i+1}\right)\in\partial h(y_{i+1})$.
Notably, \ref{eq:OptISTA_Alg_A} has the auxiliary sequence $\{w_i\}_{i\in[0:N]}$ that was not explicitly present in \ref{eq:OptISTA_Alg}.
Also we define the sequence $\{\tilde{\theta}_i\}_{i=0,\ldots,N-1}$, which is similar to $\theta$--sequence with slight modification at the last step, defined as
    \[
    \tilde{\theta}_{i}=\begin{cases}
\theta_i & \textup{if }i\in[0:N-2],\\
\frac{2\theta_{N-1}+\theta_N-1}{2}, & \textup{if }i=N-1.
\end{cases}
\]
Then we have the following two lemmas.
\begin{lem}\label{c:equivalentform}

Denote
 $x$ and $y$ sequences generated by (OptISTA) and (OptISTA-A) by $\{x_k,y_k\}_{k=0}^N$ and $\{\hat{x}_k,\hat{y}_k\}_{k=0}^N$, respectively, then 

\[
\hat{x}_k=x_k,  \quad \hat{y}_k=y_k
\]
for every $k=0,\dots,N$. Hence,  \eqref{eq:OptISTA_Alg_A} is indeed equivalent to \eqref{eq:OptISTA_Alg}.
\end{lem}
\begin{proof}
Note that $y_0=x_0=\hat{y}_0=\hat{x}_0$ by definition. Assume $x_i=\hat{x}_i$ and $y_i=\hat{y}_i$ for $i\in[0:k]$. Then, 
$\hat{y}_{k+1}=y_{k+1}$ is immediate. To prove $\hat{x}_{k+1}=x_{k+1}$, observe that
\begin{align*}
    \hat{x}_{k+1}&=\left(1-\frac{1}{\theta_{k+1}}\right)z_{k+1}+\frac{1}{\theta_{k+1}}w_{k+1}\\
    &\stackrel{(\star)}{=}\left(1-\frac{1}{\theta_{k+1}}\right)z_{k+1}+\frac{1}{\theta_{k+1}}\left(\theta_k x_k-(\theta_k-1)z_k- \frac{2\theta_{k}}{L}\nabla f(x_{k})- \frac{2\theta_{k}}{L}h'(y_{k+1})\right)\\
    &=\left(1-\frac{1}{\theta_{k+1}}\right)z_{k+1}-\frac{\theta_k-1}{\theta_{k+1}}z_k+\frac{\theta_k}{\theta_{k+1}}\left( x_k- \frac{2}{L}\nabla f(x_{k})- \frac{2}{L}h'(y_{k+1})\right)\\
        &=\left(\frac{\theta_k-1}{\theta_{k+1}}\right)(z_{k+1}-z_k)+z_{k+1}-\frac{\theta_k}{\theta_{k+1}}z_{k+1}+\frac{\theta_k}{\theta_{k+1}}\left( x_k- \frac{2}{L}\nabla f(x_{k})- \frac{2}{L}h'(y_{k+1})\right)\\
              &=z_{k+1}+\left(\frac{\theta_k-1}{\theta_{k+1}}\right)(z_{k+1}-z_k)+\frac{\theta_k}{\theta_{k+1}}\left( x_k-z_{k+1}- \frac{2}{L}\nabla f(x_{k})- \frac{2}{L}h'(y_{k+1})\right)\\
                                          &\stackrel{(\circ)}{=}z_{k+1}+\left(\frac{\theta_k-1}{\theta_{k+1}}\right)(z_{k+1}-z_k)+\frac{\theta_k}{\theta_{k+1}}\left( x_k-z_{k+1} + \frac{2}{\gamma_k}\left(y_{k+1}-y_k\right)\right)\\
                            &=z_{k+1}+\left(\frac{\theta_k-1}{\theta_{k+1}}\right)(z_{k+1}-z_k)+\frac{\theta_k}{\theta_{k+1}}\left( x_k-z_{k+1} + 2\left(z_{k+1}-x_k\right)\right)\\
                                 &=z_{k+1}+\left(\frac{\theta_k-1}{\theta_{k+1}}\right)(z_{k+1}-z_k)+\frac{\theta_k}{\theta_{k+1}}\left( z_{k+1}-x_k\right)=x_{k+1}.
\end{align*}
For $(\star)$, we used the fact that 
\[
w_{k+1} = w_k - \frac{2\theta_{k}}{L}\nabla f(x_{k})- \frac{2\theta_{k}}{L}h'(y_{k+1})
\]
and
\[
x_{k}=\left(1-\frac{1}{\theta_{k}}\right)z_{k}+\frac{1}{\theta_{k}}w_{k}
\iff  w_k = \theta_k x_k - (\theta_k-1)z_k.
\]
For $(\circ)$, we used that 
\[
y_{{\color{red}{k}}+1} - y_{\color{red}{ k}} =  -\frac{\gamma_k}{L}\nabla f(x_{k})- \frac{\gamma_k}{L}h'(y_{k+1})
\]
by the definition of $h'(y_{k+1})$ and that 
\[
z_{k+1}-x_{\color{red}{k}}=\frac{1}{\gamma_k}(y_{k+1} - y_k) = -\frac{1}{L}\nabla f(x_{k})- \frac{1}{L}h'(y_{k+1}).
\]
Therefore, 
\[
 x_k-z_{k+1}- \frac{2}{L}\nabla f(x_{k})- \frac{2}{L}h'(y_{k+1}) = x_k-z_{k+1} +2(z_{k+1}-x_{\color{red}{k}}) = 2z_{k+1}-x_k-z_{k+1}.
\]
\qed\end{proof}
\begin{lem}\label{c:x=y}
In \eqref{eq:OptISTA_Alg}, we have $x_N=y_N$.
\end{lem}
\begin{proof}
    The proof is quite technical, and therefore, we defer its proof to Appendix~\ref{s:lem2}.
\qed\end{proof}
Now we define the Lyapunov sequence $\mathcal{U}_k$ indexed by $k\in[-1:N]$. The explicit form of the sequence is quite cumbersome, so we defer the full details to Appendix~\ref{s:c}. For illustration, we present only the cases $k=N$ and $k=-1$ as follows:
\begin{align*}
   & \mathcal{U}_N=f(x_N)-f(x_\star)+h(y_N)-h(x_\star)\\
   &+\frac{L}{2\theta_N^2}\left\|w_N-x_\star+\frac{1}{L}\nabla f(x_\star)+\frac{2\theta_{N-1}}{L} h'(y_N)-\frac{\theta_N}{L}\nabla f(x_N)-\frac{2\tilde{\theta}_{N-1}}{L}h'(y_{N})\right\|^2\\
    &+\frac{L}{2\theta_N^2(\theta_N^2-1)}\left\|x_0-x_\star-\frac{\theta_N^2-1}{L}\nabla f(x_\star)-\sum_{i=0}^{N-1}\frac{2\tilde{\theta}_i}{L} h'(y_{i+1})\right\|^2\\
    &+\sum_{i\neq j, i,j\in[1:N]}\frac{\tilde{\theta}_{i-1}\tilde{\theta}_{j-1}}{L\theta_N^2(\theta_N^2-1)}\|h'(y_i)-h'(y_j)\|^2+\sum_{i=1}^{N-1}\frac{\tilde{\theta}_{i-1}^2}{L\theta_N^2}\|h'(y_i)-h'(y_{i+1})\|^2
\end{align*}
and
\begin{align*}
    \mathcal{U}_{-1} = \frac{L\|x_0-x_{\star}\|^2}{2(\theta_N^2-1)}.
\end{align*}
Then, we show $\mathcal{U}_N\le \mathcal{U}_{N-1}\le\cdots\le \mathcal{U}_1\le \mathcal{U}_0\le \mathcal{U}_{-1}$ to get
\[
f(x_N)-f(x_\star)+h(y_N)-h(x_\star) \leq \mathcal{U}_{N} \leq \cdots \leq \mathcal{U}_{-1} = \frac{L\|x_0-x_{\star}\|^2}{2(\theta_N^2-1)}.
\]
Finally, we use the fact $x_N=y_N$ by Lemma~\ref{c:x=y} to conclude that 
\[
f(y_N)+h(y_N)-f(x_\star)-h(x_\star)\le  \frac{L\|x_0-x_\star\|^2}{2(\theta_N^2-1)}.
\]
The main challenge here is to show the nonincreasing property of $\mathcal{U}_k$, and the rigorous proof is provided in Appendix~\ref{s:c}.

\section{Computer-assisted algorithmic design via PEP}
\label{s:3}

In this section, we present the double-function stepsize-optimization PEP methodology that we used to design \ref{eq:OptISTA_Alg}.

\subsection{Double-function fixed-step first-order methods}

First, we concretely specify and parameterize the class of first-order methods for solving \eqref{eq:main_problem}: we consider \emph{$N$-step double-function fixed-step first-order methods} \eqref{eq:FSFOM} defined as
\begin{equation}
\begin{aligned} 
 & y_{i+1}=x_{0}-\sum_{j\in[0:i]}\frac{\phi_{i+1,j}}{L}\nabla f(x_{j})-\sum_{j\in[0:i]}\frac{\psi_{i+1,j}}{L}h'(y_{j+1})\\
 & x_{i+1}=x_{0}-\sum_{j\in[0:i]}\frac{\alpha_{i+1,j}}{L} \nabla f(x_{j})-\sum_{j\in[0:i]}\frac{\beta_{i+1,j}}{L}h'(y_{j+1})
\end{aligned}
\tag{$N$-DF-FSFOM}\label{eq:FSFOM}
\end{equation}
for $i\in[0:N-1]$, where $h'(y_{i+1})\in\partial h(y_{i+1})$ is a subgradient $h$ at $y_{i+1}$ defined by
\begin{align*}
 & \tilde{y}_{i+1}=x_{0}-\sum_{j\in[0:i]}\frac{\phi_{i+1,j}}{L}{\color{red}{\nabla f}}(x_{j})-\sum_{j\in[0:i-1]}\frac{\psi_{i+1,j}}{L}h'(y_{j+1})\\
 & y_{i+1}=\prox_{\frac{\psi_{i+1,i}}{L}h}(\tilde{y}_{i+1})\\
 & h'(y_{i+1})=\frac{L}{\psi_{i+1,\color{red}{i}}}\left(\tilde{y}_{i+1}-y_{i+1}\right)
\end{align*}
for $i\in [0:N-1]$. 
The values of the stepsizes $\{\phi_{i,j}\}$, $\{\psi_{i,j}\}$, $\{\alpha_{i,j}\}$, and $\{\beta_{i,j}\}$, where $i,j$ have the range $1\le i\le N$ and $0\le j<i$, determine the specific instance of \eqref{eq:FSFOM}. We consider $x_N$ to be the output of \eqref{eq:FSFOM}.

The class of \eqref{eq:FSFOM} includes the proximal point method \cite{martinet1970regularisation,Rockafellar1976_monotone}, accelerated proximal point method \cite{guler1992new}, Nesterov's FGM \cite{nesterov83MomentumPaper}, OGM \cite{drori2014performance,kim2016OGM}, ISTA \cite{daubechies2004iterative,hale2008fixed,wright2009sparse,elad2006simple}, FISTA \cite{beck2009FISTA}, and other variants for the composite setup \cite{taylor2017Thesis,taylor2017CompositePEP,kim2018another}. The class essentially includes all conceivable first-order methods that (i) use $N$ evaluations of $\nabla f$,
(ii) $N$ evaluations of $\prox_{\gamma h}$ (with the value of $\gamma>0$ chosen freely every use),
(iii) interleave the evaluations of $\nabla f$ and $\prox_{\gamma h}$ in any arbitrary order, and
(iv) do not utilize function values or linesearches.

\paragraph{Discussion on stepsize dependence.}
For problem \eqref{eq:main_problem}, let $L$ be the smoothness coefficient of $f$ and $R\ge \|x_0-x_\star\|$ be an upper bound on the distance to solution. Let $N$ be the total iteration count of the \eqref{eq:FSFOM}. The stepsizes  $\{\phi_{i,j}\}$, $\{\psi_{i,j}\}$, $\{\alpha_{i,j}\}$, and $\{\beta_{i,j}\}$
 may depend on $L$, $R$, and $N$ but are otherwise predetermined.

Methods like gradient descent, Nesterov FGM, and FISTA utilize $L$ in their stepsizes. In fact, most smooth minimization methods without linesearch do so (with the notable recent exception of \cite{maliteky2020}). Subgradient methods often utilize $R$ in their stepsizes \cite[Section 3.2.3]{nesterov2003Introductory}. Stochastic first-order methods often utilize $N$ in their stepsizes \cite[Corollary 5.6]{garrigos2023handbook} and recent methods such as OGM-G \cite{kim2021OGMG} and FISTA-G \cite{lee2021geometric} also do so. Our exact optimal method \ref{eq:OptISTA_Alg} uses $N$ and $L$, but not $R$ in its stepsizes.

\subsection{Double-function stepsize-optimization PEP (DF-SO-PEP)}

Next, we introduce the double-function stepsize-optimization performance estimation problem (DF-SO-PEP), a computer-assisted methodology for finding the optimal \eqref{eq:FSFOM}.

For the sake of convenience, assume $L=1$ to eliminate the dependence on $L$. Note that we can easily extend to an arbitrary smoothness parameter $L>0$ by mere scaling. Write $\mathcal{M}_{N}$ to denote the set of all \eqref{eq:FSFOM}. Define the worst-case performance (risk) of $M\in\mathcal{M}_{N}$ as
\begin{align}
\!\! \mathcal{R}\left(M\right)
 & =
 \left(\begin{array}{ll}
\textrm{maximize}
&f(x_{N})+h(x_N)-f(x_{\star})-h(x_\star)
\\
\textrm{subject to }&\textup{$f$ is $L$-smooth convex, $h$ is closed convex proper}\\
&\textrm{there is an $x_{\star}$ minimizing $f+h$ such that $\|x_0-x_\star\|\le R$}\\
&\{(x_{i},y_i)\}_{i\in[1:N]}\textrm{ { generated by $M$ }with initial point $x_0=y_0$}\\
\end{array}\right),\tag{\ensuremath{\mathcal{O}^{\textrm{inner}}}}\label{eq:worst-case-pfm}
\end{align}
where $x_0$, $f$, and $h$ are the decision variables of the maximization problem. ($f$ and $h$ are infinite-dimensional variables.) The optimal method $M_{N}^{\star}\in\mathcal{M}_{N}$ is a solution to the following minimax optimization problem 
\begin{equation}
\begin{array}{ll}
\mathcal{R}^{\star}(\mathcal{M}_{N})=\underset{M\in\mathcal{M}_{N}}{\mbox{minimize}} & \mathcal{R}(M).\end{array}\tag{\ensuremath{\mathcal{O}^{\textrm{outer}}}}\label{eq:outer_problem}
\end{equation}

Later in this section, we demonstrate that \eqref{eq:outer_problem} can be expressed as a finite-dimensional, nonconvex yet practically solvable quadratically constrained quadratic program (QCQP), which can be globally optimized using spatial branch-and-bound optimization solvers \cite{Gurobi}. In essence, the conversion to QCQP consists of two primary steps: (1) formulating the inner maximization problem in \eqref{eq:worst-case-pfm} as a convex semidefinite programming minimization problem (SDP), and (2) transforming \eqref{eq:outer_problem} into a QCQP using the minimization SDP from the first step.

The dual variables of the aforementioned minimization SDP are referred to as the \emph{inner-dual variables}. In the resulting QCQP form of \eqref{eq:outer_problem}, the decision variables consist of the stepsizes from \eqref{eq:FSFOM} and the inner-dual variables.

\paragraph{Comparison with prior approaches.}
Our work is the first instance of a stepsize-optimization performance estimation problem (SO-PEP) framework optimizing over methods in a double-function setup. However, there have been prior work \cite{taylor2017CompositePEP,ryu2020OperatorSplittingPEP} using the PEP to evaluate the performance of a given fixed method with two or more functions. The stepsize-optimization introduces nonconvexity and makes the resulting finite-dimensional optimization problem a nonconvex QCQP rather than a convex SDP as in prior SO-PEP work \cite{drori2014performance,kim2016OGM}.  We overcome such nonconvexity using spatial branch-and-bound algorithms \cite{Gurobi,locatelli2013global,horst2013global,das2024branch}.

\subsection{Nonconvex QCQP formulation for \eqref{eq:outer_problem}}\label{sec:BnB-PEP-for-OptISTA}

As mentioned in the previous section, we transform \eqref{eq:outer_problem}
into a (nonconvex) QCQP through two distinct steps. First, 
we reformulate the inner problem \eqref{eq:worst-case-pfm} as a convex
SDP. This initial step adopts a similar approach as described in \cite{taylor2017CompositePEP}.
Second, we express the outer problem \eqref{eq:outer_problem} as a QCQP.
This second step is conceptually aligned with the method presented
in \cite{das2024branch}, which was developed to design optimal
first-order methods for minimizing a single function.

We use the following notations. Write $e_i\in\rl^{d}$ for the standard $i$-th unit vector for $i\in[0:d-1]$. Write $\rl^{m\times n}$
for the set of $m\times n$ matrices, $\mathbb{S}^{n}$ for the set
of $n\times n$ symmetric matrices, and $\mathbb{S}_{+}^{n}$ for
the set of $n\times n$ positive-semidefinite matrices. Write $\left(\cdot\odot\cdot\right)\colon\rl^{d}\times\rl^{d}\to\rl^{d\times d}$
to denote the symmetric outer product, that is, for any $x,y\in\rl^{d}$:
$x\odot y=\nicefrac{\left(xy^{\intercal}+yx^{\intercal}\right)}{2}$. 

\subsubsection{Formulating the inner problem as a convex SDP \label{par:Converting-the-inner-problem-into-SDP}}

Let $d>0$, $L>0$, and $R>0$. Let $\mathcal{F}_{0,L}$ denote the set of $L$-smooth convex functions on $\rl^d$ and $\mathcal{F}_{0,\infty}$ the set of closed convex proper functions on $\rl^d$.

\paragraph{Infinite-dimensional inner optimization problem.}
Write \eqref{eq:worst-case-pfm} as 
\begin{align*}
  \mathcal{R}(M)
 & =\left(\begin{aligned} & \textrm{maximize}\quad f(x_{N})+h(x_{N})-f(x_{\star})-h(x_{\star})\\
 & \textrm{subject to}\quad\\
 & f\in\FZerL,\,h\in\FZerInf,\quad{\color{gray}\rhd\;\textrm{function class}}\\
 & \begin{rcases}\begin{aligned}
&y_{i+1}  =x_{0}-\sum_{j\in[0,i]}\frac{\phi_{i+1,j}}{L}\nabla f(x_{j})-\sum_{j\in[0,i]}\frac{\psi_{i+1,j}}{L}h'(y_{j+1}),\quad i\in[0:N-1],\\
&x_{i+1}  =x_{0}-\sum_{j\in[0,i]}\frac{\alpha_{i+1,j}}{L}\nabla f(x_{j})-\sum_{j\in[0,i]}\frac{\beta_{i+1,j}}{L}h'(y_{j+1}),\quad i\in[0:N-1],
\end{aligned}
\end{rcases} & {\color{gray}\rhd\;\textrm{method $M$}}\\
 & \nabla f(x_{\star})+h'(x_{\star})=0,\;x_{\star}=0,\quad{\color{gray}\rhd\;\textrm{optimal solution}}\\
 & \|x_{0}-x_{\star}\|^{2}\leq R^{2},\quad{\color{gray}\rhd\;\textrm{initial condition}}
\end{aligned}
\right)
\end{align*}
where $f$, $h$, $x_{0},\ldots,x_{N}$, and $y_{1},\ldots,y_{N}$
are the decision variables. Furthermore,
we set $x_{\star}=0$ without loss of generality.

As is, both $f$ and $h$ are infinite-dimensional
decision variables.

\paragraph{Interpolation argument.}
We now convert the infinite-dimensional optimization problem into a
finite-dimensional one with the following interpolation result.
\begin{lem}[{{{$\FZerL$- and $\FZerInf$-interpolation
\cite[Theorem 4]{taylor2017InterpolationFMuLPEP}\label{Thm:Interpolation-inequality-FmuL}}}}]
Let $I$ be an index set, 
and let $\{(x_{i},g_{i},f_{i})\}_{i\in I}\subseteq\mathbb{R}^{d}\times\mathbb{R}^{d}\times\mathbb{R}$.
Let $0<L\leq\infty$. There exists $f\in\FZerL$ 
satisfying $f(x_{i})=f_{i}$ and $g_{i}\in\partial f(x_{i})$ for
all $i\in I$ if and only if \footnote{This can be viewed as a discretization of the following condition
\cite[Theorem 2.1.5, Equation (2.1.10)]{nesterov2003Introductory}:
$f\in\FZerL$ if and only if 
\begin{align*}
\quad f(y)\geq f(x)+ & \langle\nabla f(x);\,y-x\rangle+\frac{1}{2L}\|\nabla f(x)-\nabla f(y)\|^{2},\quad\forall x,y\in\rl^{d}.
\end{align*}
} 
\begin{align*}
f_{i} & \geq f_{j}+\langle g_{j},\,x_{i}-x_{j}\rangle+\frac{1}{2L}\|g_{i}-g_{j}\|^{2},\quad\forall\,i,j\in I.
\end{align*}
When $L=\infty$, we mean $\frac{1}{2L}\|g_{i}-g_{j}\|^{2}=0$.
\end{lem}


For notation convenience, define
\begin{align*}
 & w_{\star}=x_{\star}=0,\,w_{i}=x_{i}\ \textup{for}\ i\in[0:N],\,w_{N+i}=y_{i}\ \textup{for}\ i\in[1:N],\\
 & I_{f}=[0:N]\cup\{\star\},\\
 & I_{h}=[N:2N]\cup\{\star\},\\
 & f(w_{i})=f_{i},\,\nabla f(w_{i})=f_{i}'\textrm{ for }i\in I_{f},\\
 & h(w_{i})=h_{i},\,h'(w_{i})=h_{i}',\textrm{ for }i\in I_{h}.
\end{align*}

\paragraph{Finite-dimensional maximization problem.}

Using Lemma~\ref{Thm:Interpolation-inequality-FmuL} and the new notation
above, reformulate \eqref{eq:worst-case-pfm} as 
\begin{align*}
 \mathcal{R}(M)\nonumber = \left(\begin{array}{l}
\textrm{maximize}\quad f_{N}+h_{N}-f_{\star}-h_{\star}\\
\textrm{subject to}\\
f_{i}\geq f_{j}+\langle f_{j}',\,w_{i}-w_{j}\rangle+\frac{1}{2L}\Vert f_{i}'-f_{j}'\Vert^{2},\quad i,j\in I_{f},i\neq j,\\
h_{i}\geq h_{j}+\langle h_{j}',\,w_{i}-w_{j}\rangle,\quad i,j\in I_{h},i\neq j,\\
w_{N+i+1}=w_{0}-\sum_{k=0}^{i}\frac{\phi_{i+1,k}}{L}f_{k}'-\sum_{k=0}^{i}\frac{\psi_{i+1,k}}{L}h_{N+k+1}',\quad i\in[0:N-1],\\
w_{i+1}=w_{0}-\sum_{k=0}^{i}\frac{\alpha_{i+1,k}}{L}f_{k}'-\sum_{k=0}^{i}\frac{\beta_{i+1,k}}{L}h_{N+k+1}',\quad i\in[0:N-1],\\
f_{\star}'+h_{\star}'=0,\\
w_{\star}=0,\\
\|w_{0}-w_{\star}\|^{2}\leq R^{2},
\end{array}\right)
\end{align*}
where the decision variables are $\{w_{i},f_{i}',f_{i}\}_{i\in I_{f}}\subseteq\rl^{d}\times\rl^{d}\times\rl$
and $\{w_{i},h_{i}',h_{i}\}_{i\in I_{h}}\subseteq\rl^{d}\times\rl^{d}\times\rl$. The optimization problem is now finite-dimensional, although nonconvex.

\paragraph{Grammian formulation.}

Next, we formulate \eqref{eq:worst-case-pfm} into a (finite-dimensional convex) SDP.
Let 
\[
\begin{alignedat}{1}H & =[\ w_{0}\ |\ f_{\star}'\ |\ f_{0}'\ |\cdots|\ f_{N}'\ |\ h_{N}'\ |\ h_{N+1}'\ |\cdots|\ h_{2N}'\ ]\in\mathbb{R}^{d\times(2N+4)},\\
G & =H^{\intercal}H\in\mathbb{S}_{+}^{2N+4},\\
F & =[\ f_{\star}\ |\ f_{0}\ |\cdots|\ f_{N}\ |\ h_{\star}\ |\ h_{N}\ |\ h_{N+1}\ |\cdots|\ h_{2N}\ ]\in\mathbb{R}^{1\times(2N+4)}.
\end{alignedat}
\]
Note that $\rank G\leq d$. Define the following notation for selecting
columns and elements of $H$ and $F$: 
\begin{equation*}
\begin{alignedat}{1} 
 & \textbf{w}_{\star}=\textbf{0}\in\rl^{2N+4},\;\textbf{w}_{0}=e_{0}\in\rl^{2N+4},\\
&\textbf{f}_{\star}'=e_{1}\in\rl^{2N+4},\; \textbf{f}_{i}'=e_{i+2}\in\rl^{2N+4}\textrm{ for }i\in[0:N],\\
 & \textbf{h}_{\star}'=-e_{1}\in\rl^{2N+4},\;\textbf{h}_{N+i}'=e_{N+i+3}\in\rl^{2N+4}\textrm{ for }i\in[0:N],\\
 & \textbf{w}_{N+i+1}=\textbf{w}_{0}-\sum_{k=0}^{i}\frac{\phi_{i+1,k}}{L}\textbf{f}_{k}'-\sum_{k=0}^{i}\frac{\psi_{i+1,k}}{L}\textbf{h}_{N+k+1}'\in\rl^{2N+4}\textrm{ for }i\in[0:N-1],\\
 & \textbf{w}_{i+1}=\textbf{w}_{0}-\sum_{k=0}^{i}\frac{\alpha_{i+1,k}}{L}\textbf{f}_{k}'-\sum_{k=0}^{i}\frac{\beta_{i+1,k}}{L}\textbf{h}_{N+k+1}'\in\rl^{2N+4}\in\textrm{ for }i\in[0:N-1],\\
 &\textbf{f}_{\star}=e_{0}\in\rl^{2N+4},\;  \textbf{f}_{i}=e_{i+1}\in\rl^{2N+4}\textrm{ for }i\in[0:N],\\
 & \textbf{h}_{\star}=e_{N+2}\in\rl^{2N+4},\;\textbf{h}_{N+i}=e_{N+i+3}\in\rl^{2N+4}\textrm{ for }i\in[0:N].
\end{alignedat}
\end{equation*}
Note that $\mathbf{w}_{i}$ depends linearly on $\{\phi_{i,j}\}$,
$\{\psi_{i,j}\}$, $\{\alpha_{i,j}\}$, and $\{\beta_{i,j}\}$. This
notation is defined so that for all $i\in[0:2N]\cup\{\star\}$
we have 
\[
\begin{aligned} & w_{i}=H\textbf{w}_{i},\\
 & f_{i}'=H\textbf{f}_{i}',\;h_{i}'=H\mathbf{h}_{i}',\\
 & f_{i}=F\mathbf{f}_{i},\;h_{i}=F\mathbf{h}_{i}.
\end{aligned}
\]
Also, our choice ensures that $f_{\star}'+h_{\star}'=H\left(\textbf{f}_{\star}'+\mathbf{h}_{\star}'\right)=H\left(e_{1}-e_{1}\right)=0$.
Furthermore, for $i,j$ that belong to either $I_{f}$ or $I_{h}$,
define: 
\[
\begin{alignedat}{1} & A_{i,j}^{f}(\phi,\psi,\alpha,\beta)=\textbf{f}_{j}'\odot(\textbf{w}_{i}-\textbf{w}_{j}),\\
 & A_{i,j}^{h}(\phi,\psi,\alpha,\beta)=\textbf{h}_{j}'\odot(\textbf{w}_{i}-\textbf{w}_{j}),\\
 & B_{i,j}(\phi,\psi,\alpha,\beta)=(\textbf{w}_{i}-\textbf{w}_{j})\odot(\textbf{w}_{i}-\textbf{w}_{j}),\\
 & C_{i,j}^{f}=(\textbf{f}_{i}'-\textbf{f}_{j}')\odot(\textbf{f}_{i}'-\textbf{f}_{j}'),\\
 & a_{i,j}^f=\textbf{f}_j-\textbf{f}_i,\; a_{i,j}^h=\textbf{h}_j-\textbf{h}_i.
\end{alignedat}
\]
Note that $A_{i,j}^{f}(\phi,\psi,\alpha,\beta)$ and $A_{i,j}^{h}(\phi,\psi,\alpha,\beta)$
are affine and $B_{i,j}(\phi,\psi,\alpha,\beta)$ is quadratic in
the entries of $\{\phi_{i,j}\}$, $\{\psi_{i,j}\}$, $\{\alpha_{i,j}\}$,
and $\{\beta_{i,j}\}$. This notation is defined so that 
\[
\begin{alignedat}{1} & \left\langle f_{j}',\,w_{i}-w_{j}\right\rangle =\tr GA_{i,j}^{f}(\phi,\psi,\alpha,\beta),\;\left\langle h_{j}',\,w_{i}-w_{j}\right\rangle =\tr GA_{i,j}^{h}(\phi,\psi,\alpha,\beta),\\
 & \|w_{i}-w_{j}\|^{2}=\tr GB_{i,j}(\phi,\psi,\alpha,\beta),\\
 & \|f_{i}'-f_{j}'\|^{2}=\tr GC_{i,j}^{f},\\
 & f_{j}-f_{i}=Fa_{i,j}^{f},\;h_{j}-h_{i}=Fa_{i,j}^{h}.
\end{alignedat}
\]
for all $i,j$ that belong to either $I_{f}$ or $I_{h}$. Using this
notation, formulate \eqref{eq:worst-case-pfm} as 
\begin{align*}
 \mathcal{R}(M)
= & \left(\begin{array}{l}
\textrm{maximize}\quad Fa_{\star,N}^{f}+Fa_{\star,N}^{h}\\
\textrm{subject to}\\
Fa_{i,j}^{f}+\tr A_{i,j}^{f}(\phi,\psi,\alpha,\beta)G+\frac{1}{2L}\tr C_{i,j}^{f}G\leq0,\ i,j\in I_{f},i\neq j,\\
Fa_{i,j}^{h}+\tr A_{i,j}^{h}(\phi,\psi,\alpha,\beta)G\leq0,\ i,j\in I_{h},i\neq j,\\
G\succeq0,\;\rank(G)\le d,\\
\tr {\color{red}{ B_{0,\star}}}G\leq R^{2},
\end{array}\right)
\end{align*}
where $G\in\rl^{(2N+4)\times(2N+4)}$ and $F\in\mathbb{R}^{1\times(2N+4)}$
are the decision variables. The equivalence relies on the fact that
given a $G\in\mathbb{S}_{+}^{2N+4}$ satisfying $\rank(G)\le d$,
there exists a $H\in\rl^{d\times(2N+4)}$ such that $G=H^{\intercal}H$.
The argument is further detailed in \cite[$\mathsection$3.2]{taylor2017InterpolationFMuLPEP}.
This formulation is not yet a convex SDP due to the rank constraint
$\rank(G)\le d$.

\paragraph{SDP representation.}

Next, we make the following large-scale assumption.

\begin{assumption} \label{large-scale-assumption} We have $d\geq2N+4$.
\end{assumption}

Under this assumption, the constraint $\rank G\leq d$ becomes vacuous,
since $G\in\mathbb{S}_{+}^{2N+4}$. We drop the rank constraint
and formulate \eqref{eq:worst-case-pfm} as a convex SDP 
\begin{align}
 & \mathcal{R}(M)= \left(\begin{array}{l}
\textrm{maximize}\quad Fa_{\star,N}^{f}+Fa_{\star,N}^{h}\\
\textrm{subject to}\\
Fa_{i,j}^{f}+\tr A_{i,j}^{f}(\phi,\psi,\alpha,\beta)G+\frac{1}{2L}\tr C_{i,j}^{f}G\leq0,\ i,j\in I_{f},i\neq j,{\color{gray}\quad\rhd\textup{\;dual var.\;}\lambda_{i,j}\geq0}\\
Fa_{i,j}^{h}+\tr A_{i,j}^{h}(\phi,\psi,\alpha,\beta)G\leq0,\ i,j\in I_{h},i\neq j,{\color{gray}\quad\rhd\textup{\;dual var.\;}\tau_{i,j}\geq0}\\
-G\preceq0,{\color{gray}\quad\rhd\textup{\;dual var.\;}Z\succeq0}\\
\tr{\color{red}{ B_{0,\star}}}G-R^{2}\leq0 ,{\color{gray}\quad\rhd\textup{\;dual var.\;}\nu\geq0}
\end{array}\right)\label{eq:worst-case-pfm-sdp-1}
\end{align}
where $F\in\rl^{1\times(2N+4)}$ and $G\in\rl^{(2N+4)\times(2N+4)}$
are the decision variables. We represent the dual variables corresponding
to the right-hand side constraints using the notation {$\rhd\textup{\;dual var.}$},
which will be employed in later sections. It is important to note
that omitting the rank constraint does not constitute a relaxation;
the optimization problem \eqref{eq:worst-case-pfm-sdp-1} and its
solution remain independent of dimension $d$, as long as the large-scale
assumption $d\ge2N+4$ is satisfied. Further discussion on this topic
can be found in \cite[$\mathsection$3.3]{taylor2017InterpolationFMuLPEP}.

\paragraph{Dualization.}

We now utilize convex duality to recast \eqref{eq:worst-case-pfm},
initially a maximization problem, as a minimization problem. Taking
the dual of \eqref{eq:worst-case-pfm-sdp-1} yields 
\begin{align}
 & \overline{\mathcal{R}}(M) =\left(\begin{array}{l}
\textrm{minimize}\quad\nu R^{2}\\
\textrm{subject to}\\
-a_{\star,N}^{f}-a_{\star,N}^{h}+\sum_{i,j\in I_{f},i\neq j}\lambda_{i,j}a_{i,j}^{f}+\sum_{i,j\in I_{h},i\neq j}\tau_{i,j}a_{i,j}^{h}=0,\\
\nu B_{0,\star}+\sum_{i,j\in I_{f},i\neq j}\lambda_{i,j}\left(A_{i,j}^{f}(\phi,\psi,\alpha,\beta)+\frac{1}{2L}C_{i,j}^{f}\right)\\
\quad+\sum_{i,j\in I_{h},i\neq j}\tau_{i,j}A_{i,j}^{h}(\phi,\psi,\alpha,\beta)=Z,\\
Z\succeq0,\\
\nu\geq0,\\
\lambda_{i,j}\geq0,\quad i,j\in I_{f},i\neq j,\\
\tau_{i,j}\geq0,\quad i,j\in I_{h},i\neq j,
\end{array}\right)\label{eq:worst-case-pfm-dual-1}
\end{align}
where $\nu\in\rl$, $\lambda=\{\lambda_{i,j}\}_{i,j\in I_{f}:i\neq j}$,
$\tau=\{\tau_{i,j}\}_{i,j\in I_{h},i\neq j}$ with $\lambda_{i,j}\in\rl$,
$\tau_{i,j}\in\rl$, and $Z\in\mathbb{S}_{+}^{2N+4}$ are the decision
variables. (Note, we write $\overline{\mathcal{R}}$ rather than $\mathcal{R}$
here.) We call $\nu$, $\lambda$, $\tau$ and $Z$ the
\emph{inner-dual variables}. By weak duality of convex SDPs, we have
\[
\mathcal{R}(M)\le\overline{\mathcal{R}}(M).
\]
In convex SDPs, strong duality often holds but not always. For the
sake of simplicity, we assume strong duality holds. 
Note that because
we establish a matching lower bound, we know, after the fact, that strong duality holds.

\begin{assumption}\label{strong-duality-assumption} Strong duality
holds between \eqref{eq:worst-case-pfm-sdp-1} and \eqref{eq:worst-case-pfm-dual-1},
i.e., 
\[
\mathcal{R}(M)=\overline{\mathcal{R}}(M).
\]
\end{assumption}

Again, we do not need this assumption in our specific setup as we establish a matching lower bound.
However, in a generic variant of our setup, Assumption \ref{strong-duality-assumption}
can be eliminated by employing the reasoning outlined in \cite[Claim~4]{park2024optimal};
however, the approach is somewhat complicated.

At this juncture, our formulation diverges from the previous steps of \cite{taylor2017CompositePEP}.

\subsubsection{Formulating the outer problem \eqref{eq:outer_problem} as a QCQP
\label{subsec:Converting-minmax-into-a-bnb-pep}}

With the inner problem \eqref{eq:worst-case-pfm} formulated as a minimization
problem, the outer optimization problem \eqref{eq:outer_problem}
now involves joint minimization over both the inner dual variables
and the stepsizes. Although the inner problem is convex, the outer
minimization problem is not convex in all variables. The nonconvex
outer optimization problem requires minimizing over the stepsizes
$\phi$, $\psi$, $\alpha$, and $\beta$ and the inner dual variables of \eqref{eq:worst-case-pfm-dual-1}.
We directly address this nonconvexity by formulating \eqref{eq:outer_problem}
as a (nonconvex) QCQP and solving it exactly using spatial branch-and-bound
algorithms.
To achieve this, we substitute the semidefinite constraint with a
quadratic constraint using the Cholesky factorization.
\begin{lem}[{{{\cite[Corollary 7.2.9]{horn2012matrix}\label{Lem:quadratic-characterization-psd-1}}}}]
A matrix $Z\in\mathbb{S}^{n}$ is positive semidefinite if and only
if it has a Cholesky factorization $PP^{\intercal}=Z$, where $P\in\rl^{n\times n}$
is lower triangular with nonnegative diagonals. 
\end{lem}

Using Lemma~\ref{Lem:quadratic-characterization-psd-1}, we have 
\begin{align*}
 \left(Z\succeq 0\right)
\quad\Leftrightarrow \quad
 & \left(\begin{array}{l}
P\textrm{ is lower triangular with nonnegative diagonals},\\
PP^{\intercal}=Z.
\end{array}\right)\\
\quad\Leftrightarrow \quad &  \left(\begin{array}{l}
P_{j,j}\geq0,\quad j\in[1:2N+4],\\
P_{i,j}=0,\quad1\le i<j\le2N+4,\\
\sum_{k=1}^{j}P_{i,k}P_{j,k}=Z_{i,j},\quad1\le j\le i\le2N+4.
\end{array}\right)
\end{align*}

We now formulate \eqref{eq:outer_problem}, the problem to find
an optimal \eqref{eq:FSFOM}, as the following (nonconvex) QCQP:

\[ 
\mathcal{R}^{\star}(\mathcal{M}_{N}) =\left(\begin{array}{l}
\textrm{minimize}\quad\nu R^{2}\\
\textrm{subject to}\\
-a_{\star,N}^{f}-a_{\star,N}^{h}+\sum_{i,j\in I_{f},i\neq j}\lambda_{i,j}a_{i,j}^{f}+\sum_{i,j\in I_{h},i\neq j}\tau_{i,j}a_{i,j}^{h}=0,\\
\nu B_{0,\star}+\sum_{i,j\in I_{f},i\neq j}\lambda_{i,j}\left(A_{i,j}^{f}(\phi,\psi,\alpha,\beta)+\frac{1}{2L}C_{i,j}^{f}\right)\\
\quad+\sum_{i,j\in I_{h},i\neq j}\tau_{i,j}A_{i,j}^{h}(\phi,\psi,\alpha,\beta)=Z,\\
P\textup{ is lower triangular with nonnegative diagonals},\\
PP^{\intercal}=Z,\\
\nu\geq0,\\
\lambda_{i,j}\geq0,\quad i,j\in I_{f},i\neq j,\\
\tau_{i,j}\geq0,\quad i,j\in I_{h},i\neq j,
\end{array}\right)
\]
where the inner dual variables $\nu$, $\lambda$, $\tau$, $Z$,
and the stepsizes $\phi,\psi,\alpha,\beta$ are the decision variables.
We now have a nonconvex QCQP that can be
solved to global optimality using spatial branch-and-bound algorithms
\cite{Gurobi,locatelli2013global,horst2013global,das2024branch}.
The optimal stepsizes found will correspond to \ref{eq:OptISTA_Alg}.

\subsection{{\color{black}{Finding the analytical solution of the QCQP}}}

We use a spatial branch-and-bound algorithm to find global solutions to the nonconvex QCQP equivalent to \eqref{eq:outer_problem} for $N=1,\ldots5$. The numerical solutions allow us to guess the analytical forms of the optimal stepsizes and inner-dual variables as a function of $N$. Surprisingly, if we additionally assume $\alpha=\beta$ and $\phi=\psi$ in \eqref{eq:outer_problem}, many of the values were similar to the stepsize of OGM without hurting the optimal performance. This observation, along with some trial-and-error, enables to recover the analytical form. We then solve the nonconvex QCQP for larger values of $N$ (e.g., $N=6,\ldots,25$) and confirm that the numerical values of optimal stepsizes and inner-dual variables agree with our conjectured guess. This gives us confidence that our guess correctly represents the optimal stepsizes and inner-dual variables, although this is not yet a formal proof. The formal proof of Theorem~\ref{thm:OptISTA-rate} is obtained by identifying the inner-dual variables of \eqref{eq:outer_problem}.

We list optimal stepsizes $\phi,\psi,\alpha,\beta$ and dual variables $\nu$, $\lambda$, $\tau$, $Z$ as follows. For primal variables,
\[
\alpha_{i+1,j}=\begin{cases} \alpha_{j+1,j}+\sum_{k=j+1}^{i}\left(\frac{2\theta_j}{\theta_{k+1}}-\frac{1}{\theta_{k+1}}\alpha_{k,j}
\right) \textup{ if } j\in[0:i-1] \\ 1+\frac{2\theta_i-1}{\theta_{i+1}} \textup{ if } j=i \end{cases} \quad \mathrm{and} \quad \beta= \alpha.
\]
\[
\phi_{i+1,j} = \alpha_{N,j} \textup{ for } {\color{red}{0}}\le i \le {\color{red}{N-1}}, 0\le j \le i\quad \mathrm{and} \quad \psi= \phi.
\]
For inner-dual variables, $\nu = \frac{1}{2(\theta_N^2-1)}$ and 
\[
\lambda=\begin{cases}
 \lambda_{\star,i}=\frac{2\theta_i}{\theta_N^2} \textup{ for } i\in[0:N-1]  \\ 
 \lambda_{\star,N}=\frac{1}{\theta_N} \\
 \lambda_{i,i+1}=\frac{2\theta_i^2}{\theta_N^2} \textup{ for } i\in[0:N-1]\\
\lambda_{i,j}=0 \ \mathrm{otherwise},
\end{cases}
\, \tau=\begin{cases}
    \tau_{\star,N+i}=\frac{2\tilde{\theta}_{i-1}}{\theta_N^2-1} \textup{ for } i\in[1:N] \\
\tau_{N+i,N+j}=\frac{2\tilde{\theta}_{j-1}}{\theta_N^2-2\theta_i^2+\theta_i}-\frac{2\tilde{\theta}_{j-1}}{\theta_N^2-2\theta_{i-1}^2+\theta_{i-1}} \textup{ for } 1\leq i<j\leq N\\
    \tau_{N+i+1,N+i}=\frac{\theta_i-1}{\theta_N^2-2\theta_i^2+\theta_i} \textup{ for } i\in[1:N-1]\\ 
    \tau_{i,j}=0 \ \mathrm{otherwise}, 
    \end{cases}
\]
and
\[
\quad Z=\nu B_{0,\star}+\sum_{i,j\in I_{f},i\neq j}\lambda_{i,j}\left(A_{i,j}^{f}(\phi,\psi,\alpha,\beta)+\frac{1}{2L}C_{i,j}^{f}\right)\\
\quad+\sum_{i,j\in I_{h},i\neq j}\tau_{i,j}A_{i,j}^{h}(\phi,\psi,\alpha,\beta).
\]

\paragraph{Potential non-uniqueness of the optimal method.}

In solving $(\mathcal{O}^{\mathrm{outer}})$  numerically, we additionally assumed $\alpha=\beta$ and $\phi=\psi$ since we observed that adding those constraints did not worsen the worst-case performance. Under this additional assumption, our numerics lead us to believe that the following values are unique:
\begin{alignat*}{3}
\tau_{\star,N+i}&=\frac{2\tilde{\theta}_{i-1}}{\theta_N^2-1} \qquad&&\textup{ for } i\in[1:N],\\
    \phi_{i,0}&=\frac{2(\theta_N^2-1)}{\theta_N^2} \qquad&&\textup{ for }  i\in[1:N],\\ 
    \phi_{N,i}&=\alpha_{N,i} \qquad&&\textup{ for } i\in[0:N-1].
\end{alignat*}
However, the remaining values of $\phi,\psi$ and $\tau$ seem not to be unique, even with the additional constraints $\alpha=\beta$ and $\phi=\psi$. The presented stepsizes and dual variables of \ref{eq:OptISTA_Alg} were chosen as they were the most analytically tractable that the authors could find. It may be that \ref{eq:OptISTA_Alg} is the only analytically tractable choice, or it may be that there is a simpler exact optimal algorithm. Further exploring the set of exact optimal algorithms beyond OptISTA is an interesting direction of future work.

\paragraph{Global optimality through lower bound.}
The resulting proof of Theorem~\ref{thm:OptISTA-rate} is likely globally optimal over methods of the form \eqref{eq:FSFOM} as it agrees with the global numerical solutions of the spatial branch-and-bound solvers, but Theorem~\ref{thm:OptISTA-rate}, by itself, does not guarantee global optimality. {\color{black}{In other words, the proof of Theorem~\ref{thm:OptISTA-rate}, by itself, shows a valid upper bound on the last-iterate function value suboptimality, but it does not show whether or not the proof can be improved.}} Rather, global optimality is established through the matching lower bound of Section~\ref{s:4}. In Section~\ref{s:4}, we prove that all deterministic first-order methods (precisely defined later) respect a lower bound that exactly matches the upper bound of Theorem~\ref{thm:OptISTA-rate}. Since \eqref{eq:FSFOM} are instances of deterministic first-order methods, global optimality is proved.

\section{Exact matching lower bounds}
\label{s:4}
In this section, we present complexity lower bounds that establish exact optimality of \ref{eq:OptISTA_Alg} and \ref{alg:OPPA}. The explicit construction uses double-function semi-interpolated zero-chain construction, which extends the prior construction of \cite{drori2022LowerBoundITEM} to the double-function setup.


\subsection{Composite optimization lower bound}
Let $N>0$ be the total iteration count and $x_0=z_0\in \rl^d$ be a starting point.
We say a method satisfies the \emph{double-function span condition} if it produces an output $x_N$ satisfying:
\begin{alignat*}{3}
&\delta_i\in \{0,1\} &&\textup{ for }i=0,\dots,2N-1,\\
& \sum^{2N-1}_{i=0}\delta_i=\sum^{2N-1}_{i=0}(1-\delta_i)=N,
&&
\!\!\!\!\!\!\!\!\!\!\!\!\!\!\!\!\!\!\!\!\!\!\!\!\!\!\!\!\!\!\!\!\!\!\!\!\!\!\!{\color{gray}\rhd\;\textup{exactly $N$ evaluations of $\nabla f$ and $\prox_{\gamma_i h}$ in any order}}
\\
&d_i=\left\{
\begin{array}{ll}
\nabla f(z_i) &\textup{ if $\delta_i=0$, }\\
 z_i-\prox_{\gamma_i h}(z_i) &\textup{for some $\gamma_i>0$, if $\delta_i=1$,}
\end{array}
\right.
&&\textup{ for }i=0,\dots,2N-1,\\
&z_{i}\in x_0+\mathrm{span}\{d_0,\dots,d_{i-1}\}
&&\textup{ for }i=1,\dots,2N-1,\\
&x_N\in x_0+\mathrm{span}\{d_0,\dots,d_{2N-1}\}.
&&
\!\!\!\!\!\!\!\!\!\!\!\!\!\!\!\!\!\!\!\!\!\!\!\!\!\!\!\!\!\!\!\!\!\!\!\!\!\!\!{\color{gray}\rhd\;\textup{call output of method $x_N$ for consistency with Section~\ref{s:3}}}
\end{alignat*}

To clarify, we make no assumptions about the order in which the gradient and proximal oracles are used. We only assume that the gradient and proximal oracles are each called exactly $N$ times.

\begin{theorem}
\label{thm:OptISTA-lb}
Let $L>0$, $R>0$, $N>0$, and $d\ge N+1$.
Let $x_0\in \rl^d$ be any starting point and $x_N$ be generated by a method satisfying the double-function span condition. Then, there is an $f\colon \rl^d\rightarrow\rl$ that is $L$-smooth and convex and an $h\colon \rl^d\rightarrow\rl\cup\{\infty\}$ that is an indicator function of a nonempty closed convex set such that there is an $x_\star\in \argmin(f+h)$ satisfying $\|x_0-x_\star\|= R$ and
\[
f(x_{N})+h(x_{N})-f(x_\star)-h(x_\star)\ge\frac{L\|x_{0}-x_{\star}\|^{2}}{2(\theta_{N}^{2}-1)},
\]
where $\theta_{N}$ is as defined for \ref{eq:OptISTA_Alg}.
\end{theorem}

Prior lower bounds for the single-function setup of minimizing an $L$-smooth convex function $f$ (without $h$) are applicable to \ref{eq:OptISTA_Alg}.
Such bounds include the classical $3L\|x_0-x_\star\|^2/(32(N+1)^2)$ bound by \cite[Theorem 2.1.7]{nesterov2003Introductory} and the more modern $L\|x_{0}-x_{\star}\|^{2}/(2\theta_{N}^{2})$ bound by \cite{drori2017LowerBoundOGM}. However, the composite optimization setup is a harder problem setup, and our lower bound, based on a double-function construction, establishes a stronger (larger) lower bound and exactly matches the rate of \ref{eq:OptISTA_Alg}.

The fact that our construction uses an $h$ that is an indicator function (so that $\prox_{\gamma_ih}$ is a projection onto a convex set) implies that \ref{eq:OptISTA_Alg} is also an optimal projected-gradient type method. Loosely speaking, if we write $\mathcal{PC}$ to denote the problem complexity, then
\[
\mathcal{PC}\Big(\begin{array}{ll}
\underset{x}{\mbox{minimize}} & f(x)\end{array}\Big)<\mathcal{PC}\Big(\begin{array}{ll}
\underset{x\in C}{\mbox{minimize}} & f(x)\end{array}\Big)=\mathcal{PC}\Big(\begin{array}{ll}
\underset{x}{\mbox{minimize}} & f(x)+h(x)\end{array}\Big)
\]
where $C$ denotes nonempty closed convex sets, if the methods utilize $\nabla f$, $\Pi_C$, and $\prox_{\gamma_i h}$ as oracles, where $\Pi_C$ denotes the projection onto $C$.

\subsection{Double-function semi-interpolated zero-chain construction}
Next, we describe the double-function semi-interpolated zero-chain construction used to establish the lower bound of Theorem~\ref{thm:OptISTA-lb}. Let $I=[0:N]\cup\{\star\}$, let $x_i,g_i\in\mathbb{R}^{N+1}$, $f_i\in\mathbb{R}$ for $i\in I$, and let $\Delta_I=\{ \alpha\in\rl^{|I|} \,|\, \sum_{i\in I}\alpha_i=1, \ \alpha_i\geq0 \text{ for } i\in I\}$.
We define $f\colon\mathbb{R}^{N+1} \rightarrow \mathbb{R}$ as
\[
f(x)=\max_{\alpha\in\Delta_{I}}\tfrac{L}{2}\|x\|^{2}-\tfrac{L}{2}\|x-\tfrac{1}{L}\sum_{i\in I}\alpha_{i}g_{i}\|^{2}+\sum_{i\in I}\alpha_{i}\left(f_{i}+\tfrac{1}{2L}\|g_{i}-Lx_{i}\|^{2}-\tfrac{L}{2}\|x_{i}\|^{2}\right),
\]
 and $h\colon\mathbb{R}^{N+1} \rightarrow \mathbb{R}\cup\{\infty\}$ as $h(x)=\delta_{C}(x)$, where $\delta_{C}$ is the indicator function of the convex set $C=x_{\star}+\mathrm{cone}\{g_{0},g_{1},\ldots,g_{N}\}$ (here $\textrm{cone}$ denotes the set of conic combinations), satisfying $\delta_{C}(x)=0$ if $x\in C$ and $\delta_{C}(x)=\infty$
if $x\notin C$. The construction of $f$ is the semi-interpolated zero-chain construction and it is $L$-smooth and convex \cite[Theorem 1]{drori2022LowerBoundITEM}.

In the single-function semi-interpolated zero-chain construction of \cite{drori2022LowerBoundITEM}, the authors carefully choose $x_i,g_i$ and $f_i$ so that the corresponding parameterized function $f$ satisfies so-called \emph{first-order zero-chain property}:
\[
\!\!\!\!\!\!
x\in\mathrm{span}\{e_{0},\dots,e_{i-1}\}\Rightarrow\big[\begin{array}{c}
\nabla f(x)\in\mathrm{span}\{e_{0},\dots,e_{i}\}\end{array}\big]
\]
for all $i\in [0:N]$ and $\gamma>0$ (we use the convention $\mathrm{span}\{\}=\{0\}$) and standard unit vectors $e_0,\dots,e_N\in \mathbb{R}^{N+1}$. This property makes each iteration of a first-order method contained in a subspace spanned by unit vectors. However, similar properties have not been known for the proximal operator of a possibly nonsmooth CCP function. Our contribution extends the single-function semi-interpolated zero-chain construction of [26] to identify conditions under which the proximal operators also exhibit zero-chain-like properties on $\delta_C$, and the following lemma demonstrates this.
\begin{lem}
\label{lem:lower-bound-main-body}
Let $e_0,\dots,e_N\in \mathbb{R}^{N+1}$ be the standard unit vectors. 
    If 
\begin{align*}
    &g_i=La_ie_i, \ a_i>0 , \ \text{ for } i\in[0:N] ,\\
   &x_0=0, \qquad x_i\in-\mathrm{cone}\{e_0,\ldots,e_{i-1}\}\ \text{ for }i\in[1:N], \qquad x_\star\in -\mathrm{cone}\{e_0,e_1,\ldots,e_N\},\\
   &f_j-\tfrac{1}{L} \langle g_i, g_j \rangle +\tfrac{1}{2L}\|g_j-Lx_j\|^2 -\tfrac{L}{2}\|x_j\|^2\geq f_k-\tfrac{1}{L} \langle g_i, g_k \rangle +\tfrac{1}{2L}\|g_k-Lx_k\|^2 -\tfrac{L}{2}\|x_k\|^2 \nonumber \\
   &\qquad \qquad \qquad \qquad \qquad \qquad \qquad \qquad \quad \text{ for } \ i\in[0:N],\ j\in[0:N-1],\ k\in[j+1:N],\\
&\sigma_i\geq0,\ i\in[0:N],\qquad \sum_{i=0}^{N}\sigma_i=1,\qquad    g_\star=\sum_{i=0}^N\sigma_ig_i, \\
     &\sum_{i=0}^N \sigma_i \left(
 f_i+\tfrac{1}{2L}\|g_i-Lx_i\|^2-\tfrac{L}{2}\|x_i\|^2
 \right) =  f_\star+\tfrac{1}{2L}\|g_\star-Lx_\star\|^2-\tfrac{L}{2}\|x_\star\|^2,
\end{align*}
then, for all $i\in [0:N]$ and $\gamma>0$ (we use the convention $\mathrm{span}\{\}=\{0\}$),
\[
\!\!\!\!\!\!
x\in\mathrm{span}\{e_{0},\dots,e_{i-1}\}\Rightarrow\big[\begin{array}{c}
\nabla f(x)\in\mathrm{span}\{e_{0},\dots,e_{i}\}\end{array}\big]
\]
\[
x\in\mathrm{span}\{e_{0},\dots,e_{i-1}\}\Rightarrow\big[\begin{array}{c}
\prox_{\gamma h}(x)=\prox_{h}(x)\in\mathrm{span}\{e_{0},\dots,e_{i-1}\}\end{array}\big]
\]
and
\[
\inf_{x\in \mathrm{span}\{e_0,\dots,e_{N-1}\}}f(x)\ge f_N.
\]
\end{lem}
To clarify, the $x_0,\dots,x_N$ in Lemma~\ref{lem:lower-bound-main-body} is unrelated to the $x$-, $y$-, and $z$-iterates in \ref{eq:OptISTA_Alg}. We devote the next section to the proof.

\subsection{Proof of Lemma~\ref{lem:lower-bound-main-body} and Theorem~\ref{thm:OptISTA-lb}}\label{s:e-1}
Let $I=[0:N]\cup\{\star\}$, and let $x_i,g_i\in\mathbb{R}^{N+1}$ and $f_i\in\mathbb{R}$ for $i\in I$.
As in the previous section, define the $L$-smooth convex function  $f\colon\mathbb{R}^{N+1} \rightarrow \mathbb{R}$ as
\[
v(x,\alpha)=\tfrac{L}{2}\|x\|^{2}-\tfrac{L}{2}\|x-\tfrac{1}{L}\sum_{i\in I}\alpha_{i}g_{i}\|^{2}+\sum_{i\in I}\alpha_{i}\left(f_{i}+\tfrac{1}{2L}\|g_{i}-Lx_{i}\|^{2}-\tfrac{L}{2}\|x_{i}\|^{2}\right),
\]
\[
f(x)=\max_{\alpha\in\Delta_{I}}v(x,\alpha)=\max_{\alpha\in\Delta_{I}}\tfrac{L}{2}\|x\|^{2}-\tfrac{L}{2}\|x-\tfrac{1}{L}\sum_{i\in I}\alpha_{i}g_{i}\|^{2}+\sum_{i\in I}\alpha_{i}\left(f_{i}+\tfrac{1}{2L}\|g_{i}-Lx_{i}\|^{2}-\tfrac{L}{2}\|x_{i}\|^{2}\right),
\]
 and $h\colon\mathbb{R}^{N+1} \rightarrow \mathbb{R}\cup\{\infty\}$ as $h(x)=\delta_{C}(x)$, where $C=x_{\star}+\mathrm{cone}\{g_{0},g_{1},\ldots,g_{N}\}$. 
We start with some preliminary lemmas.
\begin{lem}\label{lem:fiszerochain}
   Let $\Delta_{I\setminus\{\star\}}=\{ (\tilde{\alpha}_0, \ldots, \tilde{\alpha}_N ) \,|\, \sum_{i=0}^N\tilde{\alpha}_i=1 ,\ \tilde{\alpha}_i\geq0 \text{ for } i\in [0:N]\}$.  Suppose,
\[
 \tilde{f}(x)=\max_{\tilde{\alpha}\in \Delta_{I\setminus \{\star\}}}v(x,\tilde{\alpha})=\max_{\tilde{\alpha}\in \Delta_{I\setminus \{\star\}}}
 \tfrac{L}{2}\|x\|^2-\frac{L}{2}\|x-\tfrac{1}{L}\sum_{i\in I\setminus\{\star\}}\tilde{\alpha}_i g_i\|^2+
 \sum_{i\in I\setminus\{\star\}}\tilde{\alpha}_i\left(
 f_i+\tfrac{1}{2L}\|g_i-Lx_i\|^2-\tfrac{L}{2}\|x_i\|^2
 \right).
\]
and
\begin{align*}
    &x_0=0,\\
    &\langle g_i, g_j \rangle =   \langle g_i, g_k\rangle \ \mathrm{for} \   0\leq i<j\leq N-1,  \ k\in K_j,\\
    &f_j-\frac{1}{L} \langle g_i, g_j \rangle +\frac{1}{2L}\|g_j-Lx_j\|^2 -\frac{L}{2}\|x_j\|^2\geq f_k-\frac{1}{L} \langle g_i, g_k \rangle +\frac{1}{2L}\|g_k-Lx_k\|^2 -\frac{L}{2}\|x_k\|^2 \nonumber \\
    &\mathrm{for} \   i \in I, \ j\in[0:N-1], \ k\in K_j,
\end{align*}
and $g_j$ is linearly separable from $\{g_k\}_{k\in K_j}$, where
\[
K_j=\{k\in I\setminus\{\star\} : g_k\  \mathrm{is \ linearly \ independent \ of} \{g_0, \ldots, g_j\}\}.
\]
Then if $x\in\mathrm{span}\{g_0,\ldots,g_{i-1}\}$, then $\nabla f(x)\in\mathrm{span}\{g_0,\ldots,g_{i}\}$ for $i\in[0:N-1]$.
\end{lem}
\begin{proof}
    This is an instance of \cite[Theorem 3]{drori2022LowerBoundITEM}.
\qed\end{proof}
\begin{lem}\label{lem:convexinterpolation}
The following inequality holds for any $x\in\rl^d$. 
\[
f(x)\geq f_i+\langle g_i,x-x_i\rangle ,\ \text{ for } i\in I.
\]
\end{lem}
\begin{proof}
     We refer the reader to \cite[Theorem 1]{drori2022LowerBoundITEM}.
\qed\end{proof}
\begin{lem}\label{lem:ftildeisf}
Let $\Delta_{I\setminus\{\star\}}=\{ (\alpha_0, \ldots, \alpha_N ) \,|\, \sum_{i=0}^N\alpha_i=1 \ \alpha_i\geq0 \text{ for } i\in [0:N]\}$ and let 
\[
 \tilde{f}(x)=\max_{\tilde{\alpha}\in \Delta_{I\setminus \{\star\}}}v(x,\tilde{\alpha})=\max_{\tilde{\alpha}\in \Delta_{I\setminus \{\star\}}}
 \tfrac{L}{2}\|x\|^2-\frac{L}{2}\|x-\tfrac{1}{L}\sum_{i\in I\setminus\{\star\}}\tilde{\alpha}_i g_i\|^2+
 \sum_{i\in I\setminus\{\star\}}\tilde{\alpha}_i\left(
 f_i+\tfrac{1}{2L}\|g_i-Lx_i\|^2-\tfrac{L}{2}\|x_i\|^2
 \right).
\]
If
\begin{align*}
    &g_\star=\sum_{i=0}^N \sigma_ig_i \ \text{for some} \  \sigma_i\geq 0, \ \sum_{i=0}^N \sigma_i=1,\\
    &\sum_{i=0}^N \sigma_i \left(
 f_i+\frac{1}{2L}\|g_i-Lx_i\|^2-\frac{L}{2}\|x_i\|^2
 \right) =  f_\star+\frac{1}{2L}\|g_\star-Lx_\star\|^2-\frac{L}{2}\|x_\star\|^2,
\end{align*}
Then $f\equiv\tilde{f}$.
\end{lem}
 \begin{proof}
     It is clear that $\tilde{f}\leq f$ pointwise. Now 
     fix $x$ and let $\alpha^\star=(\alpha_0^\star, \ldots, \alpha_N^\star, \alpha_\star^\star)\in\Delta_{I}$ be optimal for $x$ in $f$. Then define \[
     \tilde{\alpha}^\star=(\alpha_0^\star+\sigma_0\alpha_\star^\star,\ldots,\alpha_N^\star+\sigma_N\alpha_\star^\star)\in \Delta_{I\setminus\{\star\}}.
     \]
     Under the assumptions of the lemma,
     \[
     v(x,\tilde{\alpha})\leq \tilde{f}(x)\leq f(x)=v(x,\alpha^\star)=v(x,\tilde{\alpha}^\star) \quad \forall \tilde{\alpha}\in \Delta_{I\setminus\{\star\}}.
     \]
     Therefore $\tilde{\alpha}^\star$ is optimal for $x$ in $\tilde{f}$ and $f(x)=\tilde{f}(x)$.
 \qed\end{proof}
\begin{lem}\label{lem:indicator2}
Suppose $\{g_i\}_{i\in[0:N]}$ are orthogonal and 
\[
    x_\star=-\sum_{i=0}^N s_ig_i \ \text{for some} \  s_i\geq 0.
\]
Then, for any $\gamma>0$ and $J\subseteq [0:N]$,
if $x\in\mathrm{span}\{g_i\}_{i\in J}$, then
$\mathbf{prox}_{\gamma h}(x)\in\mathrm{span}\{g_i\}_{i\in J}$.

\end{lem}
\begin{proof}
Let $h_i(x)=\delta_{\{x\,|\, x\geq-s_i\}}(x)$ for $i\in[0:N]$. If we represent $x$ and $x_\star$ into coordinates with respect to orthogonal basis $\{g_i\}_{i\in [0:N]}$, i.e., 
\[
x=(t_0,t_1,\ldots,t_N) \quad
x_\star=(-s_0,-s_1,\ldots,-s_N).
\]
Then $x\in C$ if and only if $t_i\geq -s_i$ for $i\in[0:N]$. 
    Also, It is well known that
    \[
    \mathbf{prox}_{\gamma h}(x)=\Pi_C(x)=\argmin_{y\in C}\|y-x\|^2
    \]
and $\mathbf{prox}_{\gamma h}(\cdot)$ splits coordinate-wise:
\[
\mathbf{prox}_{\gamma h}(x)= ( \mathbf{prox}_{\gamma h_0}(t_0), \ldots, \mathbf{prox}_{\gamma h_N}(t_N) ),
\]
\[
 \mathbf{prox}_{\gamma h_i}(t_i)=\argmin_{y_i\geq -s_i}\|y_i-x_i\|^2.
\]
Then if $i\notin J$, 
\[
\mathbf{prox}_{\gamma h_i}(t_i)=\mathbf{prox}_{\gamma h_i}(0)=0.
\]
Therefore $\mathbf{prox}_{\gamma h}(x)\in\{g_i\}_{i\in J}$.
\qed\end{proof}

Now we can prove Lemma~\ref{lem:lower-bound-main-body}.

\begin{proof}\textit{of Lemma~\ref{lem:lower-bound-main-body}.}
Recall the conditions of Lemma~\ref{lem:lower-bound-main-body}:
\begin{align}
    &g_i=La_ie_i, \ a_i>0 , \ \text{ for } i\in[0:N] ,\label{eq:lem1orthogonal}\\
   &x_0=0, \qquad x_i\in-\mathrm{cone}\{e_0,\ldots,e_{i-1}\}\ \text{ for }i\in[1:N], \qquad x_\star\in-\mathrm{cone}\{e_0,e_1,\ldots,e_N\}, \label{eq:lem1condonx}\\
   &f_j-\tfrac{1}{L} \langle g_i, g_j \rangle +\tfrac{1}{2L}\|g_j-Lx_j\|^2 -\tfrac{L}{2}\|x_j\|^2\geq f_k-\tfrac{1}{L} \langle g_i, g_k \rangle +\tfrac{1}{2L}\|g_k-Lx_k\|^2 -\tfrac{L}{2}\|x_k\|^2 \nonumber \\
   &\qquad \qquad \qquad \qquad \qquad \qquad \qquad \qquad \quad \text{ for } \ i\in[0:N],\ j\in[0:N-1],\ k\in[j+1:N],\label{eq:lem1zerochain}\\
&\sigma_i\geq0,\ i\in[0:N],\qquad \sum_{i=0}^{N}\sigma_i=1,\qquad    g_\star=\sum_{i=0}^N\sigma_ig_i, \label{eq:lem1convexsum}\\
&\sum_{i=0}^N \sigma_i \left(
 f_i+\tfrac{1}{2L}\|g_i-Lx_i\|^2-\tfrac{L}{2}\|x_i\|^2
 \right) =  f_\star+\tfrac{1}{2L}\|g_\star-Lx_\star\|^2-\tfrac{L}{2}\|x_\star\|^2.\label{eq:lem1convexsumfunction}
\end{align}
Under the orthogonality assumption \eqref{eq:lem1orthogonal}, 
\[
K_j=\{k\in I\setminus\{\star\} : g_k\  \mathrm{is \ linearly \ independent \ of} \{g_0, \ldots, g_j\}\}=[\, j+1:N],
\]
\[
\langle g_i, g_j \rangle =   \langle g_i, g_k\rangle=0 \ \mathrm{for} \   0\leq i<j\leq N-1,  \ k\in K_j,
\]
and $g_j$ is linearly separable from $\{g_k\}_{k\in K_j}$. 
Together with \eqref{eq:lem1zerochain}, we can apply Lemma~\ref{lem:fiszerochain}. Therefore, we get: 
\[
x\in\mathrm{span}\{e_{0},\dots,e_{i-1}\}\Rightarrow
\nabla \tilde{f}(x)\in\mathrm{span}\{e_{0},\dots,e_{i}\} \text{ for } i\in[0:N-1].
\] Additionally, we also have \eqref{eq:lem1convexsum} and \eqref{eq:lem1convexsumfunction}, so can apply Lemma~\ref{lem:ftildeisf} to conclude that
\[
x\in\mathrm{span}\{e_{0},\dots,e_{i-1}\}\Rightarrow
\nabla f(x)=\nabla \tilde{f}(x)\in\mathrm{span}\{e_{0},\dots,e_{i}\}\text{ for } i\in[0:N-1].
\]
For $i=N$, it is immediate from the fact that $\nabla f(x)=\sum_{i\in I}\alpha_i^\star g_i$ where $\alpha^\star$
is optimal for $x$. Now using the orthogonality condition \eqref{eq:lem1orthogonal} and condition on $x_\star$ of \eqref{eq:lem1condonx}, we can apply Lemma~\ref{lem:indicator2} to get
\[
x\in\mathrm{span}\{e_{0},\dots,e_{i}\}\Rightarrow\prox_{\gamma h}(x)=\prox_{h}(x)\in\mathrm{span}\{e_{0},\dots,e_{i}\} \text{ for } i\in[0:N],
\]
for any $\gamma>0$. Now we are left to show
\[
\inf_{x\in \mathrm{span}\{e_0,\dots,e_{N-1}\}}f(x)\ge f_N.
\] 
By Lemma~\ref{lem:convexinterpolation}, we have
\[
f(x)\geq f_N+ \langle g_N, x- x_N\rangle.
\]
We take $\inf_{x\in \mathrm{span}\{e_0,\dots,e_{N-1}\}}$ on both sides to conclude  
\begin{align*}
    \inf_{x\in \mathrm{span}\{e_0,\dots,e_{N-1}\}}f(x)\ge f_N+\inf_{x\in \mathrm{span}\{e_0,\dots,e_{N-1}\}}\langle g_N, x \rangle-\langle g_N, x_N \rangle =f_N-\langle g_N, x_N \rangle=f_N.
\end{align*}
where the first equality is from orthogonality of \eqref{eq:lem1orthogonal} and the second equality is from condition on $x_N$ in \eqref{eq:lem1condonx}.
\qed\end{proof}

It remains to find a choice of $\sigma_0,\dots,\sigma_N$, $a_0,\dots,a_N$, $x_0,\dots,x_N,x_\star$, $g_\star$, $f_0,\dots,f_N$, and $f_\star$ while satisfying the constraints of Lemma~\ref{lem:lower-bound-main-body}. Since the value of $f_N$ serves as a lower bound for \ref{eq:OptISTA_Alg}, we find the choice that maximizes $f_N$.
In Appendix~\ref{s:d}, we show that the choice below leads to the lower bound of Theorem~\ref{thm:OptISTA-lb}:
\begin{align*}
& \sigma_{i}=\tfrac{2\theta_{i}}{\theta_{N}^{2}},\ i\in[0:N-1],\qquad \sigma_{N}=\tfrac{1}{\theta_{N}},\\
 & \zeta_{N+1}=\tfrac{(\theta_{N}-1)R^2}{\theta_{N}^{2}(2\theta_{N}-1)},\quad\zeta_{N}=\tfrac{\theta_{N}}{\theta_{N}-1}\zeta_{N+1},\quad\zeta_{i}=\tfrac{2\theta_{i}}{2\theta_{i}-1}\zeta_{i+1}\ ,i\in[0:N-1],\\
 & a_{i}=\tfrac{1}{\theta_{N}^{2}-1}\cdot\tfrac{\zeta_{i}}{\sigma_{i}\sqrt{\zeta_{i}-\zeta_{i+1}}},\ i\in[0:N],\qquad x_{i}=-(\theta_{N}^{2}-1)\sum_{k=0}^{i-1}\sigma_{k}a_{k}e_{k},\\
 & f_{i}=\tfrac{L}{2}a_{i}^{2}(4\theta_{i}-1)-\tfrac{LR^2}{2(\theta_{N}^{2}-1)^{2}}\,,i\in[0:N-1],\qquad f_{N}=\tfrac{LR^2}{2(\theta_{N}^{2}-1)},\\
 & x_{\star}=-(\theta_{N}^{2}-1)\sum_{k=0}^{N}\sigma_{k}a_{k}e_{k},\qquad g_{\star}=-\tfrac{L}{\theta_{N}^{2}-1}x_{\star}, \qquad f_{\star}=0.
\end{align*}

\subsection{Proximal minimization lower bound}
\label{ss:oppa-lb}
Consider the problem
\[
\begin{array}{ll}
\underset{x\in\mathbb{R}^{d}}{\mbox{minimize}} & h(x),\end{array}
\]
where $h\colon\rl^{d}\to\rl\cup\{\infty\}$ is closed, convex, and proper (possibly nonsmooth). Assume a minimizer $x_{\star}$ exists (not necessarily unique).

Let $N>0$ be the total iteration count. Let the proximal stepsizes $\gamma_{0},\gamma_{1},\ldots,\gamma_{N-1}$ be pre-specified positive numbers.
The Optimized Proximal Point Algorithm (\ref{alg:OPPA}), presented by Monteiro-Svaiter \cite{Monteiro2013hybrid} and Barr\'e-Taylor-Bach \cite{Barre2023principle}, has the form 
\begin{equation}
    \begin{aligned}
 &y_{i+1}=\prox_{\gamma_{i}h}\left(x_{i}\right)\\
 &x_{i+1}=y_{i+1}+\frac{\rho_{i+1}(\eta_{i}-\rho_{i})}{\rho_{i}\eta_{i+1}}\left(y_{i+1}-y_{i}\right)+\frac{\rho_{i+1}\eta_{i}}{\rho_{i}\eta_{i+1}}\left(y_{i+1}-x_{i}\right)
\end{aligned}\tag{OPPA}\label{alg:OPPA}
\end{equation}
for $i=0,\dots,N-1$, where $x_0=y_0\in \rl^d$ is a starting point and
\begin{equation}\label{eq:rhoeta}
\rho_{i}=\frac{\gamma_{i}}{\gamma_{0}}\,\,\,\textup{for }i=0,\dots,N-1,\qquad
\begin{aligned} 
 & \eta_{i}=\begin{cases}
1 & \textup{if }i=0,\\
\frac{\rho_{i}+\sqrt{\rho_{i}^{2}+\frac{4\rho_{i}}{\rho_{i-1}}\eta_{i-1}^{2}}}{2} & \textup{if }1\leq i\leq N-1.
\end{cases}
\end{aligned}
\end{equation}
To clarify, this version of \ref{alg:OPPA} is equivalent to A-HPE of \cite{Monteiro2013hybrid} and ORI-PPA \cite{Barre2023principle} in the specific cases detailed in Section~\ref{ss:prior-work}, but is expressed slightly differently. The prior results \cite[Theorem 1]{Barre2023principle} establishes the rate:
\[
h(y_{N})-h(x_\star) \leq\frac{\gamma_{N-1}\|x_{0}-x_{\star}\|^{2}}{4\gamma_{0}^{2}\eta_{N-1}^{2}}.
\]

We note that \eqref{eq:OptISTA_Alg} does not reduce to \eqref{alg:OPPA} when $f=0$. More precisely, given $f=0$, $N\ge 2$, and any $L>0$, there is no choice of $\gamma_0,\dots, \gamma_{N-1}$ for \eqref{alg:OPPA} that makes  \eqref{alg:OPPA} equivalent to \eqref{eq:OptISTA_Alg} with the  $\gamma_0,\dots, \gamma_{N-1}$ as specified by the definition of \eqref{eq:OptISTA_Alg}.
(Recall, \eqref{alg:OPPA} allows the user to choose $\gamma_0,\dots, \gamma_{N-1}$ while \eqref{eq:OptISTA_Alg} specifies the values for $\gamma_0,\dots, \gamma_{N-1}$.)
This contrasts with how \eqref{eq:OptISTA_Alg} reduces to OGM when $g=0$.

We say a method satisfies the \emph{proximal span condition} if it produces an output $x_N$ satisfying:
\begin{alignat*}{3}
&x_{i}\in x_0+\mathrm{span}\{
x_0-\prox_{\gamma_0 h}(x_0)
,\dots,
x_{i-1}-\prox_{\gamma_{i-1} h}(x_{i-1})
\}
\qquad\textup{ for }i=1,\dots,N.
\end{alignat*}

\begin{theorem}
\label{thm:opm-lb}
Let $R>0$, $N>0$, and $d\ge N+1$.
Let $x_0\in \rl^d$ be any starting point and $x_N$ be generated by an algorithm satisfying the proximal span condition. Then, {\color{red}{for any $\varepsilon>0$}}, there is a closed convex proper function $h\colon \rl^d\rightarrow\rl\cup\{\infty\}$ such that there is an $x_\star\in \argmin h$ satisfying $\|x_0-x_\star\|=R$ and
\[
h(x_{N})-h(x_\star) \geq\frac{\gamma_{N-1}\|x_{0}-x_{\star}\|^{2}}{4\gamma_{0}^{2}\eta_{N-1}^{2}}{\color{red}{-\varepsilon}}.
\]
where $\eta_{N-1}$ is as defined in \eqref{eq:rhoeta}.
\end{theorem}

We devote the next section to the proof. As a short remark, $\eta$ is a function of ``ratio of $\gamma_i$", so its direct dependence on $\gamma_i$ is somewhat complicated. However, if $\frac{\gamma_i}{\gamma_0}\approx 1$ for all $i$, then $\eta_N\approx \frac{N}{2}$ and thus we have $\mathcal{O}(1/N^2)$ lower bound on the proximal setup. This concludes that $\mathcal{O}(1/N^2)$ upper bound of G\"{u}ler's second method \cite{guler1992new} with uniform proximal stepsize is indeed optimal.

To the best of our knowledge, there are no prior lower bounds for the proximal minimization setup. There are lower bounds for the proximal point method applied to the more general monotone inclusion problems \cite{diakonikolas2022potential,park2022}, but these results lower bound the rate of convergence of fixed-point residuals rather than function values.

\subsection{Proof of Theorem \ref{thm:opm-lb}}
{\color{red}{
In this section, let $I=[0:N]\cup\{\star\}$, $h_i\in\mathbb{R}$ for $i\in I$, and $a_i\in\mathbb{R}$ for $i\in [0:N]$. Denote $e_0,\dots, e_N$ for standard unit vectors in $\mathbb{R}^{N+1}$. We assume proximal stepsizes $\{\gamma_i\}_{i\in[0:N-1]}$ are given and fixed. We define the following notations for closed, convex, and proper functions on $\rl^{N+1}$, which are parametrized by $(\{h_i\}_{i\in I}, \{a_i\}_{i\in [0:N-1]})$:
\begin{gather*}
H_i(x)\triangleq
h_i+\langle a_ie_i,x\rangle \textup{ for } i\in[0:N],
\\
H_{[0:i]}(x)\triangleq\max_{j\in [0:i]}\left\{
h_j+\langle a_je_j,x\rangle
\right\}=\max_{j\in [0:i]}H_j(x) \text{ for }i\in[0:N],\\
 H_\star(x) \triangleq h_\star, \quad H_{I}(x)\triangleq\max_{j\in I} H_j(x).
\end{gather*}
}}
We start the proof with preliminary lemmas. 
\begin{lem}\label{lem:proximalzerochain1}
    Suppose $x$ and $i\in[0:N-1]$ are fixed. If
    \[
    H_i(\mathbf{prox}_{\gamma_i H_{[0:i]}}(x)) \geq H_j(\mathbf{prox}_{\gamma_i H_{[0:i]}}(x)),
    \]
    for all $i<j\le N,  j=\star$, then
    \[
    \mathbf{prox}_{\gamma_i H_{[0:i]}}(x)=\mathbf{prox}_{\gamma_i H_I}(x).
    \]
\end{lem}
\begin{proof}
For notational convenience, let $z=\mathbf{prox}_{\gamma_i H_{[0:i]}}(x)$. Under the given condition $H_i(z)\ge H_j(z)$ for all $i<j\le N,  j=\star$, the following holds for every $\tilde{z}$:
\begin{align*}
H_{I}(z)+\frac{\|z-x\|^2}{2\gamma_i}=\max_{\ell\in I}H_\ell(z)+\frac{\|z-x\|^2}{2\gamma_i}&=\max_{\ell \in [0:i]}H_\ell(z)+\frac{\|z-x\|^2}{2\gamma_i} \\
&\leq \max_{\ell \in [0:i]}H_\ell(\tilde{z})+\frac{\|\tilde{z}-x\|^2}{2\gamma_i}\\
&\leq \max_{\ell\in I}H_\ell(\tilde{z})+\frac{\|\tilde{z}-x\|^2}{2\gamma_i}\\
&=H_I(\tilde{z})+\frac{\|\tilde{z}-x\|^2}{2\gamma_i}.
\end{align*}
The first inequality is from the definition of $z$. The second inequality follows from $[0:i]\subseteq I=\{0,1,\ldots,N,\star\}$.
       So $z$ minimizes  $H_I(\cdot)+\frac{\|\cdot-x\|^2}{2\gamma_i}$ and therefore $z=\mathbf{prox}_{\gamma_i H_{[0:i]}}(x)=\mathbf{prox}_{\gamma_i H_I}(x)$.
\qed\end{proof}
The following lemma is a well-known fact of pointwise maximum function.

\begin{lem}\label{lem:pointwisemaximum}
The following holds for $i\in [0:N-1]$:
     \[
     \partial H_{[0:i]}(x) \mathrel{\color{red}{\subseteq}} \left\{\sum_{\ell=0}^{i}\sigma_\ell a_\ell e_\ell \ |  \ \sigma_{\ell}\ge0, \ \sum_{\ell=0}^{i}\sigma_\ell=1 \right\}.
     \]
 \end{lem}
 \begin{proof}
     We refer the readers to \cite[Proposition 2.3.12]{clarke1983oan}.
 \qed\end{proof}
 Now we have the following lemma, which is analogous to Lemma~\ref{lem:lower-bound-main-body}.
\begin{lem}\label{lem:opm-lower-bound-main-body}
 If
\begin{align}
&h_i-\gamma_ia_i^2\geq h_j
\text{ for }  i\in[0:N-1] ,\  j\in\{i+1,\ldots,N,\star\},
\label{eq:proxzerochain}\end{align}
then the following holds for $i\in[0:N-1]$. We use the convention $\mathrm{span}\{\} = \{0\}$.
\[
x\in\mathrm{span}\{e_{0},\dots,e_{i-1}\}\Rightarrow \prox_{\gamma_{i} H_I}(x)\in\mathrm{span}\{e_{0},\dots,e_{i}\}.
\]
\end{lem}
\begin{proof}
    Fix $i\in[0:N-1]$ and assume $x\in\mathrm{span}\{e_0,\ldots,e_{i-1}\}$ and consider 
    $z=\mathbf{prox}_{\gamma_i H_{[0:i]}}(x)$. Then, 
    \[
    0\in \partial H_{[0:i]}(z)+\frac{1}{\gamma_i}(z-x)
    \]
    and hence
    \[
    z\in x-\gamma_i\partial H_{[0:i]}(z).
    \]
    Therefore by Lemma~\ref{lem:pointwisemaximum}, 
    \[
    z=\mathbf{prox}_{\gamma_i H_{[0:i]}}(x)=x-\gamma_i\sum_{\ell=0}^i\sigma_\ell a_\ell e_\ell,
    \]
    where $\sigma_\ell\geq0$ and $\sum_{\ell=0}^i \sigma_\ell = 1$. Now we show that $z$ satisfies the condition of Lemma~\ref{lem:proximalzerochain1}:
    {\color{red}{
    \begin{align*}
    H_i(\mathbf{prox}_{\gamma_i H_{[0:i]}}(x))&=H_i\Big(x-\gamma_i\sum_{\ell=0}^{i}\sigma_\ell a_\ell e_\ell\Big)
    = h_i+\left\langle a_i e_i, x-\gamma_i\sum_{\ell=0}^i\sigma_\ell a_\ell e_\ell\right\rangle\\
    &=h_i+\left\langle a_i e_i, x\right\rangle-\gamma_i\left\langle a_i e_i, \sum_{\ell=0}^i\sigma_\ell a_\ell e_\ell\right\rangle\\
    &=h_i-\gamma_i\sigma_ia_i^2{\color{gray}{\qquad\rhd\, \langle e_i, x \rangle=0 }}\\
    &\geq h_i-\gamma_ia_i^2.
    \end{align*}
    For $j\in\{i+1,\ldots,N\}$,
    \begin{align*}
     H_j(\mathbf{prox}_{\gamma_i H_{[0:i]}}(x))&=H_j\Big(x-\gamma_i\sum_{\ell=0}^i\sigma_\ell a_\ell e_\ell\Big)\\
    &=h_j+\left\langle a_je_j, x-\gamma_i\sum_{\ell=0}^i\sigma_\ell a_\ell e_\ell\right\rangle\\
    &=h_j+\left\langle a_je_j, x\right\rangle-\gamma_i\left\langle a_je_j, \sum_{\ell=0}^i\sigma_\ell a_\ell e_\ell\right\rangle\\
    &=h_j.{\color{gray}{\qquad\rhd\, \langle e_j, x \rangle=0 }}
    \end{align*}
For  $j=\star$, 
\[
 H_{\star}(\mathbf{prox}_{\gamma_i H_{[0:i]}}(x)) = h_\star
\]
is by definition. }}
Therefore by \eqref{eq:proxzerochain},
\[
 H_i(\mathbf{prox}_{\gamma_i H_{[0:i]}}(x)) \geq H_j(\mathbf{prox}_{\gamma_i H_{[0:i]}}(x))
\]
for all $j\in\{i+1,\ldots,N,\star\}$.
Then by Lemma~\ref{lem:proximalzerochain1}, 
\[
\mathbf{prox}_{\gamma_i H_{[0:i]}}(x)=\mathbf{prox}_{\gamma_i H_I}(x)\in\mathrm{span}\{e_0,\ldots,e_{i}\}.
\]
\qed\end{proof}
We now introduce the sufficient conditions for Lemma~\ref{lem:opm-lower-bound-main-body}. {\color{red}{
\begin{lem}\label{lem:lbOPMcor1}
   Let $a_i,b_i>0$  for $i\in[0:N]$ .
    If 
    \begin{align}
        &a_ib_i-a_{i+1}b_{i+1}-\gamma_ia_i^2\geq0, \ i\in[0:N-1] \label{eq:proxleminterpolate}
    \end{align}
  then the following choice of $\{h_i\}_{i\in I}$:
  \begin{equation}\label{eq:proxchoiceofpoints1}
  \begin{aligned}
        &h_i=a_ib_i, \, i\in[0:N],\quad h_\star=0
    \end{aligned}
    \end{equation}
    satisfies \eqref{eq:proxzerochain}, which is the condition of Lemma~\ref{lem:opm-lower-bound-main-body}.
\end{lem}

\begin{proof}
Fix $i\in[0:N-1]$ and assume $j\in\{i+1,\ldots,N\}$. Then,
    \[
    h_i  = a_ib_i
    \]
    and
    \[
    h_j  + \gamma_ia_i^2 = a_jb_j+ \gamma_ia_i^2.
    \]
    Therefore
    \begin{align*}
        h_i  -\gamma_ia_i^2-h_j &=a_ib_i-\gamma_ia_i^2-a_jb_j\\
        &\geq a_{i+1}b_{i+1}-a_jb_j\\
        &\geq a_{i+1}b_{i+1}-a_jb_j-\gamma_{i+1}a_{i+1}^2\\
        &\geq a_{i+2}b_{i+2}-a_jb_j\\
        &\geq \cdots\geq a_{j-1}b_{j-1}-a_jb_j\geq \gamma_{j-1}a_{j-1}^2\geq 0 .
    \end{align*}
If $j=\star$, then
\[
h_i  -h_\star  - \gamma_ia_i^2 = a_ib_i-\gamma_i a_i^2\geq a_{i+1}b_{i+1}\geq0.
\]
So, we have \eqref{eq:proxzerochain} for $j\in\{i+1,\ldots,N,\star\}$. 
\qed\end{proof}

Now we present our choice of  $a_0,\dots,a_N$, $b_0, \dots,b_N$ for Lemma~\ref{lem:lbOPMcor1} for given $R>0$, $N>0$, and $\varepsilon>0$. First, define the followings: 
\begin{equation}
\begin{aligned}
     & \zeta_{N}=\tfrac{\gamma_{N-1}R^2}{4\gamma_0^2\eta_{N-1}^2},\quad\zeta_{i}=\tfrac{\tfrac{2\eta_{i}}{\rho_i}}{\tfrac{2\eta_{i}}{\rho_i}-1}\zeta_{i+1}\ ,i\in[0:N-1],\\
 & \tilde{a}_{i}=\tfrac{\sqrt{\zeta_{i}-\zeta_{i+1}}}{\sqrt{\gamma_i}},\ i\in[0:N-1],\qquad \tilde{b}_{i}=\tfrac{\sqrt{\gamma_i}\zeta_{i}}{\sqrt{\zeta_{i}-\zeta_{i+1}}},\ i\in[0:N-1],\\
 &\tilde{\varepsilon}=\mathrm{min}\{\tfrac{\varepsilon}{2}, \tfrac{R^2}{\zeta_N}, R\},
\end{aligned}
\label{eq:proxchoiceofpoints-1}
\end{equation}
where $\{\rho_i\}_{i\in[0:N-1]}$ and $\{\eta_i\}_{i\in[0:N-1}$ are as defined in \eqref{eq:rhoeta}. Note that $\{\zeta_i\}$ is strictly decreasing by definition and thus $\tilde{a}_i$ and $\tilde{b}_i$ are well-defined. Then, define  $a_0,\dots,a_N$, $b_0, \dots,b_N$ as:
\begin{equation}
\begin{aligned}
 & a_{i}=\tfrac{\sqrt{R^2-\tilde{\varepsilon}^2}}{R}\tilde{a}_i,\ i\in[0:N-1],\qquad a_{N}=\tfrac{R^2-\tilde{\varepsilon}^2}{R^2}\cdot\tfrac{\zeta_N-\tilde{\varepsilon}}{\tilde{\varepsilon}},\\
 &{b}_{i}=\tfrac{\sqrt{R^2-\tilde{\varepsilon}^2}}{R}\tilde{b}_i,\ i\in[0:N-1],\qquad b_{N}=\tilde{\varepsilon}.
\end{aligned}
\label{eq:proxchoiceofpoints-2}
\end{equation}
\begin{lem}\label{lem:proxDrorilem}
The following holds for $a_i$ and $b_i$ in \eqref{eq:proxchoiceofpoints-2}.
\begin{align*}  
&a_ib_i-a_{i+1}b_{i+1}-\gamma_ia_i^2\ge0, \ i\in[0:N-1],\\
&\sum_{i=0}^{N}b_i^2=R^2.
\end{align*}
\end{lem}
\begin{proof}
By the definition, for $i\in[0:N-2]$,
    \[
    a_{i+1}b_{i+1}=\frac{R^2-\tilde{\varepsilon}^2}{R^2}\tilde{a}_{i+1}\tilde{b}_{i+1}=\frac{R^2-\tilde{\varepsilon}^2}{R^2}\zeta_{i+1},
    \]
    and
    \[
    a_{i}b_{i}-\gamma_ia_i^2=\frac{R^2-\tilde{\varepsilon}^2}{R^2}\zeta_i-\gamma_i\frac{R^2-\tilde{\varepsilon}^2}{ R^2}\frac{\zeta_i-\zeta_{i+1}}{\gamma_i}=\frac{R^2-\tilde{\varepsilon}^2}{R^2}\zeta_{i+1}.
    \]
So, $a_{i}b_{i}-a_{i+1}b_{i+1}-\gamma_i a_{i}^2=0$.
For $i=N-1$,
    \[
    a_{N}b_{N}=\frac{R^2-\tilde{\varepsilon}^2}{R^2}(\zeta_N-\tilde{\varepsilon}),
    \]
    and
    \[
    a_{N-1}b_{N-1}-\gamma_{N-1}a_{N-1}^2=\frac{R^2-\varepsilon^2}{R^2}\zeta_{N-1}-\gamma_{N-1}\frac{R^2-\varepsilon^2}{ R^2}\frac{\zeta_{N-1}-\zeta_{N}}{\gamma_{N-1}}=\frac{R^2-\varepsilon^2}{R^2}\zeta_{N}.
    \]
So, $a_{N-1}b_{N-1}-a_{N}b_{N}-\gamma_{N-1}a_{N-1}^2=\frac{R^2-\tilde{\varepsilon}^2}{R^2}\tilde{\varepsilon}\ge0$.
Now we prove the second equation. Suppose, 
\begin{align}
\sum_{i=0}^{N-1}\tilde{b}_i^2=R^2.
\label{eq:sumoftildeb}\end{align}
Then,
\[
\sum_{i=0}^N b_i^2=\sum_{i=0}^{N-1} b_i^2+\tilde{\varepsilon}^2=\frac{R^2-\tilde{\varepsilon}^2}{R^2}\sum_{i=0}^{N-1} \tilde{b}_i^2+\tilde{\varepsilon}^2=R^2.
\]
So, we prove \eqref{eq:sumoftildeb} instead:
\begin{align*}
    \sum_{i=0}^{N-1}\tilde{b}_i^2&=\sum_{i=0}^{N-1}\gamma_i\frac{\zeta_i^2}{\zeta_i-\zeta_{i+1}}
    =\sum_{i=0}^{N-1}\gamma_i\frac{\zeta_i^2}{\zeta_i(1-\frac{\frac{2\eta_i}{\rho_i}-1}{\frac{2\eta_i}{\rho_i}})}\\
    &=\sum_{i=0}^{N-1}\gamma_i\frac{\zeta_i}{\frac{1}{\frac{2\eta_i}{\rho_i}}}=\sum_{i=0}^{N-1}2\gamma_i\eta_i\zeta_i\frac{1}{\rho_i}=\gamma_0\sum_{i=0}^{N-1}2\eta_i\zeta_i.
\end{align*}
{\color{black}{
The term $\sum_{i=0}^{N-1}2\eta_i\zeta_i$ can be rearranged as:
\begin{align}
    \sum_{i=0}^{N-1}2\eta_i\zeta_i&=\sum_{i=0}^{N-1}2\eta_i\zeta_N\prod_{j=i}^{N-1}\frac{\frac{2\eta_j}{\rho_j}}{\frac{2\eta_j}{\rho_j} -1} \nonumber \\
    &=2\zeta_N\eta_{N-1}\frac{\frac{2\eta_{N-1}}{\rho_{N-1}}}{\frac{2\eta_{N-1}}{\rho_{N-1}} -1}+2\zeta_N\sum_{i=0}^{N-2}\frac{\frac{2\eta_i^2}{\rho_i}}{\frac{2\eta_i}{\rho_i} -1}\prod_{j=i+1}^{N-1}\frac{\frac{2\eta_j}{\rho_j}}{\frac{2\eta_j}{\rho_j} -1}\nonumber\\
    &=2\zeta_N\eta_{N-1}\frac{\frac{2\eta_{N-1}}{\rho_{N-1}}}{\frac{2\eta_{N-1}}{\rho_{N-1}} -1}+2\zeta_N\sum_{i=0}^{N-2}\frac{\frac{2\eta_{i+1}^2}{\rho_{i+1}}-2\eta_{i+1}}{\frac{2\eta_{i}}{\rho_i} -1}\prod_{j=i+1}^{N-1}\frac{\frac{2\eta_j}{\rho_j}}{\frac{2\eta_j}{\rho_j} -1}{\color{gray}\quad\rhd\,\textup{using } \frac{2\eta_{i}^2}{\rho_{i}}=\frac{2\eta_{i+1}^2}{\rho_{i+1}}-2\eta_{i+1}} \nonumber\\
    &=2\zeta_N\eta_{N-1}\frac{\frac{2\eta_{N-1}}{\rho_{N-1}}}{\frac{2\eta_{N-1}}{\rho_{N-1}} -1}+2\zeta_N\sum_{i=0}^{N-2}\frac{\frac{2\eta_{i+1}}{\rho_{i+1}}-2}{\frac{2\eta_{i}}{\rho_i} -1}\eta_{i+1}\prod_{j=i+1}^{N-1}\frac{\frac{2\eta_j}{\rho_j}}{\frac{2\eta_j}{\rho_j} -1}\nonumber\\
     &=2\zeta_N\eta_{N-1}\frac{\frac{2\eta_{N-1}}{\rho_{N-1}}}{\frac{2\eta_{N-1}}{\rho_{N-1}} -1}+2\zeta_N\sum_{i=0}^{N-2}\frac{\frac{2\eta_{i+1}}{\rho_{i+1}}-2}{\frac{2\eta_{i}}{\rho_i} -1}\frac{\frac{2\eta_{i+1}^2}{\rho_{i+1}}}{\frac{2\eta_{i+1}}{\rho_{i+1}}-1}\prod_{j=i+2}^{N-1}\frac{\frac{2\eta_j}{\rho_j}}{\frac{2\eta_j}{\rho_j} -1}\nonumber\\
     &=\cdots\nonumber\\
     &=2\zeta_N\eta_{N-1}\frac{\frac{2\eta_{N-1}}{\rho_{N-1}}}{\frac{2\eta_{N-1}}{\rho_{N-1}} -1}+2\zeta_N\sum_{i=0}^{N-2}\prod_{j=i}^{N-2}\frac{\frac{2\eta_{j+1}}{\rho_{j+1}}-2}{\frac{2\eta_{j}}{\rho_j} -1}\eta_{N-1}\frac{\frac{2\eta_{N-1}}{\rho_{N-1}}}{\frac{2\eta_{N-1}}{\rho_{N-1}}-1}\nonumber\\
     &=2\zeta_N\eta_{N-1}\frac{\frac{2\eta_{N-1}}{\rho_{N-1}}}{\frac{2\eta_{N-1}}{\rho_{N-1}} -1}+2\zeta_N\eta_{N-1}\frac{\frac{2\eta_{N-1}}{\rho_{N-1}}}{\frac{2\eta_{N-1}}{\rho_{N-1}}-1}\sum_{i=0}^{N-2}\prod_{j=i}^{N-2}\frac{\frac{2\eta_{j+1}}{\rho_{j+1}}-2}{\frac{2\eta_{j}}{\rho_j} -1}.\label{eq:sumandproductoflbOPM}
\end{align}
Meanwhile,
\begin{align*}
    \sum_{i=0}^{N-2}\prod_{j=i}^{N-2}\frac{\frac{2\eta_{j+1}}{\rho_{j+1}}-2}{\frac{2\eta_{j}}{\rho_j} -1}=\frac{\frac{2\eta_{N-1}}{\rho_{N-1}}-2}{\frac{2\eta_{N-2}}{\rho_{N-2}} -1}\Big(\frac{\frac{2\eta_{N-2}}{\rho_{N-2}}-2}{\frac{2\eta_{N-3}}{\rho_{N-3}} -1}\Big(\cdots\Big(\frac{\frac{2\eta_{2}}{\rho_{2}}-2}{\frac{2\eta_{1}}{\rho_{1}} -1}\Big(\frac{\frac{2\eta_{1}}{\rho_{1}}-2}{\frac{2\eta_{0}}{\rho_{0}}-1 }+1\Big)+1\Big)\cdots\Big)+1\Big).
\end{align*}
Since $\frac{2\eta_0}{\rho_0}-1=1$, by the telescopic product, it reduces to
\[
 \sum_{i=0}^{N-2}\prod_{j=i}^{N-2}\frac{\frac{2\eta_{j+1}}{\rho_{j+1}}-2}{\frac{2\eta_{j}}{\rho_j} -1}=\frac{2\eta_{N-1}}{\rho_{N-1}}-2.
\]
Then \eqref{eq:sumandproductoflbOPM} reduces to 
\begin{align*}
    \sum_{i=0}^{N-1}2\eta_i\zeta_i&=2\zeta_N\eta_{N-1}\frac{\frac{2\eta_{N-1}}{\rho_{N-1}}}{\frac{2\eta_{N-1}}{\rho_{N-1}} -1}+2\zeta_N\eta_{N-1}\frac{\frac{2\eta_{N-1}}{\rho_{N-1}}}{\frac{2\eta_{N-1}}{\rho_{N-1}}-1}\sum_{i=0}^{N-2}\prod_{j=i}^{N-2}\frac{\frac{2\eta_{i+1}}{\rho_{i+1}}-2}{\frac{2\eta_{i}}{\rho_i} -1}\\
    &=2\zeta_N\eta_{N-1}\frac{\frac{2\eta_{N-1}}{\rho_{N-1}}}{\frac{2\eta_{N-1}}{\rho_{N-1}} -1}+2\zeta_N\eta_{N-1}\frac{\frac{2\eta_{N-1}}{\rho_{N-1}}}{\frac{2\eta_{N-1}}{\rho_{N-1}}-1}\Big(\frac{2\eta_{N-1}}{\rho_{N-1}}-2\Big)\\
    &=2\zeta_N\eta_{N-1}\frac{\frac{2\eta_{N-1}}{\rho_{N-1}}}{\frac{2\eta_{N-1}}{\rho_{N-1}} -1}\Bigg(1+\frac{2\eta_{N-1}}{\rho_{N-1}}-2\Bigg)\\
    &=2\zeta_N\eta_{N-1}\frac{2\eta_{N-1}}{\rho_{N-1}}=4\zeta_N\frac{\eta_{N-1}^2}{\rho_{N-1}}.
\end{align*}
}}
Hence,
\[
\sum_{i=0}^{N-1}\tilde{b}_i^2=\gamma_0\sum_{i=1}^{N-1}2\eta_i\zeta_i=4\gamma_0\zeta_N\frac{\eta_{N-1}^2}{\rho_{N-1}}=4\gamma_0\frac{\gamma_{N-1}}{4\gamma_0^2\eta_{N-1}^2}R^2\cdot\frac{\eta_{N-1}^2}{\rho_{N-1}}=R^2. \quad \qed
\]
\end{proof}

The following lemma is a restriction of Theorem \ref{thm:opm-lb} in the sense that it has a fixed starting point $x_0=0$ and a fixed dimension $d=N+1$. 
\begin{lem}\label{lem:opm-lb-res}
Let  $R>0$, $N>0$, $\varepsilon>0$ and
let $x_0=0\in \rl^{N+1}$. 
Consider a closed, convex, and proper function $H_I\colon\mathbb{R}^{N+1} \rightarrow \mathbb{R}\cup\{\infty\}$ defined as 
\begin{gather*}
H_i(x)\triangleq
h_i+\langle a_ie_i,x\rangle \textup{ for } i\in[0:N],
\\
H_{[0:i]}(x)\triangleq\max_{j\in [0:i]}\left\{
h_j+\langle a_je_j,x\rangle
\right\}=\max_{j\in [0:i]}H_j(x) \text{ for }i\in[0:N],\\
 H_\star(x) \triangleq h_\star, \quad H_{I}(x)\triangleq\max_{j\in I} H_j(x).
\end{gather*}
with the choice of  \eqref{eq:proxchoiceofpoints1} with \eqref{eq:proxchoiceofpoints-2}. Take $x_\star=-\sum_{i=0}^{N}b_ie_i$.
Then $x_\star\in \argmin H_I$ and satisfies $\|x_0-x_\star\|=R$ and
\[
H_I(x_{N})-H_I(x_\star)\ge\frac{\gamma_{N-1}\|x_0-x_\star\|^2}{4\gamma_0^2\eta_{N-1}^2}-\varepsilon
\]
for any $\{x_i\}_{i\in[0:N]}$ satisfying the following proximal span condition:
\begin{equation}
\begin{aligned}
&x_{i}\in x_0+\mathrm{span}\{
x_0-\prox_{\gamma_0 h}(x_0)
,\dots,
x_{i-1}-\prox_{\gamma_{i-1} h}(x_{i-1})
\}
\qquad\textup{ for }i=1,\dots,N.
\end{aligned}
\label{eq:proxfunctionspan}\end{equation}
\end{lem}
\begin{proof}
Note that for $i\in[0:N]$,
\[
H_i(x_\star) = h_i + \langle a_i e_i, x_\star \rangle=a_ib_i+\bigg\langle a_ie_i, -\sum_{k=0}^{N} b_ke_{k}\bigg\rangle =a_ib_i-a_ib_i=0,
\]
so $H_I(x_\star)=0$. Also note that $H_I(x)\ge H_{\star}(x) = h_{\star}$. Thus $x_\star\in\argmin h$ and $\|x_0-x_\star\|=R$ by Lemma~\ref{lem:proxDrorilem}.
 By Lemma~\ref{lem:proxDrorilem} and \ref{lem:lbOPMcor1}, we have the result of Lemma~\ref{lem:opm-lower-bound-main-body}. So, each proximal evaluation will give at most one new next coordinate. Then after $N$ proximal evaluations, the output $x_N$ will be in the span of $\{e_0,\ldots,e_{N-1}\}$ under the condition \eqref{eq:proxfunctionspan} of the lemma. Hence,
\begin{align*}
H_I(x_{N})-H_I(x_\star)&\geq \inf_{x\in \mathrm{span}\{e_0,\dots,e_{N-1}\}}H_I(x)-0\\
&\geq\inf_{x\in \mathrm{span}\{e_0,\dots,e_{N-1}\}}H_I(x)\\
&\ge h_N+\langle a_Ne_N,x\rangle\\
&=\frac{R^2-\tilde{\varepsilon}^2}{R^2}(\zeta_N-\tilde{\varepsilon})\\
&=\zeta_N-\frac{\tilde{\varepsilon}^2}{R^2}\zeta_N-\frac{R^2-\tilde{\varepsilon}^2}{R^2}\tilde{\varepsilon}\\
&\ge \zeta_N-\tilde{\varepsilon}-\tilde{\varepsilon} \ge \frac{\gamma_{N-1}\|x_0-x_\star\|^2}{4\gamma_0^2\eta_{N-1}^2}-\varepsilon,
\end{align*}
where the first inequality is from the definition of $H_I$.
\qed\end{proof}}}
Then we expand the condition of Lemma~\ref{lem:opm-lb-res} to arrive at Theorem \ref{thm:opm-lb}.
\begin{proof}\textit{of Theorem \ref{thm:opm-lb}}
Assume $d\geq N+1$. Take $H_I\in\mathcal{F}_{0,\infty}$ be function defined in Lemma~\ref{lem:opm-lb-res}, which is embedded in $\rl^d$. Call $\tilde{x}_\star$ to be the element of $\argmin_{x\in\mathbb{R}^d}H_I$ in Lemma~\ref{lem:opm-lb-res}. 
Now for arbitrary $x_0$, let $h\colon \mathbb{R}^d\rightarrow\mathbb{R}$ be translation of $H_I$ by $x_0$:
\[
h(x)=H_I(x-x_0).
\]
Then $h$ is closed, convex, and proper, and $x_\star\triangleq \tilde{x}_\star+x_0\in\argmin_{x\in\mathbb{R}^d}h$. Now assume $\{x_i\}_{i\in[0:N]}$ is produced from a method that satisfies the proximal span condition. That is,
\[
x_{i}\in x_0+\mathrm{span}\{
x_0-\prox_{\gamma_0 h}(x_0)
,\dots,
x_{i-1}-\prox_{\gamma_{i-1} h}(x_{i-1})
\}
\qquad\textup{ for }i=1,\dots,N.
\]
Then for $\tilde{x}_i\triangleq x_i-x_0$ , 
 $\{\tilde{x}_i\}_{i\in[0:N]}$ satisfies:
\[
\tilde{x}_i\in \tilde{x}_0+\mathrm{span}\left\{\tilde{x}_0-\mathbf{prox}_{\gamma_0 H_I}(\tilde{x}_0),\ldots, \tilde{x}_{i-1}-\mathbf{prox}_{\gamma_{i-1} H_I}(\tilde{x}_{i-1})\right\}.
\]
This is because $\tilde{x}_0=0$ and 
\[
\mathbf{prox}_{\gamma_i h}(x_i)-x_0=\mathbf{prox}_{\gamma_i H_I}(x_i-x_0)=\mathbf{prox}_{\gamma_i H_I}(\tilde{x}_i).
\]
So,
\[
\tilde{x_i}-\mathbf{prox}_{\gamma_i H_I}(\tilde{x}_i)=x_i-x_0-\mathbf{prox}_{\gamma_i h}(x_i)+x_0=x_i-\mathbf{prox}_{\gamma_i h}(x_i).
\]
Hence, we can apply Lemma \ref{lem:opm-lb-res} on $\{\tilde{x}_i\}_{i\in[0:N]}$ to get

\[
H_I(\tilde{x}_N)-H_I(\tilde{x}_\star)\ge \frac{\gamma_{N-1}\|0-\tilde{x}_\star\|^2}{4\gamma_0^2\eta_{N-1}^2}{\color{red}{-\varepsilon}}=\frac{\gamma_{N-1}\|x_0-{x}_\star\|^2}{4\gamma_0^2\eta_{N-1}^2}{\color{red}{-\varepsilon}}=\frac{\gamma_{N-1}R^2}{4\gamma_0^2\eta_{N-1}^2}{\color{red}{-\varepsilon}}.
\]
Then we finally get
\begin{align*}
h(x_N)-h(\tilde{x}_\star+x_0)&=H_I(\tilde{x}_N)-H_I(\tilde{x}_\star)\\
&\geq\frac{\gamma_{N-1}\|0-\tilde{x}_\star\|^2}{4\gamma_0^2\eta_{N-1}^2}{\color{red}{-\varepsilon}}=\frac{\gamma_{N-1}\|x_0-{x}_\star\|^2}{4\gamma_0^2\eta_{N-1}^2}{\color{red}{-\varepsilon}}=\frac{\gamma_{N-1}R^2}{4\gamma_0^2\eta_{N-1}^2}{\color{red}{-\varepsilon}},
\end{align*}
and $h$ is are our desired function.
\qed\end{proof}

\subsection{Generalization to deterministic first-order methods via resisting oracle technique}

Using Nemirovski's resisting oracle technique \cite{nemirovski1983problem,carmon2020stationary1}, we can extend the lower bound of Theorems~\ref{thm:OptISTA-lb} and \ref{thm:opm-lb} to a broader class of deterministic methods that do not necessarily satisfy the span condition. The methods we consider are deterministic methods whose behavior only depends on the output of the oracle evaluations. In Appendix~\ref{s:e}, we precisely define the classes as ``$N$-step deterministic double-oracle method'' and ``$N$-step deterministic proximal-oracle method'' in Appendix~\ref{s:f} . For the composite minimization setup, we make no assumptions about the order in which the gradient and proximal oracles are used. For these classes of deterministic methods, we retain the same lower bound.

\begin{theorem}
\label{thm:OptISTA-lb2}
Let $L>0$, $R>0$, $N>0$, and $d\ge 3N+1$. Then for any starting point $x_0$ and $N$-step deterministic double-oracle method, whose output we call $x_N$, there is an $f\colon \rl^d\rightarrow\rl$ that is $L$-smooth and convex and an $h\colon \rl^d\rightarrow\rl\cup\{\infty\}$ that is an indicator function of nonempty closed convex set such that there is an $x_\star\in \argmin(f+h)$ satisfying $\|x_0-x_\star\|=R$ and
\[
f(x_{N})+h(x_{N})-f(x_\star)-h(x_\star)\ge\frac{L\|x_{0}-x_{\star}\|^{2}}{2(\theta_{N}^{2}-1)},
\]
where $\theta_{N}$ is as defined for \ref{eq:OptISTA_Alg}.
\end{theorem}

\begin{theorem}
\label{thm:opm-lb2}
Let $R>0$, $N>0$, {\color{red}{$\varepsilon>0$}} and $d\ge 2N+1$. Let $\gamma_{0},\gamma_{1},\ldots,\gamma_{N-1}$ be positive numbers. Then, for any starting point $x_0$, $N$-step deterministic proximal-oracle method, whose output we call $x_N$, there is an $h\colon \rl^d\rightarrow\rl\cup\{\infty\}$ that is closed convex proper function such that there is an $x_\star\in \argmin h$ satisfying $\|x_0-x_\star\|=R$ and
\[
h(x_{N})-h(x_\star) \geq\frac{\gamma_{N-1}\|x_{0}-x_{\star}\|^{2}}{4\gamma_{0}^{2}\eta_{N-1}^{2}}{\color{red}{-\varepsilon}},
\]
where $\eta_{N-1}$ is as defined in \eqref{eq:rhoeta}.
\end{theorem}

\section{Conclusion}
In this work, we present two methodologies, the double-function PEP and the double-function semi-interpolated zero-chain construction, and use them to find the exact optimal accelerated composite optimization method \ref{eq:OptISTA_Alg} and establish an exact matching lower bound. We also follow an analogous approach to establish a lower bound in the proximal setup that matches the prior upper bound of \ref{alg:OPPA}. By establishing exact optimality, this work concludes the search for the fastest methods, with respect to the performance measure of \textit{worst-case function value suboptimality}, for the proximal, projected-gradient, and proximal-gradient setups in composite minimization problem involving a {\color{black}{high-dimensional}} smooth convex function and a closed proper convex function. The presented methodology has broad potential applications beyond the setups considered in this paper. The analysis and design of optimization methods with the splitting-method-structure of Douglas--Rachford \cite{DouglasRachford1956_numerical,LionsMercier1979_splitting,ryu2022LSCMO}, ADMM \cite{GlowinskiMarrocco1975_lapproximation,GabayMercier1976_dual,FortinGlowinski1983_decompositioncoordination,BoydParikhChuPeleatoEckstein2011_distributed}, and Frank--Wolfe \cite{frank1956,CGFW2022} are interesting directions of future work.

\appendix


\section{Method reference}
We quickly list relevant first-order methods and their convergence rates.
\paragraph{Nesterov's FGM
\cite{nesterov83MomentumPaper}:}
    \begin{align*}
        &y_{i+1}=x_i-\frac{1}{L}\nabla f(x_i) \\
        &x_{i+1}=y_{i+1}+\frac{\theta_i-1}{\theta_{i+1}}\left(y_{i+1}-y_i\right)
    \end{align*}
    with starting point $x_0=y_0$, with $f\in\mathcal{F}_{0,L}$, and with the sequence $\theta_i$ satisfying $\theta_{0}=1$ and $\theta_{i}=\nicefrac{\left(1+\sqrt{1+4\theta_{i-1}^{2}}\right)}{2}$ for $i\in[1:N-1]$. Its convergence rate is
\[
f(y_N)-f(x_\star)\leq\frac{L\|x_0-x_\star\|^2}{2\theta_{N-1}^2}\leq\frac{2L\|x_0-x_\star\|^2}{(N+1)^2}.
\]

\paragraph{OGM \cite{kim2016OGM}:}

    \begin{align*}
        &y_{i+1}=x_i-\frac{1}{L}\nabla f(x_i) \\
        &x_{i+1}=y_{i+1}+\frac{\theta_i-1}{\theta_{i+1}}\left(y_{i+1}-y_i\right)+\frac{\theta_i}{\theta_{i+1}}\left(y_{i+1}-x_i\right)
    \end{align*}
    with starting point $x_0=y_0$, with $f\in\mathcal{F}_{0,L}$, and with the sequence $\theta_i$ satisfying $\theta_{0}=1$, $\theta_{i}=\nicefrac{\left(1+\sqrt{1+4\theta_{i-1}^{2}}\right)}{2}$
for $i\in[1:N-1]$, and $\theta_N=\nicefrac{\left(1+\sqrt{1+8\theta_{N-1}^{2}}\right)}{2}$. Its convergence rate is
\[
f(x_N)-f(x_\star)\leq\frac{L\|x_0-x_\star\|^2}{2\theta_{N}^2}\leq\frac{L\|x_0-x_\star\|^2}{(N+1)^2}.
\]

\paragraph{FISTA \cite{beck2009FISTA}:}
    \begin{align*}
        &y_{i+1}=\prox_{\frac{1}{L}h}\left(x_i-\frac{1}{L}\nabla f(x_i)\right) \\
        &x_{i+1}=y_{i+1}+\frac{\theta_i-1}{\theta_{i+1}}\left(y_{i+1}-y_i\right)
    \end{align*}
    with starting point $x_0=y_0$, with $f\in\mathcal{F}_{0,L}$ and $h\in\mathcal{F}_{0,\infty}$ , and with the sequence $\theta_i$ satisfying $\theta_{0}=1$, $\theta_{i}=\nicefrac{\left(1+\sqrt{1+4\theta_{i-1}^{2}}\right)}{2}$
for $i\in[1:N-1]$. Its convergence rate is
\[
f(y_N)+h(y_N)-f(x_\star)-h(x_\star)\leq\frac{L\|x_0-x_\star\|^2}{2\theta_{N-1}^2}\leq\frac{2L\|x_0-x_\star\|^2}{(N+1)^2}.
\]

\paragraph{G\"{u}ler's second method \cite{guler1992new}:}
    \begin{align*}
  &y_{i+1}=\prox_{\gamma_{i}h}(x_{i})\\
  &x_{i+1}=y_{i+1}+\frac{\theta_{i}-1}{\theta_{i+1}}(y_{i+1}-y_{i})+\frac{\theta_{i}}{\theta_{i+1}}(y_{i+1}-x_{i})\end{align*}
with starting point $x_0=y_0$, with $h\in\FZerInf$, with $\gamma_{i+1}\geq\gamma_{i}>0$ for $i\in[0:N-1]$ , and with the sequence $\theta_{i}$ satisfying 
$\theta_{0}=1$ and $\theta_{i}=\nicefrac{\left(1+\sqrt{1+4\theta_{i-1}^{2}}\right)}{2}$
for $i\in[1:N-1]$. Its convergence rate is 
\[
h(y_N)-h(x_\star)\leq\frac{\|x_0-x_\star\|^2}{4\gamma_0\theta_{N-1}^2}.
\]

\paragraph{OPPA\cite{Monteiro2013hybrid,Barre2023principle}:}
\begin{align*}
 &y_{i+1}=\prox_{\gamma_{i}h}\left(x_{i}\right)\\
 &x_{i+1}=y_{i+1}+\frac{\rho_{i+1}(\eta_{i}-\rho_{i})}{\rho_{i}\eta_{i+1}}\left(y_{i+1}-y_{i}\right)+\frac{\rho_{i+1}\eta_{i}}{\rho_{i}\eta_{i+1}}\left(y_{i+1}-x_{i}\right)
\end{align*}
with starting point $x_0=y_0$, with $h\in\FZerInf$,  with the sequence $\rho_i$ satisfying $\nicefrac{\gamma_i}{\gamma_0}$ for $i\in[0:N-1]$ and $\eta_{i}$ satisfying 
$\eta_{0}=1$ and $\eta_{i}=\nicefrac{\left(\rho_i+\sqrt{\rho_i^2+4{\rho_i\eta_{i-1}^2}/{\rho_{i-1}}}\right)}{2}$
for $i\in[1:N-1]$. Its convergence rate is
\[
h(y_{N})-h(x_\star) \leq\frac{\gamma_{N-1}\|x_{0}-x_{\star}\|^{2}}{4\gamma_{0}^{2}\eta_{N-1}^{2}}.
\]

\paragraph{OptISTA:}
    \begin{align*}
&y_{i+1}=\prox_{\frac{\gamma_{i}}{L}h}\left(y_{i}-\frac{\gamma_i}{L}\nabla f(x_i)\right)\\ &z_{i+1}=x_{i}+\frac{1}{\gamma_{i}}\left(y_{i+1}-y_{i}\right)\\
 & x_{i+1}=z_{i+1}+\frac{\theta_{i}-1}{\theta_{i+1}}\left(z_{i+1}-z_{i}\right)+\frac{\theta_{i}}{\theta_{i+1}}\left(z_{i+1}-x_{i}\right)
\end{align*}
with starting point $x_0=y_0=z_0$, with $f\in\mathcal{F}_{0,L}$ and $h\in\FZerInf$,  with the sequence
$\theta_{i}$ satisfying 
$\theta_{0}=1$, $\theta_{i}=\nicefrac{\left(1+\sqrt{1+4\theta_{i-1}^{2}}\right)}{2}$ for $i\in[1:N-1]$, and $\theta_N=\nicefrac{\left(1+\sqrt{1+8\theta_{N-1}^{2}}\right)}{2}$, and with $\gamma_{i}=\nicefrac{2\theta_i\left(\theta_N^2-2\theta_i^{\color{red}{2}}+\theta_i\right)}{\theta_N^2}>0$ 
for $i\in[0:N-1]$. Its convergence rate is
\[
f(y_N)+h(y_N)-f(x_\star)-h(x_\star)\leq\frac{L\|x_0-x_\star\|^2}{2(\theta_{N}^2-1)}\leq\frac{L\|x_0-x_\star\|^2}{(N+1)^2}.
\]

\section{Omitted proof of Lemma~\ref{c:x=y}}\label{s:lem2}
We begin the section by introducing some lemmas.
\begin{lem}\label{lem:hstepsize}
Let $\{\alpha_{i,j}\}_{i\in[1:N],j\in[0:N-1]}$ and $\{h_{i,j}\}_{i\in[1:N],j\in[0:N-1]}$ be
\[
\alpha_{i+1,j}=\begin{cases} \alpha_{j+1,j}+\sum_{k=j+1}^{i}\left(\frac{2\theta_j}{\theta_{k+1}}-\frac{1}{\theta_{k+1}}\alpha_{k,j}
\right) \textup{ if } j\in[0:i-1], \\ 1+\frac{2\theta_i-1}{\theta_{i+1}} \textup{ if } j=i,\end{cases}
\]
\[
h_{i+1,j}=\begin{cases}
   \alpha_{i+1,j}-\alpha_{i,j} \textup{ if } j\in[0:i-1],\\
   \alpha_{i+1,j} \textup{ if } j=i.
\end{cases}
\]
Then,
 \[
 \alpha_{i+1,j}=\sum_{k=j+1}^{i+1} h_{k,j} \textup{ for } j\in[0:i]
\textup{ and }
h_{i+1,j}=\begin{cases}
    \frac{\theta_i-1}{\theta_{i+1}}h_{i,j} \textup{ if } j\in[0:i-2],\\
    \frac{\theta_i-1}{\theta_{i+1}}\left(h_{i,i-1}-1\right) \textup{ if } j=i-1,\\
    1+\frac{2\theta_i-1}{\theta_{i+1}}\textup{ if }  j=i.
\end{cases}
\]
\end{lem}
\begin{proof}
Note that
 \[
 \alpha_{i+1,j}=\begin{cases}
    \alpha_{i,j}+h_{i+1,j}, \quad j\in[0:i-1],\\
    h_{i+1,j}, \quad j=i.
 \end{cases}
 \]
 So, for $j\in[0:i-1]$,  we get the following equation: 
 \[
 \alpha_{i+1,j}=\alpha_{i,j}+h_{i+1,j}=\alpha_{i-1,j}+h_{i,j}+h_{i+1,j}=\cdots=\alpha_{j+1,j}+\cdots+h_{i+1,j}=h_{j+1,j}+\cdots+h_{i+1,j}=\sum_{k=j+1}^{i+1} h_{k,j}.
 \]
 Since $\alpha_{i+1,j}=h_{i+1,j}$, we have
 \[
 \alpha_{i+1,j}=\sum_{k=j+1}^{i+1} h_{k,j} \textup{ for } j\in[0:i].
 \]    
 For the second equality, we refer the reader to \cite[Proposition 3]{kim2016OGM}.
\qed\end{proof}
Now we convert $x$ and $y$-iterates of \ref{eq:OptISTA_Alg} into \eqref{eq:FSFOM} form.
\begin{lem}\label{lem:equivalenceOptISTA}
Let
\[
\gamma_i=\frac{2\theta_i}{\theta_N^2}\left(\theta_N^2-2\theta_i^2+\theta_i\right) \textup{ for } i\in[0:N-1],
\]
\[
\alpha_{i+1,j}=\begin{cases} \alpha_{j+1,j}+\sum_{k=j+1}^{i}\left(\frac{2\theta_j}{\theta_{k+1}}-\frac{1}{\theta_{k+1}}\alpha_{k,j}
\right) \textup{ if } j\in[0:i-1], \\ 1+\frac{2\theta_i-1}{\theta_{i+1}} \textup{ if } j=i.\end{cases}
\]
Then,    $\{x_1,\ldots,x_{N}\}$ and  $\{y_1,\ldots,y_N\}$ of the following
    \eqref{eq:FSFOM} form:
\begin{equation}
    \begin{aligned}
    y_{i+1}&=x_0-\sum_{j\in[0:i]}\frac{\gamma_j}{L}\nabla f(x_j)- \sum_{j\in[0:i]}\frac{\gamma_j}{L}h'(y_{j+1}) \quad\textup{ for } i\in[0:N-1]\\
     x_{i+1}&=x_0-\sum_{j\in[0:i]}\frac{\alpha_{i+1,j}}{L}\nabla f(x_j)- \sum_{j\in[0:i]}\frac{\alpha_{i+1,j}}{L}h'(y_{j+1})
     \quad \textup{ for } i\in[0:N-1]
    \end{aligned}
        \label{eq:equivalentFSFOM}
\end{equation}
    is equal to $x$-iterates and $y$-iterates generated by \ref{eq:OptISTA_Alg} respectively.
\end{lem}
\begin{proof}
    Denote $\{\hat{x}_1,\ldots,\hat{x}_{N}\}$ and $\{\hat{y}_1,\ldots,\hat{y}_{N}\}$ to be sequence generated by \ref{eq:OptISTA_Alg}. We use mathematical induction. Since we have the same starting point $x_0=y_0$, $\hat{y}_1=y_1$ is immediate from the uniqueness of proximal operator, and
    \[
    \hat{y}_1=\prox_{\frac{\gamma_{0} }{L}h}  \Big( y_0-\frac{\gamma_0}{L} \nabla f(x_0) \Big)=x_0-\frac{\gamma_0}{L}\nabla f(x_0)-\frac{\gamma_0}{L}h'(\hat{y}_1).
    \]
    For $x$-iterate, we have,
    \begin{align*}
    \hat{x}_1&=z_{1} + \frac{\theta_0-1}{\theta_{1}}(z_1-z_0) + \frac{\theta_0}{\theta_{1}}(z_1-x_0)=x_0+\frac{1}{\gamma_0}(\hat{y}_1-y_0) +\frac{1}{\theta_1}\left(   x_0+\frac{1}{\gamma_0}(\hat{y}_1-y_0)-x_0\right)\\
    &=x_0-\frac{1}{\gamma_0}\cdot\frac{\gamma_0}{L}\left(\nabla f(x_0)+h'(\hat{y}_1)\right)-\frac{1}{\theta_1}\cdot\frac{1}{\gamma_0}\cdot\frac{\gamma_0}{L}\left(\nabla f(x_0)+h'(\hat{y}_1)\right)\\
    &=x_0-\frac{1}{L}\left(1+\frac{1}{\theta_1}\right)\nabla f(x_0) -\frac{1}{L}\left(1+\frac{1}{\theta_1}\right)h'(\hat{y}_1)\\
    &=x_0-\frac{\alpha_{1,0}}{L}\nabla f(x_0) -\frac{\alpha_{1,0}}{L}h'(\hat{y}_1)=x_1.
    \end{align*}
    Now suppose $x_k=\hat{x}_k$ and $y_k=\hat{y}_k$ for  $1\leq k \leq i$ and $1\leq i \leq N-1$. Then, we get $\hat{y}_{i+1}=y_{i+1}$ 
    from the uniqueness of the proximal operator, and 
    \begin{align*}
    \hat{y}_{i+1}&=\prox_{\frac{\gamma_{i} }{L}h}  \Big( \hat{y}_i-\frac{\gamma_i}{L} \nabla f(x_i) \Big)=y_i-\frac{\gamma_i}{L} \nabla f(x_i)-\frac{\gamma_i}{L} h'(\hat{y}_{i+1})\\
    &=x_0-\sum_{j\in[0:i-1]}\frac{\gamma_j}{L}\nabla f(x_j)- \sum_{j\in[0:i-1]}\frac{\gamma_j}{L}h'(y_{j+1})-\frac{\gamma_i}{L} \nabla f(x_i)-\frac{\gamma_i}{L} h'(\hat{y}_{i+1})\\
    &=x_0-\sum_{j\in[0:i]}\frac{\gamma_j}{L}\nabla f(x_j)- \sum_{j\in[0:i]}\frac{\gamma_j}{L}h'(y_{j+1}).
    \end{align*}
    For $x$-iterate, we have,
    \begin{align*}
    {z}_{i+1}&=x_0-\sum_{j\in[0:i-1]}\frac{\alpha_{i,j}}{L}\nabla f({x}_{j})-\sum_{j\in[0:i-1]}\frac{\alpha_{i,j}}{L}h'({y}_{j+1})-\frac{1}{\gamma_i}\left(\frac{\gamma_i}{L}\nabla f({x}_{i})+\frac{\gamma_i}{L}h'({y}_{i+1})\right)\\
    &=x_0-\sum_{j\in[0:i-1]}\frac{\alpha_{i,j}}{L}\nabla f({x}_{j})-\sum_{j\in[0:i-1]}\frac{\alpha_{i,j}}{L}h'({y}_{j+1})-\frac{1}{L}\nabla f({x}_{i})-\frac{1}{L}h'({y}_{i+1}).
    \end{align*}   
Therefore,
\begin{align}
        \hat{x}_{i+1}&=z_{i+1} + \frac{\theta_i-1}{\theta_{i+1}}(z_{i+1}-z_i) + \frac{\theta_i}{\theta_{i+1}}(z_{i+1}-\hat{x}_i)\nonumber \\
        &=x_0-\sum_{j\in[0:i-1]}\frac{\alpha_{i,j}}{L}\nabla f({x}_{j})-\sum_{j\in[0:i-1]}\frac{\alpha_{i,j}}{L}h'({y}_{j+1})-\frac{1}{L}\nabla f({x}_{i})-\frac{1}{L}h'({y}_{i+1})\nonumber\\
        &-\frac{\theta_i-1}{\theta_{i+1}}\left(\sum_{j\in[0:i-2]}\frac{h_{i,j}}{L}\nabla f({x}_{j})+\sum_{j\in[0:i-2]}\frac{h_{i,j}}{L}h'({y}_{j+1})\right) {\color{gray}\qquad\rhd\,\textup{  definition of } h_{i,j}}\nonumber\\
        &-\frac{\theta_i-1}{\theta_{i+1}}\left(\frac{\alpha_{i,i-1}-1}{L}\nabla f({x}_{i-1})+\frac{\alpha_{i,i-1}-1}{L}h'({y}_{i})\right)\nonumber\\
        &-\frac{\theta_i-1}{\theta_{i+1}}\left(\frac{1}{L}\nabla f({x}_{i})+\frac{1}{L}h'({y}_{i+1})\right)-\frac{\theta_i}{\theta_{i+1}}\left(\frac{1}{L}\nabla f({x}_{i})+\frac{1}{L}h'({y}_{i+1})\right).\label{eq:xiteratecalculation1}
    \end{align}
By the second equation of Lemma~\ref{lem:hstepsize}, We have 
\begin{align*}
 \frac{\theta_i-1}{\theta_{i+1}}h_{i,j}=h_{i+1,j} \textup{ for } j\in[0:i-2],    
\end{align*}
\begin{align*}
  \frac{\theta_i-1}{\theta_{i+1}}\left(\alpha_{i,i-1}-1\right)=h_{i+1,i-1}.
\end{align*}
So \eqref{eq:xiteratecalculation1} reduces to 
\begin{align*}
    \hat{x}_{i+1}&=x_0-\sum_{j\in[0:i-1]}\frac{\alpha_{i,j}}{L}\nabla f({x}_{j})-\sum_{j\in[0:i-1]}\frac{\alpha_{i,j}}{L}h'({y}_{j+1})-\frac{1}{L}\nabla f({x}_{i})-\frac{1}{L}h'({y}_{i+1})\\
        &-\sum_{j\in[0:i-2]}\frac{h_{i+1,j}}{L}\nabla f({x}_{j})-\sum_{j\in[0:i-2]}\frac{h_{i+1,j}}{L}h'({y}_{j+1})-\frac{h_{i+1,i-1}}{L}\nabla f({x}_{i-1})-\frac{h_{i+1,i-1}}{L}h'({y}_{i})\\
        &-\frac{\theta_i-1}{\theta_{i+1}}\left(\frac{1}{L}\nabla f({x}_{i})+\frac{1}{L}h'({y}_{i+1})\right)-\frac{\theta_i}{\theta_{i+1}}\left(\frac{1}{L}\nabla f({x}_{i})+\frac{1}{L}h'({y}_{i+1})\right)\\
        &=x_0-\sum_{j\in[0:i-1]}\frac{\alpha_{i+1,j}}{L}\nabla f({x}_{j})-\sum_{j\in[0:i-1]}\frac{\alpha_{i+1,j}}{L}h'({y}_{j+1})-\frac{1}{L}\left(1+\frac{2\theta_i-1}{\theta_{i+1}}\right)\nabla f({x}_{i})-\frac{1}{L}\left(1+\frac{2\theta_i-1}{\theta_{i+1}}\right)h'({y}_{i+1})\\
        &=x_0-\sum_{j\in[0:i]}\frac{\alpha_{i+1,j}}{L}\nabla f({x}_{j})-\sum_{j\in[0:i]}\frac{\alpha_{i+1,j}}{L}h'({y}_{j+1})=x_{i+1}.
\end{align*}
\qed\end{proof}
We show $x_N=y_N$ by using the following lemma.
\begin{lem}\label{lem:recursive}
   $\alpha_{N,j}$ is completely characterized by the following recurrent relationship:
   \[
   \frac{\theta_{j+1}}{2\theta_{j}-1}\left(\alpha_{N,j}-1\right)-1= \frac{\theta_{j+1}-1}{2\theta_{j+1}-1}\left(\alpha_{N,j+1}-1\right).
   \]
\end{lem}
\begin{proof}
    By Lemma~\ref{lem:hstepsize}, we have the following for $j\in[0:N-1]$:
\begin{align*}
&\alpha_{N,j}=\sum_{k=j+1}^N h_{k,j}=h_{j+1,j}+\frac{\theta_{j+1}-1}{\theta_{j+2}}(h_{j+1,j}-1)+\frac{\theta_{j+2}-1}{\theta_{j+3}}\frac{\theta_{j+1}-1}{\theta_{j+2}}(h_{j+1,j}-1)+\\
&\cdots+\prod_{\ell=j+1}^{N-1}\frac{\theta_{\ell}-1}{\theta_{\ell+1}}\left(h_{j+1,j}-1\right)=h_{j+1,j}+\sum_{k=j+1}^{N-1}\prod_{\ell=j+1}^{k}\frac{\theta_{\ell}-1}{\theta_{\ell+1}}\left(h_{j+1,j}-1\right).
\end{align*}
So we have
\begin{align*}
    \alpha_{N,j}-1=\left(1+\sum_{k=j+1}^{N-1}\prod_{\ell=j+1}^{k}\frac{\theta_{\ell}-1}{\theta_{\ell+1}}\right)\left(h_{j+1,j}-1\right)=\left(1+\sum_{k=j+1}^{N-1}\prod_{\ell=j+1}^{k}\frac{\theta_{\ell}-1}{\theta_{\ell+1}}\right)\cdot\frac{2\theta_{j}-1}{\theta_{j+1}}.
\end{align*}
Multiply both sides by $\frac{\theta_{j+1}}{2\theta_{j}-1}$ and subtract $1$ to get
\begin{align}
 \frac{\theta_{j+1}}{2\theta_{j}-1}\left(\alpha_{N,j}-1\right)-1=\sum_{k=j+1}^{N-1}\prod_{\ell=j+1}^{k}\frac{\theta_{\ell}-1}{\theta_{\ell+1}}.    \label{recursive1}
\end{align}
Similarly for $j+1$, we have the following:
\begin{align}
    \frac{\theta_{j+1}-1}{2\theta_{j+1}-1}\left(\alpha_{N,j+1}-1\right)&=\left(1+\sum_{k=j+2}^{N-1}\prod_{\ell=j+2}^{k}\frac{\theta_{\ell}-1}{\theta_{\ell+1}}\right)\cdot\frac{2\theta_{j+1}-1}{\theta_{j+2}}\cdot\frac{\theta_{j+1}-1}{2\theta_{j+1}-1} \nonumber \\
    &=\left(1+\sum_{k=j+2}^{N-1}\prod_{\ell=j+2}^{k}\frac{\theta_{\ell}-1}{\theta_{\ell+1}}\right)\cdot\frac{\theta_{j+1}-1}{\theta_{j+2}}=\sum_{k=j+1}^{N-1}\prod_{\ell=j+1}^{k}\frac{\theta_{\ell}-1}{\theta_{\ell+1}}.   \label{recursive2}
\end{align}
Combining \eqref{recursive1} and \eqref{recursive2}, we get the desired result.
\qed\end{proof}
Now we are ready to prove Lemma~\ref{c:x=y}.

\begin{proof}\textit{of Lemma~\ref{c:x=y}.}
By Lemma~\ref{lem:equivalenceOptISTA}, it suffices to show the following holds for $j\in[0:N-1]$.
\[
     \gamma_j=\frac{2\theta_j}{\theta_N^2}\left(\theta_N^2-2\theta_j^2+\theta_j\right)=\alpha_{N,j}.
\]
    First we show that $\gamma_{N-1}=\alpha_{N,N-1}$:
    \begin{align*}
         \gamma_{N-1}&=\frac{2\theta_{N-1}}{\theta_N^2}\left(\theta_N^2-2\theta_{N-1}^2+\theta_{N-1}\right)=\frac{2\theta_{N-1}}{\theta_N^2}\left(\theta_N+\theta_{N-1}\right)=\frac{2\theta_N\theta_{N-1}+2\theta_{N-1}^2}{\theta_N^2}\\
         &=\frac{2\theta_N\theta_{N-1}+\theta_N^2-\theta_N}{\theta_N^2}=\frac{2\theta_{N-1}+\theta_N-1}{\theta_N}=1+\frac{2\theta_{N-1}-1}{\theta_N}=\alpha_{N,N-1}.
    \end{align*}
 Now we complete the proof by showing that $\gamma_j$ satisfies the same defining recurrent relationship satisfied by $\alpha_{N,j}$ in Lemma~\ref{lem:recursive}. For $j\in[0:N-2]$,
\begin{align*}
    \frac{\theta_{j+1}}{2\theta_j-1}\left(\gamma_j-1\right)-1&=\frac{\theta_{j+1}}{2\theta_j-1}\cdot\left(\frac{2\theta_j}{\theta_N^2}\left(\theta_N^2-2\theta_j^2+\theta_j\right)-1\right)-1\\
    &=\frac{2\theta_{j+1}}{\theta_N^2}\cdot\frac{\theta_j}{2\theta_j-1}\left(\theta_N^2-2\theta_j^2+\theta_j\right)-\frac{\theta_{j+1}}{2\theta_j-1}-1\\
    &=\frac{2\theta_{j+1}\theta_N^2}{\theta_N^2}\cdot\frac{\theta_j}{2\theta_j-1}-\frac{2\theta_{j+1}\theta_j^2}{\theta_N^2}-\frac{\theta_{j+1}}{2\theta_j-1}-1\\
 &=\frac{2\theta_{j+1}\theta_j}{2\theta_j-1}-\frac{2\theta_{j+1}\theta_j^2}{\theta_N^2}-\frac{\theta_{j+1}}{2\theta_j-1}-1\\
 &=\frac{\theta_{j+1}\left(2\theta_j-1\right)}{2\theta_j-1}-\frac{2\theta_{j+1}\theta_j^2}{\theta_N^2}-1\\
 &=\theta_{j+1}-\frac{2\theta_{j+1}(\theta_{j+1}^2-\theta_{j+1})}{\theta_N^2}-1=\theta_{j+1}-\frac{2\theta_{j+1}^2(\theta_{j+1}-1)}{\theta_N^2}-1
\end{align*}
and
\begin{align*}
    \frac{\theta_{j+1}-1}{2\theta_{j+1}-1}\left(\gamma_{j+1}-1\right)&= \frac{\theta_{j+1}-1}{2\theta_{j+1}-1}\cdot\left(\frac{2\theta_{j+1}}{\theta_N^2}\left(\theta_N^2-2\theta_{j+1}^2+\theta_{j+1}\right)-1\right)\\
   &= \frac{\theta_{j+1}-1}{2\theta_{j+1}-1}\cdot\left(2\theta_{j+1}-1-\frac{2\theta_{j+1}^2\left(2\theta_{j+1}-1\right)}{\theta_N^2}\right)\\
    &=\theta_{j+1}-1-\frac{2\theta_{j+1}^2\left(\theta_{j+1}-1\right)}{\theta_N^2}.
\end{align*}
\qed\end{proof}

\section{Omitted proof of Theorem 1}\label{s:c}
First, we prove preliminary lemmas that will be used in the main proof.
\begin{lem}\label{lem:propertyoftheta}
    Let $\{{\theta}_i\}_{i=0,\ldots,N}$  defined as
    \[
    \theta_{i}=\begin{cases}
1, & \textup{if }i=0\\
\frac{1+\sqrt{1+4\theta_{i-1}^{2}}}{2}, & \textup{if }1\leq i\leq N-1,\\
\frac{1+\sqrt{1+8\theta_{N-1}^{2}}}{2}, & \textup{if }i=N.
\end{cases} 
\]
    Then $\{{\theta}_i\}_{i=0,\ldots,N}$ satisfies 
    \begin{align}
        &\theta_{i+1}^2-\theta_{i+1}-\theta_{i}^2=0  \textup{ for } i\in[0:N-2], \label{eq:recursivetheta1}\\
         &\theta_N^2-\theta_N-2\theta_{N-1}^2=0,\label{eq:recursivetheta2}\\
         &\sum_{j=0}^i \theta_j=\theta_i^2  \textup{ for } i\in[0:N-1].\label{eq:sumoftheta}
    \end{align}
\end{lem}
\begin{proof}
    Since $\theta_{i+1}$ is a root of $x^2-x-\theta_{i}^2=0$ for $i\in[0:N-2]$, we have \eqref{eq:recursivetheta1}. Similarly, $\theta_N$ is a root of $x^2-x-2\theta_{N-1}^2=0$, we have \eqref{eq:recursivetheta2}. For \eqref{eq:sumoftheta},
    \[
    \sum_{j=0}^{i}\theta_j=\theta_0+\sum_{j=1}^{i}\left(\theta_j^2-\theta_{j-1}^2\right)=\theta_0+\theta_i^2-\theta_0^2=\theta_i^2,
    \]
    where $i\in[0:N-1]$ and the second equality is from the telescopic sum. 
    \qed\end{proof}
    
\begin{lem}\label{lem:propertyoftildetheta}
    Let  $\{\tilde{\theta}_i\}_{i=0,\ldots,N-1}$ defined as
    \[
    \tilde{\theta}_{i}=\begin{cases}
\theta_i & \textup{if }i\in[0:N-2],\\
\frac{2\theta_{N-1}+\theta_N-1}{2} & \textup{if }i=N-1.
\end{cases}
    \]
    Then,
\begin{align}
     \sum_{j=0}^{N-1}\tilde{\theta}_j=\frac{\theta_N^2-1}{2}.\label{eq:sumoftildetheta}
\end{align}
\end{lem}
\begin{proof}
    \[
    \sum_{i=0}^{N-1}\tilde{\theta}_i=\sum_{i=0}^{N-2}\theta_i+\frac{2\theta_{N-1}+\theta_N-1}{2}=\frac{2\theta_{N-2}^2+2\theta_{N-1}+\theta_N-1}{2}
  =\frac{2\theta_{N-1}^2+\theta_N-1}{2}=\frac{\theta_N^2-1}{2}.
    \]
    where the second equality uses \eqref{eq:sumoftheta}, the third uses
    \eqref{eq:recursivetheta1}, and the fourth uses \eqref{eq:recursivetheta2} of Lemma~\ref{lem:propertyoftheta}.
\qed\end{proof}
Now, we are ready to start the main proof. We provide the explicit form of the Lyapunov sequence $\mathcal{U}_{k}$ for $k\in[-1:N]$. We first define $\{\tau_{i,j}\}$ where $i,j\in[0:N]\cup\{\star\}$ as
\[
\tau=\begin{cases}
\tau_{\star,i}=\frac{2\tilde{\theta}_{i-1}}{\theta_{N}^{2}-1} \textup{ if }i\in[1:N]\\
\tau_{i,j}=\frac{2\tilde{\theta}_{j-1}}{\theta_{N}^{2}-2\theta_{i}^{2}+\theta_{i}}-\frac{2\tilde{\theta}_{j-1}}{\theta_{N}^{2}-2\theta_{i-1}^{2}+\theta_{i-1}}\textup{ if } 1\leq i<j\leq N,\\
\tau_{i+1,i}=\frac{\theta_{i}-1}{\theta_{N}^{2}-2\theta_{i}^{2}+\theta_{i}}\textup{ if }i\in[1:N-1],\\
\tau_{i,j}=0\ \mathrm{otherwise}.
\end{cases}
\]
Let $\theta_{-1}=0$ and define $\{\mathcal{F}_k\}_{k\in[-1:N]}$ to be:
\begin{itemize}
    \item[$\bullet$]$k=N$
\begin{align*}
   & \mathcal{F}_N=f(x_N)-f(x_\star)+\frac{L}{2\theta_N^2}\left\|w_N-x_\star+\frac{1}{L}\nabla f(x_\star)+\frac{2\theta_{N-1}}{L} h'(y_N)-\frac{\theta_N}{L}\nabla f(x_N)-\frac{2\tilde{\theta}_{N-1}}{L}h'(y_{N})\right\|^2,
\end{align*}
 \item[$\bullet$]$k\in[-1:N-1]$
\begin{align*}
   & \mathcal{F}_{k}=\frac{2\theta_{k}^2}{\theta_N^2}\left(f(x_{k})-f(x_\star)\right)+\frac{L}{2\theta_N^2}\left\|w_{k+1}-x_\star+\frac{1}{L}\nabla f(x_\star)+\frac{2\theta_{k}}{L}h'(y_{k+1})\right\|^2-\left(\frac{1}{2L}-\frac{\theta_{k}^2}{L\theta_N^2}\right)\|\nabla f(x_\star)\|^2-\frac{\theta_{k}^2}{L\theta_N^2}\|\nabla f(x_{k})\|^2,
\end{align*}
\end{itemize}
and $\{\mathcal{H}_k\}_{k\in[-1:N]}$ to be
\begin{itemize}
    \item[$\bullet$]$k=N$
\begin{align*}
   & \mathcal{H}_N=h(y_N)-h(x_\star)\\
    &+\frac{L}{2\theta_N^2(\theta_N^2-1)}\left\|x_0-x_\star-\frac{\theta_N^2-1}{L}\nabla f(x_\star)-\sum_{i=0}^{N-1}\frac{2\tilde{\theta}_i}{L} h'(y_{i+1})\right\|^2\\
    &+\sum_{i\neq j, i,j\in[1:N]}\frac{\tilde{\theta}_{i-1}\tilde{\theta}_{j-1}}{L\theta_N^2(\theta_N^2-1)}\|h'(y_i)-h'(y_j)\|^2+\sum_{i=1}^{N-1}\frac{\tilde{\theta}_{i-1}^2}{L\theta_N^2}\|h'(y_i)-h'(y_{i+1})\|^2,
\end{align*}
 \item[$\bullet$]$k\in[1:N-1]$
\begin{align*}
   & \mathcal{H}_{k}=\sum_{i,j\in\{\star,1,\dots,k\}}\tau_{i,j}\left(h(y_j)-h(y_i)\right)\\
    &+\frac{L}{2\theta_N^2(\theta_N^2-1)}\left\|x_0-x_\star-\frac{\theta_N^2-1}{L}\nabla f(x_\star)-\sum_{i=0}^{k-1}\frac{2{\theta}_i}{L} h'(y_{i+1})\right\|^2\\
    &+\sum_{i\neq j, i,j\in[1:k]}\frac{{\theta}_{i-1}{\theta}_{j-1}}{L\theta_N^2(\theta_N^2-1)}\|h'(y_i)-h'(y_j)\|^2+\sum_{i=1}^{k-1}\frac{{\theta}_{i-1}^2}{L\theta_N^2}\|h'(y_i)-h'(y_{i+1})\|^2\\
    &+\frac{2\theta_{k-1}^2}{L\theta_N^2}\langle \nabla f(x_{k}), h'(y_{k})\rangle+\sum_{i=1}^{k}\sum_{\ell=k}^{N-1}\frac{2\tilde{\theta}_{\ell}{\theta}_{i-1}}{L\theta_N^2(\theta_N^2-1)}\|h'(y_i)\|^2+\frac{{\theta}_{k-1}^2}{L\theta_N^2}\|h'(y_{k})\|^2,
\end{align*}
 \item[$\bullet$]$k=0,-1$
\begin{align*}
   & \mathcal{H}_{0}=\mathcal{H}_{-1}=\frac{L}{2\theta_N^2(\theta_N^2-1)}\left\|x_0-x_\star-\frac{\theta_N^2-1}{L}\nabla f(x_\star)\right\|^2,
\end{align*}
\end{itemize}
Then we let 
\[
\mathcal{U}_k=\mathcal{F}_k+\mathcal{H}_k \qquad k\in[-1:N].
\]
In order to show that $\{\mathcal{U}_k\}_{k\in[-1:N]}$ is nonincreasing, we start with calculating $\mathcal{F}_{k}-\mathcal{F}_{k+1}$ for $k\in[-1:N-1]$. We define $\theta_{-1}=0$.
\begin{lem}
The following holds.
\[
\mathcal{F}_{N-1}-\mathcal{F}_{N}\geq\frac{2\tilde{\theta}_{N-1}}{\theta_N^2}\left \langle w_N-x_\star+\frac{1}{L}\nabla f(x_\star), h'(y_N)  \right\rangle+\frac{\tilde{\theta}_{N-1}(4\theta_{N-1}-2\tilde{\theta}_{N-1})}{L\theta_N^2}\|h'(y_N)\|^2 
\]
\[
\mathcal{F}_{k}-\mathcal{F}_{k+1}\geq\frac{2\theta_k}{\theta_N^2}\left\langle w_{k+1}-x_\star+\frac{1}{L}\nabla f(x_\star), h'(y_{k+1})\right \rangle+\frac{2\theta_k^2}{L\theta_N^2}\|h'(y_{k+1})\|^2+\frac{2\theta_k^2}{L\theta_N^2}\langle h'(y_{k+1}), \nabla f(x_{k+1})\rangle. \qquad k\in[-1:N-2].
\]
\end{lem}
\begin{proof}
We expand the squares of $\mathcal{F}_{N-1}$ and $\mathcal{F}_{N}$ to get:
\begin{align*}
    &\mathcal{F}_{N-1}-\mathcal{F}_{N}=\frac{2\theta_{N-1}^2}{\theta_N^2}\left(f(x_{N-1})-f(x_\star)\right)-(f(x_N)-f(x_\star))-\frac{1}{2L\theta_N}\|\nabla f(x_\star)\|^2-\frac{\theta_{N-1}^2}{L\theta_N^2}\|\nabla f(x_{N-1})\|^2\\
    &+\frac{L}{2\theta_N^2}\left\|w_N-x_\star+\frac{1}{L}\nabla f(x_\star)+\frac{2\theta_{N-1}}{L}h'(y_N)\right\|^2-\frac{L}{2\theta_N^2}\left\|w_N-x_\star+\frac{1}{L}\nabla f(x_\star)+\frac{2\theta_{N-1}}{L} h'(y_N)\right\|^2\\
   &+\frac{L}{\theta_N^2}\left \langle\frac{\theta_N}{L}\nabla f(x_N)+\frac{2\tilde{\theta}_{N-1}}{L}h'(y_{N}),w_N-x_\star+\frac{1}{L}\nabla f(x_\star)+\frac{2\theta_{N-1}}{L} h'(y_N)\right\rangle-\frac{L}{2\theta_N^2}\left\|\frac{\theta_N}{L}\nabla f(x_N)+\frac{2\tilde{\theta}_{N-1}}{L}h'(y_{N})\right\|^2\\
    &=\frac{2\theta_{N-1}^2}{\theta_N^2}\left(f(x_{N-1})-f(x_\star)-\frac{1}{2L}\|\nabla f(x_{N-1})\|^2\right)-(f(x_N)-f(x_\star))-\frac{1}{2L\theta_N}\|\nabla f(x_\star)\|^2\\
   &+\frac{L}{\theta_N^2}\left \langle\frac{\theta_N}{L}\nabla f(x_N)+\frac{2\tilde{\theta}_{N-1}}{L}h'(y_{N}),w_N-x_\star+\frac{1}{L}\nabla f(x_\star)+\frac{2\theta_{N-1}}{L} h'(y_N)\right\rangle\\
   &-\frac{1}{2L}\left\|\nabla f(x_N)\right\|^2-\frac{2\tilde{\theta}_{N-1}}{L\theta_N}\left\langle \nabla f(x_N), h'(y_N)\right\rangle-\frac{2\tilde{\theta}_{N-1}^2}{L\theta_N^2}\|h'(y_N)\|^2\\
    &=\frac{2\theta_{N-1}^2}{\theta_N^2}\left(f(x_{N-1})-f(x_\star)-\frac{1}{2L}\|\nabla f(x_{N-1})\|^2\right)-(f(x_N)-f(x_\star))-\frac{1}{2L\theta_N}\|\nabla f(x_\star)\|^2\\
   &+\frac{1}{\theta_N}\left \langle w_N-x_\star+\frac{1}{L}\nabla f(x_\star), \nabla f(x_N)  \right\rangle+\frac{2\tilde{\theta}_{N-1}}{\theta_N^2}\left \langle w_N-x_\star+\frac{1}{L}\nabla f(x_\star), h'(y_N)  \right\rangle\\
   &+\frac{2\theta_{N-1}}{L\theta_N}\left\langle \nabla f(x_N), h'(y_N)\right\rangle +\frac{4\theta_{N-1}\tilde{\theta}_{N-1}}{L\theta_N^2}\|h'(y_N)\|^2-\frac{1}{2L}\left\|\nabla f(x_N)\right\|^2-\frac{2\tilde{\theta}_{N-1}}{L\theta_N}\left\langle \nabla f(x_N), h'(y_N)\right\rangle-\frac{2\tilde{\theta}_{N-1}^2}{L\theta_N^2}\|h'(y_N)\|^2\\
    &=\frac{2\theta_{N-1}^2}{\theta_N^2}\left(f(x_{N-1})-f(x_\star)-\frac{1}{2L}\|\nabla f(x_{N-1})\|^2\right)-\frac{\theta_{N}^2}{\theta_N^2}\left(f(x_N)-f(x_\star)+\frac{1}{2L}\|\nabla f(x_N)\|^2\right)\\
   &-\frac{1}{2L\theta_N}\|\nabla f(x_\star)\|^2+\frac{1}{\theta_N}\left \langle w_N-x_\star+\frac{1}{L}\nabla f(x_\star), \nabla f(x_N)  \right\rangle+\frac{2\tilde{\theta}_{N-1}}{\theta_N^2}\left \langle w_N-x_\star+\frac{1}{L}\nabla f(x_\star), h'(y_N)  \right\rangle\\
   &-\frac{\theta_N-1}{L\theta_N}\left\langle \nabla f(x_N), h'(y_N)\right\rangle+\frac{4\theta_{N-1}\tilde{\theta}_{N-1}-2\tilde{\theta}_{N-1}^2}{L\theta_N^2}\|h'(y_N)\|^2\\
      &=\frac{\theta_N^2-\theta_N}{\theta_N^2}\left(f(x_{N-1})-f(x_\star)-\frac{1}{2L}\|\nabla f(x_{N-1})\|^2\right)-\frac{\theta_{N}^2}{\theta_N^2}\left(f(x_N)-f(x_\star)+\frac{1}{2L}\|\nabla f(x_N)\|^2\right) {\color{gray}\qquad\rhd\,\theta_{N-1}^2=\theta_N^2-\theta_N}\\
   &-\frac{1}{2L\theta_N}\|\nabla f(x_\star)\|^2+\frac{1}{\theta_N}\left \langle w_N-x_\star, \nabla f(x_N)  \right\rangle+\frac{1}{L\theta_N}\left \langle\nabla f(x_\star), \nabla f(x_N)  \right\rangle+\frac{2\tilde{\theta}_{N-1}}{\theta_N^2}\left \langle w_N-x_\star+\frac{1}{L}\nabla f(x_\star), h'(y_N)  \right\rangle\\
   &-\frac{\theta_N-1}{L\theta_N}\left\langle \nabla f(x_N), h'(y_N)\right\rangle+\frac{\tilde{\theta}_{N-1}(4\theta_{N-1}-2\tilde{\theta}_{N-1})}{L\theta_N^2}\|h'(y_N)\|^2\\
         &=\frac{\theta_N^2-\theta_N}{\theta_N^2}\left(f(x_{N-1})-f(x_\star)-\frac{1}{2L}\|\nabla f(x_{N-1})\|^2-f(x_N)+f(x_\star)-\frac{1}{2L}\|\nabla f(x_N)\|^2\right)\\
         &+\frac{\theta_{N}}{\theta_N^2}\left(f(x_\star)-f(x_N)-\frac{1}{2L}\|\nabla f(x_N)\|^2-\frac{1}{2L}\|\nabla f(x_\star)\|^2\right)\\
   &+\frac{1}{\theta_N}\left \langle w_N-x_N, \nabla f(x_N)  \right\rangle+\frac{1}{\theta_N}\left \langle x_N-x_\star, \nabla f(x_N)  \right\rangle+\frac{1}{L\theta_N}\left \langle\nabla f(x_\star), \nabla f(x_N)  \right\rangle+\frac{2\tilde{\theta}_{N-1}}{\theta_N^2}\left \langle w_N-x_\star+\frac{1}{L}\nabla f(x_\star), h'(y_N)  \right\rangle\\
   &-\frac{\theta_N-1}{L\theta_N}\left\langle \nabla f(x_N), h'(y_N)\right\rangle+\frac{\tilde{\theta}_{N-1}(4\theta_{N-1}-2\tilde{\theta}_{N-1})}{L\theta_N^2}\|h'(y_N)\|^2\\
        &=\frac{\theta_N^2-\theta_N}{\theta_N^2}\left(f(x_{N-1})-\frac{1}{2L}\|\nabla f(x_{N-1})\|^2-f(x_N)-\frac{1}{2L}\|\nabla f(x_N)\|^2\right)\\
    &+\frac{\theta_{N}}{\theta_N^2}\left(f(x_\star)-f(x_N)-\frac{1}{2L}\|\nabla f(x_N)\|^2-\frac{1}{2L}\|\nabla f(x_\star)\|^2+\frac{1}{L}\left \langle\nabla f(x_\star), \nabla f(x_N)  \right\rangle+\left \langle x_N-x_\star, \nabla f(x_N)  \right\rangle\right)\\
   &+\frac{1}{\theta_N}\left \langle w_N-x_N, \nabla f(x_N)  \right\rangle+\frac{2\tilde{\theta}_{N-1}}{\theta_N^2}\left \langle w_N-x_\star+\frac{1}{L}\nabla f(x_\star), h'(y_N)  \right\rangle\\
   &-\frac{\theta_N-1}{L\theta_N}\left\langle \nabla f(x_N), h'(y_N)\right\rangle+\frac{\tilde{\theta}_{N-1}(4\theta_{N-1}-2\tilde{\theta}_{N-1})}{L\theta_N^2}\|h'(y_N)\|^2
                 \end{align*}
   \begin{align*}
           &\geq\frac{\theta_N^2-\theta_N}{\theta_N^2}\left(f(x_{N-1})-\frac{1}{2L}\|\nabla f(x_{N-1})\|^2-f(x_N)-\frac{1}{2L}\|\nabla f(x_N)\|^2\right)+\frac{1}{\theta_N}\left \langle w_N-x_N, \nabla f(x_N)  \right\rangle  {\color{gray}\qquad\rhd\,\textrm{cocoercivity of $f$}}\\
   &+\frac{2\tilde{\theta}_{N-1}}{\theta_N^2}\left \langle w_N-x_\star+\frac{1}{L}\nabla f(x_\star), h'(y_N)  \right\rangle\\
   &-\frac{\theta_N-1}{L\theta_N}\left\langle \nabla f(x_N), h'(y_N)\right\rangle+\frac{\tilde{\theta}_{N-1}(4\theta_{N-1}-2\tilde{\theta}_{N-1})}{L\theta_N^2}\|h'(y_N)\|^2\\
            &=\frac{\theta_N^2-\theta_N}{\theta_N^2}\left(f(x_{N-1})-\frac{1}{2L}\|\nabla f(x_{N-1})\|^2-f(x_N)-\frac{1}{2L}\|\nabla f(x_N)\|^2\right)\\
            &+\frac{\theta_N-1}{\theta_N}\left \langle x_N-x_{N-1}+\frac{1}{L}\nabla f(x_{N-1})+\frac{1}{L} h'(y_{N}), \nabla f(x_N)  \right\rangle{\color{gray}\qquad\rhd\,w_N=\theta_{N}x_N-(\theta_N-1)z_N ,z_N=x_{N-1}-\frac{1}{L}\nabla f(x_{N-1})-\frac{1}{L}h'(y_N)}\\
   &+\frac{2\tilde{\theta}_{N-1}}{\theta_N^2}\left \langle w_N-x_\star+\frac{1}{L}\nabla f(x_\star), h'(y_N)  \right\rangle-\frac{\theta_N-1}{L\theta_N}\left\langle \nabla f(x_N), h'(y_N)\right\rangle+\frac{\tilde{\theta}_{N-1}(4\theta_{N-1}-2\tilde{\theta}_{N-1})}{L\theta_N^2}\|h'(y_N)\|^2   \\
               &=\frac{\theta_N^2-\theta_N}{\theta_N^2}\left(f(x_{N-1})-f(x_N)-\frac{1}{2L}\|\nabla f(x_{N-1})\|^2-\frac{1}{2L}\|\nabla f(x_N)\|^2\right)\\
            &+\frac{\theta_N-1}{\theta_N}\left \langle x_N-x_{N-1}+\frac{1}{L}\nabla f(x_{N-1}), \nabla f(x_N)  \right\rangle\\
   &+\frac{2\tilde{\theta}_{N-1}}{\theta_N^2}\left \langle w_N-x_\star+\frac{1}{L}\nabla f(x_\star), h'(y_N)  \right\rangle+\frac{\tilde{\theta}_{N-1}(4\theta_{N-1}-2\tilde{\theta}_{N-1})}{L\theta_N^2}\|h'(y_N)\|^2  \\
                &=\frac{\theta_N^2-\theta_N}{\theta_N^2}\left(f(x_{N-1})-\frac{1}{2L}\|\nabla f(x_{N-1})\|^2-f(x_N)-\frac{1}{2L}\|\nabla f(x_N)\|^2+\langle x_N-x_{N-1}, \nabla f(x_N)\rangle +\frac{1}{L}\langle \nabla f(x_{N-1}), \nabla f(x_N)\rangle \right)\\
   &+\frac{2\tilde{\theta}_{N-1}}{\theta_N^2}\left \langle w_N-x_\star+\frac{1}{L}\nabla f(x_\star), h'(y_N)  \right\rangle+\frac{\tilde{\theta}_{N-1}(4\theta_{N-1}-2\tilde{\theta}_{N-1})}{L\theta_N^2}\|h'(y_N)\|^2  \\
   &\geq\frac{2\tilde{\theta}_{N-1}}{\theta_N^2}\left \langle w_N-x_\star+\frac{1}{L}\nabla f(x_\star), h'(y_N)  \right\rangle+\frac{\tilde{\theta}_{N-1}(4\theta_{N-1}-2\tilde{\theta}_{N-1})}{L\theta_N^2}\|h'(y_N)\|^2 {\color{gray}\qquad\rhd\,\textrm{cocoercivity of $f$}} .
\end{align*}
For $k\in[-1:N-2]$, 
\begin{align*}
    &\mathcal{F}_{k}-\mathcal{F}_{k+1}=\frac{2\theta_k^2}{\theta_N^2}\left(f(x_{k})-f(x_\star)\right)+\frac{L}{2\theta_N^2}\left\|w_{k+1}-x_\star+\frac{1}{L}\nabla f(x_\star)+\frac{2\theta_k}{L}h' (y_{k+1})\right\|^2-\left(\frac{1}{2L}-\frac{\theta_k^2}{L\theta_N^2}\right)\|\nabla f(x_\star)\|^2-\frac{\theta_{k}^2}{L\theta_N^2}\|\nabla f(x_{k})\|^2\\
    &-\frac{2\theta_{k+1}^2}{\theta_N^2}\left(f(x_{k+1})-f(x_\star)\right)-\frac{L}{2\theta_N^2}\left\|w_{k+2}-x_\star+\frac{1}{L}\nabla f(x_\star)+\frac{2\theta_{k+1}}{L}h'(y_{k+2})\right\|^2+\left(\frac{1}{2L}-\frac{\theta_{k+1}^2}{L\theta_N^2}\right)\|\nabla f(x_\star)\|^2+\frac{\theta_{k+1}^2}{L\theta_N^2}\|\nabla f(x_{k+1})\|^2\\
        &=\frac{2\theta_k^2}{\theta_N^2}\left(f(x_{k})-f(x_\star)\right)-\frac{2\theta_{k+1}^2}{\theta_N^2}\left(f(x_{k+1})-f(x_\star)\right)+\frac{L}{2\theta_N^2}\left\|w_{k+1}-x_\star+\frac{1}{L}\nabla f(x_\star)+\frac{2\theta_{k}}{L}h' (y_{k+1})\right\|^2\\
        &+\left(\frac{\theta_k^2-\theta_{k+1}^2}{L\theta_N^2}\right)\|\nabla f(x_\star)\|^2
        -\frac{\theta_{k}^2}{L\theta_N^2}\|\nabla f(x_{k})\|^2
    -\frac{L}{2\theta_N^2}\left\|w_{k+1}-x_\star+\frac{1}{L}\nabla f(x_\star)-\frac{2\theta_{{k+1}}}{L}\nabla f(x_{k+1})\right\|^2+\frac{\theta_{k+1}^2}{L\theta_N^2}\|\nabla f(x_{k+1})\|^2\\
          &=\frac{2\theta_k^2}{\theta_N^2}\left(f(x_{k})-f(x_\star)-\frac{1}{2L}\|\nabla f(x_k)\|^2\right)-\frac{2\theta_{k+1}^2}{\theta_N^2}\left(f(x_{k+1})-f(x_\star)-\frac{1}{2L}\|\nabla f(x_{k+1})\|^2 \right)\\
          &+\frac{L}{\theta_N^2}\left\langle w_{k+1}-x_\star+\frac{1}{L}\nabla f(x_\star), \frac{2\theta_k}{L}h'(y_{k+1})+\frac{2\theta_{k+1}}{L}\nabla f(x_{k+1})\right\rangle -\frac{2\theta_{k+1}^2}{L\theta_N^2}\|\nabla f(x_{k+1})\|^2\\
          &+\frac{2\theta_k^2}{L\theta_N^2}\|h'(y_{k+1})\|^2-\frac{\theta_{k+1}}{L\theta_N^2}\|\nabla f(x_\star)\|^2\\
            &=\frac{2\theta_{k+1}^2-2\theta_{k+1}}{\theta_N^2}\left(f(x_{k})-f(x_{k+1})-\frac{1}{2L}\|\nabla f(x_k)\|^2-\frac{1}{2L}\|\nabla f(x_{k+1})\|^2\right)+\frac{2\theta_{k+1}}{\theta_N^2}\left(f(x_{\star})-f(x_{k+1})-\frac{1}{2L}\|\nabla f(x_{k+1})\|^2 \right)\\
          &+\frac{2\theta_{k+1}}{\theta_N^2}\left\langle w_{k+1}-x_\star, \nabla f(x_{k+1})\right\rangle +\frac{2\theta_{k+1}}{L\theta_N^2}\langle \nabla f(x_\star), \nabla f(x_{k+1})\rangle+\frac{2\theta_k}{\theta_N^2}\left\langle w_{k+1}-x_\star+\frac{1}{L}\nabla f(x_\star), h'(y_{k+1})\right \rangle\\
          &+\frac{2\theta_k^2}{L\theta_N^2}\|h'(y_{k+1})\|^2-\frac{\theta_{k+1}}{L\theta_N^2}\|\nabla f(x_\star)\|^2 {\color{gray}\qquad\rhd\,\theta_{k}^2=\theta_{k+1}^2-\theta_{k+1}}
                           \end{align*}
   \begin{align*}
               &=\frac{2\theta_{k+1}^2-2\theta_{k+1}}{\theta_N^2}\left(f(x_{k})-f(x_{k+1})-\frac{1}{2L}\|\nabla f(x_k)\|^2-\frac{1}{2L}\|\nabla f(x_{k+1})\|^2\right)\\
               &+\frac{2\theta_{k+1}}{\theta_N^2}\left(f(x_{\star})-f(x_{k+1})-\frac{1}{2L}\|\nabla f(x_{k+1})\|^2 -\frac{1}{2L}\|\nabla f(x_\star)\|^2 +\frac{1}{L}\langle \nabla f(x_\star), \nabla f(x_{k+1})\rangle +\langle x_{k+1} - x_{\star}, \nabla f(x_{k+1}) \rangle \right)\\
          &+\frac{2\theta_{k+1}}{\theta_N^2}\left\langle w_{k+1}-x_{k+1}, \nabla f(x_{k+1})\right\rangle+\frac{2\theta_k}{\theta_N^2}\left\langle w_{k+1}-x_\star+\frac{1}{L}\nabla f(x_\star), h'(y_{k+1})\right \rangle+\frac{2\theta_k^2}{L\theta_N^2}\|h'(y_{k+1})\|^2 \\ 
               &\geq \frac{2\theta_{k+1}^2-2\theta_{k+1}}{\theta_N^2}\left(f(x_{k})-f(x_{k+1})-\frac{1}{2L}\|\nabla f(x_k)\|^2-\frac{1}{2L}\|\nabla f(x_{k+1})\|^2\right)+\frac{2\theta_{k+1}}{\theta_N^2}\left\langle w_{k+1}-x_{k+1}, \nabla f(x_{k+1})\right\rangle\\
               &+\frac{2\theta_k}{\theta_N^2}\left\langle w_{k+1}-x_\star+\frac{1}{L}\nabla f(x_\star), h'(y_{k+1})\right \rangle+\frac{2\theta_{k}^2}{L\theta_N^2}\|h'(y_{\color{red}{k+1}})\|^2 {\color{gray}\qquad\rhd\,\textrm{cocoercivity of $f$}}\\
                              &=\frac{2\theta_{k+1}^2-2\theta_{k+1}}{\theta_N^2}\left(f(x_{k})-f(x_{k+1})-\frac{1}{2L}\|\nabla f(x_k)\|^2-\frac{1}{2L}\|\nabla f(x_{k+1})\|^2\right)\\
                              &+\frac{2\theta_{k+1}(\theta_{k+1}-1)}{\theta_N^2}\left\langle x_{k+1}-x_k+\frac{1}{L}\nabla f(x_k)+\frac{1}{L}h'(y_{k+1}), \nabla f(x_{k+1})\right\rangle\color{gray}\qquad\rhd\,w_{k+1}=\theta_{k+1}x_{k+1}-(\theta_{k+1}-1)z_{k+1} \\
               &+\frac{2\theta_k}{\theta_N^2}\left\langle w_{k+1}-x_\star+\frac{1}{L}\nabla f(x_\star), h'(y_{k+1})\right \rangle+\frac{2\theta_{k}^2}{L\theta_N^2}\|h'(y_{k+1})\|^2 \\ 
                                  &=\frac{2\theta_{k+1}^2-2\theta_{k+1}}{\theta_N^2}\left(f(x_{k})-f(x_{k+1})-\frac{1}{2L}\|\nabla f(x_k)\|^2-\frac{1}{2L}\|\nabla f(x_{k+1})\|^2+\langle x_{k+1}-x_k, \nabla f(x_{k+1})\rangle +\frac{1}{L}\langle \nabla f(x_k), \nabla f(x_{k+1})\rangle \right)\\
            &+\frac{2\theta_k}{\theta_N^2}\left\langle w_{k+1}-x_\star+\frac{1}{L}\nabla f(x_\star), h'(y_{k+1})\right \rangle+\frac{2\theta_k^2}{L\theta_N^2}\|h'(y_{k+1})\|^2+\frac{2\theta_k^2}{L\theta_N^2}\langle h'(y_{k+1}), \nabla f(x_{k+1})\rangle  \\
&\geq         \frac{2\theta_k}{\theta_N^2}\left\langle w_{k+1}-x_\star+\frac{1}{L}\nabla f(x_\star), h'(y_{k+1})\right \rangle+\frac{2\theta_k^2}{L\theta_N^2}\|h'(y_{k+1})\|^2+\frac{2\theta_k^2}{L\theta_N^2}\langle h'(y_{k+1}), \nabla f(x_{k+1})\rangle. {\color{gray}\qquad\rhd\,\textrm{cocoercivity of $f$}}
\end{align*}
\qed\end{proof}
Now we begin calculating $\mathcal{H}_{k}-\mathcal{H}_{k+1}$ for $k\in[-1:N-1]$. We first prove the following two lemmas.
\begin{lem}\label{lem:functionvalueforh}
The following equality holds.
\[
\sum_{i,j\in\{\star,1,\ldots,N\}}\tau_{i,j}\left(h(y_j)-h(y_i)\right)=\sum_{i=1}^N\tau_{\star,i}\left(h(y_i)-h(x_\star)\right)+\sum_{i<j\in[1:N]}\tau_{i,j}\left(h(y_j)-h(y_i)\right)+\sum_{i=1}^{N-1}\tau_{i+1,i}\left(h(y_i)-h(y_{i+1})\right)=h(y_N)-h(x_\star)
\]  
\end{lem}
\begin{proof}
We prove the lemma by showing \eqref{eq:functionvalueforh1}, \eqref{eq:functionvalueforh2}, and \eqref{eq:functionvalueforh3}.
\begin{align}
\sum_{i=1}^{N}\tau_{\star,i}=1,  \label{eq:functionvalueforh1} 
\end{align}
\begin{align}
\tau_{\star,N}+\sum_{i=1}^{N-1}\tau_{i,N}-\tau_{N,N-1}=1 , \ 
    \ \tau_{\star,1}+\tau_{2,1}-\sum_{i=2}^N\tau_{1,i}=0, \label{eq:functionvalueforh2}
\end{align}
\begin{align}
 \tau_{\star,i}+\sum_{j=1}^{i-1}\tau_{j,i}+\tau_{i+1,i} -\sum_{j=i+1}^{N}\tau_{i,j}-\tau_{i,i-1}=0 \textup{ for }  i\in[2:N-1]. \label{eq:functionvalueforh3}
\end{align}
For \eqref{eq:functionvalueforh1},
\begin{align*}
  \sum_{i=1}^{N}\tau_{\star,i}=\frac{2}{\theta_N^2-1}\sum_{i=1}^N\tilde{\theta}_{i-1} =\frac{2}{\theta_N^2-1}\cdot\frac{\theta_N^2-1}{2}=1. {\color{gray}\qquad\rhd\,\textup{using \eqref{eq:sumoftildetheta} of Lemma~\ref{lem:propertyoftildetheta} }}
\end{align*}
For \eqref{eq:functionvalueforh2}, 
\begin{align*}
&\tau_{\star,N} +\sum_{i=1}^{N-1}\tau_{i,N}-\tau_{N,N-1}
=\frac{2\tilde{\theta}_{N-1}}{\theta_N^2-1}+2\tilde{\theta}_{N-1}\sum_{i=1}^{N-1}\left(\frac{1}{\theta_N^2-2\theta_i^2+\theta_i}-\frac{1}{\theta_N^2-2\theta_{i-1}^2+\theta_{i-1}}\right)-\frac{\theta_{N-1}-1}{\theta_N^2-2\theta_{N-1}^2+\theta_{N-1}}\\
&=\frac{2\tilde{\theta}_{N-1}}{\theta_N^2-1}+2\tilde{\theta}_{N-1}\left(\frac{1}{\theta_N^2-2\theta_{N-1}^2+\theta_{N-1}}-\frac{1}{\theta_N^2-1}\right)-\frac{\theta_{N-1}-1}{\theta_N^2-2\theta_{N-1}^2+\theta_{N-1}}{\color{gray}\qquad\rhd\,\textup{telescoping sum}}\\
&=\frac{2\tilde{\theta}_{N-1}-\theta_{N-1}+1}{\theta_N^2-2\theta_{N-1}^2+\theta_{N-1}}=\frac{2\theta_{N-1}+\theta_N-1-\theta_{N-1}+1}{\theta_N+\theta_{N-1}}=1,  {\color{gray} \qquad \rhd\ \textup{substituting } \tilde{\theta}_{N-1}=\frac{2\theta_{N-1}+\theta_N-1}{2}}
\end{align*}
and
\begin{align*}
    &\tau_{\star,1}+\tau_{2,1}-\sum_{i=2}^N \tau_{1,i}=\frac{2}{\theta_N^2-1}+\frac{\theta_1-1}{\theta_N^2-2\theta_1^2+\theta_1}- \left(\frac{2}{\theta_N^2-2\theta_1^2+\theta_1}-\frac{2}{\theta_N^2-1}\right)\sum_{i=2}^N\tilde{\theta}_{i-1}\\
&=\frac{2}{\theta_N^2-1}+\frac{\theta_1-1}{\theta_N^2-2\theta_1^2+\theta_1}- \left(\frac{2}{\theta_N^2-2\theta_1^2+\theta_1}-\frac{2}{\theta_N^2-1}\right)\left(\frac{\theta_N^2-1}{2}-1\right){\color{gray}\qquad\rhd\,\textup{using \eqref{eq:sumoftildetheta} of Lemma~\ref{lem:propertyoftildetheta} }}\\
&=\frac{1}{\theta_N^2-1}\left(2+(\theta_N^2-1)-2\right)+\frac{1}{\theta_N^2-2\theta_1^2+\theta_1}\left(\theta_1-1-(\theta_N^2-1)+2\right)\\
&=1+\frac{1}{\theta_N^2-\theta_1-2}\left(-\theta_N^2+\theta_1+2\right)=0. {\color{gray}\qquad\rhd\,\textup{using \eqref{eq:recursivetheta1} of Lemma~\ref{lem:propertyoftheta} }}
\end{align*}
Lastly, for \eqref{eq:functionvalueforh3} when $i\in[2:N-1]$,
\begin{align*}
&\tau_{\star,i}+\sum_{j=1}^{i-1}\tau_{j,i}+\tau_{i+1,i}-\sum_{j=i+1}^N \tau_{i,j}-\tau_{i,i-1}\\
&=\frac{2\tilde{\theta}_{i-1}}{\theta_N^2-1}+2\tilde{\theta}_{i-1}\sum_{j=1}^{i-1}\left(\frac{1}{\theta_N^2-2\theta_j^2+\theta_j}-\frac{1}{\theta_N^2-2\theta_{j-1}^2+\theta_{j-1}}\right)+\frac{\theta_i-1}{\theta_N^2-2\theta_i^2+\theta_i}\\
&- \left(\frac{2}{\theta_N^2-2\theta_i^2+\theta_i}-\frac{2}{\theta_N^2-2\theta_{i-1}^2+\theta_{i-1}}\right)\sum_{j=i+1}^N\tilde{\theta}_{j-1}-\frac{\theta_{i-1}-1}{\theta_N^2-2\theta_{i-1}^2+\theta_{i-1}}\\
&=\frac{2\tilde{\theta}_{i-1}}{\theta_N^2-1}+2\tilde{\theta}_{i-1}\left(\frac{1}{\theta_N^2-2\theta_{i-1}^2+\theta_{i-1}}-\frac{1}{\theta_N^2-1}\right)+\frac{\theta_i-1}{\theta_N^2-2\theta_i^2+\theta_i}{\color{gray}\qquad\rhd\,\textup{telescoping sum}}\\
&-\left(\frac{2}{\theta_N^2-2\theta_i^2+\theta_i}-\frac{2}{\theta_N^2-2\theta_{i-1}^2+\theta_{i-1}}\right)\left(\frac{\theta_N^2-1}{2}-\tilde{\theta}_{i-1}^2\right){\color{gray}\qquad\rhd\,\textup{using \eqref{eq:sumoftheta} of Lemma~\ref{lem:propertyoftheta}, \eqref{eq:sumoftildetheta} of Lemma~\ref{lem:propertyoftildetheta}}}\\
&-\frac{\theta_{i-1}-1}{\theta_N^2-2\theta_{i-1}^2+\theta_{i-1}}\\
&=\frac{1}{\theta_N^2-2\theta_{i}^2+\theta_{i}}\left(\theta_i-1-(\theta_N^2-1)+2\tilde{\theta}_{i-1}^2\right)+\frac{1}{\theta_N^2-2\theta_{i-1}^2+\theta_{i-1}}\left(2\tilde{\theta}_{i-1}+(\theta_N^2-1)-2\tilde{\theta}_{i-1}^2-\theta_{i-1}+1\right)\\
&=\frac{1}{\theta_N^2-2\theta_{i}^2+\theta_{i}}\left(-\theta_N^2+\theta_i+2\theta_i^2-2\theta_i\right){\color{gray}\qquad\rhd\,\textup{using \eqref{eq:recursivetheta1} of Lemma~\ref{lem:propertyoftheta} and } \tilde{\theta}_{i-1}=\theta_{i-1}}\\
&+\frac{1}{\theta_N^2-2\theta_{i-1}^2+\theta_{i-1}}\left(\theta_N^2-2\theta_{i-1}^2+\theta_{i-1}\right) {\color{gray}\qquad\rhd\,\textup{using } \tilde{\theta}_{i-1}=\theta_{i-1}}\\
&=-1+1=0.    
\end{align*}
\qed\end{proof}
\begin{lem}
The following holds.
\begin{align*}
\mathcal{H}_{N-1}-\mathcal{H}_{N}&=\sum_{i=1}^{N-1}\tau_{i,N}\left(h(y_i)-h(y_N)\right)+\tau_{\star,N}\left(h(x_\star)-h(y_N)\right)+\tau_{N,N-1}\left(h(y_{N})-h(y_{N-1})\right)\\
&+\frac{2\tilde{\theta}_{N-1}}{\theta_N^2(\theta_N^2-1)}\left\langle x_0-x_\star-\frac{\theta_N^2-1}{L}\nabla f(x_\star),h'(y_N)\right \rangle -\frac{\tilde{\theta}_{N-1}+\theta_{N-2}^2}{L\theta_N^2}\|h'(y_N)\|^2+\frac{2\theta_{N-2}^2}{L\theta_N^2}\langle  h'(y_{N-1}) , h'(y_N)+\nabla f(x_{N-1})\rangle.
\end{align*}
For $k\in[0:N-2]$,
\begin{align*}
\mathcal{H}_{k}-\mathcal{H}_{k+1}&=\sum_{i=1}^{k}\tau_{i,k+1}\left(h(y_i)-h(y_{k+1})\right)+\tau_{\star,k+1}\left(h(x_\star)-h(y_{k+1})\right)+\tau_{k+1,k}\left(h(y_{k+1})-h(y_{k})\right)\\
    &+\frac{2\theta_{k}}{\theta_N^2(\theta_N^2-1)}\left\langle x_0-x_\star-\frac{\theta_N^2-1}{L}\nabla f(x_\star),h'(y_{k+1})\right \rangle-\frac{2\theta_k^2}{L\theta_N^2}\|h'(y_{k+1})\|^2+\frac{2{\theta}_{k-1}^2}{L\theta_N^2}\langle h'(y_k), h'(y_{k+1})+\nabla f(x_k)\rangle \\
    &-\frac{2\theta_k^2}{L\theta_N^2}\langle h'(y_{k+1}), \nabla f(x_{k+1})\rangle.
\end{align*}
\end{lem}
\begin{proof}
For $k=N-1$,
\begin{align*}
    &\mathcal{H}_{N-1}-\mathcal{H}_{N}=\sum_{i,j\in\{\star,1,\dots, N-1\}} \tau_{i,j}(h(y_j)-h(y_i)) - h(y_N)+ h(x_\star)\\
    &+\frac{L}{2\theta_N^2(\theta_N^2-1)}\left[\left\|x_0-x_\star-\frac{\theta_N^2-1}{L}\nabla f(x_\star)-\sum_{i=0}^{N-2}\frac{2{\theta}_i}{L} h'(y_{i+1})\right\|^2-\left\|x_0-x_\star-\frac{\theta_N^2-1}{L}\nabla f(x_\star)-\sum_{i=0}^{N-1}\frac{2\tilde{\theta}_i}{L} h'(y_{i+1})\right\|^2\right]\\
   &+\sum_{i\neq j, i,j\in[1:N-1]}\frac{{\theta}_{i-1}{\theta}_{j-1}}{L\theta_N^2(\theta_N^2-1)}\|h'(y_i)-h'(y_j)\|^2+\sum_{i=1}^{N-2}\frac{{\theta}_{i-1}^2}{L\theta_N^2}\|h'(y_i)-h'(y_{i+1})\|^2\\
   &-\sum_{i\neq j, i,j\in[1:N]}\frac{\tilde{\theta}_{i-1}\tilde{\theta}_{j-1}}{L\theta_N^2(\theta_N^2-1)}\|h'(y_i)-h'(y_j)\|^2-\sum_{i=1}^{N-1}\frac{\tilde{\theta}_{i-1}^2}{L\theta_N^2}\|h'(y_i)-h'(y_{i+1})\|^2\\
   &+\frac{2\theta_{N-2}^2}{L\theta_N^2}\langle \nabla f(x_{N-1}), h'(y_{N-1})\rangle+\sum_{i=1}^{N-1}\frac{2\tilde{\theta}_{N-1}{\theta}_{i-1}}{L\theta_N^2(\theta_N^2-1)}\|h'(y_i)\|^2+\frac{{\theta}_{N-2}^2}{L\theta_N^2}\|h'(y_{N-1})\|^2\\
       &=\sum_{i,j\in\{\star,1,\dots, N-1\}} \tau_{i,j}(h(y_j)-h(y_i)) - \sum_{i,j\in\{\star,1,\dots, N\}} \tau_{i,j}(h(y_j)-h(y_i))\\
    & +\frac{L}{\theta_N^2(\theta_N^2-1)}\left\langle x_0-x_\star-\frac{\theta_N^2-1}{L}\nabla f(x_\star)-\sum_{i=0}^{N-2}\frac{2\theta_i}{L}h'({\color{red}{y_{i+1}}}),\frac{2\tilde{\theta}_{N-1}}{L}h'(y_N)\right \rangle-\frac{2\tilde{\theta}_{N-1}^2}{L\theta_N^2(\theta_N^2-1)}\|h'(y_N)\|^2 \\
   &-\sum_{i=1}^{N-1}\frac{2\tilde{\theta}_{N-1}\tilde{\theta}_{i-1}}{L\theta_N^2(\theta_N^2-1)}\|h'(y_N)-h'(y_{i})\|^2-\frac{\tilde{\theta}_{N-2}^2}{L\theta_N^2}\|h'(y_N)-h'(y_{N-1})\|^2\\
   &+\frac{2\theta_{N-2}^2}{L\theta_N^2}\langle \nabla f(x_{N-1}), h'(y_{N-1})\rangle+\sum_{i=1}^{N-1}\frac{2\tilde{\theta}_{N-1}{\theta}_{i-1}}{L\theta_N^2(\theta_N^2-1)}\|h'(y_i)\|^2+\frac{{\theta}_{N-2}^2}{L\theta_N^2}\|h'(y_{N-1})\|^2\\
    &=\sum_{i=1}^{N-1}\tau_{i,N}\left(h(y_i)-h(y_N)\right)+\tau_{\star,N}\left(h(x_\star)-h(y_N)\right)+\tau_{N,N-1}\left(h(y_{N})-h(y_{N-1})\right)\\
      & +\frac{L}{\theta_N^2(\theta_N^2-1)}\left\langle x_0-x_\star-\frac{\theta_N^2-1}{L}\nabla f(x_\star),\frac{2\tilde{\theta}_{N-1}}{L}h'(y_N)\right \rangle-\sum_{i=1}^{N-1}\frac{4\theta_{i-1}\tilde{\theta}_{N-1}}{L\theta_N^2(\theta_N^2-1)}\langle h'(y_i), h'(y_N)\rangle -\frac{2\tilde{\theta}_{N-1}^2}{L\theta_N^2(\theta_N^2-1)}\|h'(y_N)\|^2 \\
   &-\sum_{i=1}^{N-1}\frac{2\tilde{\theta}_{N-1}\tilde{\theta}_{i-1}}{L\theta_N^2(\theta_N^2-1)}\|h'(y_N)\|^2+\sum_{i=1}^{N-1}\frac{4\theta_{i-1}\tilde{\theta}_{N-1}}{L\theta_N^2(\theta_N^2-1)}\langle h'(y_i), h'(y_N)\rangle-\sum_{i=1}^{N-1}\frac{2\tilde{\theta}_{N-1}{\theta}_{i-1}}{L\theta_N^2(\theta_N^2-1)}\|h'(y_i)\|^2\\
   &-\frac{\tilde{\theta}_{N-2}^2}{L\theta_N^2}\|h'(y_N)-h'(y_{N-1})\|^2+\frac{2\theta_{N-2}^2}{L\theta_N^2}\langle \nabla f(x_{N-1}), h'(y_{N-1})\rangle+\sum_{i=1}^{N-1}\frac{2\tilde{\theta}_{N-1}{\theta}_{i-1}}{L\theta_N^2(\theta_N^2-1)}\|h'(y_i)\|^2+\frac{{\theta}_{N-2}^2}{L\theta_N^2}\|h'(y_{N-1})\|^2\\
   &=\sum_{i=1}^{N-1}\tau_{i,N}\left(h(y_i)-h(y_N)\right)+\tau_{\star,N}\left(h(x_\star)-h(y_N)\right)+\tau_{N,N-1}\left(h(y_{N})-h(y_{N-1})\right)\\
        & +\frac{L}{\theta_N^2(\theta_N^2-1)}\left\langle x_0-x_\star-\frac{\theta_N^2-1}{L}\nabla f(x_\star),\frac{2\tilde{\theta}_{N-1}}{L}h'(y_N)\right \rangle -\frac{2\tilde{\theta}_{N-1}^2}{L\theta_N^2(\theta_N^2-1)}\|h'(y_N)\|^2 \\
   &-\sum_{i=1}^{N-1}\frac{2\tilde{\theta}_{N-1}\tilde{\theta}_{i-1}}{L\theta_N^2(\theta_N^2-1)}\|h'(y_N)\|^2-\frac{\tilde{\theta}_{N-2}^2}{L\theta_N^2}\|h'(y_N)-h'(y_{N-1})\|^2+\frac{2\theta_{N-2}^2}{L\theta_N^2}\langle \nabla f(x_{N-1}), h'(y_{N-1})\rangle+\frac{{\theta}_{N-2}^2}{L\theta_N^2}\|h'(y_{N-1})\|^2\\
   &=\sum_{i=1}^{N-1}\tau_{i,N}\left(h(y_i)-h(y_N)\right)+\tau_{\star,N}\left(h(x_\star)-h(y_N)\right)+\tau_{N,N-1}\left(h(y_{N})-h(y_{N-1})\right)\\
    & +\frac{2\tilde{\theta}_{N-1}}{\theta_N^2(\theta_N^2-1)}\left\langle x_0-x_\star-\frac{\theta_N^2-1}{L}\nabla f(x_\star),h'(y_N)\right \rangle -\frac{2\tilde{\theta}_{N-1}^2}{L\theta_N^2(\theta_N^2-1)}\|h'(y_N)\|^2 \\
   &-\frac{\tilde{\theta}_{N-1}}{L\theta_N^2(\theta_N^2-1)}\left(\theta_N^2-1-2\tilde{\theta}_{N-1}\right)\|h'(y_N)\|^2-\frac{\tilde{\theta}_{N-2}^2}{L\theta_N^2}\|h'(y_N)\|^2+\frac{2\theta_{N-2}^2}{L\theta_N^2}\langle h'(y_N), h'(y_{N-1})\rangle-\frac{\theta_{N-2}^2}{L\theta_N^2}\|h'(y_{N-1})\|^2\\
   &+\frac{2\theta_{N-2}^2}{L\theta_N^2}\langle \nabla f(x_{N-1}), h'(y_{N-1})\rangle+\frac{{\theta}_{N-2}^2}{L\theta_N^2}\|h'(y_{N-1})\|^2 {\color{gray}\qquad\rhd\,\sum_{i=1}^{N-1}\tilde{\theta}_{i-1}=\sum_{i=0}^{N-1}\tilde{\theta}_{i}-\tilde{\theta}_{N-1}=\frac{\theta_N^2-1}{2}-\tilde{\theta}_{N-1}}
   \end{align*}
   \begin{align*}
   &=\sum_{i=1}^{N-1}\tau_{i,N}\left(h(y_i)-h(y_N)\right)+\tau_{\star,N}\left(h(x_\star)-h(y_N)\right)+\tau_{N,N-1}\left(h(y_{N})-h(y_{N-1})\right)\\
    & +\frac{2\tilde{\theta}_{N-1}}{\theta_N^2(\theta_N^2-1)}\left\langle x_0-x_\star-\frac{\theta_N^2-1}{L}\nabla f(x_\star),h'(y_N)\right \rangle -\frac{\tilde{\theta}_{N-1}}{L\theta_N^2}\|h'(y_N)\|^2-\frac{\tilde{\theta}_{N-2}^2}{L\theta_N^2}\|h'(y_N)\|^2\\
    &+\frac{2\theta_{N-2}^2}{L\theta_N^2}\langle h'(y_N), h'(y_{N-1})\rangle+\frac{2\theta_{N-2}^2}{L\theta_N^2}\langle \nabla f(x_{N-1}), h'(y_{N-1})\rangle\\
    &=\sum_{i=1}^{N-1}\tau_{i,N}\left(h(y_i)-h(y_N)\right)+\tau_{\star,N}\left(h(x_\star)-h(y_N)\right)+\tau_{N,N-1}\left(h(y_{N})-h(y_{N-1})\right)\\
        & +\frac{2\tilde{\theta}_{N-1}}{\theta_N^2(\theta_N^2-1)}\left\langle x_0-x_\star-\frac{\theta_N^2-1}{L}\nabla f(x_\star),h'(y_N)\right \rangle -\frac{\tilde{\theta}_{N-1}+\theta_{N-2}^2}{L\theta_N^2}\|h'(y_N)\|^2+\frac{2\theta_{N-2}^2}{L\theta_N^2}\langle h'(y_{N-1}) , h'(y_N)+\nabla f(x_{N-1})\rangle
\end{align*}

For $k\in[-1:N-2]$,
\begin{align*}
    &\mathcal{H}_{k}-\mathcal{H}_{k+1}= \sum_{i,j\in\{\star,1,\dots,k\}}\tau_{i,j}\left(h(y_j)-h(y_i)\right)-\sum_{i,j\in\{\star,1,\dots,k+1\}}\tau_{i,j}\left(h(y_j)-h(y_i)\right)\\
    &+\frac{L}{2\theta_N^2(\theta_N^2-1)}\left[\left\|x_0-x_\star-\frac{\theta_N^2-1}{L}\nabla f(x_\star)-\sum_{i=0}^{k-1}\frac{2{\theta}_i}{L} h'(y_{i+1})\right\|^2-\left\|x_0-x_\star-\frac{\theta_N^2-1}{L}\nabla f(x_\star)-\sum_{i=0}^{k}\frac{2{\theta}_i}{L} h'(y_{i+1})\right\|^2\right]\\
    &+\sum_{i\neq j, i,j\in[1:k]}\frac{{\theta}_{i-1}{\theta}_{j-1}}{L\theta_N^2(\theta_N^2-1)}\|h'(y_i)-h'(y_j)\|^2+\sum_{i=1}^{k-1}\frac{{\theta}_{i-1}^2}{L\theta_N^2}\|h'(y_i)-h'(y_{i+1})\|^2\\
    &+\frac{2\theta_{k-1}^2}{L\theta_N^2}\langle \nabla f(x_{k}), h'(y_{k})\rangle+\sum_{i=1}^{k}\sum_{\ell=k}^{N-1}\frac{2\tilde{\theta}_{\ell}{\theta}_{i-1}}{L\theta_N^2(\theta_N^2-1)}\|h'(y_i)\|^2+\frac{{\theta}_{k-1}^2}{L\theta_N^2}\|h'(y_{k})\|^2\\
    &-\sum_{i\neq j, i,j\in[1:k+1]}\frac{{\theta}_{i-1}{\theta}_{j-1}}{L\theta_N^2(\theta_N^2-1)}\|h'(y_i)-h'(y_j)\|^2-\sum_{i=1}^{k}\frac{{\theta}_{i-1}^2}{L\theta_N^2}\|h'(y_i)-h'(y_{i+1})\|^2\\
    &-\frac{2\theta_{k}^2}{L\theta_N^2}\langle \nabla f(x_{k+1}), h'(y_{k+1})\rangle-\sum_{i=1}^{k+1}\sum_{\ell=k+1}^{N-1}\frac{2\tilde{\theta}_{\ell}{\theta}_{i-1}}{L\theta_N^2(\theta_N^2-1)}\|h'(y_i)\|^2-\frac{{\theta}_{k}^2}{L\theta_N^2}\|h'(y_{k+1})\|^2\\
    &=\sum_{i=1}^{k}\tau_{i,k+1}\left(h(y_i)-h(y_{k+1})\right)+\tau_{\star,k+1}\left(h(x_\star)-h(y_{k+1})\right)+\tau_{k+1,k}\left(h(y_{k+1})-h(y_{k})\right)\\
    &+\frac{L}{2\theta_N^2(\theta_N^2-1)}\left[ \left\|x_0-x_\star-\frac{\theta_N^2-1}{L}\nabla f(x_\star)-\sum_{i=0}^{k}\frac{2{\theta}_i}{L} h'(y_{i+1})\right\|^2-\left\|x_0-x_\star-\frac{\theta_N^2-1}{L}\nabla f(x_\star)-\sum_{i=0}^{k-1}\frac{2{\theta}_i}{L} h'(y_{i+1})\right\|^2\right]\\
   &-\sum_{i=1}^{k}\frac{2{\theta}_{k}{\theta}_{i-1}}{L\theta_N^2(\theta_N^2-1)}\|h'(y_{k+1})-h'(y_{i})\|^2-\frac{{\theta}_{k-1}^2}{L\theta_N^2}\|h'(y_k)-h'(y_{k+1})\|^2\\
    &+\frac{2\theta_{k-1}^2}{L\theta_N^2}\langle \nabla f(x_{k}), h'(y_{k})\rangle-\frac{2\theta_k^2}{L\theta_N^2}\langle h'(y_{k+1}), \nabla f(x_{k+1})\rangle+\sum_{i=1}^{k}\sum_{\ell=k}^{N-1}\frac{2\tilde{\theta}_{\ell}{\theta}_{i-1}}{L\theta_N^2(\theta_N^2-1)}\|h'(y_i)\|^2-\sum_{i=1}^{k+1}\sum_{\ell=k+1}^{N-1}\frac{2\tilde{\theta}_{\ell}{\theta}_{i-1}}{L\theta_N^2(\theta_N^2-1)}\|h'(y_i)\|^2\\
    &+\frac{\theta_{k-1}^2}{L\theta_N^2}\|h'(y_k)\|^2-\frac{{\theta}_{k}^2}{L\theta_N^2}\|h'(y_{k+1})\|^2\\
        &=\sum_{i=1}^{k}\tau_{i,k+1}\left(h(y_i)-h(y_{k+1})\right)+\tau_{\star,k+1}\left(h(x_\star)-h(y_{k+1})\right)+\tau_{k+1,k}\left(h(y_{k+1})-h(y_{k})\right)\\
    &+\frac{2\theta_{k}}{\theta_N^2(\theta_N^2-1)}\left\langle x_0-x_\star-\frac{\theta_N^2-1}{L}\nabla f(x_\star)-\sum_{i=0}^{k-1}\frac{2\theta_i}{L}h'(y_{i+1}),h'(y_{k+1})\right \rangle-\frac{2{\theta}_{k}^2}{L\theta_N^2(\theta_N^2-1)}\|h'(y_{k+1})\|^2 \\
                    \end{align*}
    \begin{align*}
   &-\sum_{i=1}^{k}\frac{2{\theta}_{k}{\theta}_{i-1}}{L\theta_N^2(\theta_N^2-1)}\|h'(y_{k+1})-h'(y_{i})\|^2-\frac{{\theta}_{k-1}^2}{L\theta_N^2}\|h'(y_k)-h'(y_{k+1})\|^2\\
    &+\frac{2\theta_{k-1}^2}{L\theta_N^2}\langle \nabla f(x_{k}), h'(y_{k})\rangle-\frac{2\theta_k^2}{L\theta_N^2}\langle h'(y_{k+1}), \nabla f(x_{k+1})\rangle+\sum_{i=1}^{k}\sum_{\ell=k}^{N-1}\frac{2\tilde{\theta}_{\ell}{\theta}_{i-1}}{L\theta_N^2(\theta_N^2-1)}\|h'(y_i)\|^2-\sum_{i=1}^{k+1}\sum_{\ell=k+1}^{N-1}\frac{2\tilde{\theta}_{\ell}{\theta}_{i-1}}{L\theta_N^2(\theta_N^2-1)}\|h'(y_i)\|^2\\
    &+\frac{\theta_{k-1}^2}{L\theta_N^2}\|h'(y_k)\|^2-\frac{{\theta}_{k}^2}{L\theta_N^2}\|h'(y_{k+1})\|^2\\
            &=\sum_{i=1}^{k}\tau_{i,k+1}\left(h(y_i)-h(y_{k+1})\right)+\tau_{\star,k+1}\left(h(x_\star)-h(y_{k+1})\right)+\tau_{k+1,k}\left(h(y_{k+1})-h(y_{k})\right)\\
    &+\frac{2\theta_{k}}{\theta_N^2(\theta_N^2-1)}\left\langle x_0-x_\star-\frac{\theta_N^2-1}{L}\nabla f(x_\star),h'(y_{k+1})\right \rangle-\sum_{i=0}^{k-1}\frac{4\theta_k \theta_i}{L\theta_N^2(\theta_N^2-1)}\langle h'(y_{i+1}), h'(y_{k+1})\rangle-\frac{2{\theta}_{k}^2}{L\theta_N^2(\theta_N^2-1)}\|h'(y_{k+1})\|^2 \\
   &-\sum_{i=1}^{k}\frac{2{\theta}_{k}{\theta}_{i-1}}{L\theta_N^2(\theta_N^2-1)}\|h'(y_{k+1})\|^2-\sum_{i=1}^{k}\frac{2{\theta}_{k}{\theta}_{i-1}}{L\theta_N^2(\theta_N^2-1)}\|h'(y_{i})\|^2+\sum_{i=1}^{k}
    \frac{4\theta_k\theta_{i-1}}{L\theta_N^2(\theta_N^2-1)}\langle h'(y_{i}), h'(y_{k+1})\rangle-\frac{{\theta}_{k-1}^2}{L\theta_N^2}\|h'(y_k)-h'(y_{k+1})\|^2\\
    &+\frac{2\theta_{k-1}^2}{L\theta_N^2}\langle \nabla f(x_{k}), h'(y_{k})\rangle-\frac{2\theta_k^2}{L\theta_N^2}\langle h'(y_{k+1}), \nabla f(x_{k+1})\rangle+\sum_{i=1}^{k}\sum_{\ell=k}^{N-1}\frac{2\tilde{\theta}_{\ell}{\theta}_{i-1}}{L\theta_N^2(\theta_N^2-1)}\|h'(y_i)\|^2-\sum_{i=1}^{k}\sum_{\ell=k+1}^{N-1}\frac{2\tilde{\theta}_{\ell}{\theta}_{i-1}}{L\theta_N^2(\theta_N^2-1)}\|h'(y_i)\|^2\\
    &-\sum_{\ell=k+1}^{N-1}\frac{2\tilde{\theta}_{\ell}{\theta}_{k}}{L\theta_N^2(\theta_N^2-1)}\|h'(y_{k+1})\|^2+\frac{\theta_{k-1}^2}{L\theta_N^2}\|h'(y_k)\|^2-\frac{\theta_{k}^2}{L\theta_N^2}\|h'(y_{k+1})\|^2\\
               &=\sum_{i=1}^{k}\tau_{i,k+1}\left(h(y_i)-h(y_{k+1})\right)+\tau_{\star,k+1}\left(h(x_\star)-h(y_{k+1})\right)+\tau_{k+1,k}\left(h(y_{k+1})-h(y_{k})\right)\\
    &+\frac{2\theta_{k}}{\theta_N^2(\theta_N^2-1)}\left\langle x_0-x_\star-\frac{\theta_N^2-1}{L}\nabla f(x_\star),h'(y_{k+1})\right \rangle-\frac{2{\theta}_{k}^2}{L\theta_N^2(\theta_N^2-1)}\|h'(y_{k+1})\|^2\\
   &-\sum_{i=0}^{k-1}\frac{2{\theta}_{k}\tilde{\theta}_{i}}{L\theta_N^2(\theta_N^2-1)}\|h'(y_{k+1})\|^2-\sum_{i=1}^{k}\frac{2{\theta}_{k}{\theta}_{i-1}}{L\theta_N^2(\theta_N^2-1)}\|h'(y_{i})\|^2-\frac{{\theta}_{k-1}^2}{L\theta_N^2}\|h'(y_k)-h'(y_{k+1})\|^2\\
    &+\frac{2\theta_{k-1}^2}{L\theta_N^2}\langle \nabla f(x_{k}), h'(y_{k})\rangle-\frac{2\theta_k^2}{L\theta_N^2}\langle h'(y_{k+1}), \nabla f(x_{k+1})\rangle+\sum_{i=1}^{k}\frac{2{\theta}_{k}{\theta}_{i-1}}{L\theta_N^2(\theta_N^2-1)}\|h'(y_i)\|^2-\sum_{\ell=k+1}^{N-1}\frac{2\tilde{\theta}_{\ell}{\theta}_{k}}{L\theta_N^2(\theta_N^2-1)}\|h'(y_{k+1})\|^2\\
    &+\frac{\theta_{k-1}^2}{L\theta_N^2}\|h'(y_k)\|^2-\frac{\theta_{k}^2}{L\theta_N^2}\|h'(y_{k+1})\|^2\\
            &=\sum_{i=1}^{k}\tau_{i,k+1}\left(h(y_i)-h(y_{k+1})\right)+\tau_{\star,k+1}\left(h(x_\star)-h(y_{k+1})\right)+\tau_{k+1,k}\left(h(y_{k+1})-h(y_{k})\right)\\
    &+\frac{2\theta_{k}}{\theta_N^2(\theta_N^2-1)}\left\langle x_0-x_\star-\frac{\theta_N^2-1}{L}\nabla f(x_\star),h'(y_{k+1})\right \rangle-\frac{2{\theta}_{k}^2}{L\theta_N^2(\theta_N^2-1)}\|h'(y_{k+1})\|^2 \\
   &-\frac{2{\theta}_{k}\theta_{k-1}^2}{L\theta_N^2(\theta_N^2-1)}\|h'(y_{k+1})\|^2-\frac{{\theta}_{k-1}^2}{L\theta_N^2}\|h'(y_k)-h'(y_{k+1})\|^2\\
    &+\frac{2\theta_{k-1}^2}{L\theta_N^2}\langle \nabla f(x_{k}), h'(y_{k})\rangle-\frac{2\theta_k^2}{L\theta_N^2}\langle h'(y_{k+1}), \nabla f(x_{k+1})\rangle-\frac{(\theta_N^2-1-2\theta_k^2){\theta}_{k}}{L\theta_N^2(\theta_N^2-1)}\|h'(y_{k+1})\|^2 {\color{gray}\qquad\rhd\,\sum_{\ell=k+1}^{N-1}\tilde{\theta}_{\ell}=\sum_{\ell=0}^{N-1}\tilde{\theta}_{\ell}-\sum_{\ell=0}^{k}\tilde{\theta}_{\ell}=\frac{\theta_N^2-1}{2}-\theta_k^2}\\
    &+\frac{\theta_{k-1}^2}{L\theta_N^2}\|h'(y_k)\|^2-\frac{\theta_{k}^2}{L\theta_N^2}\|h'(y_{k+1})\|^2\\
                &=\sum_{i=1}^{k}\tau_{i,k+1}\left(h(y_i)-h(y_{k+1})\right)+\tau_{\star,k+1}\left(h(x_\star)-h(y_{k+1})\right)+\tau_{k+1,k}\left(h(y_{k+1})-h(y_{k})\right)\\
    &+\frac{2\theta_{k}}{\theta_N^2(\theta_N^2-1)}\left\langle x_0-x_\star-\frac{\theta_N^2-1}{L}\nabla f(x_\star),h'(y_{k+1})\right \rangle-\frac{2{\theta}_{k}^2}{L\theta_N^2(\theta_N^2-1)}\|h'(y_{k+1})\|^2 \\
   &-\frac{2{\theta}_{k}\theta_{k-1}^2}{L\theta_N^2(\theta_N^2-1)}\|h'(y_{k+1})\|^2-\frac{{\theta}_{k-1}^2}{L\theta_N^2}\|h'(y_k)\|^2-\frac{{\theta}_{k-1}^2}{L\theta_N^2}\|h'(y_{k+1})\|^2+\frac{2{\theta}_{k-1}^2}{L\theta_N^2}\langle h'(y_k), h'(y_{k+1})\rangle \\
    &+\frac{2\theta_{k-1}^2}{L\theta_N^2}\langle \nabla f(x_{k}), h'(y_{k})\rangle-\frac{2\theta_k^2}{L\theta_N^2}\langle h'(y_{k+1}), \nabla f(x_{k+1})\rangle-\frac{{\theta}_{k}}{L\theta_N^2}\|h'(y_{k+1})\|^2+\frac{2{\theta}_{k}^3}{L\theta_N^2(\theta_N^2-1)}\|h'(y_{k+1})\|^2\\
    &+\frac{\theta_{k-1}^2}{L\theta_N^2}\|h'(y_k)\|^2-\frac{\theta_{k}^2}{L\theta_N^2}\|h'(y_{k+1})\|^2
                        \end{align*}
    \begin{align*}
      &=\sum_{i=1}^{k}\tau_{i,k+1}\left(h(y_i)-h(y_{k+1})\right)+\tau_{\star,k+1}\left(h(x_\star)-h(y_{k+1})\right)+\tau_{k+1,k}\left(h(y_{k+1})-h(y_{k})\right)\\
    &+\frac{2\theta_{k}}{\theta_N^2(\theta_N^2-1)}\left\langle x_0-x_\star-\frac{\theta_N^2-1}{L}\nabla f(x_\star),h'(y_{k+1})\right \rangle-\frac{2{\theta}_{k}^2}{L\theta_N^2(\theta_N^2-1)}\|h'(y_{k+1})\|^2 \\
   &+\frac{2{\theta}_{k}(\theta_k^2-\theta_{k-1}^2)}{L\theta_N^2(\theta_N^2-1)}\|h'(y_{k+1})\|^2-\frac{{\theta}_{k-1}^2+\theta_k+\theta_k^2}{L\theta_N^2}\|h'(y_{k+1})\|^2+\frac{2{\theta}_{k-1}^2}{L\theta_N^2}\langle h'(y_k), h'(y_{k+1})\rangle \\
    &+\frac{2\theta_{k-1}^2}{L\theta_N^2}\langle \nabla f(x_{k}), h'(y_{k})\rangle-\frac{2\theta_k^2}{L\theta_N^2}\langle h'(y_{k+1}), \nabla f(x_{k+1})\rangle\\
      &=\sum_{i=1}^{k}\tau_{i,k+1}\left(h(y_i)-h(y_{k+1})\right)+\tau_{\star,k+1}\left(h(x_\star)-h(y_{k+1})\right)+\tau_{k+1,k}\left(h(y_{k+1})-h(y_{k})\right)\\
    &+\frac{2\theta_{k}}{\theta_N^2(\theta_N^2-1)}\left\langle x_0-x_\star-\frac{\theta_N^2-1}{L}\nabla f(x_\star),h'(y_{k+1})\right \rangle-\frac{2\theta_k^2}{L\theta_N^2}\|h'(y_{k+1})\|^2+\frac{2{\theta}_{k-1}^2}{L\theta_N^2}\langle h'(y_k), h'(y_{k+1})\rangle \\
    &+\frac{2\theta_{k-1}^2}{L\theta_N^2}\langle \nabla f(x_{k}), h'(y_{k})\rangle-\frac{2\theta_k^2}{L\theta_N^2}\langle h'(y_{k+1}), \nabla f(x_{k+1})\rangle
\end{align*}
\qed\end{proof}
Now we build our last part of the proof. 
\begin{lem}\label{g:theta}\
The following holds.
\[
\tilde{\theta}_{N-1}(4\theta_{N-1}-2\tilde{\theta}_{N-1})=\tilde{\theta}_{N-1}+\theta_{N-2}^2
\]
\end{lem}
\begin{proof}
        \begin{align*}
           &\tilde{\theta}_{N-1}(4\theta_{N-1}-2\tilde{\theta}_{N-1})-\tilde{\theta}_{N-1}-\theta_{N-2}^2 = 4\theta_{N-1}\tilde{\theta}_{N-1}-2\tilde{\theta}_{N-1}^2-\tilde{\theta}_{N-1}-\theta_{N-2}^2 \\
            &=\frac{1}{2}\left(8\theta_{N-1}\tilde{\theta}_{N-1}-4\tilde{\theta}_{N-1}^2-2\tilde{\theta}_{N-1}-2\theta_{N-2}^2\right)\\
            &= \frac{1}{2}\left(4\theta_{N-1}(2\theta_{N-1}+\theta_N-1)-(2\theta_{N-1}+\theta_N-1)^2-2\tilde{\theta}_{N-1}-2\theta_{N-2}^2 \right)\\
             &= \frac{1}{2}\left(4\theta_{N-1}^2+2\theta_N-\theta_N^2-1-2\tilde{\theta}_{N-1}-2\theta_{N-2}^2\right)\\
             &= \frac{1}{2}\left(2\theta_{N-1}^2+2\theta_{N-1}+2\theta_N-\theta_N^2-1-2\tilde{\theta}_{N-1}\right)= \frac{1}{2}\left(2\theta_{N-1}+\theta_N-1-{\color{red}{2}}\tilde{\theta}_{N-1}\right)=0. 
        \end{align*}
\qed\end{proof}
\begin{lem}\label{g:sumoftaus2}
The following holds for fixed $j\in[1:N]$.
\[
\tau_{\star,j}\langle h'(y_j), x_\star-y_j \rangle+\sum_{i=1}^{j-1}\tau_{i,j}\langle h'(y_j), y_i-y_j \rangle = \left\langle h'(y_j)\ , \ \tau_{\star,j}(x_\star-x_0)+\sum_{i=0}^{j-1}\frac{4\theta_i\tilde{\theta}_{j-1}}{L\theta_N^2}\left(\nabla f(x_i)+h'(y_{i+1})\right)\right\rangle
\]
\end{lem}
\begin{proof}
Using Lemma~\ref{lem:equivalenceOptISTA}, we have:
\begin{align*}
& \tau_{\star,j}\langle h'(y_j), x_\star-y_j \rangle+\sum_{i=1}^{j-1}\tau_{i,j}\langle h'(y_j), y_i-y_j \rangle = \tau_{\star,j}\left\langle h'(y_j), x_\star-x_0+\sum_{i=0}^{j-1}\frac{\gamma_i}{L}(\nabla f(x_i)+ h'(y_{i+1}))\right\rangle+\sum_{i=1}^{j-1}\tau_{i,j}\langle h'(y_j), y_i-y_j \rangle \\
&= \tau_{\star,j}\left\langle h'(y_j), x_\star-x_0+\sum_{i=0}^{j-1}\frac{\gamma_i}{L}(\nabla f(x_i)+ h'(y_{i+1}))\right\rangle+\sum_{i=1}^{j-1}\tau_{i,j}\left \langle h'(y_j), \sum_{k=i}^{j-1}\frac{\gamma_k}{L}(\nabla f(x_k)+ h'(y_{k+1}))\right \rangle\\
&= \langle h'(y_j), \tau_{\star,j}(x_\star-x_0)\rangle+ \left\langle h'(y_j), \sum_{i=0}^{j-1}\tau_{\star,j}\frac{\gamma_i}{L}(\nabla f(x_i)+ h'(y_{i+1}))\right\rangle+\left \langle h'(y_j), \sum_{i=1}^{j-1}\tau_{i,j}\sum_{k=i}^{j-1}\frac{\gamma_k}{L}(\nabla f(x_k)+ h'(y_{k+1}))\right \rangle\\
&= \langle h'(y_j), \tau_{\star,j}(x_\star-x_0)\rangle+ \left\langle h'(y_j), \sum_{i=0}^{j-1}\left(\tau_{\star,j}+\sum_{k=1}^{i}\tau_{k,j}\right)\frac{\gamma_i}{L}(\nabla f(x_i)+ h'(y_{i+1}))\right\rangle\\
&= \langle h'(y_j), \tau_{\star,j}(x_\star-x_0)\rangle+ \left\langle h'(y_j), \sum_{i=0}^{j-1}\frac{2\tilde{\theta}_{j-1}}{\theta_N^2-2\theta_i^2+\theta_i}\frac{\gamma_i}{L}(\nabla f(x_i)+ h'(y_{i+1}))\right\rangle {\color{gray}\qquad\rhd\,\textrm{telescopic sum of } \tau_{i,j}}\\
&= \langle h'(y_j), \tau_{\star,j}(x_\star-x_0)\rangle+ \left\langle h'(y_j), \sum_{i=0}^{j-1}\frac{4\theta_i\tilde{\theta}_{j-1}}{L\theta_N^2}(\nabla f(x_i)+ h'(y_{i+1}))\right\rangle
\end{align*}
\qed\end{proof}

\begin{proof}\textit{of Theorem 1}
For $k=N-1$,
\begin{align*}
    \mathcal{U}_{N-1}-\mathcal{U}_{N}&= \mathcal{F}_{N-1}-\mathcal{F}_{N}+\mathcal{H}_{N-1}-\mathcal{H}_{N}\\
    &\geq \frac{2\tilde{\theta}_{N-1}}{\theta_N^2}\left \langle w_N-x_\star+\frac{1}{L}\nabla f(x_\star), h'(y_N)  \right\rangle+\cancel{\frac{\tilde{\theta}_{N-1}(4\theta_{N-1}-2\tilde{\theta}_{N-1})}{L\theta_N^2}\|h'(y_N)\|^2} \qquad{\color{gray}\rhd\; \textrm{using Lemma }\ref{g:theta}}\\
   &+ \sum_{i=1}^{N-1}\tau_{i,N}\left(h(y_i)-h(y_N)\right)+\tau_{\star,N}\left(h(x_\star)-h(y_N)\right)+\tau_{N,N-1}\left(h(y_{N})-h(y_{N-1})\right)\\
&+\frac{2\tilde{\theta}_{N-1}}{\theta_N^2(\theta_N^2-1)}\left\langle x_0-x_\star-\frac{\theta_N^2-1}{L}\nabla f(x_\star),h'(y_N)\right \rangle -\cancel{\frac{\tilde{\theta}_{N-1}+\theta_{N-2}^2}{L\theta_N^2}\|h'(y_N)\|^2}+\frac{2\theta_{N-2}^2}{L\theta_N^2}\langle  h'(y_{N-1}) , h'(y_N)+\nabla f(x_{N-1})\rangle\\
 &=\frac{2\tilde{\theta}_{N-1}}{\theta_N^2}\left \langle x_0-x_\star+\frac{1}{L}\nabla f(x_\star) -\sum_{i=0}^{N-1}\frac{2\theta_i}{L}\nabla f(x_i)-\sum_{i=0}^{N-1}\frac{2\theta_i}{L} h'(y_{i+1}), h'(y_N)  \right\rangle \qquad{\color{gray}\rhd\; w_{i+1}= w_i - \frac{2\theta_i}{L}\nabla f(x_i)- \frac{2\theta_i}{L}h'(y_{i+1})}\\
   & +\sum_{i=1}^{N-1}\tau_{i,N}\left(h(y_i)-h(y_N)\right)+\tau_{\star,N}\left(h(x_\star)-h(y_N)\right)+\tau_{N,N-1}\left(h(y_{N})-h(y_{N-1})\right)\\
&+\frac{2\tilde{\theta}_{N-1}}{\theta_N^2(\theta_N^2-1)}\left\langle x_0-x_\star-\frac{\theta_N^2-1}{L}\nabla f(x_\star),h'(y_N)\right \rangle+\frac{2\theta_{N-2}^2}{L\theta_N^2}\langle  h'(y_{N-1}) , h'(y_N)+\nabla f(x_{N-1})\rangle\\
 & =\sum_{i=1}^{N-1}\tau_{i,N}\left(h(y_i)-h(y_N)\right)+\tau_{\star,N}\left(h(x_\star)-h(y_N)\right)+\tau_{N,N-1}\left(h(y_{N})-h(y_{N-1})\right)\\
 &+\frac{2\tilde{\theta}_{N-1}}{\theta_N^2-1}\left \langle x_0-x_\star, h'(y_N)  \right\rangle - \sum_{i=0}^{N-1}\frac{4\tilde{\theta}_{N-1}\theta_i}{L\theta_N^2}\left \langle \nabla f(x_i)+ h'(y_{i+1}), h'(y_N)  \right\rangle\\
&+\frac{2\theta_{N-2}^2}{L\theta_N^2}\langle  h'(y_{N-1}) , h'(y_N)+\nabla f(x_{N-1})\rangle\\
 & =\sum_{i=1}^{N-1}\tau_{i,N}\left(h(y_i)-h(y_N)\right)+\tau_{\star,N}\left(h(x_\star)-h(y_N)\right)+\tau_{\star,N}\left \langle x_0-x_\star, h'(y_N)  \right\rangle - \sum_{i=0}^{N-1}\frac{4\tilde{\theta}_{N-1}\theta_i}{L\theta_N^2}\left \langle \nabla f(x_i)+ h'(y_{i+1}), h'(y_N)  \right\rangle\\
&+\frac{2\theta_{N-2}^2}{L\theta_N^2}\langle  h'(y_{N-1}) , h'(y_N)+\nabla f(x_{N-1})\rangle+\tau_{N,N-1}\left(h(y_{N})-h(y_{N-1})\right)\\
 & =\sum_{i=1}^{N-1}\tau_{i,N}\left(h(y_i)-h(y_N)\right)+\tau_{\star,N}\left(h(x_\star)-h(y_N)\right)-\tau_{\star,N}\langle h'(y_N), x_\star-y_N \rangle-\sum_{i=1}^{N-1}\tau_{i,N}\langle h'(y_N), y_i-y_N \rangle\quad{\color{gray}\rhd\; \textrm{using Lemma }\ref{g:sumoftaus2}} \\
&+\tau_{N,N-1}\left(\left\langle  h'(y_{N-1}) , \frac{\gamma_{N-1}}{L}h'(y_N)+ \frac{\gamma_{N-1}}{L}\nabla f(x_{N-1})\right\rangle+h(y_{N})-h(y_{N-1})\right)\\
 & =\sum_{i=1}^{N-1}\tau_{i,N}\left(h(y_i)-h(y_N) -\langle h'(y_N), y_i-y_N\rangle \right)+\tau_{\star,N}\left(h(x_\star)-h(y_N)-\langle h'(y_N), x_\star-y_N \rangle\right)\\
&+\tau_{N,N-1}\left(\left\langle  h'(y_{N-1}) , y_{N-1}-y_N\right\rangle+h(y_{N})-h(y_{N-1})\right)\geq 0 \qquad{\color{gray}\rhd\; \textrm{convexity of } h \textrm{ and by Lemma }\ref{lem:equivalenceOptISTA}}\\
\end{align*}

For $k\in[0:N-2]$,
\begin{align*}
    \mathcal{U}_{k}-\mathcal{U}_{k+1}&= \mathcal{F}_{k}-\mathcal{F}_{k+1}+\mathcal{H}_{k}-\mathcal{H}_{k+1}\\
    &\geq \frac{2\theta_k}{\theta_N^2}\left\langle w_{k+1}-x_\star+\frac{1}{L}\nabla f(x_\star), h'(y_{k+1})\right \rangle+\cancel{\frac{2\theta_k^2}{L\theta_N^2}\|h'(y_{k+1})\|^2}+\cancel{\frac{2\theta_k^2}{L\theta_N^2}\langle h'(y_{k+1}), \nabla f(x_{k+1})\rangle}\\
   &+ \sum_{i=1}^{k}\tau_{i,k+1}\left(h(y_i)-h(y_{k+1})\right)+\tau_{\star,{k+1}}\left(h(x_\star)-h(y_{k+1})\right)+\tau_{k+1,k}\left(h(y_{k+1})-h(y_{k})\right)\\
 &+\frac{2\theta_{k}}{\theta_N^2(\theta_N^2-1)}\left\langle x_0-x_\star-\frac{\theta_N^2-1}{L}\nabla f(x_\star),h'(y_{k+1})\right \rangle-\cancel{\frac{2\theta_k^2}{L\theta_N^2}\|h'(y_{k+1})\|^2}+\frac{2{\theta}_{k-1}^2}{L\theta_N^2}\langle h'(y_k), h'(y_{k+1})+\nabla f(x_k)\rangle \\
    &\cancel{-\frac{2\theta_k^2}{L\theta_N^2}\langle h'(y_{k+1}), \nabla f(x_{k+1})\rangle}
    \end{align*}
    \begin{align*}
 &=\frac{2{\theta}_{k}}{\theta_N^2}\left \langle x_0-x_\star+\frac{1}{L}\nabla f(x_\star) -\sum_{i=0}^{k}\frac{2\theta_i}{L}\nabla f(x_i)-\sum_{i=0}^{k}\frac{2\theta_i}{L} h'(y_{i+1}), h'(y_{k+1})  \right\rangle\qquad{\color{gray}\rhd\; w_{i+1}= w_i - \frac{2\theta_i}{L}\nabla f(x_i)- \frac{2\theta_i}{L}h'(y_{i+1})}\\
   & +\sum_{i=1}^{k}\tau_{i,k+1}\left(h(y_i)-h(y_{k+1})\right)+\tau_{\star,{k+1}}\left(h(x_\star)-h(y_{k+1})\right)+\tau_{k+1,k}\left(h(y_{k+1})-h(y_{k})\right)\\
&+\frac{2{\theta}_{k}}{\theta_N^2(\theta_N^2-1)}\left\langle x_0-x_\star-\frac{\theta_N^2-1}{L}\nabla f(x_\star),h'(y_{k+1})\right \rangle+\frac{2{\theta}_{k-1}^2}{L\theta_N^2}\langle h'(y_k), h'(y_{k+1})+\nabla f(x_k)\rangle\\
 & =\sum_{i=1}^{k}\tau_{i,k+1}\left(h(y_i)-h(y_{k+1})\right)+\tau_{\star,k+1}\left(h(x_\star)-h(y_{k+1})\right)+\tau_{k+1,k}\left(h(y_{k+1})-h(y_{k})\right)\\
 &+\frac{2\theta_k}{\theta_N^2-1}\left \langle x_0-x_\star, h'(y_{k+1})  \right\rangle - \sum_{i=0}^{k}\frac{4\theta_k\theta_i}{L\theta_N^2}\left \langle \nabla f(x_i)+ h'(y_{i+1}), h'(y_{k+1})  \right\rangle+\frac{2{\theta}_{k-1}^2}{L\theta_N^2}\langle h'(y_k), h'(y_{k+1})+\nabla f(x_k)\rangle\\
 & =\sum_{i=1}^{k}\tau_{i,k+1}\left(h(y_i)-h(y_{k+1})\right)+\tau_{\star,k+1}\left(h(x_\star)-h(y_{k+1})\right)+\tau_{\star,k+1}\left \langle x_0-x_\star, h'(y_{k+1})  \right\rangle - \sum_{i=0}^{k}\frac{4\theta_k\theta_i}{L\theta_N^2}\left \langle \nabla f(x_i)+ h'(y_{i+1}), h'(y_{k+1})  \right\rangle\\
&+\frac{2\theta_{k-1}^2}{L\theta_N^2}\langle  h'(y_{k}) , h'(y_{k+1})+\nabla f(x_{k})\rangle+\tau_{k+1,k}\left(h(y_{k+1})-h(y_{k})\right)\\
 & =\sum_{i=1}^{k}\tau_{i,k+1}\left(h(y_i)-h(y_{k+1})\right)+\tau_{\star,k+1}\left(h(x_\star)-h(y_{k+1})\right)-\tau_{\star,k+1}\langle h'(y_{k+1}), x_\star-y_{k+1} \rangle-\sum_{i=1}^{k}\tau_{i,{k+1}}\langle h'(y_{k+1}), y_i-y_{k+1} \rangle \\
&+\tau_{k+1,k}\left(\left\langle  h'(y_{k}) , \frac{\gamma_{k}}{L}h'(y_{k+1})+ \frac{\gamma_{k}}{L}\nabla f(x_{k})\right\rangle+h(y_{k+1})-h(y_{k})\right)\qquad{\color{gray}\rhd\; \textrm{using Lemma }\ref{g:sumoftaus2}}\\
 & =\sum_{i=1}^{k}\tau_{i,k+1}\left(h(y_i)-h(y_{k+1}) -\langle h'(y_{k+1}), y_i-y_{k+1}\rangle \right)+\tau_{\star,k+1}\left(h(x_\star)-h(y_{k+1})-\langle h'(y_{k+1}), x_\star-y_{k+1} \rangle\right)\\
&+\tau_{k+1,k}\left(\left\langle  h'(y_{k}) , y_{k}-y_{k+1}\right\rangle-h(y_{k+1})+h(y_{k})\right)\geq 0 \qquad{\color{gray}\rhd\; \textrm{convexity of } h \textrm{ and by Lemma }\ref{lem:equivalenceOptISTA}}
\end{align*}
For $k=-1$, recall that $\theta_{-1}=0$,
\begin{align*}
    \mathcal{U}_{-1}-\mathcal{U}_{0}&= \mathcal{F}_{-1}-\mathcal{F}_{0}+\mathcal{H}_{-1}-\mathcal{H}_{0}\geq\frac{2\theta_{-1}}{\theta_N^2}\left\langle w_{0}-x_\star+\frac{1}{L}\nabla f(x_\star), h'(y_{0})\right \rangle+\frac{2\theta_{-1}^2}{L\theta_N^2}\|h'(y_{k+1})\|^2+\frac{2\theta_{-1}^2}{L\theta_N^2}\langle h'(y_0), \nabla f(x_{0})\rangle=0.
    \end{align*}
Finally, we have
\begin{align*}
     \mathcal{U}_{-1}&= \mathcal{F}_{-1}+ \mathcal{H}_{-1}\\
     &=\frac{L}{2\theta_N^2}\left\|x_0-x_\star+\frac{1}{L}\nabla f(x_\star)\right\|^2-\frac{1}{2L}\|\nabla f(x_\star)\|^2 + \frac{L}{2\theta_N^2(\theta_N^2-1)}\left\|x_0-x_\star-\frac{\theta_N^2-1}{L}\nabla f(x_\star)\right\|^2\\
     &=\frac{L}{2\theta_N^2}\|x_0-x_\star\|^2+\frac{L}{2\theta_N^2(\theta_N^2-1)}\|x_0-x_\star\|^2+\frac{1}{\theta_N^2}\left\langle x_0-x_\star, \nabla f(x_\star) \right\rangle-\frac{1}{\theta_N^2}\left\langle x_0-x_\star, \nabla f(x_\star) \right\rangle \\
     &-\frac{1}{2L}\|\nabla f(x_\star)\|^2 -\frac{\theta_N^2-1}{2L\theta_N^2}\|\nabla f(x_\star)\|^2+\frac{1}{2L\theta_N^2}\|\nabla f(x_\star)\|^2\\
     &=\frac{L}{2(\theta_N^2-1)}\|x_0-x_\star\|^2.
\end{align*}
\qed\end{proof}

\section{Omitted proof of Theorem~\ref{thm:OptISTA-lb}}\label{s:d}

Now we show that the following choice of $\sigma_0,\dots,\sigma_N$, $a_0,\dots,a_N$, $x_0,\dots,x_N,x_\star$, $g_\star$, $f_0,\dots,f_N$, and $f_\star$ satisfies the conditions of Lemma~\ref{lem:lower-bound-main-body}.
\begin{equation}
\begin{aligned}
& \sigma_{i}=\tfrac{2\theta_{i}}{\theta_{N}^{2}},\ i\in[0:N-1],\qquad \sigma_{N}=\tfrac{1}{\theta_{N}},\\
 & \zeta_{N+1}=\tfrac{(\theta_{N}-1)R^2}{\theta_{N}^{2}(2\theta_{N}-1)},\quad\zeta_{N}=\tfrac{\theta_{N}}{\theta_{N}-1}\zeta_{N+1},\quad\zeta_{i}=\tfrac{2\theta_{i}}{2\theta_{i}-1}\zeta_{i+1}\ ,i\in[0:N-1],\\
 & a_{i}=\tfrac{1}{\theta_{N}^{2}-1}\cdot\tfrac{\zeta_{i}}{\sigma_{i}\sqrt{\zeta_{i}-\zeta_{i+1}}},\ i\in[0:N],\qquad x_{i}=-(\theta_{N}^{2}-1)\sum_{k=0}^{i-1}\sigma_{k}a_{k}e_{k},\\
 & f_{i}=\tfrac{L}{2}a_{i}^{2}(4\theta_{i}-1)-\tfrac{LR^2}{2(\theta_{N}^{2}-1)^{2}}\,,i\in[0:N-1],\qquad f_{N}=\tfrac{LR^2}{2(\theta_{N}^{2}-1)},\\
 & x_{\star}=-(\theta_{N}^{2}-1)\sum_{k=0}^{N}\sigma_{k}a_{k}e_{k},\qquad g_{\star}=-\tfrac{L}{\theta_{N}^{2}-1}x_{\star}, \qquad f_{\star}=0.
\end{aligned}
\label{eq:choiceofpoints}
\end{equation}
 At this point, we already have \eqref{eq:lem1orthogonal}, and \eqref{eq:lem1condonx}. The remaining conditions are:
 \begin{align*}
    f_j-\frac{1}{L} \langle g_i, g_j \rangle +\frac{1}{2L}\|g_j-Lx_j\|^2 -\frac{L}{2}\|x_j\|^2 &\geq f_k-\frac{1}{L} \langle g_i, g_k \rangle +\frac{1}{2L}\|g_k-Lx_k\|^2 -\frac{L}{2}\|x_k\|^2 \nonumber \\
    &\text{ for } \ i\in[0:N],\ j\in[0:N-1],\ k\in[j+1:N] , \tag{\ref{eq:lem1zerochain}}
\end{align*}
\begin{align*}
&\sigma_i\geq 0, \ \sum_{i=0}^N \sigma_i=1 ,\tag{\ref{eq:lem1convexsum}}
\end{align*}
\begin{align*}
 &\sum_{i=0}^N \sigma_i \left(
 f_i+\tfrac{1}{2L}\|g_i-Lx_i\|^2-\tfrac{L}{2}\|x_i\|^2
 \right) =  f_\star+\tfrac{1}{2L}\|g_\star-Lx_\star\|^2-\tfrac{L}{2}\|x_\star\|^2.\tag{\ref{eq:lem1convexsumfunction}}
\end{align*}
We provide the proofs of equations \eqref{eq:lem1zerochain}, \eqref{eq:lem1convexsum}, and \eqref{eq:lem1convexsumfunction} in separate lemmas. The following two lemmas state the sufficient conditions to ensure \eqref{eq:lem1zerochain}.
\begin{lem}\label{lem:f_Nequalto}
Choice of $f_N$ in \eqref{eq:choiceofpoints} can be rearranged as 
\[
f_N=\frac{LR^2}{2(\theta_N^2-1)}=\frac{L}{2}a_{N}^2(2\theta_{N}-1)-\frac{LR^2}{2(\theta_N^2-1)^2}.
\]
\end{lem}
\begin{proof}
The right hand side can be expanded as
     \begin{align*}
     \frac{L}{2}a_N^2(2\theta_N-1)-\frac{LR^2}{2(\theta_N^2-1)^2}&=\frac{L}{2}\cdot\frac{1}{(\theta_N^2-1)^2}\frac{\zeta_N^2}{\sigma_N^2(\zeta_N-\zeta_{N+1})}(2\theta_N-1)-\frac{LR^2}{2(\theta_N^2-1)^2}\\
     &=\frac{L}{2}\cdot\frac{\theta_N^2}{(\theta_N^2-1)^2}\frac{\zeta_N^2}{\zeta_N-\zeta_{N+1}}(2\theta_N-1)-\frac{LR^2}{2(\theta_N^2-1)^2}.
     \end{align*}
     Furthermore, $\frac{\zeta_{N}^2}{\zeta_{N}-\zeta_{N+1}}$ can be further simplified as
     \begin{align*}
             \frac{\zeta_{N}^2}{\zeta_{N}-\zeta_{N+1}}=\frac{\zeta_{N}^2}{\left(1-\frac{\theta_{N}-1}{\theta_N}\right)\zeta_{N}}=\theta_{N}\zeta_{N} =\frac{\theta_N^2}{\theta_N-1}\zeta_{N+1}=\frac{\theta_N^2}{\theta_N-1}\frac{\theta_N-1}{\theta_N^2(2\theta_N-1)}R^2=\frac{R^2}{2\theta_N-1}.
     \end{align*}
     Hence,
     \begin{align*}
        \frac{L}{2}\cdot\frac{\theta_N^2}{(\theta_N^2-1)^2}\frac{\zeta_N^2}{\zeta_N-\zeta_{N+1}}(2\theta_N-1)-\frac{LR^2}{2(\theta_N^2-1)^2}
     &=\frac{L}{2}\cdot\frac{\theta_N^2}{(\theta_N^2-1)^2}(2\theta_N-1)\frac{R^2}{2\theta_N-1}-\frac{LR^2}{2(\theta_N^2-1)^2}  \\   
     &=\frac{LR^2}{2}\cdot\frac{\theta_N^2}{(\theta_N^2-1)^2} -\frac{LR^2}{2(\theta_N^2-1)^2}\\
     &=\frac{(\theta_N^2-1)LR^2}{2(\theta_N^2-1)^2}=\frac{LR^2}{2(\theta_N^2-1)}.
     \end{align*}
     \qed\end{proof}
\begin{lem}\label{lem:lem1zerochain-1}
If
\[
2a_{j+1}^2\theta_{j+1}-2a_j^2\theta_j+a_j^2\leq0 \ ,j\in[0:N-2],
\]
\[
a_N^2\theta_N-2a_{N-1}^2\theta_{N-1}+a_{N-1}^2\leq0
\]
for \eqref{eq:choiceofpoints}, then  \eqref{eq:lem1zerochain} holds.
\end{lem}
\begin{proof}
 It is enough to prove \eqref{eq:lem1zerochain} for $ i\in[0:N],\ j\in[0:N-1]$, and $k=j+1$. Also note that $\langle g_i,x_i\rangle=0$ for $i\in[0:N]$. We first plug in \eqref{eq:choiceofpoints} to \eqref{eq:lem1zerochain}. Then,
 \[
 f_j-\frac{1}{L} \langle g_i, g_j \rangle +\frac{1}{2L}\|g_j-Lx_j\|^2 -\frac{L}{2}\|x_j\|^2 \geq f_{j+1}-\frac{1}{L} \langle g_i, g_{j+1} \rangle +\frac{1}{2L}\|g_{j+1}-Lx_{j+1}\|^2 -\frac{L}{2}\|x_{j+1}\|^2 
 \]
 reduces to 
\begin{align}
 f_j-\frac{1}{L} \langle g_i, g_j \rangle +\frac{1}{2L}\|g_j\|^2 \geq f_{j+1}-\frac{1}{L} \langle g_i, g_{j+1} \rangle +\frac{1}{2L}\|g_{j+1}\|^2  \label{zerochainreduced}
\end{align}
Now divide the problem into smaller cases:\\
     \item[$\bullet$] \emph{Case1-1}: $j\in[0:N-2]$, \, $i<j$ or $i>j+1$ \\\\
    Then \eqref{zerochainreduced} is equivalent to 
    \[
    \frac{L}{2}a_j^2(4\theta_{j}-1)-\frac{LR^2}{2(\theta_N^2-1)^2}+\frac{L}{2}a_j^2\geq  \frac{L}{2}a_{j+1}^2(4\theta_{j+1}-1)-\frac{LR^2}{2(\theta_N^2-1)^2}+\frac{L}{2}a_{j+1}^2.
    \]
    It reduces to
    \[
    2a_j^2\theta_j\geq 2a_{j+1}^2\theta_{j+1},
    \]
    which is true since $ \displaystyle 2a_j^2\theta_j-2a_{j+1}^2\theta_{j+1}\geq a_j^2\geq0$.
    \item[$\bullet$] \emph{Case1-2}: $j\in[0:N-2]$, \, $i=j$ \\\\
     Then \eqref{zerochainreduced} is equivalent to 
    \[
    \frac{L}{2}a_j^2(4\theta_{j}-1)-\frac{LR^2}{2(\theta_N^2-1)^2}-La_j^2+\frac{L}{2}a_j^2\geq  \frac{L}{2}a_{j+1}^2(4\theta_{j+1}-1)-\frac{LR^2}{2(\theta_N^2-1)^2}+\frac{L}{2}a_{j+1}^2.
    \]
    It reduces to
    \[
    2a_j^2\theta_j-a_j^2\geq 2a_{j+1}^2\theta_{j+1},
    \]
     which is true by the assumption of the lemma.
    \item[$\bullet$] \emph{Case1-3}: $j\in[0:N-2]$, \, $i=j+1$ \\\\
     Then \eqref{zerochainreduced} is equivalent to 
    \[
    \frac{L}{2}a_j^2(4\theta_{j}-1)-\frac{LR^2}{2(\theta_N^2-1)^2}+\frac{L}{2}a_j^2\geq  \frac{L}{2}a_{j+1}^2(4\theta_{j+1}-1)-La_{j+1}^2-\frac{LR^2}{2(\theta_N^2-1)^2}+\frac{L}{2}a_{j+1}^2.
    \]
    It reduces to
    \[
    2a_j^2\theta_j\geq 2a_{j+1}^2\theta_{j+1}-a_{j+1}^2,
    \]
     which is true by the assumption of the lemma.
     \item[$\bullet$] \emph{Case2-1}: $j=N-1$, \, $i<j$ \\\\
     Using Lemma~\ref{lem:f_Nequalto},  \eqref{zerochainreduced} is equivalent to 
    \[
    \frac{L}{2}a_{N-1}^2(4\theta_{N-1}-1)-\frac{LR^2}{2(\theta_N^2-1)^2}+\frac{L}{2}a_{N-1}^2\geq  \frac{L}{2}a_{N}^2(2\theta_{N}-1)-\frac{LR^2}{2(\theta_N^2-1)^2}+\frac{L}{2}a_{N}^2.
    \]
    It reduces to
    \[
    2a_{N-1}^2\theta_{N-1}\geq a_{N}^2\theta_{N},
    \]
     which is true since $ \displaystyle 2a_{N-1}^2\theta_{N-1}-a_{N}^2\theta_{N}\geq a_{N-1}^2\geq0$.
    \item[$\bullet$] \emph{Case2-2}: $j=N-1$, \, $i=N-1$ \\\\
     Using Lemma~\ref{lem:f_Nequalto},  \eqref{zerochainreduced} is equivalent to 
    \[
    \frac{L}{2}a_{N-1}^2(4\theta_{N-1}-1)-La_{N-1}^2-\frac{LR^2}{2(\theta_N^2-1)^2}+\frac{L}{2}a_{N-1}^2\geq  \frac{L}{2}a_{N}^2(2\theta_{N}-1)-\frac{LR^2}{2(\theta_N^2-1)^2}+\frac{L}{2}a_{N}^2.
    \]
    It reduces to
    \[
    2a_{N-1}^2\theta_{N-1}-a_{N-1}^2\geq a_{N}^2\theta_{N},
    \]
 which is true by the assumption of the lemma.
   \item[$\bullet$] \emph{Case2-3}: $j=N-1$, \, $i=N$ \\\\
     Using Lemma~\ref{lem:f_Nequalto},   \eqref{zerochainreduced} is equivalent to 
    \[
    \frac{L}{2}a_{N-1}^2(4\theta_{N-1}-1)-\frac{LR^2}{2(\theta_N^2-1)^2}+\frac{L}{2}a_{N-1}^2\geq  \frac{L}{2}a_{N}^2(2\theta_{N}-1)-\frac{LR^2}{2(\theta_N^2-1)^2}-La_{N}^2+\frac{L}{2}a_{N}^2.
    \]
    It reduces to
    \[
    2a_{N-1}^2\theta_{N-1}\geq a_{N}^2\theta_{N}-a_{N}^2,
    \]
 which is true by the assumption of the lemma.
    \qed\end{proof}
The following lemma proves \eqref{eq:lem1zerochain}.
\begin{lem}\label{lem:proofoflem1zerochain}
Choice of $a_i$ in \eqref{eq:choiceofpoints} satisfies 
\[
2a_{j+1}^2\theta_{j+1}-2a_j^2\theta_j+a_j^2\leq0 \ ,j\in[0:N-2],
\]
\[
a_N^2\theta_N-2a_{N-1}^2\theta_{N-1}+a_{N-1}^2\leq0.
\]
\end{lem}
\begin{proof}
Recall that
     \[
 a_i=\frac{1}{\theta_N^2-1}\cdot\frac{\zeta_i}{\sigma_i\sqrt{\zeta_i-\zeta_{i+1}}} \, \text{ for } i\in[0:N].
 \]
 Then for $j\in[0:N-2]$,
\[
     2a_{j+1}^2\theta_{j+1}-2a_{j}^2\theta_{j}+a_j^2=\frac{2\theta_{j+1}}{(\theta_N^2-1)^2}\cdot\frac{\zeta_{j+1}^2}{\sigma_{j+1}^2(\zeta_{j+1}-\zeta_{j+2})}-\frac{2\theta_j}{(\theta_N^2-1)^2}\cdot\frac{\zeta_j^2}{\sigma_j^2(\zeta_j-\zeta_{j+1})}+\frac{1}{(\theta_N^2-1)^2}\cdot\frac{\zeta_j^2}{\sigma_j^2(\zeta_j-\zeta_{j+1})}.
\]
For convenience, drop the constant term $\frac{1}{(\theta_N^2-1)^2}>0$ to get:
\begin{align}
\frac{2\theta_{j+1}\zeta_{j+1}^2}{\sigma_{j+1}^2(\zeta_{j+1}-\zeta_{j+2})}-\frac{2\theta_j\zeta_j^2}{\sigma_j^2(\zeta_j-\zeta_{j+1})}+\frac{\zeta_j^2}{\sigma_j^2(\zeta_j-\zeta_{j+1})}.\label{eq:constanttermdrop1}
\end{align}
Note that 
\[
\frac{\zeta_j^2}{\zeta_j-\zeta_{j+1}}=\frac{\left(\frac{2\theta_j}{2\theta_j-1}\right)^2\zeta_{j+1}^2}{\left(\frac{2\theta_j}{2\theta_j-1}-1\right)\zeta_{j+1}}=\frac{4\theta_j^2\zeta_{j+1}}{2\theta_j-1},
\]
\[
\frac{\zeta_{j+1}^2}{\zeta_{j+1}-\zeta_{j+2}}=\frac{\zeta_{j+1}^2}{\left(1-\frac{2\theta_{j+1}-1}{2\theta_{j+1}}\right)\zeta_{j+1}}=2\theta_{j+1}\zeta_{j+1}.
\]
Plugging into \eqref{eq:constanttermdrop1} gives:
\begin{align*}
\frac{2\theta_{j+1}\zeta_{j+1}^2}{\sigma_{j+1}^2(\zeta_{j+1}-\zeta_{j+2})}-\frac{2\theta_j\zeta_j^2}{\sigma_j^2(\zeta_j-\zeta_{j+1})}+\frac{\zeta_j^2}{\sigma_j^2(\zeta_j-\zeta_{j+1})}
    &=\frac{4\theta_{j+1}^2\zeta_{j+1}}{\sigma_{j+1}^2}-\frac{8\theta_j^3\zeta_{j+1}}{\sigma_j^2(2\theta_j-1)}+ \frac{4\theta_j^2\zeta_{j+1}}{\sigma_j^2(2\theta_j-1)}\\
    &=\theta_N^4\zeta_{j+1}- \theta_N^4\cdot\frac{2\theta_j\zeta_{j+1}}{2\theta_j-1}+ \theta_N^4\cdot\frac{\zeta_{j+1}}{2\theta_j-1}\\
    &=\theta_N^4\zeta_{j+1}\left(1-\frac{2\theta_j-1}{2\theta_j-1}\right)=0.
\end{align*}
For $j=N-1$,
\[
a_N^2\theta_N-2a_{N-1}^2\theta_{N-1}+a_{N-1}^2=\frac{\theta_{N}}{(\theta_N^2-1)^2}\cdot\frac{\zeta_{N}^2}{\sigma_{N}^2(\zeta_{N}-\zeta_{N+1})}-\frac{2\theta_{N-1}}{(\theta_N^2-1)^2}\cdot\frac{\zeta_{N-1}^2}{\sigma_{N-1}^2(\zeta_{N-1}-\zeta_{N})}+\frac{1}{(\theta_N^2-1)^2}\cdot\frac{\zeta_{N-1}^2}{\sigma_{N-1}^2(\zeta_{N-1}-\zeta_{N})}.
\]
Similarly, drop the constant term to get:
\begin{align}
\frac{\theta_{N}\zeta_{N}^2}{\sigma_{N}^2(\zeta_{N}-\zeta_{N+1})}-\frac{2\theta_{N-1}\zeta_{N-1}^2}{\sigma_{N-1}^2(\zeta_{N-1}-\zeta_{N})}+\frac{\zeta_{N-1}^2}{\sigma_{N-1}^2(\zeta_{N-1}-\zeta_{N})}.\label{eq:constanttermdrop2}    
\end{align}
Note that 
\[
\frac{\zeta_{N-1}^2}{\zeta_{N-1}-\zeta_{N}}=\frac{\left(\frac{2\theta_{N-1}}{2\theta_{N-1}-1}\right)^2\zeta_{N}^2}{\left(\frac{2\theta_{N-1}}{2\theta_{N-1}-1}-1\right)\zeta_{N}}=\frac{4\theta_{N-1}^2\zeta_{N}}{2\theta_{N-1}-1},
\]
\[
\frac{\zeta_{N}^2}{\zeta_{N}-\zeta_{N+1}}=\frac{\zeta_{N}^2}{\left(1-\frac{\theta_{N}-1}{\theta_N}\right)\zeta_{N}}=\theta_{N}\zeta_{N}.
\]
Plugging into \eqref{eq:constanttermdrop2} gives:
\begin{align*}
\frac{\theta_{N}\zeta_{N}^2}{\sigma_{N}^2(\zeta_{N}-\zeta_{N+1})}-\frac{2\theta_{N-1}\zeta_{N-1}^2}{\sigma_{N-1}^2(\zeta_{N-1}-\zeta_{N})}+\frac{\zeta_{N-1}^2}{\sigma_{N-1}^2(\zeta_{N-1}-\zeta_{N})}&=\frac{\theta_N^2{\color{red}{\zeta_N}}}{\sigma_N^2}-\frac{8\theta_{N-1}^3\zeta_{N}}{\sigma_{N-1}^2(2\theta_{N-1}-1)}+\frac{4\theta_{N-1}^2\zeta_{N}}{\sigma_{N-1}^2(2\theta_{N-1}-1)}\\
&=\theta_N^4{\color{red}{\zeta_N}}-\theta_N^4\cdot\frac{2\theta_{N-1}\zeta_{N}}{2\theta_{N-1}-1}+\theta_N^4\cdot\frac{\zeta_{N}}{2\theta_{N-1}-1}\\
&=\theta_N^4\zeta_{N}\left(1-\frac{2\theta_{N-1}-1}{2\theta_{N-1}-1}\right)=0.
\end{align*}
\qed\end{proof}
The next lemma proves \eqref{eq:lem1convexsum} directly.
\begin{lem}\label{lem:proofoflem1convexsum}
The following holds for \eqref{eq:choiceofpoints}.
\[
\sum_{i=0}^N\sigma_i =1 ,\ \sigma_i\geq0 \ ,i\in[0:N].
\]
\end{lem}
\begin{proof}
\begin{align*}
\sum_{i=0}^N\sigma_i=\sum_{i=0}^{N-1}\sigma_i+\sigma_N&=\sum_{i=0}^{N-1}\frac{2\theta_i}{\theta_N^2}+\frac{1}{\theta_N}\\
&=\frac{1}{\theta_N^2}\cdot2\theta_{N-1}^2+\frac{1}{\theta_N}{\color{gray}\qquad\rhd\,\textup{using \eqref{eq:sumoftheta} of Lemma~\ref{lem:propertyoftheta}}}\\
&=\frac{2\theta_{N-1}^2+\theta_N}{\theta_N^2}=1.{\color{gray}\qquad\rhd\,\textup{using \eqref{eq:recursivetheta2} of Lemma~\ref{lem:propertyoftheta}}}
\end{align*}
Also, positivity of $\sigma_i$ for $i\in[0:N]$ is clear from the definition of $\{\theta\}_{i\in[0:N]}$.
\qed\end{proof}
The following two lemmas prove \eqref{eq:lem1convexsumfunction}.
\begin{lem}\label{lem:Drorilemma}
The following holds for \eqref{eq:choiceofpoints}.
\[
\|x_0-x_\star\|^2=\|x_\star\|^2=\sum_{i=0}^N\frac{\zeta_i^2}{\zeta_i-\zeta_{i+1}}=R^2.
\]
\end{lem}
\begin{proof}
    We refer the reader to \cite[Lemma 3]{drori2017LowerBoundOGM}.
\qed\end{proof}
\begin{lem}\label{lem:proofoflem1convexsumfunctions}
The following holds for \eqref{eq:choiceofpoints}.
\[
\sum_{i=0}^N \sigma_i \left(
 f_i+\tfrac{1}{2L}\|g_i-Lx_i\|^2-\tfrac{L}{2}\|x_i\|^2
 \right) =  f_\star+\tfrac{1}{2L}\|g_\star-Lx_\star\|^2-\tfrac{L}{2}\|x_\star\|^2.
 \]
\end{lem}
\begin{proof}
Recall $\langle x_i, g_i \rangle=0$ for $i\in[0:N]$. Then,
 \begin{align*}
  &\sum_{i=0}^N \sigma_i \left(
 f_i+\frac{1}{2L}\|g_i-Lx_i\|^2-\frac{L}{2}\|x_i\|^2
 \right) =  \sum_{i=0}^N \sigma_i \left(
 f_i+\frac{1}{2L}\|g_i\|^2\right){\color{gray}\qquad\rhd\,\textup{expanding } \|g_i-Lx_i\|^2}\\
 &=\sum_{i=0}^{N-1} \sigma_i\left(
 f_i+\frac{1}{2L}\|g_i\|^2\right)+\sigma_N\left(f_N+\frac{1}{2L}\|g_N\|^2\right)\\
 &=\sum_{i=0}^{N-1} \sigma_i\left(
 \frac{L}{2}a_i^2(4\theta_i-1)-\frac{LR^2}{2(\theta_N^2-1)^2}+\frac{L}{2}a_i^2\right)\\
 &+\sigma_N\left(\frac{L}{2}a_N^2(2\theta_N-1)-\frac{LR^2}{2(\theta_N^2-1)^2}+\frac{L}{2}a_N^2\right){\color{gray}\qquad\rhd\,\text{using  Lemma }\ref{lem:f_Nequalto}}\\
  &=\sum_{i=0}^{N-1} 2L\sigma_i
 a_i^2\theta_i-\sum_{i=0}^{N-1}\sigma_i\frac{LR^2}{2(\theta_N^2-1)^2}+L\sigma_Na_N^2\theta_N-\sigma_N\frac{LR^2}{2(\theta_N^2-1)^2}\\
 &=L\theta_N^2\sum_{i=0}^{N-1}\sigma_i^2
 a_i^2+L\sigma_Na_N^2\theta_N-\frac{LR^2}{2(\theta_N^2-1)^2}.{\color{gray}\qquad\rhd\,\text{using  Lemma }\ref{lem:proofoflem1convexsum}}
 \end{align*}
 Here, $\displaystyle \sum_{i=0}^{N-1}\sigma_i^2a_i^2$ is reduced to
 \[
\sum_{i=0}^{N-1}\sigma_i^2a_i^2=\sum_{i=0}^N\sigma_i^2a_i^2-{\color{red}{\sigma_N^2a_N^2}}=\frac{1}{(\theta_N^2-1)^2}\sum_{i=0}^N\sigma_i^2\frac{\zeta_i^2}{\sigma_i^2(\zeta_i-\zeta_{i+1})}-{\color{red}{\sigma_N^2a_N^2}}
=\frac{R^2}{(\theta_N^2-1)^2}-{\color{red}{\sigma_N^2a_N^2}},
 \]
where the last equality is from Lemma~\ref{lem:Drorilemma}. So,
 \begin{align*}
 L\theta_N^2\sum_{i=0}^{N-1}\sigma_i^2
 a_i^2+L\sigma_Na_N^2\theta_N-\frac{LR^2}{2(\theta_N^2-1)^2}&=\theta_N^2\cdot\frac{LR^2}{(\theta_N^2-1)^2}-L\theta_N^2\sigma_N^2a_N^2+L\sigma_Na_N^2\theta_N-\frac{LR^2}{2(\theta_N^2-1)^2}\\
 &=\theta_N^2\cdot\frac{LR^2}{(\theta_N^2-1)^2}-L\theta_N^2\frac{1}{\theta_N^2}a_N^2+L\frac{1}{\theta_N}a_N^2\theta_N-\frac{LR^2}{2(\theta_N^2-1)^2}\\
 &=\frac{2\theta_N^2-1}{2(\theta_N^2-1)^2}LR^2.
 \end{align*}
 Meanwhile, 
 \begin{align*}
 f_\star+\frac{1}{2L}\|g_\star-Lx_\star\|^2-\frac{L}{2}\|x_\star\|^2=
 \frac{1}{2L}\|g_\star\|^2-\langle g_\star, x_\star \rangle&=\frac{1}{2L}\left\|-\frac{L}{\theta_N^2-1}x_\star\right\|^2-\left\langle-\frac{L}{\theta_N^2-1}x_\star, x_\star\right\rangle\\
 &=\frac{LR^2}{2(\theta_N^2-1)^2}+\frac{LR^2}{\theta_N^2-1}=\frac{2\theta_N^2-1}{2(\theta_N^2-1)^2}LR^2.
 \end{align*}
\qed\end{proof}
Therefore, by the aforementioned lemmas, we are able to apply Lemma~\ref{lem:lower-bound-main-body} to \eqref{eq:choiceofpoints}. We now show semi-interpolating conditions to establish the lower bound of Theorem \ref{thm:OptISTA-lb}.  
\begin{lem}\label{lem:semiinterpolation}
Let
\[
f(x)=\max_{\alpha\in\Delta_{I}}v(x,\alpha)=\max_{\alpha\in\Delta_{I}}\tfrac{L}{2}\|x\|^{2}-\tfrac{L}{2}\|x-\tfrac{1}{L}\sum_{i\in I}\alpha_{i}g_{i}\|^{2}+\sum_{i\in I}\alpha_{i}\left(f_{i}+\tfrac{1}{2L}\|g_{i}-Lx_{i}\|^{2}-\tfrac{L}{2}\|x_{i}\|^{2}\right),
\]
and for some $i\in I$, suppose
\[
f_i\geq f_j+\langle g_j,x_i-x_j\rangle+\frac{1}{2L}\| g_i-g_j \|^2 \quad \text{ for all } j\in I.
\]
Then $f(x_i)=f_i$ and $\nabla f(x_i)=g_i$
\end{lem}
\begin{proof}
     We refer the reader to \cite[Theorem 1]{drori2022LowerBoundITEM}.
\qed\end{proof}
\begin{lem}\label{lem:indicator1}
If $\{g_i\}_{i\in[0:N]}$ are orthogonal and 
\[
    g_\star=\sum_{i=0}^N\sigma_ig_i \ \text{for some} \  \sigma_i\geq 0,
\]
then, $-g_\star\in\partial h(x_\star)$.
\end{lem}
\begin{proof}
    Since $x_\star\in C$, 
    \[
    \partial h(x_\star)=\{y\,|\, \langle y, c-x_\star \rangle \leq 0 \quad \forall c\in C\}.
    \]
Let $\tilde{z}=x_\star+\sum_{i=0}^N c_ig_i$ for $c_i\geq 0$.
Then,
\begin{align*}
    \langle -g_\star, \tilde{z}-x_\star \rangle = \left\langle -g_\star, x_\star+\sum_{i=0}^N c_ig_i-x_\star \right\rangle=
    \left\langle -\sum_{i=0}^N\sigma_ig_i \ , \ \sum_{i=0}^N c_ig_i\right\rangle
    =-\sum_{i=0}^N \sigma_i c_i \|g_i\|^2 \leq 0.
\end{align*}
\qed\end{proof}
\begin{lem}\label{lem:semiinterpolationf}
The choice of \eqref{eq:choiceofpoints} satisfies the following:
\[
f_\star\geq f_i+\langle g_i,x_\star-x_i\rangle+\frac{1}{2L}\| g_\star-g_i \|^2  \text{ for } i\in I ,
\]
and therefore $f(x_\star)=f_\star$ and $\nabla f(x_\star)=g_\star$ by Lemma~\ref{lem:semiinterpolation}.
\end{lem}
\begin{proof}
In fact, we will show that:
\[
f_i=f_\star-\langle g_i,x_\star-x_i\rangle-\frac{1}{2L}\| g_\star-g_i \|^2  \text{ for } i\in I.
\]
For $i\in[0:N]$,
\begin{align*}
&f_\star-\langle g_i,x_\star-x_i\rangle-\frac{1}{2L}\| g_\star-g_i \|^2\\
&=0-\left\langle La_ie_i, -(\theta_N^2-1)\sum_{k=0}^N\sigma_ka_ke_k+(\theta_N^2-1)\sum_{k=0}^{i-1}\sigma_ka_ke_k\right\rangle-\frac{1}{2L}\left\|L\sum_{k=0}^N\sigma_ka_ke_k-La_ie_i\right\|^2\\
&=\left\langle La_ie_i, (\theta_N^2-1)\sum_{k=i}^N\sigma_ka_ke_k\right\rangle-\frac{1}{2L}\left\|L\sum_{k=0}^N\sigma_ka_ke_k\right\|^2+\frac{1}{L}\left\langle L\sum_{k=0}^N\sigma_ka_ke_k, La_ie_i\right\rangle-\frac{1}{2L}\|La_ie_i\|^2\\
&=L(\theta_N^2-1)\sigma_ia_i^2-\frac{1}{2L}\left\|\frac{L}{\theta_N^2-1}x_\star\right\|^2+L\sigma_ia_i^2-\frac{L}{2}a_i^2\\
&=\frac{L}{2}a_i^2\left( 2\sigma_i\left(\theta_N^2-1\right)+2\sigma_i-1\right)-\frac{LR^2}{2(\theta_N^2-1)^2}.
{\color{gray}\qquad\rhd\, \left\|\frac{L}{\theta_N^2-1}x_\star\right\|^2= \frac{L^2R^2}{(\theta_N^2-1)^2 }}
\end{align*}
If $i\in[0:N-1]$,
\begin{align*}
   \frac{L}{2}a_i^2\left( 2\sigma_i\left(\theta_N^2-1\right)+2\sigma_i-1\right)-\frac{LR^2}{2(\theta_N^2-1)^2}
   &=\frac{L}{2}a_i^2\left( \frac{4\theta_i}{\theta_N^2}\left(\theta_N^2-1\right)+\frac{4\theta_i}{\theta_N^2}-1\right)-\frac{LR^2}{2(\theta_N^2-1)^2}\\
   &=\frac{L}{2}a_i^2\left( 4\theta_i-1\right)-\frac{LR^2}{2(\theta_N^2-1)^2}=f_i.
\end{align*}
If $i=N$,
\begin{align*}
   \frac{L}{2}a_N^2\left( 2\sigma_N\left(\theta_N^2-1\right)+2\sigma_N-1\right)-\frac{LR^2}{2(\theta_N^2-1)^2}
   &=\frac{L}{2}a_N^2\left( \frac{2}{\theta_N}\left(\theta_N^2-1\right)+\frac{2}{\theta_N}-1\right)-\frac{LR^2}{2(\theta_N^2-1)^2}\\
   &=\frac{L}{2}a_N^2\left( 2\theta_N-1\right)-\frac{LR^2}{2(\theta_N^2-1)^2}=f_N.
   {\color{gray}\qquad\rhd\, \text{using  Lemma } \ref{lem:f_Nequalto}}
\end{align*}
\qed\end{proof}

The following lemma is a restriction of Theorem \ref{thm:OptISTA-lb} in the sense that it has a fixed starting point $x_0=z_0=0$ and fixed dimension $d=N+1$. 
\begin{lem}\label{lem:Optista-lb-res}
Let $L>0$, $R>0$, $N>0$, and
let $x_0=z_0=0\in \rl^{N+1}$. 
Let $L$-smooth and convex function $f\colon\mathbb{R}^{N+1} \rightarrow \mathbb{R}$ as
\[
f(x)=\max_{\alpha\in\Delta_{I}}\tfrac{L}{2}\|x\|^{2}-\tfrac{L}{2}\|x-\tfrac{1}{L}\sum_{i\in I}\alpha_{i}g_{i}\|^{2}+\sum_{i\in I}\alpha_{i}\left(f_{i}+\tfrac{1}{2L}\|g_{i}-Lx_{i}\|^{2}-\tfrac{L}{2}\|x_{i}\|^{2}\right),
\]
 and closed, convex, and proper function $h\colon\mathbb{R}^{N+1} \rightarrow \mathbb{R}\cup\{\infty\}$ as $h(x)=\delta_{C}(x)$, where $\delta_{C}$ is the indicator function of the convex set $C=x_{\star}+\mathrm{cone}\{g_{0},g_{1},\ldots,g_{N}\}$, with the choice of \eqref{eq:choiceofpoints}.

Then $x_\star\in \argmin(f+h)$ and satisfies $\|x_0-x_\star\|=R$ and
\[
f(x_{N})+h(x_{N})-f(x_\star)-h(x_\star)\ge\frac{L\|x_{0}-x_{\star}\|^{2}}{2(\theta_{N}^{2}-1)}.
\]
for any sequence $\{ \{({z}_i,\delta_i,\gamma_i)\}_{i\in[0:2N-1]}, {x}_N \}$ satisfying the following double-function span condition:
\begin{equation}
\begin{aligned}
&\delta_i\in \{0,1\}, &&\textup{ for }i=0,\dots,2N-1,\\
& \sum^{2N-1}_{i=0}\delta_i=\sum^{2N-1}_{i=0}(1-\delta_i)=N,
\\
&d_i=\left\{
\begin{array}{ll}
\nabla f(z_i) &\textup{ if $\delta_i=0$, }\\
 z_i-\prox_{\gamma_i h}(z_i) &\textup{ if $\delta_i=1$,}
\end{array}
\right.
&&\textup{ for }i=0,\dots,2N-1,\\
&z_{i}\in x_0+\mathrm{span}\{d_0,\dots,d_{i-1}\},
&&\textup{ for }i=1,\dots,2N-1,\\
&x_N\in x_0+\mathrm{span}\{d_0,\dots,d_{2N-1}\}.
\end{aligned}
\label{eq:doublefunctionspan}
\end{equation}
To clarify, $x_N$ in \eqref{eq:doublefunctionspan} is unrelated to $x_N$ in \eqref{eq:choiceofpoints}.
\end{lem}
\begin{proof}
By Lemma~\ref{lem:proofoflem1convexsum} and Lemma~\ref{lem:indicator1}, we have $-g_\star\in\partial h(x_\star)$. By Lemma~\ref{lem:semiinterpolationf}, we have $g_\star=\nabla f(x_\star)$. Therefore, $x_\star\in\argmin(f+h)$ with $f(x_\star)+h(x_\star)=f_\star+0=0$, and $\|x_0-x_\star\|=R$ by Lemma~\ref{lem:Drorilemma}. 
 Now by Lemma~\ref{lem:lower-bound-main-body}, when the starting point is $x_0=z_0=0$, each gradient evaluation will give at most one new next coordinate, and proximal evaluations do not introduce any new coordinate. So, after $N$ gradient and $N$ proximal evaluations respectively (order does not matter), the output $x_N$ will be in the span of $\{e_0,\ldots,e_{N-1}\}$ under the condition \eqref{eq:doublefunctionspan} of the lemma. Now we apply the second result of Lemma~\ref{lem:lower-bound-main-body} to conclude that
\begin{align*}
f(x_{N})+h(x_{N})-f(x_\star)-h(x_\star)&\geq \inf_{x\in \mathrm{span}\{e_0,\dots,e_{N-1}\}}\left(f(x)+h(x)\right)-0\\
&\geq\inf_{x\in \mathrm{span}\{e_0,\dots,e_{N-1}\}}f(x)\ge f_N=\frac{L\|x_{0}-x_{\star}\|^{2}}{2(\theta_{N}^{2}-1)}. 
\end{align*}
 \qed\end{proof}
Then we expand the condition of Lemma~\ref{lem:Optista-lb-res} to arrive at Theorem \ref{thm:OptISTA-lb}.

\begin{proof}\textit{of Theorem \ref{thm:OptISTA-lb}}
Assume $d\geq N+1$. Take $f_0\in\mathcal{F}_{0,L}$ and $h_0\in\mathcal{F}_{0,\infty}$ be functions defined in Lemma~\ref{lem:Optista-lb-res}, which are embedded in $\rl^d$. Call  $\tilde{x}_\star$ to be the element of $\argmin_{x\in\mathbb{R}^d}(f_0+h_0)$ in Lemma~\ref{lem:Optista-lb-res}.
Now for arbitrary $x_0=z_0$, let $f\colon \mathbb{R}^d\rightarrow\mathbb{R}$ and $h\colon \mathbb{R}^d\rightarrow\mathbb{R}$ be translation of $f_0$ and $h_0$ by $x_0$:
\[
f(x)=f_0(x-x_0) ,\quad h(x)=h_0(x-x_0).
\]
Then $f$ is $L$-smooth and convex, $h$ is an indicator function of nonempty closed convex set, and $x_\star\triangleq \tilde{x}_\star+x_0\in\argmin_{x\in\mathbb{R}^d}(f+h)$. Now assume the sequence $\{ \{(z_i,\delta_i,\gamma_i)\}_{i\in[0:2N-1]}, x_N \}$ is produced from a method that satisfies the double-function span condition. That is,
\[
z_i\in z_0+\mathrm{span}\left\{d_0, \ldots, d_{i-1} \right\},
\]
\[
x_N\in z_0+\mathrm{span}\left\{d_0, \ldots, d_{2N-1} \right\},
\]
\text{where}
\[
d_i=\begin{cases}
\nabla f(z_i) \ \text{ if $\delta_i=0$ },\\
 z_i-\mathbf{prox}_{\gamma_i h}(z_i) \ \text{if $\delta_i=1$}.
\end{cases}
\]
Then for $\tilde{z}_i\triangleq z_i-x_0=z_i-z_0$ and $\tilde{x}_N\triangleq x_N-x_0=x_N-z_0$, 
$\{ \{(\tilde{z}_i,\delta_i,\gamma_i)\}_{i\in[0:2N-1]}, \tilde{x}_N \}$ satisfies:
\[
\tilde{z}_i\in \tilde{z}_0+\mathrm{span}\left\{\tilde{d}_0, \ldots, \tilde{d}_{i-1}\right\},
\]
\[
\tilde{x}_N\in \tilde{z}_0+\mathrm{span}\left\{\tilde{d}_0, \ldots, \tilde{d}_{2N-1} ,\right\}
\]
\text{where}
\[
\tilde{d}_i=\begin{cases}
\nabla f_0(\tilde{z}_i) \ \text{ if $\delta_i=0$ },\\
 \tilde{z}_i-\mathbf{prox}_{\gamma_i h_0}(\tilde{z}_i) \ \text{if $\delta_i=1$}.
\end{cases}
\]
This is because $\tilde{z}_0=0$ and 
\[
\nabla f(z_i)=\nabla f(\tilde{z}_i+x_0)=\nabla f_0(\tilde{z}_i+x_0-x_0)=\nabla f_0(\tilde{z}_i),
\]
\[
\mathbf{prox}_{\gamma_i h}(z_i)-x_0=\mathbf{prox}_{\gamma_i h_0}(z_i-x_0)=\mathbf{prox}_{\gamma_i h_0}(\tilde{z}_i).
\]
So,
\[
\tilde{z_i}-\mathbf{prox}_{\gamma_i h_0}(\tilde{z}_i)=z_i-x_0-\mathbf{prox}_{\gamma_i h}(z_i)+x_0=z_i-\mathbf{prox}_{\gamma_i h}(z_i).
\]
Hence, we can apply Lemma~\ref{lem:Optista-lb-res} on $\{ \{(\tilde{z}_i,\delta_i,\gamma_i)\}_{i\in[0:2N-1]}, \tilde{x}_N \}$ to get
\[
f_0(\tilde{x}_N)+h_0(\tilde{x}_N)-f_0(\tilde{x}_\star)-h_0(\tilde{x}_\star)\geq\frac{L\|0-\tilde{x}_{\star}\|^{2}}{2(\theta_{N}^{2}-1)}=\frac{L\|x_0-{x}_{\star}\|^{2}}{2(\theta_{N}^{2}-1)}=\frac{LR^{2}}{2(\theta_{N}^{2}-1)}.
\]
Then we finally get
\begin{align*}
f(x_N)+h(x_N)-f(\tilde{x}_\star+x_0)-h(\tilde{x}_\star+x_0)&=f_0(\tilde{x}_N)+h_0(\tilde{x}_N)-f_0(\tilde{x}_\star)-h_0(\tilde{x}_\star)\\
&\geq\frac{L\|0-\tilde{x}_{\star}\|^{2}}{2(\theta_{N}^{2}-1)}=\frac{L\|x_0-{x}_{\star}\|^{2}}{2(\theta_{N}^{2}-1)}=\frac{LR^{2}}{2(\theta_{N}^{2}-1)},
\end{align*}
and $f$ and $h$ are our desired functions.
\qed\end{proof}

\section{Omitted proof of Theorem \ref{thm:OptISTA-lb2}}\label{s:e}
We now make a formal definition of the deterministic $N$-step double-oracle method.

A \emph{deterministic $N$-step double-oracle method} $\mathbf{A}$ is a mapping from initial point $z_{0}$ and a tuple of function $(f,h)$ to $\mathbf{A}(z_0,(f,h))=\{ \{(z_i,\delta_i,\gamma_i)\}_{i\in[0:2N-1]}, x_{N} \}$. Here, we use $x_N$ for consistency with Section \ref{s:3}, and we call $x_N$ the \emph{approximate solution}, or simply the \emph{output}. The sequence $\mathbf{A}(z_0,(f,h))=\{ \{(z_i,\delta_i,\gamma_i)\}_{i\in[0:2N-1]}, x_{N} \}$ depends on $(f,h)$ only through $N$ queries to gradient oracle $\mathcal{O}_f(z_i)=(\nabla f(z_i),f(z_i))$ and  $N$ to proximal oracle $\mathcal{O}_h(z_i,\gamma_i)=(\mathbf{prox}_{\gamma_i h}(z_i), h(z_i))$. Here, $z_i$ denotes the next (gradient or proximal) query point to an oracle, $\delta_i\in\{0,1\}$ is the indicator that tells a method to which oracle to query next. Specifically, if $\delta_i=0$, $z_i$ is fed into $\mathcal{O}_f(\cdot)$ , otherwise $z_i$ is fed into $\mathcal{O}_h(\cdot, \gamma_i)$, where
and $\gamma_i>0$ is $i$-th proximal stepsize.  
For further precision, we define
\begin{gather*}
(\delta_0,\gamma_0)=\mathbf{A}_0(z_{0}),
\\(z_i,\delta_i,\gamma_i)=\mathbf{A}_i\left(\{(z_j,\delta_j,\gamma_j)\}_{j=0}^{i-1},  \{(1-\delta_j)\mathcal{O}_f(z_j), \delta_{\color{red}{j}}\mathcal{O}_h(z_j,\gamma_j)\}_{j=0}^{i-1} \right)\qquad\text{ for }i\in[1:2N-1],
\\x_{N}=\mathbf{A}_{2N}\left(\{(z_j,\delta_j,\gamma_j)\}_{j=0}^{2N-1},\{(1-\delta_j)\mathcal{O}_f(z_{\color{red}{j}}), \delta_j\mathcal{O}_h(z_j,\gamma_j)\}_{j=0}^{2N-1} \right),
\end{gather*}
where we use, for notational shorthand,
\begin{itemize}
\item[$\bullet$] $0\cdot\mathcal{O}_f=\emptyset$ ($\mathcal{O}_f$ is not evaluated)
\item[$\bullet$] $0\cdot\mathcal{O}_h=\emptyset$ ($\mathcal{O}_h$ is not evaluated)
\item[$\bullet$] $1\cdot\mathcal{O}_f=\mathcal{O}_f$ ($\mathcal{O}_f$ is evaluated)
\item[$\bullet$] $1\cdot\mathcal{O}_h=\mathcal{O}_h$ ($\mathcal{O}_h$ is evaluated).
\end{itemize}
Additionally, the number of queries to $\mathcal{O}_f$ and $\mathcal{O}_h$ must each be exactly $N$. i.e., 
\begin{align*}
    \sum_{j=0}^{2N-1}\delta_j=\sum_{j=0}^{2N-1}(1-\delta_j)= N.
\end{align*}

Theorem~\ref{thm:OptISTA-lb2} expands the result of Theorem~\ref{thm:OptISTA-lb} to any deterministic $N$-step double-oracle methods using resisting oracle technique \cite{nemirovski1983problem,carmon2020stationary1,park2022}. 
To get the desired result, we use \emph{double-function zero-respecting} condition, which is similar but slightly more general than the double-function span condition. 
The sequence $\{ \{({z}_i,\delta_i,\gamma_i)\}_{i\in[0:2N-1]}, {x}_N \}$ is said to be \emph{double-function zero-respecting} with respect to $(f,h)$ if
\begin{equation*}
\begin{aligned}
&\delta_i\in \{0,1\}, &&\textup{ for }i=0,\dots,2N-1,\\
& \sum^{2N-1}_{i=0}\delta_i=\sum^{2N-1}_{i=0}(1-\delta_i)=N,
\\
&d_i=\left\{
\begin{array}{ll}
\nabla f(z_i) &\textup{ if $\delta_i=0$, }\\
 z_i-\prox_{\gamma_i h}(z_i) &\textup{if $\delta_i=1$ ,}
\end{array}
\right.
&&\textup{ for }i=0,\dots,2N-1,\\
&\mathrm{supp}\{z_{i}\}\subseteq\mathrm{supp}\{d_0,\dots,d_{i-1}\},
&&\textup{ for }i=0,\dots,2N-1,\\
&\mathrm{supp}\{x_N\}\subseteq\mathrm{supp}\{d_0,\dots,d_{2N-1}\},
\end{aligned}
\end{equation*}
where $\mathrm{supp}\{z_i\}\triangleq\{j\in[0:d-1]\,|\, \langle e_j, z_i \rangle \neq 0\}$ and $\mathrm{supp}\{d_0,\ldots,d_i\}\triangleq\bigcup_{j=0}^i\mathrm{supp}\{d_j\}$. Note that by definition, $z_0=0$ for any double-function zero-respecting sequences.

We say a matrix $U\in\rl^{m\times n}$ is \emph{orthogonal} if columns $\{u_i\}_{i\in[0:n-1]}\in\rl^{m}$ of $U$ are mutually orthonormal, or equivalently, $U^\intercal U=I_n$.

The following lemmas are building blocks for resisting oracle technique.
 
\begin{lem}\label{lem:fUlemma}
    For any orthogonal matrix   $U\in\mathbb{R}^{d'\times d}$ with $d'\geq d$ and any vector $x_0\in\mathbb{R}^{d'}$, if $f\colon \mathbb{R}^d \rightarrow \mathbb{R}$ is $L$-smooth and convex, then $f_U\colon\mathbb{R}^{d'} \rightarrow \mathbb{R}$ defined by 
    \[
    f_U(x)\triangleq f(U^\intercal(x-x_0))
    \]
    is $L$-smooth and convex.
\end{lem}
\begin{proof}
Note that $\nabla f_U(x)=U\nabla f(U^\intercal(x-x_0))$. For any $x,y\in\mathbb{R}^{d'}$,
\begin{align*}
      \|\nabla f_U(x)-\nabla f_U(y)\|&=\| U\nabla f(U^\intercal(x-x_0))-U\nabla f(U^\intercal(y-x_0))\|\\
      &=\| \nabla f(U^\intercal(x-x_0))-\nabla f(U^\intercal(y-x_0))\| {\color{gray}{\qquad\rhd\, U \text{ is orthogonal }}}
\end{align*}
Therefore,
\begin{align*}
    &f_U(y)-f_U(x)-\langle \nabla f_U(x), y-x \rangle -\frac{1}{2L}\|\nabla f_U(x)-\nabla f_U(y)\|^2
    =f(U^\intercal(y-x_0))-f(U^\intercal(x-x_0))\\
    &-\langle U\nabla f(U^\intercal(x-x_0)), y-x \rangle -\frac{1}{2L}\| \nabla f(U^\intercal(x-x_0))-\nabla f(U^\intercal(y-x_0))\|^2\\
    &=f(U^\intercal(y-x_0))-f(U^\intercal(x-x_0))-\langle \nabla f(U^\intercal(x-x_0)), U^\intercal(y-x) \rangle \\
    &-\frac{1}{2L}\| \nabla f(U^\intercal(x-x_0))-\nabla f(U^\intercal(y-x_0))\|^2\geq 0,
\end{align*}
where the last inequality is from $L$-smoothness and convexity of $f$. 
\qed\end{proof}
\begin{lem}\label{lem:hUlemma}
    For any orthogonal matrix   $U\in\mathbb{R}^{d'\times d}$ with $d'\geq d$ and any vector $x_0$, if $h\colon \mathbb{R}^d \rightarrow \mathbb{R}$ is an indicator function of nonempty closed convex set, then  
    $h_U\colon\mathbb{R}^{d'} \rightarrow \mathbb{R}$ defined by 
    \[
    h_U(x)\triangleq h(U^\intercal(x-x_0))
    \]
    is an indicator function of nonempty closed convex set.
\end{lem}
\begin{proof}
    Assume $h=\delta_C$ for nonempty closed convex set $C\subseteq \mathbb{R}^d$ and let
    \[
    \tilde{C}\triangleq\{x \,|\, U^\intercal(x-x_0)\in C\}\subseteq\mathbb{R}^{d'}.
    \]
    Then, $h_U=\delta_{\tilde{C}}$. Now we claim that $\tilde{C}$ is nonempty, closed, and convex. Take $c\in C$. Then since $Uc+x_0\in\tilde{C}$, $\tilde{C}\neq\emptyset$. $\tilde{C}$ is closed since it is preimage of $C$ under continuous transformation $U^\intercal(\cdot-x_0)$. For convexity, suppose $x,y\in\tilde{C}$ and $t\in[0,1]$. Then,
    \begin{align*}
        U^\intercal\left((1-t)x+ty-x_0\right)&=U^\intercal\left((1-t)(x-x_0)+t(y-x_0)\right)\\
        &=(1-t)U^\intercal(x-x_0)+tU^\intercal(y-x_0)\in C
    \end{align*}
    where the last inclusion is from the convexity of $C$. Therefore we conclude $(1-t)x+ty\in\tilde{C}$.
\qed\end{proof}
From now, $U\in\rl^{d'\times d}$ is orthogonal, and $f_U$ and $h_U$ are as defined in Lemma~\ref{lem:fUlemma} and \ref{lem:hUlemma} respectively.
\begin{lem}\label{lem:fUhUlemma}
Suppose $f\colon\mathbb{R}^{d} \rightarrow \mathbb{R}$ is $L$-smooth convex and $h\colon\mathbb{R}^{d} \rightarrow \mathbb{R}$ is indicator function of nonempty closed convex set.  If  $\tilde{x}_\star\in\argmin_{x\in\mathbb{R}^{d}}(f+h)$ exists, then $x^\star\triangleq U \tilde{x}_\star+x_0\in\argmin_{x\in\mathbb{R}^{d'}}(f_U+h_U)$.
\end{lem}
\begin{proof}
Note that $U^\intercal U=I_d$. For any $x\in\mathbb{R}^{d'}$,
\begin{align*}
f_U({x}_\star)+h_U({x}_\star)&=f(U^\intercal(U\tilde{x}_\star+x_0-x_0))+h(U^\intercal(U\tilde{x}_\star+x_0-x_0))\\
&=f(\tilde{x}_\star)+h(\tilde{x}_\star)\\
&\leq f(U^\intercal(x-x_0))+h(U^\intercal(x-x_0)) \qquad {\color{gray}{\qquad \rhd\,  \tilde{x}_\star\in\argmin_{x\in\mathbb{R}^d}(f+h)}}\\
&=f_U(x)+h_U(x).
\end{align*}
\qed\end{proof}
\begin{lem}\label{lem:respectingisspan}
 Take $f\colon \mathbb{R}^d \rightarrow \mathbb{R}$ and  $h\colon \mathbb{R}^d \rightarrow \mathbb{R}$ to be the functions that are defined in Theorem \ref{thm:OptISTA-lb}. If $\{ \{({z}_i,\delta_i,\gamma_i)\}_{i\in[0:2N-1]}, {x}_{N} \}$ is double-function zero-respecting with respect to $(f,h)$, then it satisfies the double-function span condition \eqref{eq:doublefunctionspan}.
 \end{lem}
\begin{proof}
Assume $\{ \{({z}_i,\delta_i,\gamma_i)\}_{i\in[0:2N-1]}, {x}_{N} \}$ is double-function zero-respecting with respect to $(f,h)$. Then, we have $z_0=0$ by the double-function zero-respecting assumption. From this, we have $d_0\in\mathrm{span}\{e_0\}$ by the property of $(f,h)$ in Theorem \ref{thm:OptISTA-lb} when $z_0=0$. Now suppose $d_i\in\mathrm{span}\{e_0,\ldots,e_{i}\}$ for all $0\leq i < k$ where $0<k\leq 2N-1$. Then, $\mathrm{supp}\{z_k\}\subseteq\mathrm{supp}\{d_0,\ldots,d_{k-1}\}\subseteq\{0,\ldots,k-1\}$ by the double-function zero-respecting assumption. By the property of $(f,h)$, we then have $d_k\in\{e_0,\ldots,e_{k}\}$. Therefore, $d_i\in\mathrm{span}\{e_0,\ldots,e_{i}\}$ holds for all $i\in[0:2N-1]$. This proves the equivalence of double-function span condition and double-function zero-respecting sequence, under the choice of $(f,h)$ from Theorem \ref{thm:OptISTA-lb}.
\qed\end{proof}

\begin{lem}\label{lem:resistoracle}
    Assume $L>0$, $N>0$, and $d'\geq d+2N$. Let $\mathbf{A}$ be any deterministic $N$-step double-oracle method, $x_0=z_0\in\mathbb{R}^{d'}$ be any vector, $f\colon \mathbb{R}^d \rightarrow \mathbb{R}$ is $L$-smooth and convex, $h\colon \mathbb{R}^d \rightarrow \mathbb{R}$ is an indicator function of nonempty closed convex set. Then there exists an orthogonal matrix $U\in\mathbb{R}^{d' \times d}$ and $\mathbf{A}(z_0,(f_U,h_U))=\{ \{(z_i,\delta_i,\gamma_i)\}_{i\in[0:2N-1]}, x_N \}$ such that if $\tilde{z}_i\triangleq U^\intercal(z_i-z_0)\in\mathbb{R}^{N+1}$, then $\{ \{(\tilde{z}_i,\delta_i,\gamma_i)\}_{i\in[0:2N-1]}, \tilde{x}_{N} \}$ is  double-function zero-respecting with respect to $(f,h)$.
\end{lem}
\begin{proof}
Recall the support condition of double-function zero-respecting sequences. 
\[
\mathrm{supp}\{z_i\}\subseteq\bigcup_{j=0}^{i-1}\mathrm{supp}\left\{ \nabla f(z_j) \ \textrm{or} \ z_j-\mathbf{prox}_{\gamma_{j}h}(z_j)\right\}=\mathrm{supp}\left\{d_0, \ldots, d_{i-1} \right\},
\]
\[
\mathrm{supp}\{x_N\}=\mathrm{supp}\left\{d_0, \ldots, d_{2N-1} \right\},
\]
where
\[
d_i=\begin{cases}
\nabla f(z_i) \ \text{ if $\delta_i=0$, }\\
 z_i-\mathbf{prox}_{\gamma_i h}(z_i) \ \text{if $\delta_i=1$}.
\end{cases}
\]
For $i\in[0:2N]$, let $I_i\subseteq[0:d-1]$ be
\[
I_i=\mathrm{supp}\left\{\tilde{d}_0, \ldots, \tilde{d}_{i-1} \right\},
\]
where
\[
\tilde{d}_i=\begin{cases}
\nabla f(\tilde{z}_i) \ \text{ if $\delta_i=0$, }\\
 \tilde{z}_i-\mathbf{prox}_{\gamma_i h}(\tilde{z}_i) \ \text{if $\delta_i=1$}.
\end{cases}
\]
Then $\{ \{(\tilde{z}_i,\delta_i,\gamma_i)\}_{i\in[0:2N-1]}, \tilde{x}_{N} \}$ being double-function zero-respecting with respect to $(f,h)$ is equivalent to 
\[
\mathrm{supp}\{\tilde{z}_i\}\subseteq I_i ,\ i\in[0:2N-1]\text{ and } \mathrm{supp}\{\tilde{x}_N\}\subseteq I_{2N}.
\]
Also note that
\[
\emptyset=I_0\subseteq I_1 \subseteq \cdots \subseteq I_{2N-1} \subseteq I_{2N} \subseteq [0:d-1].
\]
Now we construct the matrix $U=[u_0 \,|\, \cdots\,|\, u_{d-1}]$ by choosing appropriate $u_i\in\mathbb{R}^{d'}$ inductively. For $i=0$, $\tilde{z}_0=0$ is clear from the definition.
Since $\tilde{z}_i=U^\intercal(z_i-z_0)$, $\mathrm{supp}\{\tilde{z}_i\}\subseteq I_i$ is again equivalent to 
\[
\langle u_j, z_i-z_0 \rangle =0 \quad \forall j\notin I_i ,\ i\in[0:2N-1] \text{ and } \langle u_j, x_N-z_0 \rangle =0 ,\ \forall j\notin I_{2N}.
\]
Now for $i\in[1:2N]$, choose $\{u_j\}_{j\in I_i\setminus I_{i-1}}$ from the orthogonal complement of 
\[
S_{i-1}\triangleq\mathrm{span}\left\{\left\{z_1-z_0, \cdots, z_{i-1}-z_0\right\}\cup\{u_j\}_{j\in I_{i-1}}\right\}
\]
and let them be mutually orthonormal. At the last step, choose $\{u_j\}_{j\in [0:d-1]\setminus I_{2N}}$ from the orthogonal complement of
\[
S_{2N}\triangleq\mathrm{span}\left\{\left\{z_1-z_0, \cdots, z_{2N-1}-z_0, x_N-z_0 \right\}\cup\{u_j\}_{j\in I_{2N}}\right\}.
\]
Here, we need to verify that the dimension is large enough to choose such orthonormal set of vectors. More precisely, we have to check 
\[
\mathrm{dim}S_{i-1}^\perp\geq |I_i\setminus I_{i-1}| \text{ for } i\in[1:2N] \text{ and } \mathrm{dim}S_{2N}^\perp\geq d-|I_{2N}|,
\]
which is true from our assumption $d' \geq d+2N$. To be more precise,
\[
\mathrm{dim}S_{i-1}^\perp=d'-(i-1)-|I_{i-1}|\geq d'-2N-|I_{i-1}|\geq d-|I_{i-1}| \geq |I_i\setminus I_{i-1}|
\]
and
\[
\mathrm{dim}S_{2N}^\perp=d'-2N-|I_{2N}|\geq  d-|I_{2N}| .
\]
Then $\{u_i\}_{i=0}^{d-1}$ are orthonormal and $\langle u_j, z_i-z_0 \rangle =0$ for $j\notin  {\color{red}{I}}_i$. Therefore  $U=[u_0 \,|\, \cdots\,|\, u_{d-1}]$ satisfies the statement of the lemma. 
\qed\end{proof}
Now we prove Theorem \ref{thm:OptISTA-lb2}.
\begin{proof}\textit{of Theorem \ref{thm:OptISTA-lb2}}
Take $(f,h)$ to be tuple of $L$-smooth convex function on $\mathbb{R}^{N+1}$ and an indicator function of nonempty closed convex set on $\mathbb{R}^{N+1}$, which is defined in Theorem \ref{thm:OptISTA-lb}. Let $\tilde{x}_\star\in\argmin_{x\in\mathbb{R}^{N+1}}(f+h)$. By Lemma~\ref{lem:resistoracle}, there exists an orthogonal matrix $U\in\mathbb{R}^{d \times (N+1)}$ with $d\geq2N+N+1=3N+1$ and $\mathbf{A}(z_0,(f_U,h_U))=\{ \{(z_i,\delta_i,\gamma_i)\}_{i\in[0:2N-1]}, x_N \}$ such that if $\tilde{z}_i\triangleq U^\intercal(z_i-z_0)\in\mathbb{R}^{N+1}$, then $\{ \{(\tilde{z}_i,\delta_i,\gamma_i)\}_{i\in[0:2N-1]}, \tilde{x}_{N} \}$ is double-function zero-respecting with respect to $(f,h)$. Then by Lemma~\ref{lem:respectingisspan},  $\{ \{(\tilde{z}_i,\delta_i,\gamma_i)\}_{i\in[0:2N-1]}, \tilde{x}_{N} \}$ satisfies the double-function span condition. Therefore, by Theorem \ref{thm:OptISTA-lb},
\[
f(\tilde{x}_N)+h(\tilde{x}_N)-f(\tilde{x}_\star)-h(\tilde{x}_\star)\ge\frac{L\|0-\tilde{x}_{\star}\|^{2}}{2(\theta_{N}^{2}-1)}.
\]
Note that $x_\star\triangleq U\tilde{x}_\star+x_0\in\argmin_{x\in\mathbb{R}^{d}}(f_U+h_U)$ by Lemma~\ref{lem:fUhUlemma}. Then we finally have
\begin{align*}
 f_U(x_N)+h_U(x_N)-f_U(x_\star)-h_U(x_\star)&=f_U(x_N)+h_U(x_N)-f_U(U\tilde{x}_\star+x_0)-h_U(U\tilde{x}_\star+x_0)\\
 &=f(\tilde{x}_N)+h(\tilde{x}_N)-f(\tilde{x}_\star)-h(\tilde{x}_\star)\\
&\geq\frac{L\|0-\tilde{x}_{\star}\|^{2}}{2(\theta_{N}^{2}-1)}=\frac{L\|U\tilde{x}_{\star}\|^{2}}{2(\theta_{N}^{2}-1)}=\frac{L\|x_0-{x}_{\star}\|^{2}}{2(\theta_{N}^{2}-1)}.    
\end{align*}
Therefore, $f_U$ and $h_U$ are our desired functions.
\qed\end{proof}
\section{Omitted proof of Theorem \ref{thm:opm-lb2}}\label{s:f}
We now make a formal definition of the deterministic $N$-step proximal method.

We assume proximal stepsizes $\{\gamma_i\}_{i\in[0:N-1]}$ are given and fixed. A \emph{deterministic $N$-step proximal method} $\mathbf{A}$ is a mapping from initial point $x_{0}$ and a function $h$ to $\mathbf{A}(x_0,h)=\{x_i\}_{i\in[0:N]}$, and we call $x_N$ the \emph{approximate solution}, or simply the \emph{output}. The sequence $\mathbf{A}(x_0,h)=\{x_i\}_{i\in[0:N]}$ depends on $h$ only through $N$ queries to proximal oracle $\mathcal{O}_h(x_i,\gamma_i)=(\mathbf{prox}_{\gamma_i h}(x_i), h(x_i))$. Here, $x_i$ denotes the next query point to an oracle. More precisely, for $i\in[1:N]$,
\[
x_i=\mathbf{A}_i(x_0,h)=\mathbf{A}(\{x_j\}_{j=0}^{i-1},\mathcal{O}_h(x_0,\gamma_0),\ldots , \mathcal{O}_h(x_{{\color{red}{i}}-1},\gamma_{{\color{red}{i}}-1})).
\]

Similar to the previous section, we expand the result of Theorem \ref{thm:opm-lb} to any deterministic $N$-step proximal methods by using the resisting oracle technique \cite{nemirovski1983problem,carmon2020stationary1}. 

We use \emph{proximal zero-respecting} condition, which is similar but slightly more general than the proximal span condition. 
The sequence $\{x_i\}_{i\in[0:N]}$ is said to be \emph{proximal zero-respecting} with respect to $h$ if
\begin{equation*}
\begin{aligned}
\mathrm{supp}\{x_{i}\}\subseteq \mathrm{supp}\{
x_0-\prox_{\gamma_0 h}(x_0)
,\dots,
x_{i-1}-\prox_{\gamma_{i-1} h}(x_{i-1})
\}
\qquad\textup{ for } i=0,\dots,N.
\end{aligned}
\end{equation*}
where $\mathrm{supp}\{x\}\triangleq\{i\in[0:d-1]\,|\, \langle e_i, x \rangle \neq 0\}$. Note that by definition, $x_0=0$ for any proximal zero-respecting sequences.

The following lemmas are building blocks for the proximal version of the resisting oracle technique. 
\begin{lem}\label{lem:proxhUlemma1}
     For any orthogonal matrix   $U\in\mathbb{R}^{d'\times d}$ with $d'\geq d$ and any vector $x_0\in\mathbb{R}^{d'}$, if $h\colon \mathbb{R}^d \rightarrow \mathbb{R}$ is convex, then $h_U\colon\mathbb{R}^{d'} \rightarrow \mathbb{R}$ defined by 
    \[
    h_U(x)=h(U^\intercal(x-x_0))
    \]
    is convex.
\end{lem}
\begin{proof}
    Suppose $x,y\in\mathbb{R}^{d'}$. Then for $t\in[0,1]$,
    \begin{align*}
        h_U\left(\left(1-t\right)x+ty\right)&=h\left(U^\intercal\left(\left(1-t\right)x+ty-x_0\right)\right)\\
        &=h\left(\left(1-t\right)U^\intercal\left(x-x_0\right)+tU^\intercal\left(y-x_0\right)\right)\\
        &\leq(1-t)h\left(U^\intercal\left(x-x_0\right)\right)+th\left(U^\intercal\left(y-x_0\right)\right)\\
        &=(1-t)h_U(x)+th_U(y).
    \end{align*}
    where the inequality is from convexity of $h$. Therefore, $h_U$ is convex.
\qed\end{proof}
\begin{lem}\label{lem:proxhUlemma2}
Suppose $h\colon\mathbb{R}^{d} \rightarrow \mathbb{R}$ is convex. Let $h_U$ be as defined in Lemma~\ref{lem:proxhUlemma1}. If  $\tilde{x}_\star\in\argmin_{x\in\mathbb{R}^{d}}h$ exists, then $x^\star\triangleq U \tilde{x}_\star+x_0\in\argmin_{x\in\mathbb{R}^{d'}}h_U$.
\end{lem}
\begin{proof}
For any $x\in\mathbb{R}^{d'}$,
\begin{align*}
h_U({x}_\star)&=h(U^\intercal(U\tilde{x}_\star+x_0-x_0))=h(\tilde{x}_\star)\leq h(U^\intercal(x-x_0))=h_U(x),
\end{align*}
where the inequality is from $\tilde{x}_\star\in\argmin_{\mathbb{R}^d}h$.
\qed\end{proof}
\begin{lem}\label{lem:respectingisproxspan}
 Take $h\colon \mathbb{R}^d \rightarrow \mathbb{R}$ to be the function defined in Theorem \ref{thm:opm-lb}. If $\{x_i\}_{i\in[0:N]}$ is proximal zero-respecting with respect to $h$, then it satisfies the proximal span condition \eqref{eq:proxfunctionspan}.
 \end{lem}
\begin{proof}
Assume $\{x_i\}_{i\in[0:N]}$ is proximal zero-respecting with respect to $h$. Then, we have $x_0=0$ by the double-function zero-respecting assumption. From this, we have $d_0\in\mathrm{span}\{e_0\}$ by the property of $h$ in Theorem \ref{thm:opm-lb} when $x_0=0$. Now suppose $x_i-\prox_{\gamma_i h}(x_i)\in\mathrm{span}\{e_0,\ldots,e_{i}\}$ for all $0\leq i < k$ where $0<k\leq N-1$. Then, $\mathrm{supp}\{x_k\}\subseteq\mathrm{supp}\{x_0-\prox_{\gamma_0 h}(x_0),\ldots,x_{k-1}-\prox_{\gamma_{k-1} h}(x_{k-1})\}\subseteq\{0,\ldots,k-1\}$ by the proximal zero-respecting assumption. By the property of $h$, we then have $x_k-\prox_{\gamma_k h}(x_k)\in\{e_0,\ldots,e_{k}\}$. Therefore, $x_i-\prox_{\gamma_i h}(x_i)\in\mathrm{span}\{e_0,\ldots,e_{i}\}$ holds for all $i\in[0:N-1]$. This proves the equivalence of proximal span condition and proximal zero-respecting sequence, under the choice of $h$ from Theorem \ref{thm:opm-lb}.
\qed\end{proof}
\begin{lem}\label{lem:resistoracleOPM}
    Assume  $N>0$, and $d'\geq d+N$. Let $\mathbf{A}$ be any deterministic $N$-step proximal method, $x_0\in\mathbb{R}^{d'}$ be any vector, $h\colon \mathbb{R}^d \rightarrow \mathbb{R}$ is convex. Then there exists an orthogonal matrix $U\in\mathbb{R}^{d' \times d}$ and $\mathbf{A}(x_0, h_U)=\{x_i\}_{i\in[0:N]}$ such that if $\tilde{x}_i\triangleq U^\intercal(x_i-x_0)\in\mathbb{R}^{N+1}$, then $\{\tilde{x}_i\}_{i\in[0:N]}$ is proximal zero-respecting with respect to $h$.
\end{lem}
\begin{proof}
Recall the support condition of proximal zero-respecting sequences. 
\[
\mathrm{supp}\{x_i\}\subseteq\mathrm{supp}\left\{x_0-\mathbf{prox}_{\gamma_0 h}(x_0),\ldots, x_{i-1}-\mathbf{prox}_{\gamma_{i-1} h}(x_{i-1})\right\}.
\]
For $i\in[0:N]$, let $I_i\subseteq[0:d-1]$ be
\[
I_i=\mathrm{supp}\left\{\tilde{x}_0-\mathbf{prox}_{\gamma_0 h}(\tilde{x}_0),\ldots, \tilde{x}_{i-1}-\mathbf{prox}_{\gamma_{i-1} h}(\tilde{x}_{i-1})\right\}.
\]
Then  $\{\tilde{x}_i\}_{i\in[0:N]}$ being proximal zero-respecting with respect to $h$  is equivalent to 
\[
\mathrm{supp}\{\tilde{x}_i\}\subseteq I_i ,\ i\in[0:N].
\]
Also note that
\[
\emptyset=I_0\subseteq I_1 \subseteq \cdots \subseteq I_{N-1} \subseteq I_{N} \subseteq [0:d-1].
\]
Now we construct the matrix $U=[u_0 \,|\, \cdots\,|\, u_{d-1}]$ by choosing appropriate $u_i\in\mathbb{R}^{d'}$ inductively. For $i=0$, $\tilde{x}_0=0$ is clear from the definition.
Since $\tilde{x}_i=U^\intercal(x_i-x_0)$, $\mathrm{supp}\{\tilde{x}_i\}\subseteq I_i$ is again equivalent to 
\[
\langle u_j, x_i-x_0 \rangle =0 \quad \forall j\notin I_i ,\ i\in[0:N].
\]
Now for $i\in[1:N]$, choose $\{u_j\}_{j\in I_i\setminus I_{i-1}}$ from the orthogonal complement of 
\[
S_{i-1}\triangleq\mathrm{span}\left\{\left\{x_1-x_0, \cdots, x_{i-1}-x_0\right\}\cup\{u_j\}_{j\in I_{i-1}}\right\},
\]
so that they are mutually orthonormal. At the last step, choose $\{u_j\}_{j\in [0:d-1]\setminus I_{N}}$ from the orthogonal complement of
\[
S_{N}\triangleq\mathrm{span}\left\{\left\{x_1-x_0, \cdots, x_{N}-x_0 \right\}\cup\{u_j\}_{j\in I_{N}}\right\}.
\]
Here, we need to verify that the dimension is large enough to choose such orthonormal set of vectors. More precisely, we have to check 
\[
\mathrm{dim}S_{i-1}^\perp\geq |I_i\setminus I_{i-1}| \text{ for } i\in[1:N], 
\]
which is true from our assumption $d' \geq d+N$. To be more precise,
\[
\mathrm{dim}S_{i-1}^\perp=d'-(i-1)-|I_{i-1}|\geq d'-N-|I_{i-1}|\geq d-|I_{i-1}| \geq |I_i\setminus I_{i-1}|,
\]
\[
\mathrm{dim}S_{N}^\perp=d'-N-|I_{\color{red}{N}}|\geq  d-|I_{\color{red}{N}}|.
\]
Then $\{u_i\}_{i=0}^{d-1}$ are orthonormal and $\langle u_j, {\color{red}{x}}_i-{\color{red}{x}}_0 \rangle =0$ for $j\notin {\color{red}{I}}_i$. Therefore  $U=[u_0 \,|\, \cdots\,|\, u_{d-1}]$ satisfies the statement of the lemma. 
\qed\end{proof}
Now we prove Theorem \ref{thm:opm-lb2}.

\begin{proof}\textit{of Theorem \ref{thm:opm-lb2}}.
Take $h$ be closed, convex, and proper function on $\mathbb{R}^{N+1}$ defined in Theorem \ref{thm:opm-lb}. Let $\tilde{x}_\star\in\argmin_{x\in\mathbb{R}^{N+1}}h$. By Lemma~\ref{lem:resistoracleOPM}, there exists an orthogonal matrix $U\in\mathbb{R}^{d \times (N+1)}$ with $d\geq N+N+1=2N+1$ and $\mathbf{A}(x_0,h_U)=\{x_i\}_{i\in[0:N]}$ such that if $\tilde{x}_i\triangleq U^\intercal(x_i-x_0)\in\mathbb{R}^{N+1}$, then $\{\tilde{x}_i\}_{i\in[0:N]}$ is proximal zero-respecting with respect to $h$. Then by Lemma~\ref{lem:respectingisproxspan}, $\{\tilde{x}_i\}_{i\in[0:N]}$ satisfies the proximal span condition. Therefore by Theorem \ref{thm:opm-lb},
\[
h(\tilde{x}_N)-h(\tilde{x}_\star)\ge\frac{\gamma_{N-1}\|0-\tilde{x}_\star\|^2}{4\gamma_0^2\eta_{N-1}^2}{\color{red}{-\varepsilon}}.
\]
Note that $x_\star\triangleq U\tilde{x}_\star+x_0\in\argmin_{x\in\mathbb{R}^{d}}h_U$ by Lemma~\ref{lem:proxhUlemma2}. Then we finally have
\begin{align*}
 h_U(x_N)-h_U(x_\star)&=h_U(x_N)-h_U(U\tilde{x}_\star+x_0)\\
 &=h(\tilde{x}_N)-h(\tilde{x}_\star)\\
&\geq\frac{\gamma_{N-1}\|0-\tilde{x}_\star\|^2}{4\gamma_0^2\eta_{N-1}^2}{\color{red}{-\varepsilon}}\\
&=\frac{\gamma_{N-1}\|U\tilde{x}_\star\|^2}{4\gamma_0^2\eta_{N-1}^2}{\color{red}{-\varepsilon}}\\
&=\frac{\gamma_{N-1}\|x_0-x_\star\|^2}{4\gamma_0^2\eta_{N-1}^2}{\color{red}{-\varepsilon}}.
\end{align*}

Therefore, $h_U$ is our desired function.
\qed\end{proof}

{\color{black}{
\section{Alternate proof of Theorem~\ref{thm:OptISTA-rate}}\label{s:g}

In this section, we give another proof of Theorem~\ref{thm:OptISTA-rate}.

\subsection{Proof outline}

Assume $x_{\star}=y_{\star}=0$ without the loss of generality. For convenience,
define the notation
\[
f_{i}=f(x_{i})\ \mathrm{for}\ i\in[0:N],\ f_{\star}=f(x_{\star}),\  f'_{i}=\nabla f(x_{i})\ \mathrm{for}\ i\in[0:N],\  f_{\star}'=\nabla f(x_{\star}),
\]
\[
h_{i}=h(y_{i})\ \mathrm{for}\ i\in[1:N],\ h_{\star}=h(y_{\star}),\ h_{i}'=h'(y_{i})\ \mathrm{for}\ i\in[1:N],\ h_{\star}'=h'(y_{\star}).
\]
 Since Lemma~\ref{c:x=y} establishes that $x_N = y_N$, there is no need to separately define $h'(x_N)=h'_N$ and $h'(y_N)=h'_{2N}$ as was done in the formulation of optimization problem in Section~\ref{s:3}. Therefore, we use the index range $[1:N]$ for the function $h$, rather than $[N:2N]$, in order to keep the presentation simple. Now consider the following positive-semidefinite matrix

\begin{align}
\!\!\!\!\!\!\!\!\!
    G=\begin{bmatrix}
        \|x_0\|^2 & \langle x_0, f'_\star\rangle & \langle x_0,f'_0 \rangle & \cdots & \langle x_0,f'_N \rangle & \langle x_0, h_1' \rangle & \cdots & \langle x_0,h_N' \rangle \\
         \langle f'_\star, x_0 \rangle & \|f'_\star\|^2 & \langle f'_\star,f'_0 \rangle & \cdots & \langle f'_\star,f'_N \rangle & \langle f'_\star, h_1' \rangle & \cdots & \langle f'_\star,h_N' \rangle\\
          \langle f'_0, x_0 \rangle & \langle f'_0, f'_\star \rangle & \|f'_0\|^2 & \cdots & \langle f'_0,f'_N \rangle & \langle f'_0, h_1' \rangle & \cdots & \langle f'_0,h_N' \rangle\\
           \vdots & \vdots &\vdots & \ddots & \vdots & \vdots & \ddots & \vdots\\
                     \langle f'_N, x_0 \rangle & \langle f'_N, f'_\star \rangle & \langle f'_N, f'_0\rangle  & \cdots & \|f'_N\|^2 & \langle f'_N, h_1' \rangle & \cdots & \langle f'_N,h_N' \rangle\\
                      \langle h_1', x_0 \rangle & \langle h_1', f'_\star \rangle & \langle h_1', f'_0\rangle  & \cdots & \langle h_1', f'_N\rangle & \|h_1'\|^2 & \cdots & \langle h_1',h_N' \rangle\\
                        \vdots & \vdots &\vdots & \ddots & \vdots & \vdots & \ddots & \vdots\\
                           \langle h_N', x_0 \rangle & \langle h_N', f'_\star \rangle & \langle h_N', f'_0\rangle  & \cdots & \langle h_N', f'_N\rangle & \langle h_N', h_1'\rangle & \cdots & \|h_N'\|^2 \\
    \end{bmatrix}\in \mathbb{S}_{+}^{2N+3}.\label{eq:def_of_G}
\end{align}
Then, $G$ is a symmetric outer product and is, therefore, positive semidefinite. 

Suppose we have shown the existence of $\lambda_{i,j}\geq 0$, $\tau_{i,j}\geq 0$, and positive semidefinite matrix $Z\in \mathbb{S}_{+}^{2N+3}$ such that
\begin{equation}
\begin{alignedat}{1} & \frac{L\|x_{0}-x_{\star}\|^{2}}{2(\theta_{N}^{2}-1)}-\left(f_{N}+h_{N}-f_{\star}-h_{\star}\right)+\sum_{i,j\in\{\star,0,\ldots,N\}}\lambda_{i,j}\underbrace{\left(f_{j}-f_{i}+\langle f'_{j},x_{i}-x_{j}\rangle+\tfrac{1}{2L}\|f'_{i}-f'_{j}\|^{2}\right)}_{\le0\,\, (\textup{$L$-smoothness and convexity of \ensuremath{f}})}\\
 & \quad+\sum_{i,j\in\{\star,1,\ldots,N\}}\tau_{i,j}\underbrace{\left(h_{j}-h_{i}+\langle h_{j}',y_{i}-y_{j}\rangle\right)}_{\le0\,\,\textup{ (convexity of \ensuremath{h})}}\overset{(*)}{=}\tr(ZG)\geq0,
\end{alignedat}
\label{eq:weighted_ineq_to_proof}
\end{equation}
where we have used the fact that trace of the product of two positive-semidefinite
matrices is nonnegative \cite[Proposition 6.2.3]{lovasz2003semidefinite}.
Previously, we showed that $x_N=y_N$. Therefore, we conclude 
\begin{align}
f(y_N)+h(y_N) = f(x_N) + h(y_N) = f_{N}+h_{N}-f_{\star}-h_{\star}\leq\frac{L\|x_{0}-x_{\star}\|^{2}}{2(\theta_{N}^{2}-1)}\leq\frac{L\|x_0-x_\star\|^2}{(N+1)^2},\label{eq:upper_bound_optista_2}
\end{align}
where the last inequality of \eqref{eq:upper_bound_optista_2} is from \cite[Equation (4.11)]{daspremont2021acceleration}. 

In the following subsections, we show the proof of equality $(*)$ of equation \eqref{eq:weighted_ineq_to_proof} and establish positive semidefiniteness of $Z$ in detail. Also, to make independent verification possible, we make our PEP code publicly available at the following GitHub link:
\[\href{https://github.com/UijeongJang/OptISTA-Verification}{\texttt{https://github.com/UijeongJang/OptISTA-Verification}}
\]

For the rest of the proof, we discuss step by step how to arrive
at $(*)$ of equation \eqref{eq:weighted_ineq_to_proof}.
\paragraph{Choice of $\lambda_{i,j}$ and $\tau_{i,j}$.}
First, we define the sequence $\{\tilde{\theta}_i\}_{i=0,\ldots,N-1}$ for further use.
\[
\tilde{\theta}_{i}=\begin{cases}
\theta_i, & \textup{if }i\in[0:N-2],\\
\frac{2\theta_{N-1}+\theta_N-1}{2}, & \textup{if }i=N-1.
\end{cases}\]
Now we discuss our choice of $\{(\lambda_{i,j})\}_{i,j\in\{\star,0,1,\ldots,N\}}$
and $\{(\tau_{i,j})\}_{i,j\in\{\star,1,\ldots,N\}}$ as follows:
\begin{equation}\label{eq:lambdatau}
\lambda=\begin{cases}
\lambda_{\star,i}=\frac{2\theta_{i}}{\theta_{N}^{2}} \textup{ if }i\in[0:N-1],\\
\lambda_{\star,N}=\frac{1}{\theta_{N}},\\
\lambda_{i,i+1}=\frac{2\theta_{i}^{2}}{\theta_{N}^{2}} \textup{ if }i\in[0:N-1],\\
\lambda_{i,j}=0\ \mathrm{otherwise},
\end{cases}\quad\tau=\begin{cases}
\tau_{\star,i}=\frac{2\tilde{\theta}_{i-1}}{\theta_{N}^{2}-1} \textup{ if }i\in[1:N]\\
\tau_{i,j}=\frac{2\tilde{\theta}_{j-1}}{\theta_{N}^{2}-2\theta_{i}^{2}+\theta_{i}}-\frac{2\tilde{\theta}_{j-1}}{\theta_{N}^{2}-2\theta_{i-1}^{2}+\theta_{i-1}}\textup{ if } 1\leq i<j\leq N,\\
\tau_{i+1,i}=\frac{\theta_{i}-1}{\theta_{N}^{2}-2\theta_{i}^{2}+\theta_{i}}\textup{ if }i\in[1:N-1],\\
\tau_{i,j}=0\ \mathrm{otherwise}.
\end{cases}
\end{equation}
 Simple but tedious algebraic manipulation shows that $\lambda_{i,j}\geq0$
for $i,j\in\{\star,0,1,\ldots,N\}$ and $\tau_{i,j}\geq0$ for $i,j\in\{\star,1,\ldots,N\}$.
{We show this manipulation in Section~\ref{s:f-nonnegative} .}

\paragraph{Choice of $Z$.}
The matrix $Z$ in \eqref{eq:weighted_ineq_to_proof} is defined as follows:

\begin{equation}
{Z}=(Z_{i,j})=\begin{bmatrix}A & B & C\\
{B}^{\intercal} & {D} & {E}\\
{C}^{\intercal} & {E}^{\intercal} & {F}
\end{bmatrix}\in\mathbb{S}^{2N+3},\label{eq:def_of_slack_mat_Z}
\end{equation}
where $A\in\mathbb{S}^{2}$, $B\in\mathbb{R}^{2\times(N+1)}$, $C\in\mathbb{R}^{2\times N}$,
$D\in\mathbb{S}^{N+1}$, $E\in\mathbb{R}^{(N+1)\times N}$, and $F\in\mathbb{S}^{N}$. The submatrices $A,$$B$, $C$, $D$, $E$, and $F$ are defined as:

\begin{equation}
\begin{aligned} & A=\begin{cases}
A_{1,1}=\frac{L}{2(\theta_{N}^{2}-1)},\\
A_{1,2}=A_{2,1}=0,\\
A_{2,2}=\frac{1}{2L},
\end{cases}\quad B=\begin{cases}
B_{1,i}=-\frac{\theta_{i-1}}{\theta_{N}^{2}}\textup{ if }i\in[1:N],\\
B_{1,N+1}=-\frac{1}{2\theta_{N}},\\
B_{2,i}=-\frac{\theta_{i-1}}{L\theta_{N}^{2}}\textup{ if }i\in[1:N],\\
B_{2,N+1}=-\frac{1}{2L\theta_{N}},
\end{cases},\\
 & C=\begin{cases}
C_{1,i}=-\frac{\tilde{\theta}_{i-1}}{\theta_{N}^{2}-1}\textup{ if }i\in[1:N],\\
C_{2,i}=0\ \textup{ if }i\in[1:N],
\end{cases}\quad D=\begin{cases}
D_{i,j}=\frac{2\theta_{i-1}\theta_{j-1}}{L\theta_{N}^{2}}\textup{ if }i,j\in[1:N],\\
D_{N+1,i}=D_{i,N+1}=\frac{\theta_{i-1}}{L\theta_{N}}\textup{ if }i\in[1:N],\\
D_{N+1,N+1}=\frac{1}{2L},
\end{cases},\\
 & E=\begin{cases}
E_{i,j}=\frac{2\theta_{i-1}\tilde{\theta}_{j-1}}{L\theta_{N}^{2}}\ \textup{ if }i,j\in[1:N],\\
E_{N+1,i}=\frac{\tilde{\theta}_{i-1}}{L\theta_{N}}\textup{ if }i\in[1:N],
\end{cases}\quad F=\begin{cases}
F_{i,i}=\frac{2\theta_{i-1}\cdot2\tilde{\theta}_{i-1}}{L\theta_{N}^{2}}\textup{ if }i\in[1:N],\\
F_{i,i+1}=F_{i+1,i}=\frac{2\theta_{i-1}\tilde{\theta}_{i}-\theta_{i-1}^{2}}{L\theta_{N}^{2}}\textup{ if }i\in[1:N-1],\\
F_{i,j}=F_{j,i}=E_{i,j}=\frac{2\theta_{i-1}\tilde{\theta}_{j-1}}{L\theta_{N}^{2}}\textup{ if }i\in[1:N-2],|i-j|\geq2.
\end{cases}
\end{aligned}
\label{eq:choice_of_Z-1}
\end{equation}

We can show that $Z$ is positive-semidefinite through algebraic manipulations, {we
show this calculation in Section~\ref{s:f-positivesemi}. }
\paragraph{Characterization of LHS of \eqref{eq:weighted_ineq_to_proof}.}

Note that , the left hand side of \eqref{eq:weighted_ineq_to_proof} can be
expanded as follows:

\noindent 
\begin{align}
 & \frac{L}{2(\theta_{N}^{2}-1)}\|x_{0}-x_{\star}\|^{2}+\sum_{i,j\in\{\star,0,\ldots,N\}}\lambda_{i,j}\left(f_{j}-f_{i}+\langle f_{j}',x_{i}-x_{j}\rangle+\frac{1}{2L}\|f_{i}'-f_{j}'\|^{2}\right)\nonumber \\
 & +\sum_{i,j\in\{\star,1,\ldots,N\}}\tau_{i,j}\left(h_{j}-h_{i}+\langle h_{j}',y_{i}-y_{j}\rangle\right)-\left(f_{N}+h_{N}-f_{\star}-h_{\star}\right)\nonumber \\
 & =\underbrace{\sum_{i,j\in\{\star,0,\ldots,N\}}\lambda_{i,j}\left(f_{j}-f_{i}\right)}_{\texttt{term\_1}}+\underbrace{\sum_{i,j\in\{\star,1,\ldots,N\}}\tau_{i,j}\left(h_{j}-h_{i}\right)}_{\texttt{term\_2}}\nonumber \\
 & \quad+\underbrace{\sum_{i,j\in\{\star,0,\ldots,N\}}\lambda_{i,j}\left(\langle f_{j}',x_{i}-x_{j}\rangle+\frac{1}{2L}\|f_{i}'-f_{j}'\|^{2}\right)}_{\texttt{term\_3}}+\underbrace{\sum_{i,j\in\{\star,1,\ldots,N\}}\tau_{i,j}\langle h_{j}',y_{i}-y_{j}\rangle}_{\texttt{term\_4}}\nonumber \\
 & \quad-\left(f_{N}+h_{N}-f_{\star}-h_{\star}\right)+\frac{L}{2(\theta_{N}^{2}-1)}\|x_{0}-x_{\star}\|^{2}.\label{eq:expansion_lagrangian}
\end{align}
Next, through simple but tedious algebra, we simplify $\texttt{term\_}i$
for $i=1,\ldots,4$. We can show that $\texttt{term\_1}$ and $\texttt{term\_2}$
can be simplified as follows for our selection of
$\lambda_{i,j}$ and $\tau_{i,j}$:
\begin{align}
 & \texttt{term\_1}\equiv\sum_{i,j\in\{\star,0,\ldots,N\}}\lambda_{i,j}\left(f_{j}-f_{i}\right)=f_{N}-f_{\star},\label{eq:simplify_term_1}\\
 & \texttt{term\_2}\equiv\sum_{i,j\in\{\star,1,\ldots,N\}}\tau_{i,j}\left(h_{j}-h_{i}\right)=h_{N}-h_{\star},\label{eq:simplify_term_2}
\end{align}
which {we show in Section~\ref{s:f-functionval} }. Then, we can show that for our selection of $\lambda_{i,j}$,
$\texttt{term\_3}$ can be expanded as:

\noindent 
\begin{align}
  &\sum_{i,j\in\{\star,0,\ldots,N\}}\lambda_{i,j}\left(\langle f_{j}',x_{i}-x_{j}\rangle+\frac{1}{2L}\|f_{i}'-f_{j}'\|^{2}\right)=  -\sum_{i=0}^{N}\lambda_{\star,i}\langle f_{i}',x_{0}\rangle+\frac{1}{2L}\|f_{\star}'\|^{2}-\sum_{i=0}^{N}\frac{\lambda_{\star,i}}{L}\langle f_{\star}',f_{i}'\rangle\nonumber \\
 & \quad+\sum_{i=0}^{N-1}\frac{2\theta_{i}^{2}}{L\theta_{N}^{2}}\|f_{i}'\|^{2}+\frac{1}{2L}\|f_{N}'\|^{2}+\quad\sum_{0\leq j\leq i-1<N-1}\frac{4\theta_{i}\theta_{j}}{L\theta_{N}^{2}}\langle f_{i}',f_{j}'\rangle+\sum_{j=0}^{N-1}\frac{2\theta_{j}}{L\theta_{N}}\langle f_{N}',f_{j}'\rangle\nonumber \\
 & \quad+\sum_{1\leq j<i<N}\frac{4\theta_{i}\theta_{j-1}}{L\theta_{N}^{2}}\langle f_{i}',h_{j}'\rangle\ +\sum_{j=1}^{N-1}\frac{2\theta_{j-1}}{L\theta_{N}}\langle f_{N}',h_{j}'\rangle+\quad\sum_{i=1}^{N}\left(\frac{\lambda_{i-1,i}\alpha_{i,i-1}}{L}+\frac{\lambda_{\star,i}\alpha_{i,i-1}}{L}\right)\langle f_{i}',h_{i}'\rangle,\label{eq:cocoerc-expansion}
\end{align}
 and for our selection of $\tau_{i,j}$, $\texttt{term\_4}$ can be
expanded as:
\begin{align}
 & \sum_{i,j\in\{\star,1,\ldots,N\}}\tau_{i,j}\langle h_{j}',y_{i}-y_{j}\rangle\nonumber=-\sum_{i=1}^{N}\tau_{\star,i}\langle h_{i}',x_{0}\rangle+\sum_{0\leq j<i\leq N}\frac{4\tilde{\theta}_{i-1}\theta_{j}}{L\theta_{N}^{2}}\langle h_{i}',f_{j}'\rangle-\sum_{i=1}^{N-1}\frac{\tau_{i+1,i}\gamma_{i}}{L}\langle h_{i}',f_{i}'\rangle\nonumber \\
 & \quad+\sum_{1\leq j<i-1\leq N-1}\frac{4\tilde{\theta}_{i-1}\theta_{j-1}}{L\theta_{N}^{2}}\langle h_{i}',h_{j}'\rangle+\sum_{i=1}^{N}\frac{4\tilde{\theta}_{i-1}\theta_{i-1}}{L\theta_{N}^{2}}\|h_{i}'\|^{2}+\sum_{i=1}^{N-1}\frac{4\theta_{i-1}\tilde{\theta}_{i}-2\theta_{i-1}^{2}}{L\theta_{N}^{2}}\langle h_{i}',h_{i+1}'\rangle.\label{eq:cvx-expansion}
\end{align}

\noindent Both \eqref{eq:cocoerc-expansion} and \eqref{eq:cvx-expansion}
can be deduced through elementary but tedious symbolic calculation. We
show this in Section~\ref{s:f-expansion}.

\paragraph{Proving $(*)$ of equality \eqref{eq:weighted_ineq_to_proof}.}

\noindent Now putting the values obtained from \eqref{eq:simplify_term_1},
\eqref{eq:simplify_term_2}, \eqref{eq:cocoerc-expansion}, and \eqref{eq:cvx-expansion}
in \eqref{eq:expansion_lagrangian}, we can show that the LHS of \eqref{eq:weighted_ineq_to_proof}
would be equal to $\tr GZ$ where $G$ and $Z$ are defined in \eqref{eq:def_of_G}
and \eqref{eq:def_of_slack_mat_Z} (calculation shown in
Section~\ref{s:f-proofofstar}). This proves \eqref{eq:weighted_ineq_to_proof}
thus giving us the desired rate.

\subsection{Nonnegativity of $\lambda_{i,j}$ and $\tau_{i,j}$}\label{s:f-nonnegative}
The following lemma proves the nonnegativity of $\lambda_{i,j}$ and $\tau_{i,j}$.
\begin{lem}\label{variablesarepositive}
    Recall $\{(\lambda_{i,j})\}_{i,j\in \{\star,0,1,\ldots,N\}}$ and $\{(\tau_{i,j})\}_{i,j\in \{\star,1,\ldots,N\}}$ defined as in \eqref{eq:lambdatau}.
Then $\lambda_{i,j}\geq0$ for $i,j\in \{\star,0,1,\ldots,N\}$ and $\tau_{i,j}\geq0$ for $i,j\in \{\star,1,\ldots,N\}$.
\end{lem}
\begin{proof}
It is easy to check that 
\[
1=\theta_0<\theta_1<\cdots<\theta_{N-1} \ \mathrm{and} \ 2\theta_{N-1}^2<\theta_N^2.
\]
Then $\lambda_{i,j}\geq0$ for $i,j\in \{\star,0,1,\ldots,N\}$ is straightforward from the definition. It is also immediate to check that
\[
 \tau_{\star,i}=\frac{2\tilde{\theta}_{i-1}}{\theta_N^2-1}\geq 0 \quad i\in[1:N],
\]
\[
\tau_{i+1,i}=\frac{\theta_i-1}{\theta_N^2-2\theta_i^2+\theta_i}\geq0 \quad i\in[1:N-1].
\]
For $\tau_{i,j}$ when $1\leq i<j\leq N$,
\[
(\theta_N^2-2\theta_{i}^2+\theta_i)-(\theta_N^2-2\theta_{i-1}^2+\theta_{i-1})=2\theta_{i-1}^2-2\theta_{i}^2-\theta_{i-1}+\theta_i=(\theta_{i-1}-\theta_i)(2\theta_{i-1}+2\theta_i-1)<0,
\]
so $\{(\theta_N^2-2\theta_{i}^2+\theta_i)\}_{i\in[0:N-1]}$ is strictly decreasing and positive. Therefore,
\[
\tau_{i,j}={2\tilde{\theta}_{j-1}}\left(\frac{1}{\theta_N^2-2\theta_{i}^2+\theta_i}-\frac{1}{\theta_N^2-2\theta_{i-1}^2+\theta_{i-1}}\right)\geq0.
\]
\end{proof}

\subsection{Positive semidefiniteness of $Z$}\label{s:f-positivesemi}
The following two lemmas prove the positive semidefiniteness of $Z$.
\begin{lem}\label{lem:interlacing}
    (Cauchy's interlacing theorem) Let $A\in\mathbb{S}^n$ be a symmetric matrix and $v\in\rl^{n}$ be a nonzero vector. Let $\mu_i(\cdot)$ denote the $i$-th largest eigenvalue of a matrix. Then, 
    \[\mu_n(A)\leq\mu_n(A+vv^\intercal)\leq\mu_{n-1}(A)\leq\mu_{n-1}(A+vv^\intercal)\leq\cdots\leq\mu_1(A)\leq\mu_1(A+vv^\intercal).\]
\end{lem}
\begin{proof}
    We refer the readers to \cite[Corollary 4.3.9]{horn2012matrix}.
\end{proof}
\begin{lem}\label{lem:Zispositive}
Recall the definition of $Z$ in \eqref{eq:def_of_slack_mat_Z}
where the submatrices are defined as in \eqref{eq:choice_of_Z-1}.
Then $Z$ is positive semidefinite, i.e.,$Z\in\mathbb{S}_{+}^{2N+3}$.
\end{lem}
\begin{proof}
It is clear from the definition that $A$ is positive definite and
 \[
 A^{-1}=\begin{bmatrix}
     \frac{2(\theta_N^2-1)}{L} & 0 \\
     0 & 2L \end{bmatrix}.
 \]
So we equivalently prove the positive semidefiniteness of the Schur complement:
\[
\begin{bmatrix}
    D & E\\
    E^\intercal& F
\end{bmatrix}-\begin{bmatrix}
    B^\intercal \\
    C^\intercal
\end{bmatrix}A^{-1}\begin{bmatrix}
    B & C
\end{bmatrix} \in \mathbb{S}^{2N+1}.
\]
Let
\[
\begin{bmatrix}
    B^\intercal \\
    C^\intercal
\end{bmatrix}A^{-1}\begin{bmatrix}
    B & C
\end{bmatrix}=\begin{bmatrix}
    \bar{D} & \bar{E}\\
    \bar{E}^\intercal & \bar{F} 
\end{bmatrix},
\]
where $\bar{D}\in\mathbb{S}^{N+1}$, $\bar{E}\in\mathbb{R}^{(N+1)\times N}$, and $\bar{F}\in\mathbb{S}^{N}$. Then,
\[
\bar{D}_{i,j}=\frac{2(\theta_N^2-1)}{L}B_{1,i}B_{1,j}+2LB_{2,i}B_{2,j},
\]
\[
\bar{E}_{i,j}=\frac{2(\theta_N^2-1)}{L}B_{1,i}C_{1,j}+2LB_{2,i}C_{2,j},
\]
\[
\bar{F}_{i,j}=\frac{2(\theta_N^2-1)}{L}C_{1,i}C_{1,j}+2LC_{2,i}C_{2,j}.
\]
For $i,j\in[1:N]$,
\begin{align}
 \frac{2(\theta_N^2-1)}{L}B_{1,i}B_{1,j}+2LB_{2,i}B_{2,j} =\frac{2(\theta_N^2-1)}{L}\frac{\theta_{i-1}\theta_{j-1}}{\theta_N^4}+2L\frac{\theta_{i-1}\theta_{j-1}}{L^2\theta_N^4}=\frac{2\theta_N^2\theta_{i-1}\theta_{j-1}}{L\theta_N^4}=\frac{2\theta_{i-1}\theta_{j-1}}{L\theta_N^2}=D_{i,j}. 
\label{eq:DistildeD-1}\end{align}
For $i=N+1$ and $j\in[1:N]$,
\begin{align}
 \frac{2(\theta_N^2-1)}{L}B_{1,N+1}B_{1,j}+2LB_{2,N+1}B_{2,j} =\frac{2(\theta_N^2-1)}{L}\frac{\theta_{j-1}}{2\theta_N^3}+2L\frac{\theta_{j-1}}{2L^2\theta_N^3}=\frac{\theta_N^2\theta_{j-1}}{L\theta_N^3}=\frac{\theta_{j-1}}{L\theta_N}=D_{N+1,j}.      
\label{eq:DistildeD-2}\end{align}
For $i=j=N+1$,
\begin{align}
  \frac{2(\theta_N^2-1)}{L}B_{1,N+1}B_{1,N+1}+2LB_{2,N+1}B_{2,N+1} =\frac{2(\theta_N^2-1)}{L}\frac{1}{4\theta_N^2}+2L\frac{1}{4L^2\theta_N^2}=\frac{\theta_N^2}{2L\theta_N^2}=\frac{1}{2L}=D_{N+1,N+1}.    
\label{eq:DistildeD-3}\end{align}
So we have $\bar{D}=D$ by \eqref{eq:DistildeD-1}, \eqref{eq:DistildeD-2}, and \eqref{eq:DistildeD-3}. For $i,j\in[1:N]$,
 \begin{align}
  \frac{2(\theta_N^2-1)}{L}B_{1,i}C_{1,j}+2LB_{2,i}C_{2,j}=\frac{2(\theta_N^2-1)}{L}B_{1,i}C_{1,j}=\frac{2(\theta_N^2-1)}{L}\frac{\theta_{i-1}\tilde{\theta}_{j-1}}{\theta_N^2(\theta_N^2-1)}=\frac{2\theta_{i-1}\tilde{\theta}_{j-1}}{L\theta_N^2}=E_{i,j}.    
 \label{eq:EistildeE-1}\end{align}
 For $i=N+1$, $j\in[1:N]$,
\begin{align}
  \frac{2(\theta_N^2-1)}{L}B_{1,N+1}C_{1,j}+2LB_{2,N+1}C_{2,j}=\frac{2(\theta_N^2-1)}{L}B_{1,N+1}C_{1,j}=\frac{2(\theta_N^2-1)}{L}\frac{\tilde{\theta}_{j-1}}{2\theta_N(\theta_N^2-1)}=\frac{\tilde{\theta}_{j-1}}{L\theta_N}=E_{N+1,j}.    
 \label{eq:EistildeE-2}\end{align}
 So we have $\bar{E}=E$ by \eqref{eq:EistildeE-1} and \eqref{eq:EistildeE-2}.
Hence,
\[
\begin{bmatrix}
    D & E\\
    E^\intercal& F
\end{bmatrix}-\begin{bmatrix}
    \bar{D} & \bar{E}\\
    \bar{E}^\intercal& \bar{F}
\end{bmatrix}=
\begin{bmatrix}
    \mathbf{0} & \mathbf{0}\\
    \mathbf{0}& F-\bar{F}
\end{bmatrix}.
\]
and the problem reduces to proving the positive semidefiniteness of $S\triangleq F-\bar{F}\in\mathbb{S}^{N}$. 
For $i,j\in[1:N]$,
 \begin{align}
  \bar{F}_{i,j} =\frac{2(\theta_N^2-1)}{L}C_{1,i}C_{1,j}+2LC_{2,i}C_{2,j}=\frac{2(\theta_N^2-1)}{L}C_{1,i}C_{1,j}=\frac{2(\theta_N^2-1)}{L}\frac{\tilde{\theta}_{i-1}\tilde{\theta}_{j-1}}{(\theta_N^2-1)^2}=\frac{2\tilde{\theta}_{i-1}\tilde{\theta}_{j-1}}{L(\theta_N^2-1)}.  
\label{eq:Fmatrixcalculation1} \end{align}
 Recall that
\begin{equation}
\begin{aligned}
F_{i,i}&=\frac{4\theta_{i-1}\tilde{\theta}_{i-1}}{L\theta_N^2}\textup{ if } i\in[1:N],\\
F_{i,i+1}&=F_{i+1,i}=\frac{2\theta_{i-1}\tilde{\theta}_{i}-\theta_{i-1}^2}{L\theta_N^2}=\frac{\tilde{\theta}_{i-1}(2\tilde{\theta}_{i}-\tilde{\theta}_{i-1})}{L\theta_N^2}\textup{ if }i\in[1:N-1],\\
F_{i,j}&=F_{j,i}=\frac{2\theta_{i-1}\tilde{\theta}_{j-1}}{L\theta_N^2}=\frac{2\tilde{\theta}_{i-1}\tilde{\theta}_{j-1}}{L\theta_N^2}\textup{ if } i\in[1:N-2],\ j\geq i+2.
\end{aligned}
\label{eq:Fmatrixcalculation2}
\end{equation}
Then from \eqref{eq:Fmatrixcalculation1} and \eqref{eq:Fmatrixcalculation2}, we get
    \[S=V-W\]
 where 
  \[V=\begin{bmatrix}
                \frac{2\theta_0^2}{L\theta_N^2} &-\frac{\theta_0^2}{L\theta_N^2} & & & &  \\
                -\frac{\theta_0^2}{L\theta_N^2}& \frac{2\theta_1^2}{L\theta_N^2} & -\frac{\theta_1^2}{L\theta_N^2} & & & \\
                & -\frac{\theta_1^2}{L\theta_N^2} & \frac{2\theta_2^2}{L\theta_N^2} & -\frac{\theta_2^2}{L\theta_N^2} & & \\
                & & \ddots & \ddots & \ddots & \\
                & & & -\frac{\theta_{N-3}^2}{L\theta_N^2} & \frac{2\theta_{N-2}^2}{L\theta_N^2} & -\frac{\theta_{N-2}^2}{L\theta_N^2} \\
                & & & & -\frac{\theta_{N-2}^2}{L\theta_N^2} & \frac{2\tilde{\theta}_{N-1}(2\theta_{N-1}-\tilde{\theta}_{N-1})}{L\theta_N^2} 
            \end{bmatrix}\]
    
and 
    \[W_{ij}=\left(\frac{2}{L(\theta_N^2-1)}-\frac{2}{L\theta_N^2}\right)\tilde{\theta}_{i-1}\tilde{\theta}_{j-1}=\frac{2}{L(\theta_N^2-1)\theta_N^2}\tilde{\theta}_{i-1}\tilde{\theta}_{j-1}.
    \]
            For notational simplicity, let $\varphi_i\triangleq\frac{\theta_i^2}{L\theta_N^2}>0$ for $i \in [0:N-2]$ and $\varphi_{N-1}\triangleq\frac{2\tilde{\theta}_{N-1}(2\theta_{N-1}-\tilde{\theta}_{N-1})}{L\theta_N^2} $.
        Then for any nonzero $\xi=(\xi_0,\ldots,\xi_{N-1})^\intercal\in\mathbb{R}^N$,
        \begin{align*}
            \xi^\intercal V\xi&=\sum_{i=0}^{N-2}2\varphi_i \xi_i^2 -2\sum_{i=0}^{N-2}\varphi_i\xi_i\xi_{i+1}+\varphi_{N-1}\xi_{N-1}^2\\
            &=\sum_{i=0}^{N-2}\varphi_i (\xi_{i+1}-\xi_i)^2 +\varphi_0\xi_0^2+\sum_{i=1}^{N-2}(\varphi_i-\varphi_{i-1})\xi_i^2+(\varphi_{N-1}-\varphi_{N-2})\xi_{N-1}^2.
        \end{align*}
        We now show that $\varphi_i$ are strictly increasing and positive, which implies $\xi^\intercal V\xi>0$ and thus $V$ is positive definite. 
        For $i\in[1:N-2]$,
        \begin{align*}
        &\varphi_{i}-\varphi_{i-1}=\frac{\theta_i^2-\theta_{i-1}^2}{L\theta_N^2}=\frac{\theta_{i}}{L\theta_N^2}=\frac{\tilde{\theta}_{i}}{L\theta_N^2}>0.
        & {\color{gray}\hfill\rhd\,\textup{using \eqref{eq:recursivetheta1} of Lemma~\ref{lem:propertyoftheta} }}
         \end{align*}
        For $i=N-1$,
        \begin{align*}
           &\varphi_{N-1}-\varphi_{N-2}=\frac{2\tilde{\theta}_{N-1}(2\theta_{N-1}-\tilde{\theta}_{N-1})}{L\theta_N^2}-\frac{\theta_{N-2}^2}{L\theta_N^2} = \frac{1}{L\theta_N^2}\left(4\theta_{N-1}\tilde{\theta}_{N-1}-2\tilde{\theta}_{N-1}^2-\theta_{N-2}^2 \right)\\
            &= \frac{1}{2L\theta_N^2}\left(8\theta_{N-1}\tilde{\theta}_{N-1}-4\tilde{\theta}_{N-1}^2-2\theta_{N-2}^2 \right)\\
            &= \frac{1}{2L\theta_N^2}\left(4\theta_{N-1}(2\theta_{N-1}+\theta_N-1)-(2\theta_{N-1}+\theta_N-1)^2-2\theta_{N-2}^2 \right){\color{gray}\qquad \rhd\ \textup{substituting } \, \tilde{\theta}_{N-1}=\frac{2\theta_{N-1}+\theta_N-1}{2}}\\
             &= \frac{1}{2L\theta_N^2}\left(4\theta_{N-1}^2+2\theta_N-\theta_N^2-1-2\theta_{N-2}^2\right)\\
             &= \frac{1}{2L\theta_N^2}\left(2\theta_{N-1}^2+2\theta_{N-1}+2\theta_N-\theta_N^2-1\right){\color{gray} \qquad\rhd\,\textup{using \eqref{eq:recursivetheta1} of Lemma~\ref{lem:propertyoftheta} }}\\
              &= \frac{1}{2L\theta_N^2}\left(2\theta_{N-1}+\theta_N-1\right)=\frac{\tilde{\theta}_{N-1}}{L\theta_N^2}>0. {\color{gray}\qquad\rhd\,\textup{using \eqref{eq:recursivetheta2} of Lemma~\ref{lem:propertyoftheta} }}
        \end{align*}
        Note that $W$ 
        is positive semidefinite rank 1 matrix, 
        \[
        W=ww^\intercal \textup{ where } 
    w=\sqrt{\frac{2}{L(\theta_N^2-1)\theta_N^2}}(\tilde{\theta}_0,\ldots,\tilde{\theta}_{N-1})^\intercal,
    \]
        so we apply Lemma~\ref{lem:interlacing} with ${S}$ and $V={S}+W$.
        Then, 
        \begin{align}
 \mu_N(S)\leq\mu_N(V)\leq\mu_{N-1}(S)\leq\mu_{N-1}(V)\leq\cdots\leq\mu_1(S)\leq\mu_1(V).     
        \label{eq:interlacing}\end{align}
        From positive
        definiteness of $V$, we have $\mu_N(V)>0$. For $\mathbf{1}=(1,\ldots,1)^\intercal\in \mathbb{R}^N$, note that 
        \[
        V\mathbf{1}=\begin{bmatrix}
            \varphi_0\\
            \varphi_1-\varphi_0\\
            \vdots\\
            \varphi_{N-1}-\varphi_{N-2}
        \end{bmatrix}=\frac{1}{L\theta_N^2}\begin{bmatrix}
            \tilde{\theta}_0\\
            \tilde{\theta}_1\\
            \vdots\\
            \tilde{\theta}_{N-1}
        \end{bmatrix}, \quad
        W\mathbf{1}=\frac{2}{L(\theta_N^2-1)\theta_N^2}\cdot\left(\sum_{i=0}^{N-1}\tilde{\theta}_i\right)\begin{bmatrix}
            \tilde{\theta}_0\\
            \tilde{\theta}_1\\
            \vdots\\
            \tilde{\theta}_{N-1}
        \end{bmatrix}=\frac{1}{L\theta_N^2}\begin{bmatrix}
            \tilde{\theta}_0\\
            \tilde{\theta}_1\\
            \vdots\\
            \tilde{\theta}_{N-1}
        \end{bmatrix},
        \]
        where we used $\sum_{i=0}^{N-1}\tilde{\theta}_i=\frac{\theta_N^2-1}{2}$ by \eqref{eq:sumoftildetheta} of Lemma~\ref{lem:propertyoftheta}. Therefore $S\mathbf{1} =\mathbf{0}$, so $0$ is an eigenvalue of ${S}$. By our
        interlacing property \eqref{eq:interlacing}, we must have $\mu_N(S)=0$ and this proves the positive semidefiniteness of ${S}$.
\end{proof}
\subsection{Proof of \eqref{eq:simplify_term_1} and \eqref{eq:simplify_term_2}}\label{s:f-functionval}
The next lemma proves \eqref{eq:simplify_term_1}.
\begin{lem}\label{lem:functionvalueforf}
The following equality holds.
\[
\sum_{i,j\in\{\star,0,\ldots,N\}}\lambda_{i,j}\left(f_j-f_i\right)=\sum_{i=0}^N\lambda_{\star,i}\left(f_i-f_\star\right)+\sum_{i=0}^{N-1}\lambda_{i,i+1}\left(f_{i+1}-f_i\right)=f_N-f_\star
\]   
\end{lem}
\begin{proof}
It is equivalent to show \eqref{eq:functionvalueforf1} and \eqref{eq:functionvalueforf2}:
\begin{align}
   \sum_{i=0}^N\lambda_{\star,i}=1 , \ \lambda_{\star,0}-\lambda_{0,1}=0, \ \lambda_{\star,N}+\lambda_{N-1,N}=1 \label{eq:functionvalueforf1},
\end{align}
\begin{align}
    \lambda_{\star,i}+\lambda_{i-1,i}-\lambda_{i,i+1}=0 \textup{ for } i\in[1:N-1] \label{eq:functionvalueforf2}.
\end{align}
For \eqref{eq:functionvalueforf1},
 \begin{align*}
 \sum_{i=0}^{N}\lambda_{\star,i}=\frac{1}{\theta_N^2}\sum_{k=0}^{N-1}2\theta_k+\frac{1}{\theta_N}=\frac{1}{\theta_N^2}2\theta_{N-1}^2+\frac{\theta_N}{\theta_N^2}=1, {\color{gray}\qquad\rhd\,\textup{using \eqref{eq:recursivetheta2} and \eqref{eq:sumoftheta} of Lemma~\ref{lem:propertyoftheta} }}
\end{align*}
\[
\lambda_{\star,0}-\lambda_{0,1}=\frac{2\theta_0}{\theta_N^2}-\frac{2\theta_0^2}{\theta_N^2}=0,
\]
\begin{align*}
\lambda_{\star,N}+\lambda_{N-1,N}=\frac{1}{\theta_N}+\frac{2\theta_{N-1}^2}{\theta_N^2}=\frac{\theta_N+2\theta_{N-1}^2}{\theta_N^2}=1. {\color{gray}\qquad\rhd\,\textup{using \eqref{eq:recursivetheta2} of Lemma~\ref{lem:propertyoftheta} }}
\end{align*}
For \eqref{eq:functionvalueforf2} when $i\in[1:N-1]$,
\begin{align*}
\lambda_{\star,i}+\lambda_{i-1,i}-\lambda_{i,i+1}=\frac{2\theta_i}{\theta_N^2}+\frac{2\theta_{i-1}^2}{\theta_N^2}-\frac{2\theta_{i}^2}{\theta_N^2}=\frac{2}{\theta_N^2}\left(-\theta_i^2+\theta_i+\theta_{i-1}^2\right)=0. {\color{gray}\qquad\rhd\,\textup{using \eqref{eq:recursivetheta1} of Lemma~\ref{lem:propertyoftheta} }}
\end{align*}
\end{proof}
Recall that $\{\tau_{i,j}\}_{i,j\in\{*,1,\dots,N\}}$ is exactly the same as 
the one defined in Section C. Thus, Lemma~\ref{lem:functionvalueforh} proves \eqref{eq:simplify_term_2}.
\subsection{Proof of \eqref{eq:cocoerc-expansion} and \eqref{eq:cvx-expansion}}\label{s:f-expansion}
The next lemma proves \eqref{eq:cocoerc-expansion}.
\begin{lem}\label{lem:cocoerc-expansion}
The following equality \eqref{eq:cocoerc-expansion} holds.
\begin{align}
    &\sum_{i,j\in\{\star,0,\ldots,N\}}\lambda_{i,j}\left(\langle f_j',x_i-x_j\rangle+\frac{1}{2L}\| f_i'-f_j' \|^2\right)=
-\sum_{i=0}^N\lambda_{\star,i}\langle f_i', x_0\rangle+\frac{1}{2L}\|f_\star'\|^2-\sum_{i=0}^N\frac{\lambda_{\star,i}}{L}\langle f_\star', f_i' \rangle \nonumber \\
&+\sum_{i=0}^{N-1}\frac{2\theta_i^2}{L\theta_N^2}\|f_i'\|^2+\frac{1}{2L}\|f_N'\|^2+\sum_{0\leq j\leq i-1< N-1}\frac{4\theta_i\theta_j}{L\theta_N^2}\langle f_i', f_j'\rangle+\sum_{j=0}^{N-1}\frac{2\theta_j}{L\theta_N}\langle f_N', f_j'\rangle\nonumber \\
&+\sum_{1\leq j< i< N}\frac{4\theta_i\theta_{j-1}}{L\theta_N^2}\langle f_i', h_j'\rangle\ +\sum_{j=1}^{N-1}\frac{2\theta_{j-1}}{L\theta_N}\langle f_N', h_j'\rangle +\sum_{i=1}^N\frac{\lambda_{i-1,i}\alpha_{i,i-1}+\lambda_{\star,i}\alpha_{i,i-1}}{L}\langle f_i', h_i'\rangle. 
\tag{\ref{eq:cocoerc-expansion}}\end{align}
\end{lem}
\begin{proof}
Left-hand side equals
    \begin{align*}
 \underbrace{\sum_{i=0}^N\lambda_{\star,i}\left(\langle f_i',-x_i\rangle+\frac{1}{2L}\| f_\star'-f_i' \|^2\right)}_{\texttt{term\_a}}+\underbrace{\sum_{i=0}^{N-1}\lambda_{i,i+1}\left(\langle f_{i+1}',x_i-x_{i+1}\rangle+\frac{1}{2L}\| f_{i}'-f_{i+1}' \|^2\right)}_{\texttt{term\_b}}.
    \end{align*}
\texttt{term\_a} can be rearranged as
\begin{align}
   &\texttt{term\_a}=\sum_{i=0}^N\lambda_{\star,i}\left(\left\langle f_i',-x_0+\sum_{j=0}^{i-1}\frac{\alpha_{i,j}}{L}f_j'+\sum_{j=0}^{i-1}\frac{\alpha_{i,j}}{L}h_{j+1}'\right\rangle+\frac{1}{2L}\| f_\star'-f_i' \|^2\right)\qquad{\color{gray}{\rhd\,\textup{substituting \eqref{eq:equivalentFSFOM}}}}\nonumber \\
   &=-\sum_{i=0}^N\lambda_{\star,i}\langle f_i', x_0\rangle+\sum_{i=0}^N\lambda_{\star,i}\left\langle f_i',\sum_{j=0}^{i-1}\frac{\alpha_{i,j}}{L}f_j'+\sum_{j=0}^{i-1}\frac{\alpha_{i,j}}{L}h_{j+1}'\right\rangle\nonumber\\
   &+\sum_{i=0}^N\lambda_{\star,i}\left(\frac{1}{2L}\| f_\star'\|^2-\frac{1}{L}\langle f_\star', f_i' \rangle+\frac{1}{2L}\|f_i' \|^2\right)\qquad{\color{gray}{\rhd\,\textup{expanding }\frac{1}{2L}\| f_\star'-f_i' \|^2}}\nonumber\\
   &=-\sum_{i=0}^N\lambda_{\star,i}\langle f_i', x_0\rangle+\sum_{i=0}^N\frac{\lambda_{\star,i}}{2L}\|f_\star'\|^2-\sum_{i=0}^N\frac{\lambda_{\star,i}}{L}\langle f_\star', f_i \rangle'+\sum_{i=0}^N\frac{\lambda_{\star,i}}{2L}\|f_i'\|^2+\sum_{i=1}^N\lambda_{\star,i}\left\langle f_i',\sum_{j=0}^{i-1}\frac{\alpha_{i,j}}{L}f_j'+\sum_{j=0}^{i-1}\frac{\alpha_{i,j}}{L}h_{j+1}'\right\rangle\nonumber\\ 
   &=-\sum_{i=0}^N\lambda_{\star,i}\langle f_i', x_0\rangle+\frac{1}{2L}\|f_\star'\|^2-\sum_{i=0}^N\frac{\lambda_{\star,i}}{L}\langle f_\star', f_i \rangle'+\sum_{i=0}^N\frac{\lambda_{\star,i}}{2L}\|f_i'\|^2 {\color{gray}\qquad\rhd\,\textup{using }\sum_{i=0}^N\lambda_{\star,i}=1}\nonumber\\
   &+\sum_{0\leq j<i\leq N}\frac{\lambda_{\star,i}\alpha_{i,j}}{L}\langle f_i',f_j'\rangle+\sum_{1\leq j\leq i\leq N}\frac{\lambda_{\star,i}\alpha_{i,j-1}}{L}\langle f_i',h_j'\rangle. \label{eq:matrixforf1}
\end{align}
Similarly for \texttt{term\_b}, we have
\begin{align}
    &\texttt{term\_b}=\sum_{i=0}^{N-1}\lambda_{i,i+1}\left(\left\langle f_{i+1}',\sum_{j=0}^{i}\frac{h_{i+1,j}}{L}f_j'+\sum_{j=0}^{i}\frac{h_{i+1,j}}{L}h_{j+1}'\right\rangle+\frac{1}{2L}\| f_{i}'-f_{i+1}' \|^2\right){\color{gray}{\qquad\rhd\,\textup{substituting \eqref{eq:equivalentFSFOM}}}}\nonumber\\
    &= \sum_{i=0}^{N-1}\lambda_{i,i+1}\left\langle f_{i+1}',\sum_{j=0}^{i}\frac{h_{i+1,j}}{L}f_j'+\sum_{j=0}^{i}\frac{h_{i+1,j}}{L}h_{j+1}'\right\rangle\nonumber\\
    &+\sum_{i=0}^{N-1}\lambda_{i,i+1}\left(\frac{1}{2L}\|f_i'\|^2-\frac{1}{L}\langle f_i',f_{i+1}'\rangle+\frac{1}{2L}\|f_{i+1}'\|^2\right){\color{gray}\qquad{\color{gray}\rhd\,\textup{expanding }\frac{1}{2L}\| f_\star'-f_i' \|^2}}\nonumber\\
     &= \sum_{i=0}^{N-1}\frac{\lambda_{i,i+1}}{2L}\|f_i'\|^2
     -\sum_{i=0}^{N-1}\frac{\lambda_{i,i+1}}{L}\langle f_i',f_{i+1}'\rangle
     +\sum_{i=0}^{N-1}\frac{\lambda_{i,i+1}}{2L}\|f_{i+1}'\|^2\nonumber\\
     &+\sum_{0\leq j< i\leq N}\frac{\lambda_{i-1,i}h_{i,j}}{L}\langle f_i', f_j'\rangle +\sum_{1\leq j\leq i\leq N}\frac{\lambda_{i-1,i}h_{i,j-1}}{L}\langle f_i', h_j'\rangle. \label{eq:matrixforf2}
\end{align}
Adding \eqref{eq:matrixforf1} and \eqref{eq:matrixforf2} gives
\begin{align*}
 &\texttt{term\_a}+\texttt{term\_b} \\
 &=-\sum_{i=0}^N\lambda_{\star,i}\langle f_i', x_0\rangle+\frac{1}{2L}\|f_\star'\|^2-\sum_{i=0}^N\frac{\lambda_{\star,i}}{L}\langle f_\star', f_i \rangle'+\underbrace{\sum_{i=0}^N\frac{\lambda_{\star,i}}{2L}\|f_i'\|^2+\sum_{i=0}^{N-1}\frac{\lambda_{i,i+1}}{2L}\|f_i'\|^2+\sum_{i=0}^{N-1}\frac{\lambda_{i,i+1}}{2L}\|f_{i+1}'\|^2}_{\texttt{term\_c}} \\
&+\underbrace{\sum_{0\leq j<i\leq N}\frac{\lambda_{\star,i}\alpha_{i,j}}{L}\langle f_i',f_j'\rangle+\sum_{0\leq j< i\leq N}\frac{\lambda_{i-1,i}h_{i,j}}{L}\langle f_i', f_j'\rangle -\sum_{i=0}^{N-1}\frac{\lambda_{i,i+1}}{L}\langle f_i',f_{i+1}'\rangle}_{\texttt{term\_d}}\\
 &+\underbrace{\sum_{1\leq j\leq i\leq N}\frac{\lambda_{\star,i}\alpha_{i,j-1}}{L}\langle f_i',h_j'\rangle+\sum_{1\leq j\leq i\leq N}\frac{\lambda_{i-1,i}h_{i,j-1}}{L}\langle f_i', h_j'\rangle}_{\texttt{term\_e}}.
 \end{align*}
Now we simplify and expand \texttt{term\_c}, \texttt{term\_d}, and \texttt{term\_e}. For \texttt{term\_c},
\begin{align*}
    \texttt{term\_c}=\frac{\lambda_{\star,0}+\lambda_{0,1}}{2L}\|f_0'\|^2+\sum_{i=1}^{N-1}\frac{\lambda_{\star,i}+\lambda_{i,i+1}+\lambda_{i-1,i}}{2L}\|f_i'\|^2+\frac{\lambda_{\star,N}+\lambda_{N-1,N}}{2L}\|f_N'\|^2.
\end{align*}
We can simplify the coefficients as
\begin{align*}
&\lambda_{\star,0}+\lambda_{0,1}=\frac{2}{\theta_N^2} +\frac{2}{\theta_N^2}=\frac{4}{\theta_N^2}=\frac{4\theta_0^2}{\theta_N^2},\\
&\lambda_{\star,i}+\lambda_{i,i+1}+\lambda_{i-1,i}=\frac{2\theta_{i}}{\theta_N^2}+\frac{2\theta_{i}^2}{\theta_N^2}+\frac{2\theta_{i-1}^2}{\theta_N^2}=\frac{4\theta_{i}^2}{\theta_N^2}=\frac{4\theta_{i}^2}{\theta_N^2},
{\color{gray}\qquad\rhd\,\textup{using \eqref{eq:recursivetheta1} of Lemma~\ref{lem:propertyoftheta}}}\\
&\lambda_{\star,N}+\lambda_{N-1,N}=\frac{1}{\theta_N}+\frac{2\theta_{N-1}^2}{\theta_N^2}=\frac{\theta_N+2\theta_{N-1}^2}{\theta_N^2}=1{\color{gray}\qquad\rhd\,\textup{using \eqref{eq:recursivetheta2} of Lemma~\ref{lem:propertyoftheta}}}
\end{align*}
\textup{ for } $i \in[1:N-1]$.
Therefore,
\begin{align}
\texttt{term\_c}=\sum_{i=0}^{N-1}\frac{2\theta_i^2}{L\theta_N^2}\|f_i'\|^2 +\frac{1}{2L}\|f_N'\|^2. 
\label{eq:termc}\end{align}
For \texttt{term\_d},
\begin{align}
    &\texttt{term\_d}=\sum_{0\leq j<i\leq N}\frac{\lambda_{\star,i}\alpha_{i,j}}{L}\langle f_i',f_j'\rangle+\sum_{0\leq j< i\leq N}\frac{\lambda_{i-1,i}h_{i,j}}{L}\langle f_i', f_j'\rangle -\sum_{i=0}^{N-1}\frac{\lambda_{i,i+1}}{L}\langle f_i',f_{i+1}'\rangle\nonumber\\
    &=\sum_{0\leq j<i-1\leq N}\frac{\lambda_{\star,i}\alpha_{i,j}}{L}\langle f_i',f_j'\rangle+\sum_{i=0}^{N-1}\frac{\lambda_{\star,i+1}\alpha_{i+1,i}}{L}\langle f_i',f_{i+1}'\rangle\qquad{\color{gray}\rhd\,\textup{separating }\langle f_i',f_{i+1}'\rangle\textup{ term}}\nonumber\\
    &+\sum_{0\leq j< i-1\leq N}\frac{\lambda_{i-1,i}h_{i,j}}{L}\langle f_i', f_j'\rangle+\sum_{i=0}^{N-1}\frac{\lambda_{i,i+1}h_{i+1,i}}{L}\langle f_i', f_{i+1}'\rangle-\sum_{i=0}^{N-1}\frac{\lambda_{i,i+1}}{L}\langle f_i',f_{i+1}'\rangle\qquad{\color{gray}\rhd\,\textup{separating }\langle f_i',f_{i+1}'\rangle \textup{ term}}\nonumber\\
    &=\sum_{0\leq j< i-1\leq N}\frac{\lambda_{\star,i}\alpha_{i,j}+\lambda_{i-1,i}h_{i,j}}{L}\langle f_i', f_j'\rangle+\sum_{i=0}^{N-1}\frac{\lambda_{\star,i+1}\alpha_{i+1,i}+\lambda_{i,i+1}h_{i+1,i}-\lambda_{i,i+1}}{L}\langle f_i', f_{i+1}'\rangle.\label{eq:termdcalculation1}
\end{align}
We now simplify the coefficients. When $i\in[0:N-2]$,
\begin{align*}
&\lambda_{\star,i+1}\alpha_{i+1,i}+\lambda_{i,i+1}h_{i+1,i}-\lambda_{i,i+1}
=\left(1+\frac{2\theta_{i}-1}{\theta_{i+1}}\right)\frac{2\theta_{i+1}+2\theta_{i}^2}{\theta_N^2}-\frac{2\theta_{i}^2}{\theta_N^2}\qquad{\color{gray}\rhd\,\textup{using Lemma~\ref{lem:hstepsize}}}\\
&=\frac{2\theta_{i+1}^2}{\theta_N^2}\cdot\frac{\theta_{i+1}+2\theta_{i}-1}{\theta_{i+1}}-\frac{2\theta_{i}^2}{\theta_N^2}
=\frac{2\theta_{i+1}^2+4\theta_{i+1}\theta_{i}-2\theta_{i+1}}{\theta_N^2}-\frac{2\theta_{i}^2}{\theta_N^2}=\frac{4\theta_i\theta_{i+1}}{\theta_N^2},
\qquad{\color{gray}\rhd\,\textup{using \eqref{eq:recursivetheta1} of Lemma~\ref{lem:propertyoftheta}}}\end{align*}
and when $i=N-1$,
\begin{align*}
&\lambda_{\star,N}\alpha_{N,N-1}+\lambda_{N-1,N}h_{N,N-1}-\lambda_{N-1,N}
=\left(1+\frac{2\theta_{N-1}-1}{\theta_{N}}\right)\frac{\theta_{N}+2\theta_{N-1}^2}{\theta_N^2}-\frac{2\theta_{N-1}^2}{\theta_N^2}\qquad{\color{gray}\rhd\,\textup{using Lemma~\ref{lem:hstepsize}}}\\
&=\frac{\theta_N^2}{\theta_N^2}\cdot\frac{\theta_{N}+2\theta_{N-1}-1}{\theta_{N}}-\frac{2\theta_{N-1}^2}{\theta_N^2}
=\frac{\theta_{N}^2+2\theta_{N-1}\theta_N-\theta_N-2\theta_{N-1}^2}{\theta_N^2}=\frac{2\theta_{N-1}\theta_{N}}{\theta_N^2}.
\qquad{\color{gray}\rhd\,\textup{using \eqref{eq:recursivetheta2} of Lemma~\ref{lem:propertyoftheta}}}\end{align*}
When $i\in[0:N-1]$ and $j\in[0:i-2]$,
\begin{align}
&\lambda_{\star,i}\alpha_{i,j}+\lambda_{i-1,i}h_{i,j}=\lambda_{\star,i}\alpha_{i,j}+(\alpha_{i,j}-\alpha_{i-1,j})\lambda_{i-1,i}\qquad{\color{gray}\rhd\,\textup{definition of }h_{i,j}}\nonumber\\&=\left(\frac{2\theta_{i}}{\theta_N^2}+\frac{2\theta_{i-1}^2}{\theta_N^2}\right)\alpha_{i,j}-\frac{2\alpha_{i-1,j}\theta_{i-1}^2}{\theta_N^2}=\frac{2\alpha_{i,j}\theta_{i}^2}{\theta_N^2}-\frac{2\alpha_{i-1,j}\theta_{i-1}^2}{\theta_N^2}\qquad{\color{gray}\rhd\,\textup{using \eqref{eq:recursivetheta1} of Lemma~\ref{lem:propertyoftheta}}}\nonumber\\
&=\frac{1}{\theta_N^2}\left(\left(2\alpha_{i-1,j}+\frac{4\theta_{j}}{\theta_{i}}-\frac{2\alpha_{i-1,j}}{\theta_{i}}\right)\theta_{i}^2-2\alpha_{i-1,j}\theta_{i-1}^2\right)\qquad{\color{gray}\rhd\,\alpha_{i,j}=\alpha_{i-1,j}+\frac{2\theta_j}{\theta_{i}}-\frac{1}{\theta_i}\alpha_{i-1,j}}\nonumber\\
&=\frac{1}{\theta_N^2}\left(2\alpha_{i-1,j}\left(\theta_{i}^2-\theta_{i}-\theta_{i-1}^2\right)+4\theta_{i}\theta_{j}\right)=\frac{4\theta_i\theta_j}{\theta_N^2}.\qquad{\color{gray}\rhd\,\textup{using \eqref{eq:recursivetheta1} of Lemma~\ref{lem:propertyoftheta}}}
\label{eq:termdcalculation2}\end{align}
Similarly for $i=N$ and $j\in[0:N-2]$,
\begin{align*}
&\lambda_{\star,N}\alpha_{N,j}+\lambda_{N-1,N}h_{N,j}=\lambda_{\star,N}\alpha_{N,j}+(\alpha_{N,j}-\alpha_{N-1,j})\lambda_{N-1,N}\qquad{\color{gray}\rhd\,\textup{definition of }h_{N,j}}\\&=\left(\frac{\theta_N}{\theta_N^2}+\frac{2\theta_{N-1}^2}{\theta_N^2}\right)\alpha_{N,j}-\frac{2\alpha_{N-1,j}\theta_{N-1}^2}{\theta_N^2}=\frac{\alpha_{N,j}\theta_{N}^2}{\theta_N^2}-\frac{2\alpha_{N-1,j}\theta_{N-1}^2}{\theta_N^2}\qquad{\color{gray}\rhd\,\textup{using \eqref{eq:recursivetheta2} of Lemma~\ref{lem:propertyoftheta}}}\\
&=\frac{1}{\theta_N^2}\left(\left(\alpha_{N-1,j}+\frac{2\theta_{j}}{\theta_{N}}-\frac{\alpha_{N-1,j}}{\theta_{N}}\right)\theta_{N}^2-2\alpha_{N-1,j}\theta_{N-1}^2\right)\qquad{\color{gray}\rhd\,\alpha_{N,j}=\alpha_{N-1,j}+\frac{2\theta_j}{\theta_{N}}-\frac{1}{\theta_N}\alpha_{N-1,j}}\\
&=\frac{1}{\theta_N^2}\left(\alpha_{N-1,j}\left(\theta_{N}^2-\theta_{N}-2\theta_{N-1}^2\right)+2\theta_{N}\theta_{j}\right)=\frac{2\theta_N\theta_j}{\theta_N^2}.\qquad{\color{gray}\rhd\,\textup{using \eqref{eq:recursivetheta2} of Lemma~\ref{lem:propertyoftheta}}}
\end{align*}
Therefore, \eqref{eq:termdcalculation1} can be arranged as 
\begin{align}
\texttt{term\_d}&=\sum_{0\leq j< i-1\leq N}\frac{\lambda_{\star,i}\alpha_{i,j}+\lambda_{i-1,i}h_{i,j}}{L}\langle f_i', f_j'\rangle+\sum_{i=0}^{N-1}\frac{\lambda_{\star,i}\alpha_{i+1,i}+\lambda_{i,i+1}h_{i+1,i}-\lambda_{i,i+1}}{L}\langle f_i', f_{i+1}'\rangle\nonumber\\
&=\sum_{0\leq j< i-1<N-1}\frac{4\theta_i\theta_j}{L\theta_N^2}\langle f_i', f_j'\rangle+\sum_{j=0}^{N-2}\frac{2\theta_j}{L\theta_N}\langle f_N', f_j'\rangle + \sum_{i=0}^{N-2}\frac{4\theta_i\theta_{i+1}}{L\theta_N^2}\langle f_i', f_{i+1}' \nonumber\rangle+\frac{2\theta_{N-1}\theta_{N}}{L\theta_N^2}\langle f_{N-1}', f_{N}' \nonumber\rangle\\
&=\sum_{0\leq j\leq i-1<N-1}\frac{4\theta_i\theta_j}{L\theta_N^2}\langle f_i', f_j'\rangle+\sum_{j=0}^{N-1}\frac{2\theta_j}{L\theta_N}\langle f_N', f_j'\rangle.
\label{eq:termd}\end{align}
Similarly, \texttt{term\_e} can be simplified as,
\begin{align}
&\texttt{term\_e}=\sum_{1\leq j\leq i\leq N}\frac{\lambda_{\star,i}\alpha_{i,j-1}+\lambda_{i-1,i}h_{i,j-1}}{L}\langle f_i', h_j'\rangle\nonumber\\
&=\sum_{1\leq j<i\leq N}\frac{\lambda_{\star,i}\alpha_{i,j-1}+\lambda_{i-1,i}h_{i,j-1}}{L}\langle f_i', h_j'\rangle+\sum_{i=1}^N\frac{\lambda_{\star,i}\alpha_{i,i-1}+\lambda_{i-1,i}h_{i,i-1}}{L}\langle f_i', h_i'\rangle\nonumber\\
&=\sum_{1\leq j<i< N}\frac{\lambda_{\star,i}\alpha_{i,j-1}+\lambda_{i-1,i}h_{i,j-1}}{L}\langle f_i', h_j'\rangle+\sum_{j=0}^{N-1}\frac{2\theta_{j-1}}{L\theta_N}\langle f_N',h_j'
\rangle+\sum_{i=1}^N\frac{\lambda_{\star,i}\alpha_{i,i-1}+\lambda_{i-1,i}h_{i,i-1}}{L}\langle f_i', h_i'\rangle\nonumber\\
&=\sum_{1\leq j<i< N}\frac{4\theta_i\theta_{j-1}}{L\theta_N^2}\langle f_i', h_j'\rangle+\sum_{j=0}^{N-1}\frac{2\theta_{j-1}}{L\theta_N}\langle f_N',h_j'
\rangle+\sum_{i=1}^N\frac{\lambda_{\star,i}\alpha_{i,i-1}+\lambda_{i-1,i}h_{i,i-1}}{L}\langle f_i', h_i'\rangle.
\qquad{\color{gray}\rhd\,\textup{using \eqref{eq:termdcalculation2}}}\label{eq:terme}\end{align}
Thus we have the desired result by combining \eqref{eq:termc}, \eqref{eq:termd}, and \eqref{eq:terme}.
\end{proof}
The following lemma proves \eqref{eq:cvx-expansion}.
\begin{lem}\label{lem:cvx-expansion}
The following equality \eqref{eq:cvx-expansion} holds.
\begin{align}
 &\sum_{i,j\in\{\star,1,\ldots,N\}}\tau_{i,j}\langle h_j',y_i-y_j\rangle=  -\sum_{i=1}^N\tau_{\star,i}\langle h_i', x_0 \rangle+\sum_{0\leq j<i\leq N}\frac{4\tilde{\theta}_{i-1}\theta_{j}}{L\theta_N^2}\langle h_i', f_j' \rangle -\sum_{i=1}^{N-1}\frac{\tau_{i+1,i}\gamma_i}{L}\langle h_i',f_i'\rangle \nonumber \\
 &\sum_{1\leq j <i-1\leq N-1}\frac{4\tilde{\theta}_{i-1}\theta_{j-1}}{L\theta_N^2}\langle h_i', h_j' \rangle +\sum_{i=1}^N\frac{4\tilde{\theta}_{i-1}\theta_{i-1}}{L\theta_N^2}\|h_i'\|^2+\sum_{i=1}^{N-1}\frac{4\theta_{i-1}\tilde{\theta}_i-2\theta_{i-1}^2}{L\theta_N^2}\langle h_i', h_{i+1}' \rangle.
\tag{\ref{eq:cvx-expansion}}\end{align}
\end{lem}
\begin{proof}
    Left-hand side equals
    \begin{align*}
     \sum_{i=1}^N\tau_{\star,i}\langle h_i',-y_i\rangle + 
     \sum_{1\leq i<j\leq N}\tau_{i,j}\langle h_j',y_i-y_j\rangle+ 
     \sum_{i=1}^{N-1}\tau_{i+1,i}\langle h_{i}',y_{i+1}-y_{i}\rangle.
    \end{align*}
    The first term can be rearranged as 
    \begin{align}
     &\sum_{i=1}^N\tau_{\star,i}\langle h_i',-y_i\rangle=\sum_{i=1}^N\tau_{\star,i}\left\langle h_i',-x_0+\sum_{j=0}^{i-1}\frac{\gamma_j}{L}f_j'+\sum_{j=0}^{i-1}\frac{\gamma_j}{L}h_{j+1}'\right\rangle\qquad{\color{gray}\rhd\,\textup{substituting \eqref{eq:equivalentFSFOM}}}\nonumber\\
     &=-\sum_{i=1}^N\tau_{\star,i}\langle h_i', x_0 \rangle  +\sum_{0\leq j<i\leq N}\frac{\tau_{\star,i}\gamma_j}{L}\langle h_i', f_j' \rangle +\sum_{1\leq j\leq i\leq N}\frac{\tau_{\star,i}\gamma_{j-1}}{L}\langle h_i', h_j' \label{eq:matrixforh1}\rangle .
    \end{align}
    For the second term,
    \begin{align}
      &\sum_{1\leq i<j\leq N}\tau_{i,j}\langle h_j',y_i-y_j\rangle=\sum_{1\leq i<j\leq N}\tau_{i,j}\left\langle h_j',\sum_{k=i}^{j-1}\frac{\gamma_k}{L}f_k'+\sum_{k=i}^{j-1}\frac{\gamma_k}{L}h_{k+1}'\right\rangle, \qquad{\color{gray}\rhd\,\textup{substituting \eqref{eq:equivalentFSFOM}}}\nonumber\\
      &=\sum_{1\leq i\leq k<j\leq N}\frac{\tau_{i,j}\gamma_k}{L}\langle h_j',f_k'\rangle+\sum_{1\leq i< k\leq j\leq N}\frac{\tau_{i,j}\gamma_{k-1}}{L}\langle h_j',h_{k}'\rangle\nonumber\\
      &=\sum_{1\leq k< j\leq N}\frac{\gamma_k}{L}\sum_{i=1}^{k}\tau_{i,j}\langle h_j',f_k'\rangle+\sum_{1<k\leq j\leq N}\frac{\gamma_{k-1}}{L}\sum_{i=1}^{k-1}\tau_{i,j}\langle h_j',h_k'\rangle\nonumber\\
       &=\sum_{1\leq j< i\leq N}\frac{\gamma_j}{L}\sum_{k=1}^{j}\tau_{k,i}\langle h_i',f_j'\rangle+\sum_{1<j\leq i\leq N}\frac{\gamma_{j-1}}{L}\sum_{k=1}^{j-1}\tau_{k,i}\langle h_i',h_j'\rangle. \qquad{\color{gray}\rhd\,\textup{renaming the indices}}\label{eq:matrixforh2}
    \end{align}
    For the last term,
    \begin{align}
     &\sum_{i=1}^{N-1}\tau_{i+1,i}\langle h_{i}',y_{i+1}-y_{i}\rangle=\sum_{i=1}^{N-1}\tau_{i+1,i}\left\langle h_i',-\frac{\gamma_i}{L}f_i'-\frac{\gamma_i}{L}h_{i+1}'\right\rangle\qquad{\color{gray}\rhd\,\textup{substituting \eqref{eq:equivalentFSFOM}}}\nonumber \\
     &=-\sum_{i=1}^{N-1}\frac{\tau_{i+1,i}\gamma_i}{L}\langle h_i',f_i'\rangle-\sum_{i=1}^{N-1}\frac{\tau_{i+1,i}\gamma_i}{L}\langle h_i', h_{i+1}'\rangle. \label{eq:matrixforh3}
    \end{align}
Adding \eqref{eq:matrixforh1}, \eqref{eq:matrixforh2}, and \eqref{eq:matrixforh3} gives
\begin{align}
 &\sum_{i,j\in\{\star,1,\ldots,N\}}\tau_{i,j}\langle h_j',y_i-y_j\rangle=  -\sum_{i=1}^N\tau_{\star,i}\langle h_i', x_0 \rangle+\sum_{0\leq j<i\leq N}\left(\frac{\tau_{\star,i}\gamma_j}{L}+\frac{\gamma_j}{L}\sum_{k=1}^{j}\tau_{k,i}\right)\langle h_i', f_j' \rangle\nonumber\\
 &-\sum_{i=1}^{N-1}\frac{\tau_{i+1,i}\gamma_i}{L}\langle h_i',f_i'\rangle+\underbrace{\sum_{1\leq j\leq i\leq N}\left(\frac{\tau_{\star,i}\gamma_{j-1}}{L}+\frac{\gamma_{j-1}}{L}\sum_{k=1}^{j-1}\tau_{k,i}\right)\langle h_i', h_j' \rangle-\sum_{i=1}^{N-1}\frac{\tau_{i+1,i}\gamma_i}{L}\langle h_i', h_{i+1}'\rangle}_{\texttt{term\_f}}.
 \label{eq:cvx-expansioncalculation}
\end{align}
For \texttt{term\_f},
\begin{align}
 \texttt{term\_f}&=\sum_{1\leq j< i\leq N}\left(\frac{\tau_{\star,i}\gamma_{j-1}}{L}+\frac{\gamma_{j-1}}{L}\sum_{k=1}^{j-1}\tau_{k,i}\right)\langle h_i', h_j' \rangle\nonumber\\
 &+\sum_{i=1}^N\left(\frac{\tau_{\star,i}\gamma_{i-1}}{L}+\frac{\gamma_{i-1}}{L}\sum_{k=1}^{i-1}\tau_{k,i}\right)\|h_i'\|^2-\sum_{i=1}^{N-1}\frac{\tau_{i+1,i}\gamma_i}{L}\langle h_i', h_{i+1}'\rangle\rangle\qquad{\color{gray}\rhd\,\textup{separating }\| h_i'\|^2\textup{ term}}\nonumber \\
 &=\sum_{1\leq j< i-1\leq N-1}\left(\frac{\tau_{\star,i}\gamma_{j-1}}{L}+\frac{\gamma_{j-1}}{L}\sum_{k=1}^{j-1}\tau_{k,i}\right)\langle h_i', h_j' \rangle+\sum_{i=1}^N\left(\frac{\tau_{\star,i}\gamma_{i-1}}{L}+\frac{\gamma_{i-1}}{L}\sum_{k=1}^{i-1}\tau_{k,i}\right)\|h_i'\|^2\nonumber\\
 &+\sum_{i=0}^{N-1}\left(\frac{\tau_{\star,i+1}\gamma_{i-1}}{L}+\frac{\gamma_{i-1}}{L}\sum_{k=1}^{i-1}\tau_{k,i+1}\right)\langle h_i', h_{i+1}' \rangle-\sum_{i=1}^{N-1}\frac{\tau_{i+1,i}\gamma_i}{L}\langle h_i', h_{i+1}'\rangle\qquad{\color{gray}\rhd\,\textup{separating }\langle h_i',h_{i+1}'\rangle\textup{ term}}\nonumber \\
 &=\sum_{1\leq j< i-1\leq N-1}\left(\frac{\tau_{\star,i}\gamma_{j-1}}{L}+\frac{\gamma_{j-1}}{L}\sum_{k=1}^{j-1}\tau_{k,i}\right)\langle h_i', h_j' \rangle+\sum_{i=1}^N\left(\frac{\tau_{\star,i}\gamma_{i-1}}{L}+\frac{\gamma_{i-1}}{L}\sum_{k=1}^{i-1}\tau_{k,i}\right)\|h_i'\|^2\nonumber\\
 &+\sum_{i=0}^{N-1}\left(\frac{\tau_{\star,i+1}\gamma_{i-1}}{L}+\frac{\gamma_{i-1}}{L}\sum_{k=1}^{i-1}\tau_{k,i+1}-\frac{\tau_{i+1,i}\gamma_i}{L}\right)\langle h_i', h_{i+1}' \rangle.\label{eq:termf}
\end{align}
Now we simplify the coefficients:
\begin{align}
&\tau_{\star,i}\gamma_{j}+\gamma_{j}\sum_{k=1}^{j}\tau_{k,i}=\frac{2\tilde{\theta}_{i-1}\gamma_j}{\theta_N^2-1}+2\tilde{\theta}_{i-1}\gamma_{j}\sum_{k=1}^{j}\left(\frac{1}{\theta_N^2-2\theta_k^2+\theta_k}-\frac{1}{\theta_N^2-2\theta_{k-1}^2+\theta_{k-1}}\right)\nonumber\\
    &=\frac{2\gamma_{j}\tilde{\theta}_{i-1}}{\theta_N^2-1}+2\gamma_{j}\tilde{\theta}_{i-1}\left( \frac{1}{\theta_N^2-2\theta_{j}^2+\theta_{j}}-\frac{1}{\theta_N^2-2\theta_{0}^2+\theta_{0}} \right)=\frac{2\gamma_{j}\tilde{\theta}_{i-1}}{\theta_N^2-2\theta_{j}^2+\theta_{j}}=\frac{4\tilde{\theta}_{i-1}\theta_{j}}{\theta_N^2}\qquad{\color{gray}\rhd\,\textup{telescopic sum}
 \label{eq:termfcoefficient1} }  \end{align}
and
\begin{align}
    &\tau_{\star,i+1}\gamma_{i-1}+\gamma_{i-1}\sum_{k=1}^{i-1}\tau_{k,i+1}-\tau_{i+1,i}\gamma_i \nonumber \\
    &=\frac{2\gamma_{i-1}\tilde{\theta}_{i}}{\theta_N^2-1}+2\gamma_{i-1}\tilde{\theta}_{i}\sum_{k=1}^{i-1}\left(\frac{1}{\theta_N^2-2\theta_k^2+\theta_k}-\frac{1}{\theta_N^2-2\theta_{k-1}^2+\theta_{k-1}}\right)-\frac{\gamma_{i}(\theta_i-1)}{\theta_N^2-2\theta_i^2+\theta_i}\nonumber\\
    &=\frac{2\gamma_{i-1}\tilde{\theta}_{i}}{\theta_N^2-1}+2\gamma_{i-1}\tilde{\theta}_{i}\left(\frac{1}{\theta_N^2-2\theta_{i-1}^2+\theta_{i-1}}-\frac{1}{\theta_N^2-2\theta_{0}^2+\theta_{0}}\right)-\frac{\gamma_{i}(\theta_i-1)}{\theta_N^2-2\theta_i^2+\theta_i}\qquad{\color{gray}\rhd\,\textup{telescopic sum}}\nonumber\\
    &=\frac{2\gamma_{i-1}\tilde{\theta}_{i}}{\theta_N^2-2\theta_{i-1}^2+\theta_{i-1}}-\frac{\gamma_{i}(\theta_i-1)}{\theta_N^2-2\theta_i^2+\theta_i}=\frac{4\theta_{i-1}\tilde{\theta}_i}{\theta_N^2}-\frac{2\theta_i(\theta_i-1)}{\theta_N^2}\qquad{\color{gray}\rhd\,\textup{definition of }\gamma_{i-1}\textup{ and } \gamma_{i}}\nonumber\\
    &=\frac{4\theta_{i-1}\tilde{\theta}_i-2\theta_{i-1}^2}{\theta_N^2}.\qquad{\color{gray}\rhd\,\textup{using \eqref{eq:recursivetheta1} of Lemma~\ref{lem:propertyoftheta}} }\label{eq:termfcoefficient2}
\end{align}
By plugging \eqref{eq:termfcoefficient1} and \eqref{eq:termfcoefficient2} into \eqref{eq:cvx-expansioncalculation} and \eqref{eq:termf}, we get 
\begin{align*}
 &\sum_{i,j\in\{\star,1,\ldots,N\}}\tau_{i,j}\langle h_j',y_i-y_j\rangle=  -\sum_{i=1}^N\tau_{\star,i}\langle h_i', x_0 \rangle+\sum_{0\leq j<i\leq N}\frac{4\tilde{\theta}_{i-1}\theta_j}{L\theta_N^2}\langle h_i', f_j' \rangle\nonumber\\
 &-\sum_{i=1}^{N-1}\frac{\tau_{i+1,i}\gamma_i}{L}\langle h_i',f_i'\rangle+\sum_{1\leq j< i-1 \leq N-1}\frac{4\tilde{\theta}_{i-1}\theta_{j-1}}{L\theta_N^2}\langle h_i', h_j' \rangle+\sum_{i=1}^N\frac{4\tilde{\theta}_{i-1}\theta_{i-1}}{L\theta_N^2}\|h_i'\|^2+\sum_{i=1}^{N-1}\frac{4\theta_{i-1}\tilde{\theta}_i-2\theta_{i-1}^2}{L\theta_N^2}\langle h_i', h_{i+1}' \rangle.
\end{align*}
\end{proof}
\subsection{Proof of $(*)$ of equation \eqref{eq:weighted_ineq_to_proof}}\label{s:f-proofofstar}
We combine Lemma~\ref{lem:functionvalueforf}, \ref{lem:functionvalueforh}, \ref{lem:cocoerc-expansion}, and \ref{lem:cvx-expansion} to get the following lemma.
\begin{lem}
Take $G\in\mathbb{S}_{+}^{2N+3}$ as in \eqref{eq:def_of_G} and $Z\in\mathbb{S}_{+}^{2N+3}$ as in \eqref{eq:choice_of_Z-1}.
Then the equality of \eqref{eq:weighted_ineq_to_proof} holds.
\end{lem}
\begin{proof}
By Lemma~\ref{lem:functionvalueforf}, \ref{lem:functionvalueforh}, \ref{lem:cocoerc-expansion}, and \ref{lem:cvx-expansion}, left-most side of \eqref{eq:weighted_ineq_to_proof} equals
 \begin{align}
        &\frac{L}{2(\theta_N^2-1)}\|x_0\|^2+\frac{1}{2L}\|f_\star'\|^2-\sum_{i=0}^N\lambda_{\star,i}\langle f_i', x_0\rangle-\sum_{i=0}^N\frac{\lambda_{\star,i}}{L}\langle f_\star', f_i \rangle'-\sum_{i=1}^N\tau_{\star,i}\langle h_i', x_0 \rangle\nonumber\\
        &+\sum_{i=0}^{N-1}\frac{2\theta_i^2}{L\theta_N^2}\|f_i'\|^2+\frac{1}{2L}\|f_N'\|^2+\sum_{0\leq j\leq i-1< N-1}\frac{4\theta_i\theta_j}{L\theta_N^2}\langle f_i', f_j'\rangle+\sum_{j=0}^{N-1}\frac{2\theta_j}{L\theta_N}\langle f_N', f_j'\rangle\nonumber\\
        &+\sum_{1\leq j< i< N}\frac{4\theta_i\theta_{j-1}}{L\theta_N^2}\langle f_i', h_j'\rangle\ +\sum_{j=1}^{N-1}\frac{2\theta_{j-1}}{L\theta_N}\langle f_N', h_j'\rangle  \underline{+\sum_{i=1}^N\frac{\lambda_{i-1,i}\alpha_{i,i-1}+\lambda_{\star,i}\alpha_{i,i-1}}{L}\langle f_i', h_i'\rangle}\nonumber \\
&+\sum_{0\leq j<i\leq N}\frac{4\tilde{\theta}_{i-1}\theta_{j}}{L\theta_N^2}\langle h_i', f_j' \rangle\underline{-\sum_{i=1}^{N-1}\frac{\tau_{i+1,i}\gamma_i}{L}\langle h_i',f_i'\rangle}\nonumber\\
&+\sum_{1\leq j< i-1 \leq N-1}\frac{4\tilde{\theta}_{i-1}\theta_{j-1}}{L\theta_N^2}\langle h_i', h_j' \rangle+\sum_{i=1}^N\frac{4\tilde{\theta}_{i-1}\theta_{i-1}}{L\theta_N^2}\|h_i'\|^2+\sum_{i=1}^{N-1}\frac{4\theta_{i-1}\tilde{\theta}_i-2\theta_{i-1}^2}{L\theta_N^2}\langle h_i', h_{i+1}' \rangle.\label{eq:Zlemmasum}
    \end{align}
We collect underlined terms that has $\langle f_i', h_i^\prime \rangle$:
    \begin{align}
     &\sum_{i=1}^N\frac{\lambda_{i-1,i}\alpha_{i,i-1}+\lambda_{\star,i}\alpha_{i,i-1}}{L}\langle f_i', h_i^\prime\rangle-\sum_{i=1}^{N-1}\frac{\tau_{i+1,i}\gamma_i}{L}\langle h_i^\prime,f_i'\rangle  \nonumber \\
     &=  \sum_{i=1}^{N-1}\frac{\lambda_{i-1,i}\alpha_{i,i-1}+\lambda_{\star,i}\alpha_{i,i-1}-\tau_{i+1,i}\gamma_i}{L}\langle f_i', h_i^\prime\rangle+\frac{\lambda_{N-1,N}\alpha_{N,N-1}+\lambda_{\star,N}\alpha_{N,N-1}}{L}\langle f_N', h_N^\prime\rangle.
    \label{eq:Zlemmasumcalculation}\end{align}
Then we simplify the coefficients, for $i\in[1:N-1]$,
\begin{align}
&\lambda_{i-1,i}\alpha_{i,i-1}+\lambda_{\star,i}\alpha_{i,i-1}-\tau_{i+1,i}\gamma_i=\alpha_{i,i-1}\left(\lambda_{\star,i}+\lambda_{i-1,i}\right)-\tau_{i+1,i}\gamma_{i}\nonumber\\
    &=\left(\frac{2\theta_{i}}{\theta_N^2}+\frac{2\theta_{i-1}^2}{\theta_N^2}\right)\left(1+\frac{2\theta_{i-1}-1}{\theta_{i}}\right)-\frac{\gamma_i\left(\theta_{i}-1\right)}{\theta_N^2-2\theta_{i}^2+\theta_{i}}\nonumber\\
    &=\frac{2\theta_{i}^2}{\theta_N^2}\cdot\frac{\theta_{i}+2\theta_{i-1}-1}{\theta_{i}}-\frac{2\theta_{i}(\theta_{i}-1)}{\theta_N^2}\qquad{\color{gray}\rhd\,\textup{using \eqref{eq:recursivetheta1} of Lemma~\ref{lem:propertyoftheta}}, \textup{definition of } }\gamma_i\nonumber\\
    &=\frac{2\theta_{i}(\theta_{i}+2\theta_{i-1}-1-\theta_i+1)}{\theta_N^2}=\frac{4\theta_{i-1}\theta_{i}}{\theta_N^2}, 
\label{eq:Zlemmacoefficient1}\end{align}
and for $i=N$,
\begin{align}
&\lambda_{N-1,N}\alpha_{N,N-1}+\lambda_{\star,N}\alpha_{N,N-1} =
\left(\frac{1}{\theta_N}+\frac{2\theta_{N-1}^2}{\theta_N^2}\right)\left(1+\frac{2\theta_{N-1}-1}{\theta_N}\right)\nonumber\\
&=\frac{\theta_N+2\theta_{N-1}^2}{\theta_N^2}\cdot\frac{\theta_N+2\theta_{N-1}-1}{\theta_N}\nonumber\\
&=\frac{\theta_N+2\theta_{N-1}-1}{\theta_N}=\frac{2\tilde{\theta}_{N-1}}{\theta_N}. \qquad{\color{gray}\rhd\,\textup{using \eqref{eq:recursivetheta2} of Lemma~\ref{lem:propertyoftheta}} , \textup{definition of } \tilde{\theta}_{N-1}
\label{eq:Zlemmacoefficient2}}\end{align}
Substituting \eqref{eq:Zlemmacoefficient1} and \eqref{eq:Zlemmacoefficient2} in \eqref{eq:Zlemmasumcalculation} gives
\begin{align}
    &\sum_{i=1}^N\frac{\lambda_{i-1,i}\alpha_{i,i-1}+\lambda_{\star,i}\alpha_{i,i-1}}{L}\langle f_i', h_i^\prime\rangle-\sum_{i=1}^{N-1}\frac{\tau_{i+1,i}\gamma_i}{L}\langle h_i^\prime,f_i'\rangle\nonumber\\
    &=\sum_{i=1}^{N-1}\frac{4\theta_{i-1}\theta_i}{L\theta_N^2}\langle h_i^\prime,f_i'\rangle+\frac{\tilde{2\theta}_{N-1}}{L\theta_N}\langle h_N^\prime,f_N'\rangle.
\label{eq:Zlemmasumcaclulation2}\end{align}
Therefore, \eqref{eq:Zlemmasum} reduces to 
\begin{align}
        &\frac{L}{2(\theta_N^2-1)}\|x_0\|^2+\frac{1}{2L}\|f_\star'\|^2-\sum_{i=0}^N\lambda_{\star,i}\langle f_i', x_0\rangle-\sum_{i=0}^N\frac{\lambda_{\star,i}}{L}\langle f_\star', f_i \rangle'-\sum_{i=1}^N\tau_{\star,i}\langle h_i', x_0 \rangle\nonumber\\
        &+\sum_{i=0}^{N-1}\frac{2\theta_i^2}{L\theta_N^2}\|f_i'\|^2+\frac{1}{2L}\|f_N'\|^2+\sum_{0\leq j\leq i-1< N-1}\frac{4\theta_i\theta_j}{L\theta_N^2}\langle f_i', f_j'\rangle+\sum_{j=0}^{N-1}\frac{2\theta_j}{L\theta_N}\langle f_N', f_j'\rangle\nonumber\\
        &+\sum_{1\leq j< i< N}\frac{4\theta_i\theta_{j-1}}{L\theta_N^2}\langle f_i', h_j'\rangle\ \underline{+\sum_{j=1}^{N-1}\frac{2\theta_{j-1}}{L\theta_N}\langle f_N', h_j'\rangle}  +\sum_{i=1}^{N-1}\frac{4\theta_{i-1}\theta_i}{L\theta_N^2}\langle f_i', h_i'\rangle\nonumber \\
&+\sum_{0\leq j<i\leq N}\frac{4\tilde{\theta}_{i-1}\theta_{j}}{L\theta_N^2}\langle h_i', f_j' \rangle+\underline{\frac{2\tilde{\theta}_{N-1}}{L\theta_N}\langle h_N',f_N'\rangle}\nonumber\\
&+\sum_{1\leq j< i-1 \leq N-1}\frac{4\tilde{\theta}_{i-1}\theta_{j-1}}{L\theta_N^2}\langle h_i', h_j' \rangle+\sum_{i=1}^N\frac{4\tilde{\theta}_{i-1}\theta_{i-1}}{L\theta_N^2}\|h_i'\|^2+\sum_{i=1}^{N-1}\frac{4\theta_{i-1}\tilde{\theta}_i-2\theta_{i-1}^2}{L\theta_N^2}\langle h_i', h_{i+1}' \rangle.\nonumber
    \end{align}
Again, merge underlined terms to get: 
    \begin{align}
        &=\frac{L}{2(\theta_N^2-1)}\|x_0\|^2+\frac{1}{2L}\|f_\star'\|^2\qquad{\color{gray}\rhd\, A}\nonumber\\
        &-\sum_{i=0}^N\lambda_{\star,i}\langle f_i', x_0\rangle-\sum_{i=0}^N\frac{\lambda_{\star,i}}{L}\langle f_\star', f_i \rangle'-\sum_{i=1}^N\tau_{\star,i}\langle h_i', x_0 \rangle\qquad{\color{gray}\rhd\, B\textup{ and }C }\nonumber\\
        &+\sum_{i=0}^{N-1}\frac{2\theta_i^2}{L\theta_N^2}\|f_i'\|^2+\frac{1}{2L}\|f_N'\|^2+\sum_{0\leq j\leq i-1< N-1}\frac{4\theta_i\theta_j}{L\theta_N^2}\langle f_i', f_j'\rangle+\sum_{j=0}^{N-1}\frac{2\theta_j}{L\theta_N}\langle f_N', f_j'\rangle\qquad{\color{gray}\rhd\, D}\nonumber\\
        &+\sum_{1\leq j< i\leq N}\frac{4\theta_i\theta_{j-1}}{L\theta_N^2}\langle f_i', h_j'\rangle\ +\sum_{j=1}^{N}\frac{2\tilde{\theta}_{j-1}}{L\theta_N}\langle f_N', h_j'\rangle  +\sum_{i=1}^{N-1}\frac{4\theta_{i-1}\theta_i}{L\theta_N^2}\langle f_i', h_i'\rangle+\sum_{0\leq j<i\leq N}\frac{4\tilde{\theta}_{i-1}\theta_{j}}{L\theta_N^2}\langle h_i', f_j' \rangle\qquad{\color{gray}\rhd\, E}\nonumber\\
&+\sum_{1\leq j< i-1 \leq N-1}\frac{4\tilde{\theta}_{i-1}\theta_{j-1}}{L\theta_N^2}\langle h_i', h_j' \rangle+\sum_{i=1}^N\frac{4\tilde{\theta}_{i-1}\theta_{i-1}}{L\theta_N^2}\|h_i'\|^2+\sum_{i=1}^{N-1}\frac{4\theta_{i-1}\tilde{\theta}_i-2\theta_{i-1}^2}{L\theta_N^2}\langle h_i', h_{i+1}' \rangle.\qquad{\color{gray}\rhd\, F}\label{eq:Zlemmasum2}
    \end{align}
Meanwhile, $\mathbf{tr}(ZG)$ can be expanded as:
\begin{align}
\mathbf{tr}(ZG)&=\sum_{i,j\in[1:2N-3]}Z_{i,j}G_{i,j}=\sum_{i=1}^{2N-3}Z_{i,i}G_{i,i}+\sum_{1\leq j<i\leq2N-3}2Z_{i,j}G_{i,j} \nonumber \\
&=A_{1,1}\|x_0\|^2+A_{2,2}\|f_\star'\|^2+2A_{1,2}\langle x_0,f_\star'\rangle\nonumber\qquad{\color{gray}\rhd\, A}\\
&+\sum_{i=0}^{N}2B_{1,i+1}\langle x_0, f_{i}'\rangle+\sum_{i=0}^{N}2B_{2,i+1}\langle f_\star', f_{i}'\rangle+\sum_{i=1}^{N}2C_{1,i}\langle x_0, h_{i}'\rangle+\sum_{i=1}^{N}2C_{2,i}\langle f_\star', h_{i}'\rangle\qquad {\color{gray}\rhd\,B \textup{ and } C}\nonumber\\
&+\sum_{i=0}^{N}D_{i+1,i+1}\|f_{i}'\|^2+\sum_{0\leq j\le i-1< N-1}2D_{i+1,j+1}\langle f_{i}',f_{j}'\rangle+\sum_{j=0}^{N-1}2D_{N+1,j+1}\langle f_{N}',f_{j}'\rangle\qquad{\color{gray}\rhd\, D}\nonumber\\
&+\sum_{1\leq j< i\leq N}2E_{i+1,j}\langle f_i', h_j'\rangle\ +\sum_{j=1}^{N}2E_{N+1,j}\langle f_N', h_j'\rangle  +\sum_{i=1}^{N-1} 2E_{i+1,i}\langle f_i', h_i'\rangle+\sum_{0\leq j<i\leq N}2E_{j+1,i}\langle h_i', f_j' \rangle\qquad{\color{gray}\rhd\, E}\nonumber\\
&+\sum_{1\leq j< i-1 \leq N-1}2F_{i,j}\langle h_i', h_j' \rangle+\sum_{i=1}^N F_{i,i}\|h_i'\|^2+\sum_{i=1}^{N-1}2F_{i,i+1}\langle h_i', h_{i+1}'\rangle. \qquad{\color{gray}\rhd\, F} 
\label{eq:expandtrace}\end{align}
so we can solve for $Z_{i,j}$ by making term by term comparison with  \eqref{eq:Zlemmasum2} and \eqref{eq:expandtrace}. The correspondence is marked with $\rhd$ on each line of \eqref{eq:Zlemmasum2} and \eqref{eq:expandtrace}. Then, we can recover $Z$:
\begin{equation}
\begin{aligned} & A=\begin{cases}
A_{1,1}=\frac{L}{2(\theta_{N}^{2}-1)},\\
A_{1,2}=A_{2,1}=0,\\
A_{2,2}=\frac{1}{2L},
\end{cases},\quad B=\begin{cases}
B_{1,i}=-\frac{\theta_{i-1}}{\theta_{N}^{2}}\textup{ if }i\in[1:N],\\
B_{1,N+1}=-\frac{1}{2\theta_{N}},\\
B_{2,i}=-\frac{\theta_{i-1}}{L\theta_{N}^{2}}\textup{ if }i\in[1:N],\\
B_{2,N+1}=-\frac{1}{2L\theta_{N}},
\end{cases},\\
 & C=\begin{cases}
C_{1,i}=-\frac{\tilde{\theta}_{i-1}}{\theta_{N}^{2}-1}\textup{ if }i\in[1:N],\\
C_{2,i}=0\ \textup{ if }i\in[1:N],
\end{cases},\quad D=\begin{cases}
D_{i,j}=\frac{2\theta_{i-1}\theta_{j-1}}{L\theta_{N}^{2}}\textup{ if }i,j\in[1:N]\\
D_{N+1,i}=D_{i,N+1}=\frac{\theta_{i-1}}{L\theta_{N}}\textup{ if }i\in[1:N],\\
D_{N+1,N+1}=\frac{1}{2L},
\end{cases},\\
 & E=\begin{cases}
E_{i,j}=\frac{2\theta_{i-1}\tilde{\theta}_{j-1}}{L\theta_{N}^{2}}\ \textup{ if }i,j\in[1:N],\\
E_{N+1,i}=\frac{\tilde{\theta}_{i-1}}{L\theta_{N}}\textup{ if }i\in[1:N],
\end{cases},\quad F=\begin{cases}
F_{i,i}=\frac{2\theta_{i-1}\cdot2\tilde{\theta}_{i-1}}{L\theta_{N}^{2}}\textup{ if }i\in[1:N],\\
F_{i,i+1}=F_{i+1,i}=\frac{2\theta_{i-1}\tilde{\theta}_{i}-\theta_{i-1}^{2}}{L\theta_{N}^{2}}\textup{ if }i\in[1:N-1],\\
F_{i,j}=F_{j,i}=E_{i,j}=\frac{2\theta_{i-1}\tilde{\theta}_{j-1}}{L\theta_{N}^{2}}\textup{ if }i\in[1:N-2],|i-j|\geq2.
\end{cases}
\end{aligned}
\tag{\ref{eq:choice_of_Z-1}}
\end{equation}
    \end{proof}

\subsection{Relation between two proofs of Theorem~\ref{thm:OptISTA-rate}}

The dual variables $\lambda$, $\tau$, $\nu$ and $Z$ are related to the $\{\mathcal{F}_k\}_{k\in[-1:N]}$ $\{\mathcal{H}_k\}_{k\in[-1:N]}$ sequence in Section~\ref{s:c} in such a way that the sequence is a linear combination of interpolation inequalities weighted by a dual variable. The Lyapunov sequence $\mathcal{U}_k$ can be equivalently defined as 
\begin{equation*}
\begin{alignedat}{1} &\mathcal{U}_k =  \nu\|x_0-x_\star\|^2+\sum_{i,j\in\{\star,0,\ldots,k\}}\lambda_{i,j}\underbrace{\left(f_{j}-f_{i}+\langle f'_{j},x_{i}-x_{j}\rangle+\tfrac{1}{2L}\|f'_{i}-f'_{j}\|^{2}\right)}_{\le0\,\, (\textup{$L$-smoothness and convexity of \ensuremath{f}})}\\
 & \quad+\sum_{i,j\in\{\star,1,\ldots,k\}}\tau_{i,j}\underbrace{\left(h_{j}-h_{i}+\langle h_{j}',y_{i}-y_{j}\rangle\right).}_{\le0\,\,\textup{ (convexity of \ensuremath{h})}}
\end{alignedat}
\end{equation*}
 If one starts the analysis from here, then the dissipative property of $\mathcal{U}_k$ is straightforward from the smoothness and convex inequality of $f$ and $g$. However, the main source of difficulty is moved to showing that
 \[
\mathcal{U}_N -\left( f(x_N)+h(x_N)-f(x_\star)-h(x_\star) \right) =\textbf{tr}(ZG) \ge0,
\]
which heavily relies on positive semidefiniteness of the dual matrix variable $Z$.

The two proofs in Appendix~\ref{s:c} and Appendix~\ref{s:g} differ in that each defines the setting to make either the dissipative property or the positiveness relatively straightforward. In the proof presented in Appendix~\ref{s:c}, where positiveness is straightforward, the expression 
\begin{align*}
   & \mathcal{U}_N-\left(f(x_N)-f(x_\star)+h(y_N)-h(x_\star)\right)\\
   &=\frac{L}{2\theta_N^2}\left\|w_N-x_\star+\frac{1}{L}\nabla f(x_\star)+\frac{2\theta_{N-1}}{L} h'(y_N)-\frac{\theta_N}{L}\nabla f(x_N)-\frac{2\tilde{\theta}_{N-1}}{L}h'(y_{N})\right\|^2\\
    &+\frac{L}{2\theta_N^2(\theta_N^2-1)}\left\|x_0-x_\star-\frac{\theta_N^2-1}{L}\nabla f(x_\star)-\sum_{i=0}^{N-1}\frac{2\tilde{\theta}_i}{L} h'(y_{i+1})\right\|^2\\
    &+\sum_{i\neq j, i,j\in[1:N]}\frac{\tilde{\theta}_{i-1}\tilde{\theta}_{j-1}}{L\theta_N^2(\theta_N^2-1)}\|h'(y_i)-h'(y_j)\|^2+\sum_{i=1}^{N-1}\frac{\tilde{\theta}_{i-1}^2}{L\theta_N^2}\|h'(y_i)-h'(y_{i+1})\|^2,
\end{align*}
is readily made by the authors ourselves using the positive semidefiniteness of $Z$. In contrast, the proof in this section, where the dissipative property is made straightforward, explicitly demonstrates the positive semidefiniteness of $Z$ to the readers, relying on Cauchy’s interlacing theorem.
}}

\begin{acknowledgements}
UJ and EKR was supported by the Samsung Science and Technology Foundation (Project Number SSTF-BA2101-02). We also thank Hyunsik Chae for providing valuable feedback on the manuscript.
\end{acknowledgements}
\bibliographystyle{spmpsci}      
\bibliography{optista}   

\end{document}